\documentclass[11pt]{amsart}
\usepackage{amsmath}
\usepackage{amsxtra}
\usepackage{amscd}
\usepackage{amsthm}
\usepackage{amsfonts}
\usepackage{amssymb}
\usepackage{euscript}   
\usepackage{eepic}

\DeclareMathAlphabet{\eurm}{U}{eur}{m}{n}

\newtheorem{thm}{Theorem}[section]
\newtheorem{conj}{Conjecture}

\newtheorem{lem}[thm]{Lemma}

\theoremstyle{remark}
\newtheorem{remark}[thm]{Remark}

\theoremstyle{definition}

\numberwithin{equation}{section}

\newcommand{\bean}{\begin{eqnarray}}
\newcommand{\eean}{\end{eqnarray}}
\newcommand{\be}{\begin{displaymath}}
\newcommand{\ee}{\end{displaymath}}
\newcommand{\bea}{\begin{eqnarray*}}
\newcommand{\eea}{\end{eqnarray*}}

\newcommand{\thmref}[1]{Theorem~\ref{#1}}
\newcommand{\secref}[1]{Section~\ref{#1}}

\newcommand{\conjref}[1]{Conjecture~\ref{#1}}
\newcommand{\nc}{\newcommand}
\nc{\on}{\operatorname}
\nc{\ch}{\mbox{ch}}
\nc{\Z}{{\mathbb Z}}
\nc{\C}{{\mathbb C}}
\nc{\pone}{{\mathbb P}^1}
\nc{\pa}{\partial}
\nc{\F}{{\mathcal F}}
\nc{\arr}{\rightarrow}
\nc{\larr}{\longrightarrow}
\nc{\al}{\alpha}
\nc{\ri}{\rangle}
\nc{\lef}{\langle}
\nc{\W}{{\mathcal W}}
\nc{\la}{\lambda}
\nc{\ep}{\epsilon}
\nc{\su}{\widehat{{\mathfrak s}{\mathfrak l}}_2}
\nc{\sw}{{\mathfrak s}{\mathfrak l}}
\nc{\g}{{\mathfrak g}}
\nc{\h}{{\mathfrak h}}
\nc{\n}{{\mathfrak n}}
\nc{\N}{\widehat{\n}}
\nc{\G}{\widehat{\g}}
\nc{\De}{\Delta}
\nc{\gt}{\widetilde{\g}}
\nc{\Ga}{\Gamma}
\nc{\one}{{\mathbf 1}}
\nc{\z}{{\mathfrak Z}}
\nc{\La}{\Lambda}
\nc{\wt}{\widetilde}
\nc{\wh}{\widehat}
\nc{\cri}{_{\kappa_c}}
\nc{\kk}{h^\vee}
\nc{\sun}{\widehat{\sw}_N}
\nc{\si}{\sigma}
\nc{\el}{\ell}
\nc{\bi}{\bibitem}
\nc{\om}{\omega}
\nc{\ol}{\overline}
\nc{\ds}{\displaystyle}
\nc{\dzz}{\frac{dz}{z}}
\nc{\Res}{\on{Res}}
\nc{\mc}{\mathcal}
\nc{\Cal}{\mathcal}
\nc{\bb}{{\mathfrak b}}
\nc{\ot}{\otimes}
\nc{\R}{{\mc R}}
\nc{\yy}{{\mc Y}}
\nc{\ga}{\gamma}

\nc{\us}{\underset}
\nc{\opl}{\oplus}
\nc{\beq}{\begin{equation}}
\nc{\Fq}{{\mathbb F}_q}
\nc{\Mq}{{\mathcal M}}
\nc{\Rep}{\on{Rep}}
\nc{\sssec}{\subsubsection}
\nc{\ssec}{\subsection}
\nc{\lan}{\langle}
\nc{\ran}{\rangle}

\nc{\D}{\mathcal D} \nc{\Vect}{\on{Vect}} \nc{\ghat}{\G}
\nc{\T}{\mc T} \nc{\Tloc}{\T^\g_{\on{loc}}} \nc{\vac}{|0\ran}
\nc{\Wick}{{\mb :}} \nc{\mb}{\mathbf} \nc{\delz}{\partial_z}
\nc{\K}{{\cali K}} \nc{\cali}{\mathcal} \nc{\li}{\mathfrak l}
\nc{\lt}{\widetilde{\li}} \nc{\astar}{a^*} \nc{\cA}{{\mc A}}
\nc{\ka}{\kappa}

\def\BB{\text{\bf \em B}}  
\def\TT{\text{\bf \em T}}  
\def\FF{\text{\bf \em F}}  
\def\PP{\text{\bf \em P}}  

\nc{\OO}{{\mc O}}
\nc{\AutO}{\on{Aut}\OO}
\nc{\DerO}{\on{Der}\OO}
\nc{\DerpO}{\on{Der}_+\OO}
\nc{\Au}{{\mc A}ut}
\nc{\mf}{\mathfrak}
\nc{\V}{{\mathbb V}}
\nc{\hh}{\wh{\h}}

\nc{\pp}{{\mathfrak p}}
\nc{\mm}{{\mathfrak m}}
\nc{\rr}{{\mathfrak r}}
\nc{\ket}{\rangle}
\nc{\zz}{{\mathfrak z}}
\nc{\gr}{\on{gr}}
\nc{\Spe}{\on{Spec}}
\nc{\rv}{\crho}
\nc{\can}{\on{can}}
\nc{\Op}{\on{Op}_G(D)}
\nc{\MOp}{\on{MOp}_G(D)}
\nc{\Db}{{\mathbb D}}
\nc{\ww}{w}
\def\hat{\widehat}
\def\tilde{\widetilde}

\nc{\af}{{\mathbb A}^1}
\nc{\bs}{\backslash}
\nc{\laa}{(\la_i)}
\nc{\zn}{(z_i)}

\nc{\cla}{\check{\la}}
\nc{\cmu}{\check{\mu}}
\nc{\crho}{\check{\rho}}
\nc{\chal}{\check{\al}}
\nc{\cc}{{\mathfrak c}}
\def\bar{\overline}
\nc{\M}{{\mathbb M}}

\nc{\ZZ}{{\mc Z}}

\nc{\UU}{{\mathbb U}}

\nc{\Conn}{\on{Conn}(\Omega^{\crho})}
\nc{\Con}{\on{Conn}(\Omega^{-\rho})}
\nc{\Co}{\on{Conn}(\Omega^{\rho})}

\nc{\ppart}{(\!(t)\!)}
\nc{\pparl}{(\!(\la)\!)}
\nc{\zpart}{(\!(z)\!)}
\nc{\ppzi}{(\!(t-z_i)\!)}
\nc{\ppinf}{(\!(t^{-1})\!)}
\nc{\Ind}{\on{Ind}}
\nc{\I}{{\mathbb I}}

\nc{\Bun}{\on{Bun}}

\nc{\CC}{C}

\nc{\gtil}{\wt{\g}}
\nc{\ntil}{\wt{\n}}
\nc{\htil}{\wt{\h}}
\nc{\gbar}{\ol{\g}}
\nc{\nbar}{\ol{\n}}
\nc{\bbar}{\ol{\bb}}

\def\M{{\mathcal M}}
\def\MH{{\mathcal M}_H}
\def\CP{{\mathbb C}{\mathbb P}}
\def\L{{\mathcal L}}

\def\Tr{{\rm Tr}}
\def\R{{\mb R}}
\def\cal{\mathcal}

\def\eusm{\EuScript}

\def\tilde{\widetilde}

\def\CO{{\cal O}}
\def\O{\CO}

\nc{\mbb}{\mathbb}
\def\neg{\negthinspace}
\nc{\LG}{{}^L\neg G}
\nc{\LH}{{}^L\neg H}
\nc{\LZ}{{}^L\neg Z}

\def\CA{{\eusm A}}

\def\B{{\cal B}}
\def\D{{\cal D}}
\def\cl{{\rm cl}}
\def\A{{\cal A}}
\def\N{{\cal N}}
\def\S{{S}}
\def\M{{\cal M}}
\def\L{{\cal L}}
\nc{\Pic}{\on{Pic}}

\nc{\Irrep}{\on{Irrep}}

\def\E{{\mathcal E}}

\def\R{{\cal R}}
\def\T{{\cal T}}
\def\S{{\cal S}}
\def\Pic{{\rm Pic}}
\def\RR{\mathbb{R}}

\begin{document}

\title{Geometric Endoscopy and Mirror Symmetry}

\author[Edward Frenkel]{Edward Frenkel$^1$}\thanks{$^1$Supported in
  part by DARPA and AFOSR through the grant FA9550-07-1-0543}

\address{Department of Mathematics, University of California,
  Berkeley, CA 94720, USA }

\author[Edward Witten]{Edward Witten$^{2}$}

\address{School of Natural Sciences, Institute for Advanced Study,
  Princeton, NJ 08540, USA}\thanks{$^{2}$Supported in part by NSF
  Grant PHY-0503584.}

\vspace*{20mm}

\begin{abstract}
The geometric Langlands correspondence has been interpreted as the
mirror symmetry of the Hitchin fibrations for two dual reductive
groups. This mirror symmetry, in turn, reduces to $T$--duality on the
generic Hitchin fibers, which are smooth tori.  In this paper we study
what happens when the Hitchin fibers on the $B$-model side develop
orbifold singularities. These singularities correspond to local
systems with finite groups of automorphisms.  In the classical
Langlands Program local systems of this type are called
endoscopic. They play an important role in the theory of automorphic
representations, in particular, in the stabilization of the trace
formula. Our goal is to use the mirror symmetry of the Hitchin
fibrations to expose the special role played by these local systems in
the geometric theory. The study of the categories of $A$-branes on the
dual Hitchin fibers allows us to uncover some interesting phenomena
associated with the endoscopy in the geometric Langlands
correspondence. We then follow our predictions back to the classical
theory of automorphic functions. This enables us to test and confirm
them. The geometry we use is similar to that which is exploited in
recent work by B.-C. Ng\^o, a fact which could be significant for
understanding the trace formula.

\end{abstract}

\date{October 2007}

\maketitle

\tableofcontents



\section{Introduction}

This paper is concerned with some natural questions arising in the
study of  Langlands duality and mirror symmetry. On the
mathematics side, the question is to describe a geometric analogue
of the phenomenon of endoscopy in the theory of automorphic
representations. On the physics side, it is to explore the limit
of the $T$--duality of supersymmetric sigma models with smooth
dual tori as the target manifolds when the tori become singular.

Somewhat surprisingly, the two questions turn out to be closely
related. The reason is that, according to \cite{KW,GW}, the geometric
Langlands correspondence may be interpreted in terms of the mirror
symmetry of the {\em Hitchin fibrations} for two dual reductive
groups, $G$ and $\LG$:
$$
\begin{array}{ccccc}
{\mc M}_{H}(\LG) & \; & \; & \; & {\mc M}_H(G) \\
\; & \searrow & \; & \swarrow & \; \\
\; & \; & {\BB} & \; & \;
\end{array}
$$
Here ${\mc M}_{H}(G)$ denotes the moduli space of Higgs $G$-bundles on
a smooth Riemann surface $\CC$, and $\BB$ is the common base of the
corresponding two dual Hitchin fibrations \cite{Hi,Hi2}. The mirror
symmetry between them is realized via the fiberwise $T$--duality, in
the framework of the general Strominger--Yau--Zaslow picture
\cite{SYZ}. This duality is also closely related to the $S$-duality of
certain supersymmetric four-dimensional gauge theories corresponding
to $G$ and $\LG$ \cite{KW}.

\subsection{$T$--duality Of Singular Fibers}

The generic fibers of the Hitchin fibrations are smooth dual tori
(which may be described as generalized Prym varieties of spectral
curves when $G=SL_n$), and the $T$--duality is relatively well
understood for these smooth fibers. In particular, it sets up an
equivalence between the category of $B$-branes on the fiber $^L\neg
{\FF}_b$ of ${\mc M}_{H}(\LG)$ at $b \in {\BB}$ and the category of
$A$-branes on the dual fiber $\FF_b$ (more precisely, their connected
components). The simplest $B$-branes are the skyscraper coherent sheaves
supported at the points of $^L\neg {\FF}_b$.\footnote{The category of
$B$-branes should be considered here in the complex structure $J$, in
which ${\mc M}_{H}(\LG)$ is realized as the moduli space of flat
$\LG$-bundles. However, skyscraper sheaves are legitimate $B$-branes
in any complex structure.} Under this equivalence
of categories they correspond to the $A$-branes which are rank one
unitary flat bundles on $\FF_b$. The latter have an important
property: they are {\em eigenbranes} of certain operators which are
two-dimensional shadows of the 't Hooft line operators of the
four-dimensional gauge theory \cite{KW} and are closely related to the
Hecke correspondences on $G$-bundles. This property is dual to the
more easily established property of the skyscraper $B$-branes of being
eigenbranes of the so-called Wilson operators \cite{KW}.

It is natural to ask: what does the $T$--duality look like at the
singular fibers of the Hitchin fibrations? In particular, where
does the $T$--duality map the $B$-branes supported at the singular
points of ${\mc M}_{H}(\LG)$? In this paper we consider the case
that the singularity of the fiber is the mildest possible, namely,
an orbifold singularity. (For example, in the case of $SL_n$,
these are the only singularities if the spectral curve is
irreducible and reduced.) This turns out to be precisely the
situation of ``elliptic endoscopy'', as defined in
\cite{Ngo1,Ngo2} (see below). In the present paper we describe in
detail what happens in the case of the group $G=SL_2$ and explain
how to generalize our results to other groups.

In the case of $G=SL_2$, the singular points of ${\mc M}_{H}(\LG)$
that we are interested in correspond to the $\LG=SO_3$ local systems
(or Higgs bundles) on the curve $\CC$ which are reduced to the
subgroup $O_2 \subset SO_3$ (this is the simplest possible scenario
for elliptic endoscopy, as explained below). Generic local systems of
this type have the group of automorphisms $\Z_2$ (which is the center
of $O_2$) and therefore the corresponding points of ${\mc M}_{H}(\LG)$
are really $\Z_2$-orbifold points.  This means that the category of
$B$-branes supported at such a point is equivalent to the category
$\Rep(\Z_2)$ of representations of $\Z_2$. Thus, it has two
irreducible objects.  Therefore we expect that the dual category of
$A$-branes should also have two irreducible objects. In fact, we show
that the dual Hitchin fiber has two irreducible components in this
case, and the sought-after $A$-branes are the so-called {\em
fractional branes} supported on these two components. Only their sum
(or union) is an eigenbrane of the 't Hooft--Hecke operators
(reflecting the fact that the sole eigenbrane of the Wilson operators
in the $B$-model corresponds to the regular representation of $\Z_2$,
that is, the direct sum of its two irreducible
representations). However, we show that each of the two fractional
$A$-branes separately satisfies a certain natural modification of the
standard Hecke property (the ``fractional Hecke property''), which has
a direct generalization to other groups and is of independent
interest.

\subsection{$A$-branes And ${\mc D}$-modules}

In the conventional formulation of the geometric Langlands
correspondence (see, e.g., \cite{F:houches}, Section 6), the objects
corresponding to $\LG$-local systems on $\CC$ are the so-called Hecke
eigensheaves. These are ${\mc D}$-modules (or perverse sheaves) on the
moduli stack $\Bun_G$ of $G$-bundles on $\CC$ satisfying the Hecke
property. However, their structure is notoriously complicated and it
is difficult to analyze them explicitly. In contrast, in the new
formalism developed in \cite{KW,GW}, the Hecke eigensheaves are
replaced by the $A$-branes which are eigenbranes of the 't Hooft
operators. These $A$-branes are much easier to ``observe
experimentally'' and to analyze explicitly. Our idea is to use this
new language in order to gain insights into the structure of the
geometric Langlands correspondence -- specifically, the part that
pertains to endoscopy (and, more ambitiously, to the general
functoriality principle).

The passage from $A$-branes on ${\mc M}_H(G)$ to ${\mc D}$-modules
on $\Bun_G$ was explained in \cite{KW} (see \cite{NZ,Nadler} for a
possible alternative approach; also, see \cite{ADKMV,DHSV}, where
${\cal D}$-modules have been introduced in physics from a
different point of view). While this has not yet been made
completely rigorous mathematically, it is sufficient to describe
important characteristics of the Hecke eigensheaves associated to
eigenbranes, such as their reducibility, the open subsets of
$\Bun_G$ where the Hecke eigensheaves are represented by local
systems, the ranks of these local systems, and even their
monodromy. Thus, our results on $A$-branes have direct
implications for Hecke eigensheaves. In particular, if an
eigenbrane ${\mc A}$ decomposes into two irreducible branes ${\mc
A}_1$ and ${\mc A}_2$, then we predict that the corresponding
Hecke eigensheaf ${\mc F}$ will also decompose as a direct sum of
two ${\mc D}$-modules, ${\mc F}_1$ and ${\mc F}_2$, corresponding
to ${\mc A}_1$ and ${\mc A}_2$, respectively. Furthermore, these
two ${\mc D}$-modules should then separately satisfy the
fractional Hecke eigensheaf property alluded to above.

\subsection{From Curves Over $\C$ To Curves Over $\Fq$}

The upshot of all this is that by analyzing the categories of
$A$-branes supported on the singular Hitchin fibers, we gain insight
into the geometric Langlands correspondence. We then make another leap
of faith and postulate that the same structures on the Hecke
eigensheaves that we observe for curves over $\C$ (such as their
decomposition into two direct summands) should also hold for curves
over finite fields. In the latter case, to a fractional Hecke
eigensheaf we may associate an automorphic function on the ad\`elic
group $G({\mathbb A}_F)$ by taking the traces of the Frobenius on the
stalks (this is referred to as the {\em Grothendieck
faisceaux--fonctions dictionnaire}, see \secref{pos char}). Our
predictions for the fractional Hecke eigensheaves then get translated
into concrete predictions for the behavior of these automorphic
functions under the action of the classical Hecke operators. We show
that functions satisfying these properties do exist, and this provides
a consistency check for our conjectures.

Thus, our starting point is the homological mirror symmetry between
the categories of branes on the dual Hitchin fibrations. By applying
the following sequence of transformations:

\begin{equation}    \label{passage}
\boxed{\text{$A$-branes}} \; \; \overset{\on{over} \C}\Longrightarrow
\; \; \boxed{\D\text{-modules}}  \; \; \overset{\on{over}
  \C}\Longrightarrow \; \;
\boxed{\begin{matrix} \text{perverse} \\ \text{sheaves} \end{matrix}}
\; \; \overset{\on{over} \Fq}\Longrightarrow \; \;
\boxed{\begin{matrix} \text{automorphic} \\ \text{functions}
\end{matrix}}
\end{equation}

\bigskip

\noindent we link the structure of $A$-branes that we observe in the
study of this mirror symmetry to the classical theory of automorphic
forms. From this point of view, the $A$-branes that are eigenbranes of
the 't Hooft--Hecke operators (in the ordinary sense) are geometric
analogues of the Hecke eigenfunctions that encapsulate irreducible
automorphic representations. This leads to a tantalizing question:
what is the representation theoretic analogue of the fractional
eigenbranes into which the $A$-eigenbranes break in the endoscopic
case, when the Hitchin fiber becomes singular? R. Langlands has
previously suggested a nice analogy between irreducible automorphic
representations and elementary particles \cite{La:gibbs}. From this
point of view, the existence of fractional $A$-branes indicates the
existence of some inner, ``quark--like'', structure of automorphic
representations, which is still waiting to be fully explored and
understood.

We hope that understanding these structures will give us important
clues about the geometric meaning of endoscopy.

\subsection{Classical Endoscopy}

Endoscopy is one of the most fascinating phenomena in the classical
Langlands correspondence. To motivate it, let us recall (see, e.g.,
\cite{F:houches}, Section 2.4, for more details) the Langlands
correspondence for the group $GL_n$.  Let $C$ be a smooth projective
curve over a finite field $\Fq$, and $F$ the field of rational
functions on $C$. The Weil group $W_F$ is a dense subgroup of the
Galois group $\on{Gal}(\ol{F}/F)$ of automorphisms of the (separable)
closure $\ol{F}$ of $F$. The Langlands correspondence sets up a
bijection between $n$-dimensional ($\ell$-adic) representations of the
Weil group $W_F$ and irreducible automorphic representations of
$GL_n({\mathbb A}_F)$, where ${\mathbb A}_F$ is the ring of ad\`eles
of $F$ \cite{Dr,Drinfeld,Laf}.\footnote{This is the Langlands
correspondence for the function fields. There is a similar, but more
complicated, number fields version, in which $F$ is replaced by the
field ${\mathbb Q}$ of rational numbers or its finite extension; see,
e.g., \cite{F:houches}, Part I, for more details.} If we replace
$GL_n$ by a more general reductive group $G$, then, in the first
approximation, we should expect that irreducible automorphic
representations of $G({\mathbb A})$ would be in bijection with
($\ell$-adic) homomorphisms $\sigma$ from $W_F$ to the Langlands dual
group $\LG$ of $G$. However, it turns out that in general to each
$\sigma$ corresponds not one, but several (possibly infinitely many),
irreducible automorphic representation of the ad\`elic group
$G({\mathbb A}_F)$. The set of equivalence classes of these
representations is called the $L$-{\em packet} associated to $\sigma$,
after the work of Labesse--Langlands \cite{LL} in which this
phenomenon was discovered (for the group $G=SL_2$).

The structure of the $L$-packets is most interesting in the case of
homomorphisms $W_F \to \LG = SO_3$ that have their image contained in
the subgroup $O_2 \subset PGL_2$, but not in its connected component
$SO_2$ which implies that their group of automorphisms is disconnected
(generically, it is $\Z_2$). Thus, the automorphic representation
theory of $SL_2({\mathbb A}_F)$ is governed in part by the group
$O_2$. This group (or rather, to keep with the traditional
terminology, the subgroup $H$ of $SL_2(F)$ whose dual is $O_2$) is an
example of an {\em endoscopic group}, and this relation is an example
of the mysterious phenomenon known as the ``endoscopy.''  It was
discovered by Langlands and others in their attempt to organize
automorphic representations in a way that would be compatible with the
structure of the orbital integrals appearing on the geometric side of
the trace formula.

Let us explain this briefly, referring the reader to
\cite{L:debuts,K:cusp,Arthur:trace} for more details. The trace
formula, or rather, its ``regular elliptic part'' (to which we will
restrict ourselves here), has the following general form:

\begin{equation}    \label{trace formula}
\boxed{\begin{matrix} \text{spectral} \\ \text{side} \end{matrix}}
\quad = \quad
\boxed{\begin{matrix} \text{geometric} \\ \text{side}
\end{matrix}}
\end{equation}

\bigskip

\noindent The spectral side is equal to the sum of traces of a test
function $f$ with compact support on $G({\mathbb A}_F)$ over
irreducible tempered cuspidal automorphic representations of
$G({\mathbb A}_F)$ (we recall that those are realized in a certain
space of functions on the quotient $G(F) \bs G({\mathbb A}_F)$):
\begin{equation}    \label{spectral side}
\sum_{\sigma: W_F \to \LG} \quad \sum_{\phi \in L_\sigma}
m_\phi \on{Tr}(f,\pi_\phi).
\end{equation}
Here the sum is over a certain class of homomorphisms $\sigma: W_F \to
\LG$, which are supposed to label the $L$-packets $L_\sigma$ of
equivalence classes of irreducible automorphic representations $\{
\pi_\phi \}_{\phi \in L_\sigma}$, and $m_\phi$ denotes the
multiplicity of $\pi_\phi$ in the space of automorphic functions.

The geometric side is the sum of orbital integrals of $f$, that is,
integrals of $f$ over $G({\mathbb A}_F)$-conjugacy classes of elements
of $G(F)$.

The geometric side needs to be ``stabilized.'' This means rewriting
it as a sum of integrals over {\em stable conjugacy classes} of
elements of $G(F)$ in $G({\mathbb A}_F)$.\footnote{Two elements of
$G(k)$, where $k$ is any field, are called stably conjugate if they
are conjugate in $G(\ol{k})$. Since ${\mathbb A}_F$ is the
restricted product of completions of $F$, we obtain a natural notion
of stable conjugacy in ${\mathbb A}_F$ as well.} This is necessary
for many reasons, one of which is that without this one cannot even
hope to compare the geometric sides of the trace formulas for
different groups (see, e.g., \cite{Arthur:trace}, Sect. 27). The
resulting expression for the geometric side reads \cite{L:debuts}
\begin{equation}    \label{geom side}
\sum_H \imath(H,G) \; ST^{G-\on{reg}}_{\on{ell}}(f^H),
\end{equation}
where the sum is over the {\em elliptic endoscopic groups} $H$ of $G$
(they are {\em not} subgroups of $G$ in general), as well as $H=G$
itself, $ST^{G-\on{reg}}_{\on{ell}}(f^H)$ denotes the sum of stable
orbital integrals for the group $H$, and the $\imath(H,G)$ are certain
numbers. The elliptic endoscopic groups are defined, roughly, as the
dual groups of the centralizers of semi-simple elements in $\LG$ which
are not contained in any proper Levi subgroups of $\LG$.

Formula \eqref{geom side} hinges upon a number of assumptions, the
most important of which is the so-called transfer conjecture. It
states the existence of an assignment $f \mapsto f^H$, from functions
on $G({\mathbb A}_F)$ to those on $H({\mathbb A}_F)$, satisfying the
property that, roughly speaking, the stable orbital integrals of $f^H$
are equal to stable orbital integrals of $f$ modified by a certain
twist (see \cite{L:debuts,Arthur:trace,Dat} and references therein for
details).

A special case of the transfer conjecture (when $f$ is the
characteristic function of a maximal compact subgroup of $G$; then
$f^H$ is required to be of the same kind) is the so-called {\em
fundamental lemma}. The fundamental lemma (in the function field case)
has been recently proved by B.-C. Ng\^o \cite{Ngo2} (see also
\cite{GKM,Laumon:sp,LN,Ngo1}). More precisely, Ng\^o has proved a Lie
algebra version of the fundamental lemma, but Waldspurger has shown
that it implies the fundamental lemma for the group, as well as the
general transfer conjecture \cite{Wa1}. Also, the fundamental lemma in
the function field case that we are discussing is equivalent to the
one in the number field case, provided that the characteristics of the
residue fields are sufficiently large \cite{Wa2,CL}.

Since the geometric side of the trace formula has the form \eqref{geom
side}, it is natural to expect that the spectral side has a similar
form, that is, it may be written as a sum of terms labeled by the
elliptic endoscopic groups, with individual terms on both sides of
\eqref{trace formula} being equal. A formula of this sort has been
first established by Labesse--Langlands for $G=SL_2$ (and its inner
forms). In general, the corresponding formula was conjectured by
R. Kottwitz \cite{K:cusp}. This formula has a number of important
consequences for the theory of automorphic representations. First of
all, it leads to an explicit formula for the multiplicities $m_\phi$ of
the automorphic representations in the $L$-packet associated to a
homomorphism $\sigma: W_F \to \LG$ (appearing in \eqref{spectral
side}). The answer is a linear combination of terms associated to the
elliptic endoscopic groups $H$ such that the image of $\sigma$ is
contained in $\LH \subset \LG$.

Perhaps it would be helpful to explain here the relation between
$\sigma$ and the endoscopic groups. Suppose for simplicity that $\LG$
is a semi-simple group of adjoint type (so it has trivial center), and
the image of a homomorphism $\sigma: W_F \to \LG$ occurring in
\eqref{spectral side} has finite centralizer $S_\sigma$. Then the
Langlands duals $\LH$ of the endoscopic groups $H$ associated to
$\sigma$ are just the centralizers of non-trivial elements $s \in
S_\sigma$ in $\LG$ (hence the image of $\sigma$ is automatically
contained in $\LH$). For instance, if $\LG = SO_3$, the only subgroup
that we can obtain this way is $O_2 \subset SO_3$.\footnote{Note that
as in \cite{Ngo1,Ngo2} and contrary to the standard convention, we do
not consider here $G$ itself as an elliptic endoscopic group.}

The second, and perhaps, more important, consequence of the transfer
conjecture and the stabilized trace formula is that it gives a
concrete realization of the Langlands {\em functoriality principle}
\cite{L} for the homomorphisms $\LH \to \LG$, where $H$ are the
elliptic endoscopic groups. Namely, we obtain a natural map from
$L$-packets of automorphic representations of $H({\mathbb A}_F)$ to
those of $G({\mathbb A}_F)$ (see, e.g., \cite{Arthur:funct} and
\cite{Arthur:trace}, Sect. 26, for more details).

In short, the classical endoscopy establishes an elusive connection
between automorphic representations of $G({\mathbb A}_F)$ and those of
its endoscopic groups $H({\mathbb A}_F)$ which matches, via the trace
formula, the relation between orbital integrals for the two groups
provided by the transfer conjecture.

\subsection{Geometric Endoscopy}

In the last twenty years significant progress has been made in
translating the classical Langlands correspondence to the language of
geometry. The emerging geometric Langlands correspondence has the
advantage that it makes sense not only for curves defined over finite
fields, but also for curves over the complex field, that is, Riemann
surfaces. In this version we have the opportunity to use the vast
resources of complex algebraic geometry and thereby advance our
understanding of the general Langlands duality patterns.

For many concepts of the classical Langlands correspondence
counterparts have been found in the geometric version. One important
phenomenon that has not yet been understood geometrically is the
endoscopy. It is the goal of this paper to make the first steps in the
development of {\em geometric endoscopy}, by which we mean exposing
the special role played by the endoscopic groups in the geometric
Langlands correspondence. Since there is no obvious analogue of the
trace formula in the geometric setting, we are not trying to imitate
the stabilization of the trace formula that leads to the classical
endoscopy. Rather, we wish to describe the new structures that emerge
in the geometric Langlands correspondence for the $\LG$-local systems
of elliptic endoscopic type.  When $\LG$ is semi-simple, this means
(by analogy with the classical setting) that the group of
automorphisms of this local system is a finite group that is strictly
larger than the center of $\LG$. For example, for $\LG=SO_3$ these are
the local systems whose image is contained in $O_2 \subset SO_3$, but
not in the maximal torus of $O_2$. Then the group of automorphisms
is equal to $\Z_2$, the center of $O_2$, unless the image is contained
in a subgroup $\Z_2 \times \Z_2$ (we will mostly ignore this case in
the present paper).

Our approach is to use mirror symmetry of the Hitchin fibrations
associated to $G$ and $\LG$ and to explore the structure of the
$A$-branes corresponding to the endoscopic $\LG$-local systems
(viewed as orbifold points of ${\mc M}_H(\LG)$, on the $B$-model
side), which are realized on the corresponding singular Hitchin
fibers in ${\mc M}_H(G)$ (on the $A$-model side). The first
advantage of this approach is that the endoscopic groups (which are
rather mysterious in the classical theory, where they arise in the
process of stabilization of the trace formula for the group $G$) are
manifest: they occur naturally on the $B$-model side of mirror
symmetry. Indeed, for any subgroup $\LH \subset \LG$ there is a
natural embedding ${\mc M}_H(\LH) \hookrightarrow {\mc M}_H(\LG)$,
and a point of ${\mc M}_H(\LG)$ corresponding to an endoscopic local
system $\E$ belongs to the image of ${\mc M}_H(\LH)$ for all
endoscopic groups $H$ associated to it (i.e., those for which $\LH$
contains the image of $\E$). The second advantage, already mentioned
above, is that, at least in the generically regular semi-simple
case, the corresponding $A$-branes have a simple and transparent
structure (in contrast to Hecke eigensheaves), and this simplifies
our analysis considerably.

We then interpret the structures that we observe in the category of
$A$-branes (on the $A$-model side) in terms of ${\mc D}$-modules (for
curves over $\C$) or perverse sheaves (for curves over $\C$ or over
$\Fq$) on $\Bun_G$, which are the more standard objects in the
geometric Langlands correspondence.\footnote{Alternatively, one may
look at it from the point of view of a non-abelian version of the
Fourier--Mukai transform \cite{Laumon:fm,Roth}, suggested by
A. Beilinson and V. Drinfeld.} Finally, for curves over $\Fq$, we
study the automorphic functions associated to these perverse sheaves
(see the diagram \eqref{passage}). This results in a series of
concrete predictions:

\begin{itemize}

\item Our first prediction is that the Hecke eigensheaves
corresponding to an elliptic endoscopic $\LG$-local system $\E$ with
the finite group of automorphisms $\Gamma$ splits into a direct sum of
irreducible sheaves ${\mc F}_R$ labeled by irreducible representations
$R$ of $\Gamma$, with the multiplicity of ${\mc F}_R$ equal to $\dim
R$.

\item Our second prediction is that the sheaves ${\mc F}_R$ satisfy a
fractional Hecke property described in \secref{categories}. This
prediction is confirmed in the case of curves over $\Fq$: we check
that the functions assigned to our sheaves and satisfying the
function-theoretic analogue of the fractional Hecke property do
exist. Moreover, we express them as linear combinations of the
ordinary Hecke eigenfunctions by a kind of Fourier transform over
$\Gamma$.\footnote{This is somewhat reminiscent of the Fourier
transform observed by G. Lusztig in the theory of character sheaves
\cite{Lu}.}  Thus, it turns out that in the endoscopic case the
functions assigned to irreducible perverse sheaves are not Hecke
eigenfunctions, but linear combinations thereof.

\item Our third prediction is that if the image $b$ of $\E$ in the
Hitchin base $\BB$ is a generically regular semi-simple point, then
the group $\Gamma$ is {\em abelian} and is isomorphic to a subgroup of
the dual group of $\pi_0(\PP_b)$, where $\PP_b$ is the generalized
Prym of $b$. Furthermore, there exists a local system in the dual
Hitchin fiber $^L\neg \FF_b$ for which $\Gamma$ is isomorphic to the
dual group of $\pi_0(\PP_b)$ (note that $\pi_0(\PP_b)$ acts simply
transitively on an open dense subset of the Hitchin fiber $\FF_b$ over
$b$ in ${\mc M}_H(G)$, which is reduced in this case
\cite{Ngo2}).\footnote{In the course of writing this paper we were
informed by B.-C. Ng\^o that he was also aware of this statement.}

\end{itemize}

We hope that proper understanding of these phenomena will lead to
better understanding of endoscopy and related subjects such as the
fundamental lemma.

\subsection{Connection With The Work Of B.-C. Ng\^o}

A link between our analysis and the classical endoscopy comes from
the fact that the geometry we use is similar to what is exploited
in the recent work of Ng\^o Bao-Ch\^au \cite{Ngo1,Ngo2} (see also
the excellent survey \cite{Dat}). Ng\^o has discovered a striking
connection between the orbital integrals appearing on the
geometric side of the trace formula \eqref{trace formula} and the
cohomology of the Hitchin fibers in moduli space ${\mc M}_H(G)$
(more precisely, in generalized versions of ${\mc M}_H(G)$ which
parametrize meromorphic Higgs fields with the divisor of poles
$D$ which is sufficiently large). He used it to prove the
fundamental lemma, in the Lie algebra setting, for function fields
(for unitary groups he had done it earlier together with G. Laumon
\cite{LN}). He achieved that by interpreting the orbital integrals
as numbers of points of the Hitchin fibers $\FF_b$ in the moduli
stack of Higgs bundles ${\mc M}_H(G)$ defined for a curve over a
{\em finite field}. These numbers are in turn interpreted as
traces of the Frobenius acting on the ($\ell$-adic) cohomology of
$\FF_b$. The crucial step in Ng\^o's construction is the
decomposition of this cohomology with respect to the action of the
finite abelian group $\pi_0(\PP_b)$, where $\PP_b$ is the
generalized Prym variety associated to an elliptic point $b \in
\BB$. He identified the $\ka$-isotypic part of this decomposition,
where $\ka$ is a character of $\pi_0(\PP_b)$, with the $\ka$-part
of the decomposition of the cohomology of the corresponding
Hitchin fiber of the endoscopic group $H_\kappa$. Taking traces of
the Frobenius over these subspaces, he obtained the fundamental
lemma. (Here we should mention the earlier works
\cite{GKM,Laumon:sp} in which closely related geometric
interpretations of the fundamental lemma had been given.)

Thus, Ng\^o uses  geometry that seems very close to the geometry we
are using. Indeed, we consider the fractional $A$-branes on ${\mc
M}_H(G)$, supported on the Hitchin fiber $\FF_b$, which are
essentially labeled by $\pi_0(\PP_b)$, and Ng\^o considers the
cohomology of similarly defined Hitchin fibers and their
decomposition under $\pi_0(\PP_b)$. However, there are important
differences.

First of all, we work over $\C$, whereas Ng\^o works over $\Fq$. In
the latter setting there is no obvious analogue of the homological
mirror symmetry between ${\mc M}_H(G)$ and ${\mc M}_H(\LG)$ that is
crucial to our approach. (In fact, neither the dual group $\LG$ nor
the dual Hitchin moduli space ${\mc M}_H(\LG)$ play a role in Ng\^o's
work.) Second, Ng\^o works with generalized moduli spaces of Higgs
bundles labeled by effective divisors $D$ on the curve $\CC$. Our
moduli space ${\mc M}_H(G)$ corresponds to a divisor of the canonical
line bundle on $\CC$.\footnote{Technically, this case is outside of
the scope of Ng\^o's work, since he imposes the condition
$\deg(D)>2g-2$. However, as he explained to us, most of his results
remain true when $D$ is a divisor of the canonical line bundle.}
Third, and most importantly, the objects we assign to the connected
components of the singular Hitchin fiber $\FF_a$ -- the $A$-branes --
are objects of {\em automorphic} nature; we hope to relate them to
Hecke eigensheaves and ultimately to the automorphic functions in the
classical theory. Thus, these objects should live on the {\em spectral
side} of the trace formula \eqref{trace formula}. On the other hand,
Ng\^o relates the numbers of points of the Hitchin fibers (and their
cohomology) to the orbital integrals appearing on the {\em geometric
side} of the trace formula (more precisely, its Lie algebra
version). {\em A priori}, this has nothing to do with automorphic
representations (or automorphic functions)! The connection between
orbital integrals and automorphic representations is provided by the
trace formula, but in a rather indirect, combinatorial way. It arises
only when we sum up all orbital integrals on the geometric side, and
over all representations on the spectral side of the trace formula.

This raises the following question: could there be a {\em direct link}
between individual Hitchin fibers in the moduli space $\MH(G)$
over $\Fq$ and individual automorphic representations? After all, we
have a natural forgetful map from $\MH(G)$ to $\Bun_G$, where
unramified automorphic functions live. Could it be that the passage
from $A$-branes to Hecke eigensheaves discussed above has an analogue
in the classical theory as a passage from orbital integrals to Hecke
eigenfunctions?

In any case, we find it remarkable that the same geometry of the
Hitchin fibration that Ng\^o has used to understand the geometric side
of the trace formula is also used in our study of the geometric
endoscopy via mirror symmetry and therefore appears to be pertinent to
automorphic representations (via the correspondence
\eqref{passage}). This connection could potentially be significant as
it could shed new light on the trace formula and the theory of
automorphic representations in general.\footnote{This seems to
resonate with the views expressed by R. Langlands in his recent Shaw
Prize lecture \cite{La:shaw}.}

\subsection{Quantum Field Theory}

Finally, we wish to relate our computation to various issues in
quantum field theory. Recall that the Strominger--Yau--Zaslow
picture \cite{SYZ} relates homological mirror symmetry of two
manifolds $X$ and $Y$ to the $T$--duality of dual special
Lagrangian fibrations in $X$ and $Y$. This works especially nicely
for the generic fibers, which are smooth. The $T$--duality of
these fibers may be thought of as a kind of abelian version of
mirror symmetry (closely related to the Fourier--Mukai transform).
Therefore, important ``non-abelian'' information about the mirror
symmetry of $X$ and $Y$ is hidden in the duality of the singular
fibers. Special Lagrangian fibrations are difficult to understand
in general, but in the case of Hitchin fibrations, the
hyper-Kahler structure leads to a drastic simplification \cite{HT}
which we will exploit to analyze certain cases of singular fibers.
In particular, we show that under mirror symmetry, the orbifold
singularities on $X$ correspond to reducible Lagrangian fibers in
$Y$; we  believe that this is a fairly general phenomenon. From
this point of view, the geometric endoscopy appears as a special
case of mirror symmetry in the presence of orbifold singularities.

Geometric endoscopy might have a natural realization physically by
means of a supersymmetric domain wall, with ${\cal N}=4$
supersymmetric Yang-Mills theory of gauge group $\LH$ on one side
of the domain wall, and the same theory with gauge group $\LG$ on
the other side (with $H$ being an endoscopic group of $G$). This
domain wall should in particular give rise after duality to a
functor from the category of $A$-branes on $\MH(H)$ to the
category of $A$-branes on $\MH(G)$. This is a geometric analogue
of the notion of transfer, or the functoriality principle, in the
Langlands Program (see \secref{transfer}).

Finally, as we explain in \secref{SeibergWitten}, the Hitchin
fibrations studied in this paper also appear in Seiberg--Witten
theory.  In that context, the important four-manifold invariants
arise as contributions from the endoscopic points.  The appearance
of the same Hitchin fibrations in the two different problems can be
traced to an underlying six-dimensional quantum field theory that
can be compactified to four dimensions in two different ways.

\subsection{Plan Of The Paper}

In \secref{dualbranes} we give an overview of the connection
between the homological mirror symmetry of the dual Hitchin
fibrations and the geometric Langlands Program. In \secref{genus
one}, we take up our main example: the moduli spaces of $SL_2$ and
$SO_3$ Higgs bundles on an elliptic curve with tame ramification
at one point. We describe in detail these moduli spaces, their
singularities and the corresponding categories of branes. This
example will serve as the prototype for the general picture
developed in this paper.

In \secref{genus one A to D}, we discuss in more detail the
passage from $A$-branes to ${\mc D}$-modules. While we focus
largely on the genus one example developed in the preceding
section, many aspects of this discussion apply in a more general
setting. \secref{spectral} describes the generalization of our
results to curves of higher genus. We also compute explicitly the
action of the Wilson and 't Hooft/Hecke operators on the electric
and magnetic branes relevant to   geometric endoscopy. In
\secref{categories}, we explain how these results fit in a general
categorical formalism. In particular, we introduce the notion of
``fractional Hecke eigensheaves'' and conjecture that the ${\mc
D}$-modules associated to the fractional $A$-branes found in
\secref{spectral} are objects of this type.

Our next task is to describe the analogues of these conjectures in the
case when our curve $C$ is defined over a finite field, and to link
them to the theory of automorphic functions. In \secref{classical}, we
recall the set-up of endoscopy and $L$-packets in the classical theory
of automorphic forms. We focus in particular on the unramified case
for $G=SL_2$ analyzed first by Labesse and Langlands \cite{LL}. We
then discuss in \secref{from}, potential implications for the
classical theory of the geometric picture outlined in the earlier
sections. In particular, we compute the automorphic functions
associated to the fractional Hecke eigensheaves. We generalize our
results and conjectures to other groups in \secref{other}. Finally, in
\secref{gerbes} we discuss the tricky point that the two fractional
eigenbranes that we have found in the case of $SL_2$ are
indistinguishable. We trace this phenomenon to a certain $\Z_2$-gerbe
that appears to be a subtle, but important, ingredient of the mirror
symmetry of the Hitchin fibrations. In the Appendix, we explain the
structure of the unramified global $L$-packets for $SL_2$ in concrete
terms using the Whittaker functions.

\subsection{Acknowledgments}

E.F. wishes to express his gratitude to R. Langlands for conveying his
insights on the endoscopy. He also thanks J. Arthur, D.  Kazhdan,
R. Kottwitz, B.-C. Ng\^o, M. Olsson, T. Pantev, B.  Poonen, and
K. Ribet for their helpful answers to various questions related to
this paper, and D. Arinkin, A. Beilinson, R.  Bezrukavnikov, and
V. Lafforgue for stimulating discussions.  E.W. similarly thanks
T. Hausel, D. R. Morrison, T. Pantev, I. M. Singer, and C. Vafa. We
are also grateful to S. Gukov for his comments on the draft.

E.F. thanks the Intitute for Advanced Study for hospitality during
his visits in the course of this work.

\section{Duality, Branes, and Endoscopy}\label{dualbranes}

\subsection{Geometric Langlands Duality And Mirror
Symmetry}\label{geomdual}

 In the gauge theory approach,
the geometric Langlands correspondence is understood as a mirror
symmetry.

Let ${\cal Y}(C;G)$ be the moduli space of flat
$G$-bundles\footnote{The complex form of a Lie group is always meant
unless otherwise specified. This is in keeping with most literature
on the geometric Langlands program, but in contrast to most
literature on gauge theory including \cite{KW,GW}.} over an oriented
two-manifold $C$.  Here $G$ is a complex reductive Lie group.  Then
${\cal Y}(C;G)$ is in a natural way a complex symplectic manifold.
The complex structure of ${\cal Y}(C;G)$ comes simply from the
complex structure of $G$, and its holomorphic two-form $\Omega$ is
defined using the intersection pairing on the tangent space to
${\cal Y}(C;G)$, which is $H^1(C,{\rm ad}(Y))$.

Using the complex structure of ${\cal Y}(C;G)$ (and the triviality
of its canonical bundle, which follows from the fact that ${\cal
Y}(C;G)$ is complex symplectic), one can define a $B$-model of
${\cal Y}(C;G)$.  Similarly, viewing ${\cal Y}(C;G)$ as a real
symplectic manifold, with symplectic structure $\omega={\rm
Im}\,\Omega$, one can define an $A$-model.\footnote{The definition
of $\Omega$ is such that ${\rm Im}\,\Omega$ is cohomologically
trivial, though ${\rm Re}\,\Omega$ is not.  The definition of
$\Omega$ depends on the choice of a nondegenerate quadratic form on
$\mathfrak g$, the Lie algebra of $G$.  However, the $A$-model and
therefore the geometric Langlands duality derived from it is
independent of this choice, up to a natural isomorphism.  Once the
relation between flat bundles and Higgs bundles \cite{Hi} is
incorporated, this follows from the $\C^*$ action on the moduli
space of Higgs bundles. In the gauge theory approach \cite{KW}, it
is clear {\it a priori} that this choice is inessential, since the
dependence of the action on the gauge coupling is of the form
$\{Q,\cdot\}$.}

Now let $^L{}\neg G$ and $G$ be a dual pair of complex reductive Lie
groups.  It turns out that there is a mirror symmetry between the
$B$-model of ${\cal Y}(C;{}^L\neg G)$ and the $A$-model of ${\cal
Y}(C;G)$.  This instance of mirror symmetry was first studied by
Hausel and Thaddeus \cite{HT}. (A closely related duality has been
studied by Donagi and Pantev \cite{DP}, Hitchin \cite{Hi3}, and
Arinkin \cite{Ar}.  The relation between the two is explained in
section 5.3.) It was deduced from electric-magnetic duality of
four-dimensional supersymmetric gauge theory in \cite{KW}, following
earlier arguments \cite{BJSV,HMS}, and shown to underlie geometric
Langlands duality.\footnote{We will restrict ourselves to the most
basic form of the geometric Langlands duality. In \cite{KW}, it is
explained how one can incorporate in the gauge theory approach an
additional complex parameter, leading to what in the mathematical
literature is sometimes called quantum geometric Langlands.}

To establish mirror symmetry between ${\cal Y}(C;{}^L\neg G)$ and
${\cal Y}(C;G)$, one needs the fact that these spaces have another
interpretation as moduli spaces of Higgs bundles.  This comes from
Hitchin's equations. Unlike the complex symplectic structure of
${\cal Y}(C;G),$ which we have considered hitherto, Hitchin's
equations require a choice of conformal structure on $C$.  So
henceforth $C$ is a complex Riemann surface, not just an oriented
two-manifold.

\subsubsection{Hitchin's Equations}\label{hitch}
 Hitchin's equations \cite{Hi} are nonlinear equations for a pair
$(A,\phi)$.  $A$ is a connection on a $G$-bundle $E\to C$ with
structure group the compact form of $G$.  And $\phi$ is a one-form on
$C$ that is valued in ${\mathfrak g}$, the Lie algebra of this compact
form. Hitchin's equations read
\begin{align}\notag F-\phi\wedge \phi&=0\\ \label{otto}
  d_A\phi=d_A\star \phi& = 0.\end{align}
Here $d_A$ is the gauge-covariant exterior derivative and $\star $ is
the Hodge star operator.

If $(A,\phi)$ is a solution of Hitchin's equations, then
$\CA=A+i\phi$ is a complex-valued flat connection, and thus endows
$E$ (or its complexification) with the structure of a local
system. In particular, a solution of Hitchin's equations
determines a point in ${\cal Y}(C;G)$, the moduli space of local
systems with structure group $G$.

Alternatively, the $\bar\partial$ operator given by the $(0,1)$ part
of the operator $d_A$ endows $E$ with a holomorphic structure. And
if we write $\varphi$ for the $(1,0)$ part of $\phi$ and $K$ for the
canonical bundle of $C$, then $\varphi$ is a section of $K\otimes
{\rm ad}(E)$ and is holomorphic according to Hitchin's equations.
The pair $(E,\varphi)$ then defines what is known from a holomorphic
point of view as a Higgs bundle.

A basic result about Hitchin's equations is that the moduli space
$\MH=\MH(C;G)$ of solutions of these equations is a hyper-Kahler
manifold. In one complex structure, called $I$ in \cite{Hi}, $\MH$
is the moduli space of stable Higgs bundles, while in another
complex structure, denoted as $J$, it is the moduli space of stable
local systems and coincides with what we have earlier called ${\cal
Y}$. The fact that the same space has these dual interpretations is
one of the reasons that it is possible to say something about
geometric Langlands by studying Higgs bundles.

As a hyper-Kahler manifold, $\MH$ has a distinguished triple
$\omega_I, $ $\omega_J$, $\omega_K$ of real symplectic forms, which
are Kahler respectively in complex structures $I$, $J$, and $K=IJ$.
Similarly, there are complex two-forms
$\Omega_I=\omega_J+i\omega_K$, $\Omega_J=\omega_K+i\omega_I$,
$\Omega_K=\omega_I+i\omega_J$, that are holomorphic symplectic
forms, respectively, in complex structures $I$, $J$, and $K$. When
${\cal Y}$ is identified with $\MH$, its complex  structure
corresponds to $J$.  Moreover its holomorphic symplectic form is
$\Omega=i\Omega_J$, and $\omega={\rm Im}\,\Omega$ is equal to
$\omega_K$.

\subsubsection{The Hitchin Fibration}\label{hitchfibbo}

The Hitchin fibration is the map $\Pi:\MH\to \BB$ that takes
a Higgs bundle $(E,\varphi)$ to the characteristic polynomial of
$\varphi$. ($\BB$ is a linear space that parametrizes the
possible values of the characteristic polynomial.)  The Hitchin
fibration is holomorphic in complex structure $I$ (in which $\MH$ is
the moduli space of Higgs bundles). The fibers of the fibration are
Lagrangian submanifolds from the point of view of the holomorphic
symplectic structure $\Omega_I$, and hence also from the point of
view of $\omega_K={\rm Im}\,\Omega_I$.

Thus, from the point of view of the $A$-model that is relevant in
the geometric Langlands program, the Hitchin fibration is a
fibration by Lagrangian submanifolds.  These fibers generically are
smooth tori (holomorphic in complex structure $I$).
 This fact is related to
the interpretation of $\MH$ as a completely integrable Hamiltonian
system \cite{Hi2}, and can be seen explicitly using spectral
curves, as we will explain in \secref{spectral}.

{}From the standpoint of the $A$-model, the fibers of the Hitchin
fibration are not merely Lagrangian tori; they are special
Lagrangian.  Indeed, since they are holomorphic in one of the
complex structures (namely $I$), they minimize the volume (computed
using the hyper-Kahler metric of $\MH$) in their cohomology class.
So \cite{HT} this is an example of a fibration by special Lagrangian
tori, a geometric structure that has been proposed \cite{SYZ} to
describe mirror symmetry.

More specifically, if $G$ is simply-laced, the bases of the Hitchin
fibrations for $G$ and $^L\neg G$ are the same.  (For example, in
\secref{genus one}, we will parametrize $\BB$ by the same complex
variable $w$ both for $G=SL_2$ and for $^L\neg G=SO_3$.) For any
$G$, there are natural isomorphisms between these bases.  So we get
a picture:
\begin{equation}\label{celp}
\begin{array}{ccccc}
{\mc M}_{H}(^L\neg G) & \; & \; & \; & {\mc M}_H(G) \\
\; & \searrow & \; & \swarrow & \; \\
\; & \; & {\BB} & \; & \;
\end{array}
\end{equation}
The fibers of the two fibrations over a generic point $b\in \BB$ are
dual tori, as first argued in \cite{HT} for $G=SL_n$. (For other
groups, see \cite{DP,Hi3}.)  This is the usual SYZ picture associated
with mirror symmetry.

This particular example has several advantages. It is usually very
difficult to concretely describe a special Lagrangian fibration,
but in the case of $\MH$, the fact that the fibration is
holomorphic in complex structure $I$ makes it accessible, as we
will see in the examples discussed in Sections \ref{genus one} and
\ref{Hecke}. Also, the relation between mirror symmetry and a
special Lagrangian fibration is typically affected by what
physicists call quantum corrections by disc instantons. However,
the hyper-Kahler nature of $\MH$ ensures the absence of such
corrections and makes more straightforward the application of the
Hitchin fibration to this particular example of mirror symmetry.

\subsection{Branes And Their Duals}\label{braneduals}

 To make contact with geometric
Langlands duality, we must consider $B$-branes and $A$-branes on
$\MH$.

A simple example of a $B$-brane ${\cal B}$ is a brane of rank 1
supported at a smooth point $r\in \MH(C;{}^L\neg G)$.  Such a point
corresponds to an irreducible $^L\neg G$ local system ${\cal E}_r$
over $C$ (that is, one whose automorphism group reduces to the center
of $^L\neg G$; for the moment, assume that the center is trivial). It
is contained in a fiber $^L\neg \FF_b$ of the Hitchin
fibration of $^L\neg G$, and lies over a point $b$ in the base $\BB$
of the fibration.  We let $\FF_b$ denote the fiber over $b$ of the
dual Hitchin fibration of $\MH(C;G)$.

Mirror symmetry is understood as a $T$--duality on the fibers of the
Hitchin fibration.  So it maps ${\mc B}$ to an $A$-brane ${\mc A}$
whose support is $\FF_b$, endowed with a flat unitary line bundle
${\cal L}_r$ that depends on $r$.  This makes sense since the
generic fibers $^L\neg \FF_b$ and $\FF_b$ of the dual Hitchin
fibrations are dual tori. So a point $r\in{} ^L\neg \FF_b$
determines a flat unitary line bundle $\L_r\to \FF_b$.

What has just been described is the usual picture. What happens if
$r$ is a singular point in $\MH(C;{}^L\neg G)$, corresponding to an
$^L\neg G$ local system that has a non-trivial automorphism group?
The category of branes supported at a smooth point is equivalent to
the category of vector spaces, via the map that takes a brane to its
space of  sections. The category of branes supported at a singular
point is generally more complicated and cannot be reduced to a
single object.  Mirror symmetry or geometric Langlands duality must
be applied to this whole category.

Suppose now that the automorphism group of the local system ${\cal
E}_r$ is a non-trivial finite group $\Gamma$.  In this case, the
moduli space of $^L\neg G$ local systems can be modeled near $r$ by
a finite quotient $\C^{2n}/\Gamma$ for some $n$, with a linear
action of $\Gamma$ on $\C^{2n}$ coming from a homomorphism
$\Gamma\to Sp_{2n}\subset U_{2n}$.  ($\Gamma$ acts via a subgroup of
the symplectic group since $\MH$ is complex symplectic  and in fact
hyper-Kahler.)  We identify $r$ with the origin in $\C^{2n}/\Gamma$.
The space of sections of a brane supported at $r$ is now a
finite-dimensional vector space with an action of $\Gamma$. To
describe a brane, we have to say how $\Gamma$ acts on this vector
space.  An irreducible brane corresponds to an irreducible
representation. (For a more complete description, see Section
\ref{gerbes}.)

It makes sense to use this formalism even if $\Gamma$ is simply the
center of $^L\neg G$, in which case $r$ is a smooth point since the
center acts trivially on the fields entering in Hitchin's
equations. An irreducible representation of the center is
one-dimensional, given by a character which in \cite{KW}, section 7,
was called $\mbox{\bf \em e}_0$.  Thus the description of an
irreducible brane at a smooth point can be refined to include
specifying the character $\mbox{\bf \em e}_0$. Under duality,
$\mbox{\bf \em e}_0$ maps to a characteristic class $\mbox{\bf \em
m}_0\in H^2(C,G)$ that determines the topology of a $G$-bundle over
$C$. A brane on $\MH(C;{}^L\neg G)$ of specified $\mbox{\bf \em e}_0$
has a dual that is supported on an irreducible component of $\MH(C;G)$
with definite $\mbox{\bf \em m}_0$.

However, endoscopy arises when the automorphism group is not simply
the center, and we will illustrate the ideas assuming that the
center of $^L\neg G$ is trivial. For instance, in the example of
\secref{genus one}, we will have $^L\neg G=SO_3$, of trivial center.
This being so, $\Gamma$ acts effectively on $\C^{2n}$ and if
$\Gamma$ is non-trivial, then $r$ is an orbifold singularity of
$\MH(C;{}^L\neg G)$.  In this case, for each isomorphism class of
irreducible representation $R_i$ of $\Gamma$, there is a
corresponding irreducible brane ${\cal B}_i$ supported at the point
$r$.

Let $r^*$ be a smooth point of $\MH(C;{}^L\neg G)$, with
irreducible brane ${\cal B}^*$.  Now consider what happens as
$r^*$ approaches an orbifold point $r$.  In the limit, the
irreducible brane ${\cal B}^*$ degenerates to the brane supported
at $r$ and associated with the regular representation of $\Gamma$.
(The reason for this is that a smooth point $r^*\in
\C^{2n}/\Gamma$ corresponds to a free $\Gamma$ orbit on $\C^{2n}$;
the functions on a free orbit furnish a copy of the regular
representation.) The regular representation can be decomposed as
$\bigoplus_i n_iR_i$, where the sum runs over all irreducible
representations $R_i$ of $\Gamma$, and $n_i$ is the dimension of
$R_i$.  So the decomposition of ${\cal B}^*$ is
\begin{equation}\label{omilc}{\cal B}^*\to \bigoplus_i n_i {\cal
B}_i.\end{equation}   In the physics literature, this
decomposition was first analyzed in \cite{DM}, and the branes
${\cal B}_i$ are usually called {\em fractional branes}.

What can this decomposition mean for mirror symmetry or geometric
Langlands duality?  For $r^*$ a smooth point, the brane ${\cal B}^*$
is irreducible, and is mapped by duality to an $A$-brane $\A^*$
supported on the appropriate fiber $\FF^*$ of the Hitchin fibration of
$G$.  $\FF^*$ is irreducible as a Lagrangian submanifold, so the
corresponding brane is irreducible as an $A$-brane.  When $r^*$ is set
to $r$, the brane ${\cal B}^*$ decomposes as a sum of $B$-branes, so
the dual $A$-brane $\A^*$ must also decompose as a sum of $A$-branes,
\begin{equation}\label{omilc1}{\cal A}^*\to \bigoplus_i n_i {\cal
A}_i.\end{equation}

It is attractive if this decomposition occurs geometrically.  (In
fact, we do not know of any other way that it might occur.)  In
Sections \ref{genus one} and \ref{spectral}, we will show in detail
how a geometrical decomposition occurs for $^L\neg G=SO_3$. In this
example, $\MH(C;{}^L\neg G)$ contains $A_1$ singularities,
corresponding to local systems with automorphism group $\Z_2$. (These
are the simplest examples of endoscopic local systems.) Let $r$ be one
of those singularities. Up to isomorphism, there are two irreducible
branes supported at $r$, say $\B_+$ and $\B_-$, corresponding to
one-dimensional representations of $\Z_2$ in which the non-trivial
element acts by $+1$ or $-1$. The decomposition (\ref{omilc}) reads
\begin{equation}\label{imilc}\B^*\to \B_+\oplus
\B_-.\end{equation} Correspondingly, the relevant Hitchin fiber
for $G=SL_2$ should decompose as a sum of two components.

We will see this explicitly for $G=SL_2$ in \secref{genus one} in
genus one and in \secref{spectral} in higher genus.  Above just
those points in the base $\BB$ of the Hitchin fibration at which the
$^L\neg G=SO_3$ moduli space contains a singularity, the fiber of
the Hitchin fibration for $G=SL_2$ decomposes as a union of two
components $\FF_1$ and $\FF_2$, meeting at two double points. In the
example studied in \secref{genus one} (genus one with one
ramification point), $\FF_1$ and $\FF_2$ are both isomorphic to
$\mbb{CP}^1$.  They are smooth and are smoothly embedded in
$\MH(C;G)$ and are each $A$-branes in their own right. Our proposal
is that {\em the two fractional branes $\B_+$ and $\B_-$ supported
at the orbifold singularity $r\in \MH(C;{}^L\neg G)$ map under
duality to the $A$-branes $\A_1$ and $\A_2$ supported on $\FF_1$ and
$\FF_2$}. (Which of $\B_+$ and $\B_-$ maps to $\A_1$ and which to
$\A_2$ is a slightly subtle question that will be discussed in
Section \ref{gerbes}.)

\newpage

\vspace*{-30mm}

\begin{center}
\setlength{\unitlength}{0.8mm}
\begin{picture}(30,30)(-70,90)\label{kart}
\allinethickness{1.5pt}
\put(-24.5,35){\circle*{2}}     
\put(-5,35){\circle{40}}
\put(5,35){\circle{60}}

\thinlines
\qbezier[15](15,35)(28,40)(35,35)
\qbezier(15,35)(22,30)(35,35)
\qbezier[15](9,49)(12,58)(20,60)
\qbezier(9,49)(18,52)(20,60)
\qbezier[15](9.5,20.5)(18,18)(20,10)
\qbezier(9.5,20.5)(12,12)(20,10)

\end{picture}


\begin{picture}(90,90)(40,0)
\linethickness{1pt}
\qbezier(45,0)(-60,35)(45,70)
\qbezier(45,0)(10,35)(45,70)
\qbezier(45,0)(150,35)(45,70)
\qbezier(45,0)(80,35)(45,70)
\thinlines
\qbezier[30](-7.5,35)(15,45)(27.5,35)
\qbezier(-7.5,35)(5,25)(27.5,35)
\qbezier[20](10,15)(25,20)(33.5,15)
\qbezier(10,15)(18,10)(33.5,15)
\qbezier[20](10,55)(25,60)(33.5,55)
\qbezier(10,55)(18,50)(33.5,55)
\qbezier[30](62.5,35)(85,45)(97.5,35)
\qbezier(62.5,35)(75,25)(97.5,35)
\qbezier[20](57,15)(72,20)(80,15)
\qbezier(57,15)(65,10)(80,15)
\qbezier[20](57,55)(72,60)(80,55)
\qbezier(57,55)(65,50)(80,55)
\end{picture}
\end{center}

\noindent
\qquad \hspace*{15mm}{\small Singular Hitchin fiber in} \hspace*{39mm}
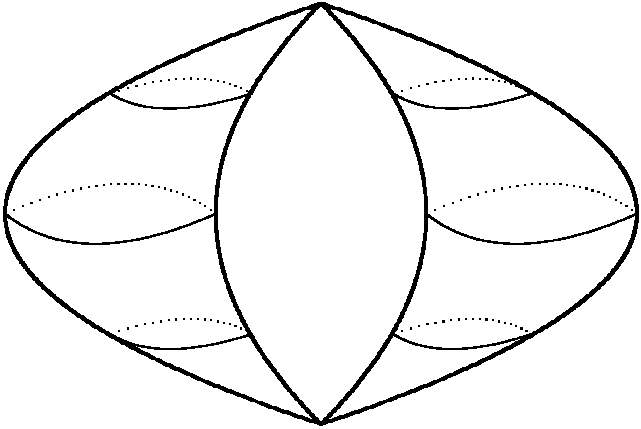
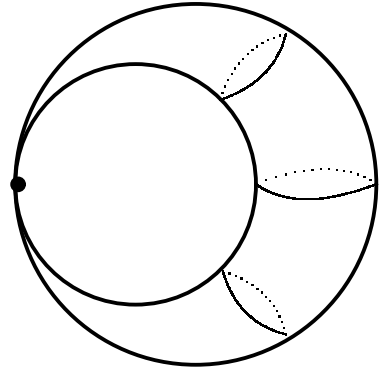
{\small Singular Hitchin fiber in}

\noindent
\qquad \hspace*{15mm}{\small the $A$-model, $G=SL_2$.} \hspace*{40mm}
  {\small the $B$-model, $^L\neg G=SO_3$.}

\vspace*{10mm}

Though the detailed analysis in this paper will be limited to
$\LG=SO_3$, $G=SL_2$, our conjecture is that a similar geometric
description of endoscopy holds for all groups (see \secref{other}).

\subsection{From $A$-Branes To ${\cal D}$-Modules}\label{dmodules}

What we have discussed so far are $A$-branes on $\MH(C;G)$ that are
dual to $B$-branes on $\MH(C;{}^L\neg G)$.  However, the geometric
Langlands dual of a $B$-brane is usually described not as an $A$-brane
but as a twisted ${\cal D}$-module on $\M$, the moduli space of
stable $G$-bundles on the curve $C$ (and, more generally, on $\Bun_G$,
the moduli stack of $G$-bundles on $C$).

The two viewpoints were reconciled in section 11 of \cite{KW}. The
most familiar $A$-branes are branes supported on a Lagrangian
submanifold (such as a fiber of the Hitchin fibration), endowed with
a flat unitary connection. However, the $A$-model on a symplectic
manifold $X$ may in general \cite{KO} have additional branes,
supported on coisotropic submanifolds whose dimension may exceed
half the dimension of $X$. In particular, let $X=T^*Y$ be the
cotangent bundle of a complex manifold $Y$, with the natural
holomorphic two-form $\Omega$. Consider the $A$-model of $X$ with
symplectic form $\omega={\rm Im}\,\Omega$.  The $A$-model admits a
special brane, the canonical coisotropic $A$-brane $\A_{cc}$, whose
support is all of $X$ and whose existence bridges the gap between
$A$-branes and $\D$-modules. The endomorphisms of $\A_{cc}$ (in
physical terms, the $\A_{cc}-\A_{cc}$ strings) can be sheafified
along $Y$ to give a sheaf of rings.  This sheaf of rings is the
sheaf of differential operators acting on $K_Y^{1/2}$, where $K_Y$
is the canonical bundle of $Y$.  We write $\D^*$ for the sheaf of
such differential operators, and refer to a sheaf of modules for
this sheaf of rings (or its generalization introduced below) as a
twisted $\D$-module. Now if $\A$ is any other $A$-brane, then ${\rm
Hom}(\A_{cc},\A)$ can be sheafified along $\M$ to give a twisted
$\D$-module.  The association $\A \to {\rm Hom}(\A_{cc},\A)$ gives a
functor from the category of $A$-branes to the category of twisted
$\D$-modules.

To apply this to $A$-branes on $\MH$, we note that although $\MH$ is
not quite a cotangent bundle, it has a Zariski open set that can be
identified with $T^*\M$, where $\M$ is the moduli space of stable
bundles. An $A$-brane $\A$ on $\MH$ can be restricted to $T^*\M$,
and this restriction is non-empty for dimensional reasons.  Then we
can apply the above construction and associate to $\A$ a twisted
$\D$-module on $\M$.

This construction has an analog in which $X$ is not the cotangent
bundle of $Y$ but an affine symplectic deformation of one. This
means that $X\to Y$ is a bundle of affine spaces, with a holomorphic
symplectic form $\Omega$, such that locally along $Y$, $X$ is
equivalent to $T^*Y$. Such an $X$ is obtained by twisting $T^*Y$ by
an element\footnote{To construct $X$, one pulls back $\chi$ to
$H^1(T^*Y,\Omega^{1,{\cl}}(T^*Y))$, and uses the symplectic form of
$T^*Y$ to map this pullback to $H^1(T^*Y,T(T^*Y))$, which classifies
deformations of $T^*Y$.  The resulting deformation is symplectic
because we start with $\Omega^{1,\cl}$.} $\chi\in
H^1(Y,\Omega^{1,{\cl}})$, where $\Omega^{1,{\cl}}$ is the sheaf of
closed one-forms on $Y$. In this situation, one can still
define\footnote{\label{pp} It is necessary to require that the
cohomology class of ${\rm Re}\,\chi$ is equal to a quantum parameter
called $\eta$ in \cite{GW}.  Except in \secref{zame}, we emphasize
the classical picture and suppress the role of $\eta$.} the
canonical coisotropic brane $\A_{cc}$, and the endomorphisms of
$\A_{cc}$ can still be sheafified along $Y$. The sheaf of rings we
get is now the sheaf of differential operators acting on
$K_Y^{1/2}\otimes \L$, where\footnote{In general, the cohomology
class of $\chi$ is not integral, so $\L$ may be the complex power of
a line bundle (or a tensor product of such) rather than an ordinary
complex line bundle. But the sheaf of differential operators acting
on $\L$, or on $K_Y^{1/2}\otimes \L$, still makes sense.} $\L$ is a
``line bundle'' with $c_1(\L)=\chi$. Hence now we get a functor from
$A$-branes on $Y$ to modules for $\D^*(\L)$, the sheaf of
differential operators acting on $K_Y^{1/2}\otimes \L$.

This construction was applied in \cite{KW} to what mathematically is
known as quantum geometric Langlands.  (Here it is necessary to
consider $X=\MH$ in a complex structure obtained by a hyper-Kahler
rotation of $I$.) More relevant for our purposes, it was applied in
\cite{GW} to the ramified case of geometric Langlands. Here one uses
the fact that in the ramified case, a Zariski open set in $\MH$ can be
identified with an affine symplectic deformation of $T^*\M$.
(This is described in detail for our example in section
\ref{compcar}.) For the ramified case of geometric Langlands, this
construction leads to the following statement: there is a natural
functor from $A$-branes of $\MH(C;G)$ to $\D^*(\L)$-modules on $\M$,
where $\M$ is now the moduli space of stable parabolic $G$-bundles on
$C$ and the first Chern class of $\L$ is the logarithm of the
monodromy of the dual $^L\neg G$ local system.

\subsubsection{Relation To A Local System}\label{furpro}

Now let us discuss how the twisted $\D$-modules arising from
$A$-branes in this situation may be related to local systems. We let
$X$ be $T^*Y$ or an affine deformation thereof, and let $\pi:X\to Y$
be the projection. Let $L$ be a compact (complex) Lagrangian
submanifold of $X$ such that the map $\pi:L\to Y$ is an $n$-fold
cover.  Then the functor from $A$-branes to twisted $\D$-modules is
expected to map a rank 1 $A$-brane supported on $L$ to a local
system on $Y$ of rank $n$ (twisted by $K_Y^{1/2}$), or in other
words to a rank $n$ complex vector bundle $V\to Y$ with a flat
connection (or a connection of central curvature in case of an
affine deformation).

This has an important generalization if $L$ is closed in $X$ but not
compact, and the map $\pi:L\to Y$ is generically $n$ to 1, but is of
lower degree on a divisor. This happens if, intuitively, some
branches of $L$ go to infinity over the divisor in question. For
simplicity, suppose  that $Y$ is a curve, as is actually the case in
the example of \secref{genus one}.  Let $u$ be a local parameter on
$Y$, and let $s$ be a function linear on the fibers of $X\to Y$ such
that locally along $Y$ the symplectic form of $X$ is
$\Omega=du\wedge ds$. Then $L$ can be described locally by an
$n$-valued function $s(u)$ and the situation that we are interested
in is that some branches of this function are singular at a point
$r\in Y$ corresponding to, say, $u=u_0$.  We assume that the $n$
branches look like
\begin{equation}\label{normo}s_i(u)\sim  c_i
(u-u_0)^{-d_i},~i=1,\dots,n.\end{equation} (The $d_i$ are not
necessarily integers, since the map $\pi:L\to Y$ may be ramified at
$u=u_0$.)  In this situation, an $A$-brane supported on $L$ will map
to a twisted $\D$-module on $Y$ that is represented by a local
system $V\to Y\backslash \{r\}$ with a singularity at $r$. The
nature of the singularity is largely determined by the $d_i$ and
$c_i$.  For example, the condition for a regular singularity is that
$d_i\leq 1$, and the monodromies at a regular singularity are then
largely determined by the $c_i$.  For this reason, we compute the
$c_i$ and $d_i$ for our example in eqn. (\ref{gesto}). In
\secref{genus one A to D}, we explain how these coefficients are
expected to be related to the singularities of the local system. We
also describe the generalization to higher dimensions.

\subsubsection{Eigenbranes and Eigensheaves}
Mirror symmetry of $\MH$ has many special properties related to its
origin in four dimensions.  For example, the Wilson and 't Hooft
line operators of four-dimensional gauge theory can be reinterpreted
in two-dimensional terms \cite{KW} and are essential for
understanding the Hecke operators of the geometric Langlands
program.

The correspondence from $A$-branes to ${\mc D}$-modules  should map an
$A$-brane which is an eigenbrane of the 't Hooft operators to a ${\mc
D}$-module on ${\mc M}$ (and more generally, on $\Bun_G$, the moduli
stack of $G$-bundles on our curve $C$) which is a Hecke eigensheaf.
These Hecke eigensheaves are the main objects of interest in the
geometric Langlands correspondence (in its usual formulation). The
geometric Langlands conjecture predicts that to each $\LG$-local
system $\E$ on $C$, one may associate a category ${\mc Aut}_E$ of
Hecke eigensheaves on $\Bun_G$. In particular, if the group of
automorphisms of the local system is trivial, then it is expected
that this category is equivalent to the category of vector spaces.
In other words, it contains a unique irreducible object, and all
other objects are direct sums of copies of this object. The
challenge is to describe what happens for local systems with
non-trivial groups of automorphisms.

However, this is rather difficult to do using the language of ${\mc
D}$-modules. In those cases in which Hecke eigensheaves have been
constructed explicitly (for instance, for $G=SL_n$), their structure
is notoriously complicated. This makes it  difficult to extract
useful information. The language of $A$-branes, on the other hand,
is much better adapted to analyzing the structure of the
corresponding categories of eigenbranes. As discussed above, the
generic eigenbranes are unitary flat local systems on the smooth
Hitchin fibers, and the special eigenbranes associated with
endoscopy are supported on the singular Hitchin fibers. It turns
out that one can describe these singular fibers, and hence the
corresponding eigenbranes, very explicitly. For instance, we observe
that the eigenbranes supported on the singular fibers break into
pieces, the ``fractional branes'' discussed earlier. Furthermore,
these fractional branes satisfy a certain modification of the
eigenbrane property discussed in \secref{Hecke} (the fractional
eigenbrane property).

We then translate these results to the language of ${\mc
D}$-modules. Thus, if an eigenbrane ${\mc A}$ decomposes into two
irreducible branes ${\mc A}_1$ and ${\mc A}_2$, then it is natural to
predict that the corresponding Hecke eigensheaf ${\mc F}$ will also
decompose as a direct sum of two ${\mc D}$-modules, ${\mc F}_1$ and
${\mc F}_2$, corresponding to ${\mc A}_1$ and ${\mc A}_2$,
respectively.  Furthermore, these two ${\mc D}$-modules should then
satisfy a fractional Hecke eigensheaf property described in
\secref{categories}.

The upshot of all this is that by analyzing the categories of $A$-branes
supported on the singular Hitchin fibers, we gain a lot of insight
into the geometric Langlands correspondence, and, hopefully, even into
the classical Langlands correspondence for curves over finite
fields. We will describe this in detail in explicit examples presented
below.

\section{Explicit Example In Genus One}    \label{genus one}

\subsection{Higgs Bundles In Genus One}\label{higgsbun}
 To construct an explicit example in which we can see
the geometric analog of endoscopy, we take $G=SL_2$, and we work on
a Riemann surface $C$ of genus $g_C=1$, with a single point $p$ of
ramification. One might think that ramification would bring an extra
complication, but actually, the case of genus 1 with a single
ramification point is particularly simple, and has often been
considered in the literature on Hitchin fibrations. For $g_C\geq 2$,
the relevant moduli spaces have higher dimension and explicit
computation is difficult; for $g_C\leq 1$, the fundamental group of
$C$ is abelian (or trivial), which in the absence of ramification
leads to complications, unrelated to endoscopy, that we prefer to
avoid here.

As we explain in \secref{Hecke}, for $g_C>1$, though explicit
computation is difficult, the method of spectral curves is a
powerful substitute.  But we prefer to begin with the case of
$g_C=1$ for which everything can be computed directly.

First we describe Higgs bundles without ramification on a Riemann
surface $C$ of genus 1. We begin with ordinary $SL_2$ Higgs
bundles, that is pairs $(E,\varphi)$ where $E$ is a rank 2 bundle
of trivial determinant. If $E$ is semi-stable, it must have the
form
\begin{equation} E=\L\oplus\L^{-1}\end{equation}
where $\L$ is a complex line bundle of degree 0.  If $\L$ is
non-trivial, $\phi$ must be in this basis
\begin{equation} \phi=\begin{pmatrix} a & 0 \\ 0 & -a
\end{pmatrix}\end{equation}
where $a$ is an ordinary holomorphic differential on $C$. The
choice of $\L$ is parametrized by a curve $C'$ (the Jacobian of
$C$) that is isomorphic to $C$, and the space of holomorphic
differentials is one-dimensional.  So we have constructed a family
of semi-stable Higgs bundles parametrized by $C'\times \C $, where
the choice of $\L$ gives a point in $C'$ and the choice of $a$
gives a point in $\C$.  (We have implicitly picked a particular
holomorphic differential on $C$ to identify the space of such
differentials with $\C$.) However, replacing $\L$ by $\L^{-1}$ and
changing the sign of $a$ gives back the same Higgs bundle, up to
isomorphism.  This operation can be understood as a gauge
transformation
\begin{equation}\begin{pmatrix}0 & 1\\ -1& 0
\end{pmatrix}.\end{equation}
So we can take the quotient by $\Z_2$ and we get a family of
semi-stable Higgs bundles parametrized by $(C'\times \C)/\Z_2$. This
actually is the moduli space  of rank two semi-stable Higgs bundles
over $C$ of trivial determinant.

\subsection{Ramification}\label{ram}
Now let us incorporate ramification.  In the context of Higgs
bundles, ramification means \cite{Sim} that $\varphi$ may have a
pole at a prescribed point $p\in C$ (or more generally at several
such points) and with a prescribed characteristic polynomial of
the polar part. We will consider the case of a simple pole. In
addition, in the fiber $E_p$ of the bundle $E$ at $p$, one is
given a $\varphi$-invariant parabolic structure, that is, a flag
that is invariant under the action of the polar part of $\varphi$.
This flag, moreover, is endowed with parabolic weights.  The whole
structure can be described uniformly by adapting Hitchin's
equations  to incorporate singularities. The behavior near $p$ is
determined by parameters $\alpha, \beta,\gamma$ valued in the Lie
algebra $\mf t$ of a maximal torus of the compact form of $G$.
Let $z$ be a local parameter near $p$ and write $z=re^{i\theta}$.
In the notation of \cite{GW}, the local behavior near $p$ is
\begin{align} \label{horseg} \notag  A & = \alpha \, d\theta+\dots\\
\phi& = \beta\frac{dr}{r}-\gamma d\theta+\dots.  \end{align}

Since the Higgs field $\varphi$ is simply the $ (1,0)$ part of
$\phi$, these equations immediately determine its polar behavior,
up to gauge-equivalence or conjugacy.  One has $\varphi\sim \sigma
\,dz/z$, where $\sigma=(\beta+i\gamma)/2$. The $\varphi$-invariant
parabolic structure is completely determined by $\varphi$ and a
choice of Borel subgroup containing $\sigma$. If $\sigma$ is
regular, as we will generally assume, there are only finitely many
choices of Borel subgroup containing $\sigma$, and a choice can be
made by ordering the eigenvalues of $\sigma$.  Consequently, for
studying ramified Higgs bundles with regular $\sigma$, the
parabolic structure (or the choice of $\alpha$) need not be
specified explicitly.

Similarly, we can determine the monodromy around $p$ of the local
system with connection $\CA=A+i\phi$. It is
\begin{equation}\label{monodromy}
M=\exp(-2\pi(\alpha-i\gamma)).\end{equation}

Now let us specialize to the case of a genus 1 curve $C$ with one
ramification point.  It is convenient to describe $C$ explicitly
by an algebraic equation
\begin{equation} y^2=f(x)\end{equation}
where we can take $f$ to be of the form
\begin{equation}\label{curl}f(x)=x^3+ax+b\end{equation}
and interpret $p$ as the point at infinity. We assume that $f$ has
distinct roots, so that $C$ is smooth.

To explicitly describe ramified Higgs bundles $(E,\varphi)$, we
first, as before, take $E=\L\oplus \L^{-1}$, where $\L$ is of
degree zero. We can take $\L=\CO(p)\otimes \CO(q)^{-1}$, where $q$
is some point in $C$, corresponding to $(x,y)=(x_0,y_0)$.  A Higgs
field will be of the form
\begin{equation}\label{conagain} \varphi=\begin{pmatrix} h & k \\ g &
-h\end{pmatrix},\end{equation} where $h$ is a section of $K\otimes
\O(p)$, $k$ is a section of $K\otimes \O(p)\otimes \L^2$, and $g$
is a section of $K\otimes \O(p)\otimes \L^{-2}$.  Of course, for
$C$ of genus 1, $K$ is trivialized by the differential $dx/y$. The
relevant line bundles all have one-dimensional spaces of
holomorphic sections, and the general forms for $k, g$, and $h$
are
\begin{align}\label{cute}\notag h & = \frac{dx}{  y}h_0  \\
              k & = \frac{dx}{ y}
              k_0\left((y-y_0)-\frac{f'(x_0)(x-x_0
              )}{2y_0}\right)\\ \notag
              g&=\frac{dx}{ y}g_0\left((y+y_0)+\frac{f'(x_0)(x-x_0
              )}{2y_0}\right)\frac{1}{(x-x_0)^2}.\end{align}
where $h_0,k_0$, and $g_0$ are complex constants, and the formulas
were obtained as follows. The formula for $h$ requires no
explanation. After trivializing $K$, $k$ is supposed to be a
holomorphic section of $\O(p)\otimes \L^2=\O(p)^3\otimes
\O(q)^{-2}$, so we have looked for a function that is holomorphic
except for a triple pole at infinity, and moreover has a double
zero at $q$.  Such a function is $(y-y_0)-f'(x_0)(x-x_0)/2y_0$.
Similarly, $g$ is supposed to be a holomorphic section of
$\O(p)\otimes \L^{-2}=\O(p)^{-1}\otimes \O(q)^2$, so we have
looked for a function that is holomorphic except for a double pole
at $q$ and vanishes at infinity.  Such a function is
$\left((y+y_0)+\frac{f'(x_0)(x-x_0
              )}{2y_0}\right)/(x-x_0)^2$.

The next step is to evaluate the characteristic polynomial of
$\varphi$.  For $G=SL_2$, this simply means that we should
compute $\Tr\,\varphi^2$, which in the present case turns out to
be
\begin{equation}\Tr\,\varphi^2=\left(\frac{dx}{y}\right)^2
\left(2h_0^2+2k_0g_0
\left(x+2x_0-\frac{f'(x_0)^2}{4f(x_0)}\right)
\right)
\end{equation}
The polar part of $\Tr\,\varphi^2$ is simply $(dx/y)^2 2xk_0g_0$,
which has a double pole at infinity.  If $z$ is a local parameter
at infinity, we have $x\sim z^{-2}$, $ y\sim z^{-3}$ and
$(dx/y)^2x\sim 4(dz/z)^2$.  The polar part of $\varphi $ is
supposed to be conjugate to $\sigma dz/z$, so we want
$\Tr\,\varphi^2\sim \Tr\,\sigma^2(dz/z)^2=2\sigma_0^2(dz/z)^2$,
where we denote the eigenvalues of $\sigma$ as $\pm \sigma_0$.  So
we set $k_0g_0=\sigma_0^2/4$, and write
\begin{equation}\label{ulg}\Tr\,\varphi^2=
\left(\frac{dx}{y}\right)^2\left(2h_0^2+\frac{\sigma_0^2}{2}
\left(x+2x_0-\frac{f'(x_0)^2}{4f(x_0)}\right) \right)
\end{equation}
The reason that only the product $k_0g_0$ is determined is that
the bundle $E=\L\oplus \L^{-1}$ has an automorphism group
$\C^\times$, acting on the two summands as multiplication by $
\lambda$ and $\lambda^{-1}$ respectively, and transforming $k_0$
and $g_0$ by $k_0\to \lambda^2k_0$, $g_0\to \lambda^{-2}g_0$.

The formula (\ref{cute}) breaks down if $y_0=0$ (because $y_0$
appears in the denominator in the formulas for $k$ and $g$) or
$y_0=\infty$ (since the formulas for $k$ and $g$ also contain
terms linear in $y_0$), or equivalently if $\L$ is of order 2.
This happens because when $\L$ is of order 2, we have
$\L\cong\L^{-1}$ and (if $\sigma\not=0$) there does not exist a
stable or semi-stable ramified Higgs bundle $(E,\varphi)$ with
$E=\L\oplus \L^{-1}=\L\oplus \L$. Indeed, with that choice of $E$,
a holomorphic section of $K\otimes {\rm ad}(E)\otimes \O(p)\cong
K\otimes \O(p)^{\oplus 3}$ cannot have a pole at $p$, and hence
the condition $\Tr\,\varphi^2\sim 2\sigma_0^2(dz/z)^2$ cannot be
obeyed. Instead, if $\L$ is of order 2, and $\sigma\not=0$, one
must take $E$ to be a non-trivial extension of $\L$ by $\L$.

\subsection{The Moduli Space}\label{modspace}

For each choice of the parameter $\sigma_0^2$, we have constructed
a family of ramified Higgs bundles.  The underlying bundle is
$E=\L\oplus \L^{-1}$, and the Higgs field $\varphi$ has been
described above. The choice of $E$ depends on the point $q$ or
equivalently the pair $(x_0,y_0)$ (with $y_0^2=f(x_0)$), and the
choice of $\varphi$ depends additionally on the parameter $h_0$.
To construct in this situation the moduli space $\MH$ of ramified
Higgs bundles, we must take account of the exchange $\tau:\L\to
\L^{-1}$, which acts by $(x_0,y_0,h_0)\to (x_0,-y_0,-h_0)$.

The pair $(q,h_0)$ or triple $(x_0,y_0,h_0)$ defines a point in
$C\times \C$, and after allowing for the symmetry $\tau$, it seems
that the moduli space is $(C\times \C)/\Z_2$, independent of
$\sigma_0$.  This is actually not quite correct, because of the
point made at the end of Section \ref{ram}.  The space $(C\times
\C)/\Z_2$ has four $A_1$ singularities, at points with $h_0=0$ and
$q$ of order 2.   For $\sigma\not=0$, the $A_1$ singularities are
deformed and the moduli space becomes smooth.  We have not seen the
deformation because our analysis does not cover the case that $q$ is
of order 2.  If $\sigma=0$ and $\alpha\not=0$, the $A_1$
singularities are resolved rather than deformed; from a hyper-Kahler
point of view, the phenomenon, for generic values of $\sigma$ and
$\alpha$, is really a simultaneous deformation and resolution, as in
\cite{Kr}.  In this paper, we will not describe this deformation or
resolution directly, but in equation (\ref{elmy}) below, the
deformed moduli space is described, with the aid of the Hitchin
fibration.

Perhaps we should make a comment here on the role of ramification in
genus 1. For $g_C>1$, a generic semi-stable bundle is actually stable,
but for $g_C=1$ (and no ramification) and simply-connected $G$, there
are no strictly stable bundles. The closest one can come is a
semi-stable bundle, such as $E=\L\oplus \L^{-1}$ for
$G=SL_2$. Likewise, in general $\MH$ parametrizes stable and
semi-stable Higgs bundles. For $g_C=1$, unramified Higgs bundles are at
best semi-stable, so $\MH$ actually parametrizes semi-stable Higgs
bundles. The situation changes with ramification. Semi-stable Higgs
bundles with one ramification point are actually stable (if
$\sigma\not=0$). The reason for this is that the potential
destabilizing sheaves of $E=\L\oplus \L^{-1}$ (namely the summands
$\L$ and $\L^{-1}$) are not $\varphi$-invariant, and so do not
contradict stability of the Higgs bundle $(E,\varphi)$.

\subsection{The Hitchin Fibration}\label{hitchfib}

Next we need to understand the Hitchin fibration.  The Hitchin
fibration is simply the map that takes a Higgs bundle
$(E,\varphi)$ to the characteristic polynomial of $\varphi$.  For
$G=SL_2$, this characteristic polynomial reduces to
$\Tr\,\varphi^2$. In the present context, according to
(\ref{ulg}), that characteristic polynomial is a polynomial in $x$
of degree 1 (times $(dx/y)^ 2$).  The coefficient of the linear
term in $x$ is fixed, and the Hitchin fibration is the map that
extracts the constant term, which we will call $w_0$. In other
words, the Hitchin fibration maps the Higgs bundle $(E,\varphi)$
to
\begin{equation}\label{zelg}w_0= 2h_0^2+\frac{\sigma_0^2}{2}
\left(2x_0-\frac{f'(x_0)^2}{4f(x_0)}\right).\end{equation}

 The
fiber of the Hitchin fibration is described by variables
$x_0,y_0,h_0$ obeying (\ref{zelg}) and
\begin{equation}\label{gelg}y_0^2=f(x_0)\end{equation}
 and subject also to the
$\Z_2$ symmetry $(x_0,y_0,h_0)\to (x_0,-y_0,-h_0)$.  Apart from
$x_0$, the $\Z_2$ invariants are $y_0^2 $, $h_0^2$, and
\begin{equation}\label{helpme}
\rho=(2/\sigma_0)y_0h_0.\end{equation} (The factor of $2/\sigma_0$
has been included for convenience.) Of these, $y_0^2$ and $h_0^2$
are equivalent to rational functions of $x_0$ according to the
last two equations, and so can be omitted, while $\rho$ obeys a
quadratic equation. If henceforth we write $u$ for $x_0$, and set
$w_0=-w\sigma_0^2/2$, then the equation obeyed by $\rho$ is
\begin{equation}\label{elmy}\rho^2=-
(2u+w)f(u)+\frac{f'(u)^2}{4}.\end{equation}

This equation for complex variables $\rho,u,w$ describes a complex
surface which is the moduli space $\MH$ of ramified $SL_2$ Higgs
bundles. (Some points at $u,\rho=\infty$ are omitted in this way
of writing the equation.) There is a simple explanation for why it
does not depend on the parameters $\alpha,\beta,\gamma$ that
characterize ramification.  Complex structure $I$ does not depend
on $\alpha$, and as long as there is only one ramification point,
$\sigma=(\beta+i\gamma)/2$ can be eliminated by rescaling
$\varphi$ (provided it is not zero), as we have done.

If we set $w$ to a fixed complex number, we get an algebraic curve
$\FF_w$ which is the fiber of the Hitchin fibration. The right hand
side is a quartic polynomial in $u$ and, for generic $w$, $\FF_w$ is
a smooth curve of genus 1.

When is $\FF_w$ singular?  This occurs precisely if two roots of the
polynomial $g(u)=(2u+w)f(u)-f'(u)^2/4$ coincide, or in other words
if its discriminant vanishes. Let $e_1, e_2, $ and $e_3$ be the
roots of the polynomial $f(u)=u^3+au+b$.  (Of course,
$e_1+e_2+e_3=0$, and we assume that the $e_i$ are distinct so that
$C$ is smooth.) The discriminant of $g$ is
$(e_1-e_2)^2(e_2-e_3)^2(e_3-e_1)^2(w-e_1)^2(w-e_2)^2(w-e_3)^2$, so
 $\FF_w$ is singular precisely if $w$ is equal to one of the $e_i$.
For example, if $w=e_1$, we find that
\begin{equation}\label{gelmy}
-(2u+w)f(u)+\frac{f'(u)^2}{4}=\frac{1}{4}\left((u-e_1)^2-(e_1-e_2)(e_1-e_3)
\right)^2.\end{equation} Of course, there are similar formulas if
$w=e_2$ or $w=e_3$. The fact that the left hand side of
(\ref{gelmy}) is a perfect square means that at $w=e_1$, the curve
$\FF_{w}$ splits as a union of two components $\FF_{e_1}^\pm$ defined by
\begin{equation}\label{zelmy}\rho=\pm\frac{1}{2}
\left((u-e_1)^2-(e_1-e_2)(e_1-e_3) \right).\end{equation}

The curves $\FF_{e_1}^\pm$ are each of genus 0.  They meet at the
two points given by
\begin{equation}\label{kondo}\rho=0,~~u=e_1 \pm\sqrt{(e_1-e_2)(e_1-e_3)}.
\end{equation}
A pair of genus 0 curves meeting at two double points gives a curve of
arithmetic genus 1.  A smooth curve of genus 1 can degenerate to such
a singular curve.  This is the behavior of the Hitchin fibration in
our example at the three fibers $w=e_1,e_2$, and $e_3$. These singular
fibers are shown on the left picture on page \pageref{kart}.  The fact
that there are two double points when $w$ equals one of the $e_i$ is
the reason that the discriminant of $g$ has a double zero at those
values of $w$.

Although the fiber $\FF_w$ of the Hitchin fibration is singular when
$w=e_i$, the moduli space $\MH$ of stable ramified $SL_2$ Higgs
bundles over $C$ is actually smooth.  Indeed, since $f\not=0$ at
the points $\rho=0$, $w=e_i$, the polynomial
$\rho^2+(2u+w)f(u)-f'(u)^2/4$ whose vanishing characterizes $\MH$
has a nonzero differential at those points.  So $\MH(C;SL_2)$ is
smooth even though the fibers of the Hitchin fibration are
singular.

What we have found is precisely the behavior that was promised in
\secref{braneduals}.  Certain fibers of the Hitchin fibration split
up as the union of two components $\FF^\pm$, leading to a
decomposition of $A$-branes.  In \secref{struc}, we explain how this
is related to singularities of the moduli space for the dual group.

\subsection{Symmetry Group}\label{symgroup}

An $SL_2$ Higgs bundle $(E,\varphi)$ can be twisted by a line
bundle $\N$ of order 2.  To be more precise, this operation is
$E\to E\otimes \N$, $\varphi\to\varphi$.  For $C$ of genus 1, the
group of line bundles of order 2 is $Q=\Z_2\times \Z_2$, and this
group must act on the moduli space $\MH$ of ramified stable Higgs
bundles.

The underlying bundle $E$ is $E=\L\oplus \L^{-1}$, where modulo the
exchange $\L\leftrightarrow \L^{-1}$, $\L$ is parametrized by
$u=x_0$.  The group $Q$ acts on $E$ by holomorphic automorphisms of
the moduli space of semi-stable bundles.  Such an automorphism
is a fractional linear transformation of the $u$-plane.  Moreover,
the condition for $\L$ to be of order 2 is invariant under the
action of $Q$.

$\L$ is trivial if $u=\infty$ and is a nontrivial line bundle of
order 2 if $u=e_1,e_2, $ or $e_3$.  Twisting with a line bundle of
order 2 therefore exchanges $u=\infty$ with one of the values
$u=e_i$ while also exchanging the other two $e$'s.  For example,
there is an element $T_1\in Q$ that exchanges $\infty$ with $e_1$
and also exchanges $e_2$ with $e_3$.  It acts by
\begin{align}\notag  u& \to \frac{e_1u+e_2e_3-e_1e_3-e_1e_2}{u-e_1}\\
               w& \to w\\ \notag
               \rho& \to
               -\frac{(e_1-e_2)(e_1-e_3)}{(u-e_1)^2}\rho.\end{align}
The fact that $w$ is invariant reflects the fact that $\varphi$,
and therefore its characteristic polynomial, is invariant under
twisting by a line bundle of order 2.  The sign in the
transformation of $\rho$ is not obvious (the equation (\ref{elmy})
that defines $\MH$ is invariant under $\rho\to -\rho$) and can be
determined from the fact that the holomorphic two-form $\Omega_I$
of $\MH$ is $Q$-invariant. $Q$ also has elements $T_2$ and $T_3$
that are obtained by cyclic permutation of $e_1,e_2,e_3$. These
are the three non-trivial elements of $Q$.

Since $w$ is $Q$-invariant, the group $Q$ acts on each fiber $\FF_w$
of the Hitchin fibration.  Let us describe the action on the
singular fiber at, say, $w=e_1$.  A short calculation with the
above formulas shows that $T_1$ maps each component $\FF^\pm_{e_1}$
to itself, and leaves fixed the two double points of eqn.
(\ref{kondo}). The curves $\FF^\pm_{e_1}$ have genus zero, and $T_1$
acts on each of these genus zero curves as a fractional linear
transformation with two fixed points, namely the double points.
(For an involution of $\mbb{CP}^1$ with two fixed points, consider
the transformation $z\to -z$ of the complex $z$-plane, with fixed
points at 0 and $\infty$.) On the other hand, $T_2$ and $T_3$
exchange the two components $\FF^\pm_{e_1}$ and also exchange the
two double points. This being the case, $T_2$ and $T_3$ act freely
on the singular curve $\FF_{e_1}$, and therefore also on all nearby
fibers of the Hitchin fibration.

\remark $Q$ acts on each fiber  $\FF_w$ of the Hitchin fibration by
translation by the group of points of order 2.  This is clear in the
spectral curve construction that is described in \secref{Hecke}.
Hence when $\FF_w$ is smooth, $Q$ acts freely.  When $\FF_w$ is
degenerating to a union of two components $\FF^\pm_w$, it has a
``short'' direction corresponding to a one-cycle that collapses at a
double point and a complementary ``long'' direction.  $T_2$ and
$T_3$ are translations in the long direction; they act freely and
exchange the two double points and the two components.  $T_1$ is a
translation in the short direction, maps each component to itself,
and has the double points as fixed points.

\subsection{Langlands Dual Group}\label{struc}

The Langlands dual group of $G=SL_2$ is $^L\neg G=SO_3$ or
equivalently $PGL_2$. We would therefore also like to understand
Higgs bundles on $C$ with structure group $SO_3$.  These are closely
related to $SL_2$ Higgs bundles, because $SO_3=SL_2/\Z_2$ is the
adjoint form of $SL_2$.

We will let $W$ denote a holomorphic $SO_3$ bundle, that is a rank
three holomorphic bundle with a nondegenerate holomorphic quadratic
form and volume form.  Since the three-dimensional representation of
$SO_3$ is the adjoint representation, we need not distinguish
between $W$ and the corresponding adjoint bundle.

The moduli space of $SO_3$ Higgs bundles $(W,\varphi)$ over a
Riemann surface $C$ has two components, distinguished\footnote{See
remark \ref{secondremark} for a more precise statement.} by the
second Stieffel-Whitney class $w_2(W)$ of the underlying bundle $W$.
We denote these components as $\MH(C;SO_3,w_2)$, where $w_2$ is
either 0 or is the nonzero element of $H^2(C,\Z_2)\cong \Z_2$, which
we call $\theta$. Let us first consider the case that $w_2(W)=0$.

The structure group of an $SO_3$-bundle $W$ with $w_2=0$ can be
lifted to $SL_2$, and we can then form an associated rank two bundle
$E$, with structure group $SL_2$. $E$ is uniquely determined up to
twisting by a line bundle of order 2. Consequently, there is a very
simple relation between the moduli space $\MH(C;SO_3,0)$ of $SO(3)$
Higgs bundles with vanishing $w_2$ and the corresponding $SL_2$
moduli space $\MH(C;SL_2)$:
\begin{equation}\label{turgid}\MH(C;SO_3,0)=\MH(C;SL_2)/Q.\end{equation}

We can immediately use our knowledge of the action of $Q$ to describe
the fiber of the Hitchin fibration of $\MH(C;SO_3,0)$ at the special
points $w=e_1,e_2,e_3$.  At, say, $w=e_1$, the effect of dividing by
$T_2$ (or $T_3$) is to identify the two components $\FF^\pm_{e_1}$.
So we can just focus on one of them, say $\FF^+_{e_1}$.  It is a curve
of genus 0 with two points identified (namely the double points that
are exchanged by $T_2$). We still must divide by $T_1$; the quotient
is again a curve of genus 0 with a double point.  So that is the
nature of the exceptional fibers of the Hitchin fibration for $^L\neg
G=SO_3$. They are shown on the right picture on page
\pageref{kart}. In particular, the special Hitchin fibers for $SO_3$
are irreducible, while those for $SL_2$ have two components.

Now we come to another crucial difference between $SL_2$ and
$SO_3$. In the case of $SL_2$, though some fibers of the Hitchin
fibration are singular, the singularities of the fibers are not
singularities of $\MH$; $\MH$ is smooth near the exceptional
fibers (and in fact everywhere, for generic $\sigma$).

But dividing $\MH(C;SL_2)$ by $Q=\Z_2\times \Z_2$ creates a
singularity.  $T_2$ acts freely on the fiber $\FF_{e_1}$, and
therefore on a small neighborhood of it.  But $T_1$ leaves fixed
the two double points of the fiber.  Dividing by $T_1$ therefore
creates two $A_1$ singularities.  These are exchanged by the
action of $T_2$, and therefore the quotient $\MH(C;SO_3,0)$ has
one $A_1$ singularity on each exceptional fiber.

\subsubsection{Relation to Geometric Endoscopy}
What we have just described is the picture promised in
\secref{braneduals}.  For $SO_3$, the special fiber of the Hitchin
fibration has an $A_1$ singularity at a point $r$.  A $B$-brane ${\cal
B}^*$ supported at a generic point $r^*$ is irreducible, but for
$r^*=r$, it can split up as a sum ${\cal B}={\cal B}_+\oplus{\cal
B}_-$. Dually, the special fiber of the Hitchin fibration for $SL_2$
has two components, so that an $A$-brane supported on this fiber can
split up as a sum of two $A$-branes, each supported on one
component. Since the components are simply-connected, an $A$-brane of
rank 1 supported on one of them has no moduli.  This is dual to the
fact that the fractional branes ${\cal B}_+$ and ${\cal B}_-$ have no
moduli. But the brane ${\cal B}={\cal B}_+\oplus {\cal B}_-$ has
moduli (since it can be deformed away from the singularity), and
dually, the sum $\A$ of the two $A$-branes can similarly be
deformed.

If we deform $\B$ to a skyscraper sheaf supported at a smooth point
$r^*$ of a nearby Hitchin fiber, then $\A$ deforms to a rank 1
$A$-brane supported on the dual Hitchin fiber, with a flat unitary
line bundle determined by $r^*$. It is also instructive to see what
happens if we deform $\B$ to a smooth point $r^*$ of the same
singular fiber. Then the dual $A$-brane is a flat unitary line
bundle on the same dual (singular) Hitchin fiber. However, this flat
line bundle now has non-trivial monodromies linking the two
irreducible components of the fiber, and therefore they can no
longer be ``pulled apart''. In other words, a generic rank 1
$A$-brane on the singular Hitchin fiber is actually irreducible, in
agreement with the fact that it corresponds to a rank 1 $B$-brane
supported at a smooth point, which is also irreducible.

\subsubsection{Relation Between The Two Components}\label{relcom}
As hyper-Kahler manifolds, the two components $\MH(SO_3,0)$ and
$\MH(SO_3,\theta)$ are distinct.  But if we view them purely as
complex symplectic manifolds in complex structure $I$, then they are
actually isomorphic as long as $\sigma\not=0$.

One can map between them by making a $\varphi$-invariant Hecke
modification at the point $p$.  This concept is developed more fully
in \secref{modif}.  In brief, decompose $E$ near $p$ as $\L_1\oplus
\L_2$ where $\L_1$ and $\L_2$ are line bundles that are
$\varphi$-invariant, in the sense that $\varphi:E\to E\otimes K$
maps $\L_i\to \L_i\otimes K$, for $i=1,2$.  Order the $\L_i$ so that
the fiber of $\L_1$ at $p$ is the eigenspace of $\sigma$ with
eigenvalue $+\sigma_0$.  (This is the step that requires
$\sigma\not=0$.)  Consider the operation that leaves $(E,\varphi)$
unchanged away from $p$ and acts near $p$ as $\L_1\oplus \L_2\to
\L_1(p)\oplus\L_2$.  When viewed as a transformation of the $SO_3$
Higgs bundle $(W,\varphi)$ (with $W={\rm ad}(E)$), this operation
establishes the isomorphism between the two components of
$\MH(SO_3)$. This isomorphism commutes with the Hitchin fibration,
since the Hecke modification does not change the characteristic
polynomial of $\varphi$.

Going back to Hitchin's equations for ramified Higgs bundles, with
the singularity postulated in eqn. (\ref{horseg}),  the
$\varphi$-invariant Hecke modification is equivalent to a shift of
$\alpha$ by a lattice vector.  (One can compensate for such a shift
by a gauge transformation that is discontinuous at $r=0$ and has the
effect of changing the natural extension of the bundle over that
point. See Section 2.1 of \cite{GW}.) In general \cite{Nak}, $\MH$
when viewed as a complex symplectic manifold in complex structure
$I$ is independent of $\alpha$ as long as $\sigma$ is regular.  The
equivalence between the two components of $\MH(SO_3)$ is a special
case of this.  For  $\sigma$ non-regular (which for $SO_3$ means
$\sigma=0$), the two components of $\MH(SO_3)$ are birational but
not isomorphic.

For unramified $SO_3$ Higgs bundles, there is no such relation
between the two components.

\subsection{$O_2$-Bundles}\label{otwo}

Our next task is to interpret better the singularities that we have
found in the $SO_3$ moduli space.

In general, for generic\footnote{For  non-regular $\sigma$, $\MH$
also has local singularities described in section 3.6 of \cite{GW}.}
$\sigma$, singularities of the moduli space $\MH$ of stable ramified
Higgs bundles come entirely from automorphisms. If a Higgs bundle
$(E,\varphi)$ has a non-trivial finite automorphism group $\Gamma$,
we should expect the corresponding point in $\MH$ to be a singular
point.  If $\Gamma$ is a finite group, the singularity will be an
orbifold singularity, locally of the form $\C^{2n}/\Gamma$ with some
$n$ and some linear action of $\Gamma$ on $\C^{2n}$.  If $\Gamma$
has positive dimension, the singularity is typically (but not
always) more severe than an orbifold singularity.

In the present case, we have encountered some $A_1$ orbifold
singularities, and this strongly suggests that the corresponding
$SO_3$ Higgs bundles $(W,\varphi)$ have automorphism group
$\Gamma=\Z_2$. To describe Higgs bundles with this automorphism
group, we will use the correspondence between Higgs bundles and
local systems given by Hitchin's equations.  This correspondence
preserves the automorphism group, so a Higgs bundle with
automorphism group $\Z_2$ corresponds to a local system with
automorphism group $\Z_2$.

The reason that it is possible to have an $SO_3$ local system with
automorphism group $\Z_2$ is that $SO_3$ contains the subgroup
$O_2$, consisting of $SO_3$ elements of the form
\begin{equation}\begin{pmatrix}* & * & 0 \\ * & * & 0 \\ 0 & 0 &
\pm 1\end{pmatrix},\end{equation} where the upper left block is an
element of $O_2$, and the sign in the lower right corner is chosen
so that the determinant equals $+1$.  $O_2$ has two components
topologically; the component that is connected to the identity
consists of group elements of lower right matrix entry $+1$, and
the disconnected component consists of elements for which that
entry is $-1$.  The subgroup of $SO_3$ that commutes with $O_2$ is
$\Z_2$, generated by
\begin{equation}\begin{pmatrix}-1 & 0 & 0 \\ 0 & -1 & 0 \\ 0 & 0 &
 1\end{pmatrix}.\end{equation}
 So an $SO_3$ local system whose structure group reduces to $O_2$ but
 which is otherwise generic will have automorphism group $\Z_2$.

 Let us describe local systems on a Riemann surface $C$ of genus
 1, first in the unramified case.  Such a local system is
 determined up to isomorphism by the monodromy elements $V_1$ and
 $V_2$ around two one-cycles in $C$.  As the fundamental group of
 $C$ is abelian, they obey
 \begin{equation}\label{hgo} V_1V_2V_1^{-1}V_2^{-1}=1.\end{equation}
 Now suppose that there is a single ramification point $p$, with a
 specified conjugacy class $M$ for the monodromy around $p$.
Then (\ref{hgo}) is modified to
\begin{equation}\label{ggo}V_1V_2V_1^{-1}V_2^{-1}=M.\end{equation}

In our case, the monodromy is $M=\exp(-2\pi(\alpha-i\gamma))$, as
in in eqn. (\ref{monodromy}).   If $V_1$ and $V_2$ take values in
$O_2$ and obey (\ref{ggo}), then $M$ must take values in the
connected component of $O_2$  and has the form
 \begin{equation}M=\begin{pmatrix}* & * & 0 \\ * & * & 0 \\ 0 & 0 &
1\end{pmatrix}. \end{equation} For such an $M$, we want to find $V_1$
and $V_2$, taking values in $O_2$, and obeying (\ref{ggo}).

An $SO_3$ local system $W$ whose structure group actually reduces
to $O_2$ splits up as $U\oplus \S$, where $W$ is a rank two local
system (with structure group $O_2\subset GL_2$) and $\S\cong
\det\,U$ is a line bundle of order 2.  If $\S$ is trivial, then
$V_1$ and $V_2$ both take values in the connected component of
$O_2$. Since this connected component is the abelian group $SO_2$,
eqn. (\ref{ggo}) is then impossible to obey for $M\not=1$.

So we must take $\S$ to be a non-trivial line bundle of order 2.
This means that either $V_1$ or $V_2$, or both, takes values in
the disconnected component of $O_2$.  We can choose the two
one-cycles with holonomies $V_1$ and $V_2$ so that $V_1$ takes
values in the disconnected component and $V_2$ in the connected
component. Any element of the disconnected component is conjugate
to
\begin{equation}\label{vone}V_1=\begin{pmatrix}1 & 0 & 0 \\ 0 & -1 & 0
\\ 0 & 0 & -1\end{pmatrix}.\end{equation} With this choice of $V_1$, we
have $V_1V_2V_1^{-1}=V_2^{-1}$, so (\ref{ggo}) reduces to
\begin{equation}V_2^2=M^{-1},\end{equation}
where $M$ and $V_2$ take values in $SO_2$.

For each $M$, this last equation has precisely two solutions,
differing (in $SO_2$) by $V_2\to -V_2$. Hence, for each
non-trivial line bundle $\S$ of order 2, we can construct
precisely two $SO_3$ local systems each with a group of
automorphisms $R\cong \Z_2$.

These two local systems, however, differ by the value of $w_2$.
Hence, one of them corresponds to a $\Z_2$ orbifold singularity on
the appropriate fiber of $\MH(C;SO_3,0)$, and one corresponds to
such a singularity on the corresponding fiber of $\MH(C;
SO_3,\theta)$.

To explain why the two choices of $V$ correspond to topologically
inequivalent bundles, consider the special case $M=1$. In this
case, the two possibilities for $V_2$ are
\begin{equation}\label{vtwo}V_2^{(1)}=1,~~V_2^{(2)}=\begin{pmatrix}-1&
0 & 0 \\ 0 & -1 & 0 \\ 0 & 0 & 1\end{pmatrix}.\end{equation} In
either case, $V_1$ and $V_2$ are diagonal, so $W$ splits as a direct
sum of three rank 1 local systems $\S_1$, $\S_2$, $\S_3$. We
therefore have
\begin{equation}\label{quirk}w_2(W)=
\sum_{1\leq i<j\leq 3}w_1(\S_i)\cup w_1(\S_j).\end{equation}

If $V_2=1$ then $\S_1$ is trivial and $\S_2$ and $\S_3$ are
isomorphic.  Hence $w_2(W)=w_1(\S_2)^2$, and this is zero since
$x^2=0$ for any $x\in H^1(C,\Z_2)$.

If $V_2=V_2^{(2)}$, then $\S_1$, $\S_2$, and $\S_3$ are the three
distinct non-trivial line bundles of order 2 over $C$.  If
$x=w_1(\S_1)$, $y=w_1(\S_2)$, then we have $w_1(\S_3)=x+y,$ since
$S_3\cong S_1\otimes S_2$. $x$ and $y$ are a basis for
$H^1(C,\Z_2)\cong\Z_2\oplus \Z_2$, and $\theta=x\cup y$ generates
$H^2(C,\Z_2)\cong \Z_2$. Evaluation of (\ref{quirk}) now gives
$w_2(W)=x\cup y$, so in particular $w_2(W)\not=0$.

Thus, for any choice of non-trivial line bundle ${\S}$ of order 2,
and any choice of $w_2(W)$, we find one $SO_3$ local system with
automorphism group $\Z_2$.  This accounts for the singularities that
we found in Section \ref{struc} for $w_2=0$, as well as the
singularities that we will find for $w_2\not=0$.

\remark\label{firstremark} In the example with $w_2\not=0$, instead
of thinking of $W=\S_1\oplus \S_2\oplus \S_3$ as an $SO_3$ local
system, we can understand it as a flat bundle with structure group
the compact form of $SO_3$, or equivalently, as a stable holomorphic
$SO_3$ bundle. As such, $W$ is the unique topologically non-trivial
flat bundle or stable holomorphic $SO_3$ bundle on a curve of genus
1.  In particular, it has no deformations; for instance,
$H^1(C,W)=\oplus_{i=1}^3 H^1(C,\S_i)=0$, as the cohomology of a
non-trivial line bundle vanishes for $C$ of genus 1.

\remark\label{secondremark}  We have been too cavalier in describing
the distinction between the two components of $\MH(C;SO_3)$. In the
unramified case, $w_2(W)$ is a well-defined topological invariant.
But in the ramified case, the definition of $w_2(W)$ depends on a
choice of extension of $W$ (at least as a topological bundle) over
the ramification point $p$.  The moduli space $\MH(C;SO_3)$ has two
components, each of which has the singularities that we have
described, but it is oversimplified to claim that one corresponds to
$w_2(W)=0$ and one to $w_2(W)\not=0$. Indeed, the two components are
exchanged under monodromy of $M$ in the group $SO_2$, whose
fundamental group is $\Z$.  The monodromy in $M$ can be achieved by
varying the parameter $\alpha$, and thus is equivalent to the
operation described in \secref{relcom}.

\subsubsection{Relation To Geometric Endoscopy}
What we have found is, again, fully in keeping with expectations
from \secref{braneduals}.  An $SO_3$ local system whose structure
group reduces to $O_2$ has automorphism group $\Z_2$.  It
corresponds to an $A_1$ singularity of $\MH(C;SO_3)$. So the
category of $B$-branes associated with this local system is
generated by two objects $\B_+$ and $\B_-$.  By contrast,
$\MH(C;SL_2)$ is smooth; some exceptional fibers of the Hitchin
fibration are singular, but the singularities of the fibers are
smooth points of the total space.

\subsection{Second Component}\label{secomp}
In \secref{otwo}, we have constructed $SO_3$ local systems with
$w_2\not=0$ whose structure group reduces to $O_2$.  These should
correspond to $A_1$ singularities of $\MH(C;SO_3,\theta)$.

It will be useful to give a direct description of this
space. $\MH(C;SO_3,\theta)$ paramet\-rizes ramified Higgs bundles
$(W,\varphi)$, where $W$ is an $SO_3$-bundle with nonzero $w_2$. It
will be sufficient for our purposes to construct a Zariski open set in
this moduli space corresponding to Higgs bundles for which $W$ is
stable. According to remark \ref{firstremark}, this means that
$W=\S_1\oplus \S_2\oplus \S_3$ is the direct sum of the three
non-trivial line bundles of order 2.  We can take $\S_i=\CO(p)\otimes
\CO(q_i)^{-1}$, where $q_i$ is the point of order 2 given by
$(x,y)=(e_i,0)$.  As an $SO_3$ bundle, $W$ should have a nondegenerate
quadratic form, which we will denote as a trace (thinking of $W$ as an
adjoint bundle).  This can be defined as follows: if $s=s_1\oplus
s_2\oplus s_3$ is a section of $W=\S_1\oplus \S_2\oplus \S_3$, then
the quadratic form is $\Tr\,s^2=\sum_i s_i^2/(x-e_i)$. The idea is
that $s_i^2$ is a section of $\CO(p)^2\otimes \CO(q_i)^{-2}$, which
can be trivialized by dividing by $x-e_i$.

The bundle $W$ has an automorphism group $Q=\Z_2\times \Z_2$.
Indeed, as a flat bundle, $W$ has monodromies $V_1$ and
$V_2^{(2)}$ presented in eqns. (\ref{vone}) and (\ref{vtwo}) and
$Q$ is the subgroup of $SO_3$ consisting of diagonal matrices with
diagonal entries $\pm 1$.

To construct $\MH(C;SO_3,\theta)$ (or rather a Zariski open set
corresponding to stable bundles $E$), we must consider all Higgs
fields $\varphi\in H^0(C,K\otimes \CO(p)\otimes W)$, impose a
suitable condition on the polar part of $\Tr\,\varphi^2$, and
divide by $Q$.  Before trying to do this, let us discuss what will
happen if  we do {\it not} divide by $Q$.

{}From the point of view of complex algebraic geometry, instead of
considering $SL_2$ Higgs bundles $(E,\varphi)$ where $E$ is a rank 2
holomorphic bundle with $\det E$ trivial, it is natural to pick a
fixed line bundle ${\cal L}$ and consider Higgs bundles
$(E,\varphi)$ with $\det E={\cal L}$.  If one replaces ${\cal L}$ by
${\cal L}\otimes {\cal N}^2$, one can compensate for this by $E\to
E\otimes {\cal N}$.  Modulo this operation, all that really matters
about ${\cal L}$ is its degree modulo 2.  Apart from the familiar
case of ${\cal L}=\CO$,  we can consider a second component with
(say) $\det E=\CO(p)$.  We write $\MH(C;SL_2^*)$ for the moduli
space of stable Higgs bundles of this type. We sometimes refer to
this as the improper component of the $SL_2$ moduli space (and we
refer to the moduli space of ordinary $SL_2$ Higgs bundles
$(E,\varphi)$ with $\det\,E$ trivial as the proper component). If
$(E,\varphi)$ is a Higgs bundle with $\det E=\CO(p)$, then, upon
setting $W={\rm ad}(E)$, we get an $SO_3$ Higgs bundle $(W,\varphi)$
with $w_2(W)\not=0$. All $(W,\varphi)$ can arise this way, and
$(E,\varphi)$ is determined uniquely from $(W,\varphi)$ up to
twisting $E$ by a line bundle of order 2.  So $\MH(C;SL_2^*)$ can be
related to $SO_3$ Higgs bundles by analogy with (\ref{turgid}):
\begin{equation} \MH(C;SO_3,\theta)=\MH(C;SL_2^*)/Q.\end{equation}

 {}From the point of view of differential geometry,
$\MH(C;SL_2^*)$ can be constructed by solving Hitchin's equations
with gauge group the compact group $SO_3$ on a bundle with
$w_2\not=0$ and dividing only by those gauge transformations that
can be lifted to $SU(2)$.

\remark \label{tiplo} The reasoning of Remark \ref{secondremark} may
also be applied to the group $SL_2$ to show that $\MH(C;SL_2)$ and
$\MH(C;SL_2^*)$ are isomorphic as complex symplectic manifolds in
complex structure $I$.

\smallskip Now let us give an explicit description.  We
consider a Higgs field $\varphi$ that is a section of $K\otimes
\CO(p)\otimes W=\oplus_{i=1}^3 K\otimes \CO(p)^2\otimes
\CO(q_i)^{-1}$. For each $i$, we can pick a section $u_i$ of
$K\otimes \CO(p)^2\otimes\CO(q_i)^{-1}$, namely $u_i=(dx/y)(x-e_i)$,
with $\Tr\,u_i^2=(dx/y)^2(x-e_i)$.  The general form of the Higgs
field is
\begin{equation}\varphi=\sum_{i=1}^3 a_i u_i,\end{equation}
with complex constants $a_i$.  This gives
\begin{equation}\label{polt}\Tr\,\varphi^2 =
\left(\frac{dx}{y}\right)^2\sum_{i=1}^3a_i^2(x-e_i).
\end{equation}
As in eqn. (\ref{ulg}), what multiplies $(dx/y)^2$ is a first order
polynomial in $x$. Letting $z\sim x^{-1/2}$ be a local parameter at
infinity, the polar part of $\Tr\,\varphi^2$ is $4(dz/z)^2\sum_i
a_i^2.$ Setting this to $2\sigma_0^2(dz/z)^2$, we require
\begin{equation}\label{mut} \sum_ia_i^2=\frac{\sigma_0^2}{2}.
\end{equation}
This affine quadric describes a Zariski open set in
$\MH(C;SL_2^*)$ (in Section \ref{sothree}, we divide by
$Q=\Z_2\times \Z_2$ and give a similar description of
$\MH(C;SO_3,\theta)$).

The constant term multiplying $(dx/y)^2$ on the right hand side of
(\ref{polt}) is $-\sum_i e_ia_i^2$.  This enables us to describe
the Hitchin fibration; it is the map from $(a_1,a_2,a_3)$ to
\begin{equation}\label{mutt}w_0=-\sum_i e_ia_i^2.\end{equation}
A fiber of the Hitchin fibration is given by the intersection of
the two quadrics (\ref{mut}) and (\ref{mutt}).  For the same
reasons as in our discussion of $SL_2$, the parameter $\sigma_0$
can be scaled out of these equations, assuming that it is nonzero.
We set $w_0=-\sigma_0^2 w/2$, and $b_i=a_i(\sqrt 2/\sigma_0)$ to
put the two quadrics in the form
\begin{align}\label{zelto}\notag  b_1^2+b_2^2+b_3^2& =1\\
                     e_1b_1^2+e_2b_2^2+e_3b_3^2& = w.\end{align}

For generic $w$, this intersection is a smooth curve of genus 1,
with some points omitted because we have assumed $W$ to be stable.
For what values of $w$ is the fiber singular?  If $f=\sum_ib_i^2-1$,
$g=\sum_ie_ib_i^2-w$, then a singularity of the fiber is a point
with $f=g=df\wedge dg=0$.  A short calculation shows that $df\wedge
dg=0$ precisely if two of $b_1,b_2$, and $b_3$ vanish. If $b_i$ is
non-vanishing for some $i$  and $b_j=0$ for $j\not=i$, then we must
have
\begin{equation} w=e_i\end{equation}
and
\begin{equation}\label{belx} b_i=\pm 1.\end{equation}
Thus, there are precisely three singular fibers, one for each choice
of $i$, just as in \secref{hitchfib}. The singular fibers are at
$w=e_i$, $i=1,2,3$, just as we found for the other component of the
moduli space.  These facts of course agree with the relation between
the two components claimed in \secref{relcom}.

Moreover, each singular fiber $\FF_w$ contains two singular points,
given in eqn. (\ref{belx}). The singular fibers consist of two
components of genus 0 joined at two double points.  To see this,
take $i=1$.  If $w=e_1$, then a linear combination of the equations
$f=0$ and $g=0$ gives $(e_2-e_1)b_2^2+(e_3-e_1)b_3^2=0$ or
\begin{equation}\label{elg}b_2=\pm b_3\sqrt{-(e_3-e_1)/(e_2-e_1)}
.\end{equation} This describes a curve $\FF_{w,0}$ that is a union
of two genus zero components   meeting at one point, $b_2=b_3=0$.
Now solving for $b_1$ via $b_1^2=1-b_3^2(e_2-e_3)/(e_2-e_1)$ gives
a double cover of $\FF_{w,0}$.  The double cover, which is the fiber
$\FF_w$ of the Hitchin fibration, is branched over two points in
each component of $\FF_{w,0}$.  A double cover of a curve of genus
zero branched at two points is still of genus zero.  So $\FF_w$
consists of two components of genus zero, meeting at the two
points $b_2=b_3=0$, $b_1=\pm 1$.

\subsubsection{$SO_3$ Moduli Space}\label{sothree}

What we have constructed so far is a Zariski open set in
$\MH(C;SL_2^*)$, together with its Hitchin fibration.  We can also
divide by $Q$, which acts by pairwise sign changes of the $b_i$,
and get a Zariski open set in $\MH(C;SO_3,\theta)$.

The invariants under pairwise sign changes of the $b_i$ are
$b_1^2$, $b_2^2$, $b_3^2$, and $z=b_1b_2b_3$.  They obey
$z^2=b_1^2b_2^2b_3^2$. Using the two equations (\ref{zelto}), one
can solve for $b_i^2, $ $i=1,2,3$ as linear functions of
$t=\sum_{i=1}^3 e_i^2b_i^2$. The equation $z^2=b_1^2b_2^2b_3^2$
becomes
\begin{equation}\label{zot} z^2 = -\prod_{1\leq i< j\leq
3}\frac{(t-w(e_i+e_j)+e_ie_j) }{(e_i - e_j)^2 }.
\end{equation}
This describes the $SO_3$ moduli space (except that one must include
points with $t,z=\infty$ for fixed $w$).

Let us use this description to find the singular fibers of the
Hitchin fibration for $\MH(C;SO_3,\theta)$.  The cubic polynomial on
the right hand side of (\ref{zot}) has distinct roots unless $w=e_i$
for some $i$, in which case precisely two of the roots coincide.  So
the only singular fibers are for $w=e_i$, and they are nodal cubic
curves, that is, copies of $\CP^1$ with two points identified.  By
expanding the cubic equation (\ref{zot}) near a double point, it is
not hard to verify that the double points of the fiber are actually
$A_1$ singularities of $\MH(C;SO_3,\theta)$. This is the same result
that we found at the end of Section \ref{struc} for the singular
fibers of $\MH(C;SO_3,0)$.  This is, of course, completely in
keeping with the claim in \secref{relcom} that the two components of
the  moduli space are equivalent in complex structure $I$.

\subsection{Relation To The Cotangent Bundle}\label{compcar}

One of the most important properties \cite{Hi} of the moduli space
$\MH$ of stable Higgs bundles (without ramification) is that it can be
approximated as $T^*\M$, where $\M$ is the moduli space of
stable bundles. ($\M$ is a Zariski open set the moduli space that
parametrizes stable and semi-stable bundles.) The reason for this is
that the cotangent space to $\M$, at a point corresponding to a
stable bundle $E$, is $H^0(C,K\otimes {\rm ad}(E))$.  So a point in
$T^*\M$ is a pair $(E,\varphi)$, or in other words a Higgs
bundle. This gives an embedding of $T^*\M$ as a Zariski open set
in $\MH$. This map is not surjective because a Higgs bundle
$(E,\varphi)$ may be stable or semi-stable, and so represent a point
in $\MH$, even if the underlying bundle $E$ is not stable.

This has an analog for ramified Higgs bundles.  In this case, one
takes $\M$ to be the moduli space of stable bundles with parabolic
structure at a point $p$, and $\MH$ to be the moduli space of stable
ramified Higgs bundles. (For a generic choice of the parabolic weight
-- the parameter called $\alpha$ in (\ref{horseg}) -- every semi-stable
parabolic bundle is stable.)  Then $\MH$ has a Zariski open set
that is not quite $T^*\M$ but is an affine symplectic deformation
of $T^*\M$, as described in section 3.6 of \cite{GW}.  We denote
such an affine symplectic deformation as $\tilde T^*\M$.  Here
$\tilde T^*\M$ denotes a complex symplectic variety with a map to
$\M$, such that locally in $\M$, $\tilde T^*\M$ is
symplectically isomorphic to $T^*\M$.  For a detailed description
of an example, see \secref{impom}.

For applications to geometric Langlands, it is important to restrict
the fibers of the Hitchin fibration from $\MH$ to $T^*\M$ or $\tilde
T^*\M$, since this is an essential step in interpreting $A$-branes in
terms of ${\cal D}$-modules, as discussed in \secref{dmodules}).  So
let us carry out this step for our example of $SL_2$ Higgs bundles on
a curve $C$ of genus 1 with one point $p$ of ramification.  Of course,
one can divide by $Q$ and make a similar discussion for $SO_3$.

\subsubsection{Calculation}\label{impom}

These questions depend only on $\MH$ as a complex symplectic
manifold in complex structure $I$.  As such, the two components
are equivalent, according to \secref{relcom}, so from now on, we
will consider the improper component.

 We have explicitly described the stable Higgs bundles
$(E,\varphi)$ where $\det E=\CO(p)$,  and $E$ is stable. $E$ is
uniquely determined up to isomorphism and we have found the space of
Higgs bundles of this kind to be precisely the affine quadric
\begin{equation}\label{hytz} b_1^2+b_2^2+b_3^2=1.\end{equation}
which we will denote for the moment as $\M_H^0$.  It differs
slightly from $\MH$, which also parametrizes stable pairs
$(E,\varphi)$ where the underlying bundle $E$ is unstable.

In this situation, what is the moduli space $\M$ of stable parabolic
bundles?  The bundle $E$ has  no deformations, but its parabolic
structure at the point $p$ can vary. The choice of parabolic
structure is the choice of a complex line ${\Psi}$ in the
two-dimensional vector space $E_p$, the fiber of $E$ at $p$. So $\M$
is a copy of $\mbb{CP}^1$, parametrizing the choice of $\Psi$.

With $E$ as above, suppose that we are given a ramified Higgs bundle
$(E,\varphi)$, where $\varphi$ has a pole at $p$ whose residue has
eigenvalues $\pm \sigma_0$. This determines a parabolic structure on
$E$, by setting $\Psi$ to be the eigenspace of the residue with
eigenvalue $+\sigma_0$.  So there is a natural map from $\M_H^0$ to
$\M$, taking $(E,\varphi)$ to the parabolic bundle with this $\Psi$.
This map has no holomorphic section, since given the bundle $E$ and
the choice of $\Psi$, there is no natural way to produce a $\varphi$
with the right pole.  It is this that leads to the affine
deformation.

Alternatively, we could get a second and equally natural parabolic
structure and map $\M_H^0\to \M$ by using the eigenspace with
eigenvalue $-\sigma_0$. This would ultimately lead to a different
functor from $A$-branes to twisted $\cal D$-modules.  In general, in
the ramified case of geometric Langlands, the mirror symmetry
between the $B$-model and the $A$-model is uniquely determined, but
the mapping from $A$-branes to twisted $\cal D$-modules depends on a
choice of Borel subgroup containing the local monodromy. The
different choices lead to ${\mc D}$-modules with different
twistings, related by a Weyl transformation.

In the present context, the choice of Borel subgroup amounts to the
choice of eigenvalue $+\sigma_0$ or $-\sigma_0$. (The isomorphism
between the two components of $\MH$ also depends on this choice of
sign.) For $SL_n$, the analog would be to pick an ordering of the
eigenvalues of the polar part of $\varphi$.

The map from the affine quadric $b_1^2+b_2^2+b_3^2=1$ to
$\mbb{CP}^1$ can of course be described without talking about Higgs
bundles.  The equation  $b_1^2+b_2^2+b_3^2=1$ is equivalent to the
statement that the matrix
\begin{equation}\label{glytz}M=\begin{pmatrix}b_1 &b_2-ib_3\\
b_2+ib_3& -b_1\end{pmatrix}\end{equation} has eigenvalues $\pm 1$.
Taking ${\Psi}$ to be the space of solutions of the equation
$M\psi=\psi$, we get a map from the affine quadric to
$\mbb{CP}^1$.  If we set $\psi=\begin{pmatrix}1\\
z\end{pmatrix}$, we find that in affine coordinates, the map to
$\mbb{CP}^1$ can be described by
$z=(1-b_1)/(b_2-ib_3)=(b_2+ib_3)/(1+b_1)$.

The fibers of the map from $\M_H^0$ to $\mbb{CP}^1$ are copies of
$\C$.  For example, the fiber for $z=0$ is given by $b_1=1$,
$b_2+ib_3=0$, and is parametrized by $b_2$.   However, $\M_H^0$ is
not a holomorphic line bundle over $\mbb{CP}^1$, because the map
from $\M_H^0$ to $\mbb{CP}^1$ has no holomorphic section.  Rather,
the structure group of the bundle $\M_H^0\to \mbb{CP}^1$ is the
group of affine linear motions of $\C$, that is the group of
transformations of the form $x\to ax+b$, $a\in \C^\times$,
$b\in\C$. We might describe $\M_H^0$ as an affine bundle of rank
1.

The group of affine linear motions has a homomorphism to
$\C^\times$, by forgetting $b$.  So to an affine bundle of rank 1,
there is always a canonically associated fiber bundle whose
structure group is $\C^\times$. In the case of the quadric, the
associated variety is the cotangent bundle of $\CP^1$. The quadric
itself admits a holomorphic symplectic form:\footnote{The
multiplicative constant $\sigma_0$ was determined by integrating
$\Omega$ over the two-cycle characterized by $b_1,b_2,b_3$ real
and comparing to the result of \cite{GW}, eqn. (3.77), for the
cohomology class of $\Omega$.}
\begin{equation}\label{holform} \Omega=i\sigma_0\frac{db_1\wedge
    db_2}{b_3}
\end{equation}
which is part of the hyper-Kahler structure on the moduli space.
 Locally on $\mbb{CP}^1$ one can pick a section of the fibration
 $\M_H^0\to \CP^1$ and identify $\M_H^0$ with the cotangent bundle of
 $\CP^1$.

To do so explicitly on the Zariski open set with $b_1\not=-1$, where
$$z=(1-b_1)/(b_2-ib_3)=(b_2+ib_3)/(1+b_1)$$ is defined, let
$v=-\sigma_0 (b_2-ib_3)$. A small computation shows that
$\Omega=dz\wedge dv$, so away from $b_1=-1$, we can embed $\CP^1$ in
$\M_H^0$ as the locus $v=0$, and identify $\M_H^0$ with the
cotangent bundle of the $z$-plane. This identification does not
extend\footnote{The fact that $\Omega$ has nonzero periods, as noted
in the last footnote, implies that it cannot be globally the
symplectic form of a cotangent bundle.} over $z=\infty$, and
globally $\M_H^0$ is what we call an affine deformation of the
cotangent bundle.  It is convenient to also introduce $\tilde
v=v/\sigma_0=-(b_2-ib_3)$.

Since $\M_H^0$ is not the cotangent bundle to $\M$, but is an affine
deformation thereof, we will denote it henceforth as $\tilde T^*\M$.

Now let us consider the fibers of the Hitchin fibration.  Their
intersection with the quadric $\tilde T^*\M$ is obtained by
supplementing the defining equation of the quadric with the equation
\begin{equation}e_1b_1^2+e_2b_2^2+e_3b_3^2=w,\end{equation}
giving an algebraic curve $\FF_w$. This curve, of course, can be
projected to $\M$, and gives a double cover of $\M$. In general,
however, the fibers of the Hitchin fibration are complete (and
generically smooth) algebraic varieties. This is an important part
of the theory in \cite{Hi}. We will explain it, in the present
context, in \secref{Hecke} using spectral covers. So the fiber
$\FF_w$ of the Hitchin fibration of $\MH$ is really, for generic
$w$, the smooth projective curve that corresponds to the affine
curve just described.  A few points are missing in the description
by affine quadrics.

We can describe $\FF_w$ as a projective curve by simply adding another
variable $b_4$, where $b_1,\dots,b_4$ are understood as homogeneous
coordinates on $\CP^3$ and obey
\begin{align}\label{potag}\notag \sum_{i=1}^3b_i^2 & = b_4^2\\
\sum_{i=1}^3e_ib_i^2 & = w b_4^2.\end{align} Missing when one
approximates $\MH$ by an affine deformation of the cotangent bundle
are the four points with $b_4=0$.  These points correspond to stable
Higgs bundles $(E,\varphi)$ where the parabolic bundle $E$ is
unstable.\footnote{$E$ must have determinant $\O(p)$ and must be of
the form $\L\otimes (\O(p)\oplus \O)$, where $\L$ is of order 2, in
order for there to exist a stable ramified Higgs bundle
$(E,\varphi)$. The four choices of $\L$ lead to the four missing
points on each fiber of the Hitchin fibration.} They form a single
orbit of the group $Q=\Z_2\times \Z_2$ of pairwise sign changes of
$b_1,b_2,b_3$. Explicitly, the values of $z$ corresponding to these
four points are
\begin{equation}\label{telmo}z=\pm\sqrt{\frac{e_2-e_1}
{e_2-e_3}}\pm \sqrt{\frac{e_3-e_1}{e_2-e_3}}.\end{equation}

To gain some insight about the ${\cal D}$-modules arising in the
geometric Langlands program, we must describe $\FF_w$ as a curve in
$\tilde T^*\M$. For this, we use the coordinates $z,\tilde v$, and
find, after some algebra, that we can describe the fiber $\FF_w$
by an explicit quadratic equation for $\tilde v$, of the form
\begin{equation}\label{plug}A(z)\tilde v^2+B(z)\tilde v+C(z)=0,
\end{equation}
with
\begin{align}\label{eelg}\notag
A(z)&=(e_2-e_3)z^4+(4e_1-2e_2-2e_3)z^2+(e_2-e_3)\\
\notag B(z)& = -4z((e_2-e_3)z^2+2e_1-e_2-e_3)\\
C(z)&=4((e_2-e_3)z^2+e_1-w).\end{align} Note that
\begin{align}\label{pelg}B=-\frac{dA}{dz}.\end{align}

If we let $z$ approach one of the four values in eqn. (\ref{telmo}),
then one of the roots of the quadratic equation for $\tilde v$ goes
to infinity.  So at any of those values of $z$,  a point in the
Hitchin fiber is ``missing,'' if we restrict to $\tilde T^*\M$. In
fact, the four critical values of $z$ are precisely the zeroes of
the polynomial $A(z)$.  At those values of $z$, $B(z)$ is nonzero,
which means that one of the two roots of the quadratic equation
(\ref{plug}) for $\tilde v$ has a pole, and one does not. It is
convenient to express the result in terms of the variable
$v=\sigma_0\tilde v$ that puts the symplectic form $\Omega$ in a
canonical form.  Let $z^*$ be any one of the zeroes of $A(z)$. The
behavior of the polar branch of $v$ near $z^*$ is $v\sim -\sigma_0
B(z)/A(z)$, which using eqn. (\ref{pelg}) reduces to
\begin{equation}\label{gesto}v\sim \sigma_0\frac{1}{z-z^*}.\end{equation}

\remark\label{endosc} For the geometric endoscopy, we must examine
in a similar way the singular fibers of the Hitchin fibration. For
example, we take $w=e_1$, so that the fiber $\FF_{e_1}$ splits into
components $\FF^\pm_{e_1}$ defined by the ratio of $b_2/b_3$, as in
eqn.  (\ref{elg}). Compactifying the two components in projective
space, we see that of the four points at $b_4=0$, two lie on
$\FF^+_{e_1}$ and two on $\FF^-_{e_1}$. (This actually follows just
from the fact that the four points are an orbit of $Q$, and that $Q$
exchanges the two components.)  If we restrict to $\tilde T^*\M$,
the two curves $\FF^\pm_{e_1}$ behave near the two relevant critical
values of $z$ precisely as found in the last paragraph.  So each
fractional $A$-brane has two points with this sort of behavior.

\subsubsection{Nilpotent Higgs Fields}\label{nilphiggs}

There is a simple way to characterize points in $\M$ over which a
fiber $\FF$ of the Hitchin fibration, when viewed as a curve in
$T^*\M$, goes to infinity.  They correspond to stable parabolic
bundles $E$ that admit a nonzero nilpotent Higgs field (whose polar
residue leaves fixed the parabolic structure of $E$). This is
relevant for characterizing twisted $\D$-modules arising in
geometric Langlands duality (see \cite{Laumon:duke}).

For example, consider the improper component of the $SL_2$ moduli
space. The Higgs field $\varphi=\sum a_iu_i$ obeys
$\Tr\,\varphi^2=\sum_i a_i^2(x-e_i)$, as we computed in eqn.
(\ref{polt}).  So if $0=\sum_ia_i^2=\sum_ie_ia_i^2$, which is the
condition defining a point in $\FF$ that does not lie in $T^*\M$,
then $\Tr\,\varphi^2$ identically vanishes and $\varphi$ is
nilpotent.

For the proper component, if $E=\L\oplus \L^{-1}$, then the
condition for the Hitchin fiber $\FF$ to go to infinity in $\tilde
T^*\M$ is that $\L$ should be of order 2.  (This is clear in the
original coordinates used in eqn. (\ref{zelg}).  $\L$ of order 2
means $x_0=\infty$ or $f(x_0)=0$; for such $x_0$, with $w_0$ fixed,
we have $h_0\to\infty$.) On the other hand, if we endow $E$ with a
parabolic structure that makes it stable, then a Higgs field whose
pole at the point $p$ leaves fixed the parabolic structure cannot be
nilpotent unless $\L$ is of order 2. Indeed, if $\L$ is not of order
2, the general form of this Higgs field is given in eqns.
(\ref{conagain}) and (\ref{cute}), which show that $\varphi$ can be
everywhere  nilpotent only if $h_0$ vanishes along with either $k_0$
or $g_0$; in that case, a parabolic structure at $p$ that is
invariant under the polar part of $\varphi$ corresponds to a complex
line $\Psi\subset E_p$ that lies in either $\L$ or $\L^{-1}$, and
$E$ is unstable. Conversely, if $\L$ is of order 2, we can make a
stable parabolic bundle $E$ that is an extension of $\L$ by $\L$ (a
direct sum $\L\oplus \L$ is unstable with any parabolic structure),
and it admits a nilpotent Higgs field $\varphi$ whose polar part
leaves fixed the parabolic structure.

Here is a brief explanation, in the context of $SL_2$ on a curve
$C$ of genus $g$, of the behavior we have just seen. (Ramification
can be incorporated in the following argument by allowing certain
poles in the Higgs field and the quadratic differentials.)  Let
$\varphi_1,\dots,\varphi_{3g-3}$ be a basis of $H^0(C,K\otimes
{\rm ad}(E))$. So a general Higgs field is
$\varphi=\sum_ia_i\varphi_i$, for some complex numbers $a_i$. The
fiber $U_E$ of $T^*\M$ at the point corresponding to $E$ is thus
parametrized by $a_1,\dots,a_{3g-3}$. Similarly, let
$w_1,\dots,w_{3g-3}$ be a basis of $H^0(C,K^2)$. A fiber $\FF$ of
the Hitchin fibration is specified by $\Tr\,\varphi^2=\sum t_i
w_i$, for some complex numbers $t_i$. Explicitly, this gives a set
of $3g-3$ quadratic equations
\begin{equation}\label{erx}
P_j(a_1,\dots,a_{3g-3})=t_j,~~j=1,\dots,3g-3.\end{equation}
Generically, such a set of equations has precisely $2^{3g-3}$
solutions, weighted by multiplicity.  These are the intersection
points of the fiber $U_E$ with $\FF$. If we add one more variable
$a_0$ to make the equations homogeneous
\begin{equation}\label{terx}
P_j(a_1,\dots,a_{3g-3})=t_ja_0^2,~~j=1,\dots,3g-3,\end{equation}
then the number of solutions in $\mathbb{CP}^{3g-3}$, weighted by
multiplicity, is always precisely $2^{3g-3}$ if it is finite.  The
number of solutions of the affine equations (\ref{erx}) is less
than this precisely if the homogeneous equations (\ref{terx}) have
solutions with $a_0=0$.  But solutions of the homogeneous
equations with $a_0=0$ correspond to Higgs fields $\varphi$ with
$\Tr\,\varphi^2=0$, that is, to nilpotent Higgs fields.  Thus,
some of the intersections of $U_E$ with $\FF$ go to infinity
precisely when $E$ admits a nonzero nilpotent Higgs field, just as
we found in our example.  (The example involves a ramified case in
which the polynomials $P_j$ in eqn. (\ref{erx}) are not
homogeneous, so that the homogeneous form of eqn. (\ref{terx}) is
slightly different.)

\subsubsection{Comparison Of The Two
Components}\label{comparisoncomponents}

What we have learned makes it possible to find an explicit mapping
between the two components of $\MH$, as predicted in
\secref{relcom}.

We want to explicitly map between the description of $\MH(SL_2)$
given by the familiar equation $P=0$, where
\begin{equation}\label{jilk}P=\rho^2+
(2u+w)f(u)-\frac{f'(u)^2}{4}\end{equation} and  its analog
\begin{equation} b_1^2+b_2^2+b_3^2=1\end{equation}
for the improper component $\MH(SL_2^*)$.

We already know from \secref{elmy} that the change of variables we
want preserves $w$, which for the improper component is
\begin{equation}w=e_1b_1^2+e_2b_2^2+e_3b_3^2.\end{equation}

In addition, the desired change of variables commutes with the
projection to $\M$ of a Zariski open set in $\MH$. The reason for
this is that the $\varphi$-invariant Hecke modification of a bundle
$E$ that was used in \secref{relcom} depends only on the parabolic
structure on $E$ that is determined by $\varphi$, and not on
$\varphi$ itself. So it commutes with the projection from $\M_H^0$
to $\M$, the moduli space of stable parabolic bundles.

For the proper component, $\M$ is parametrized by the variable
$u$, and for the improper component it is similarly parametrized
by $z=(1-b_1)/(b_2-ib_3)$.  These parametrizations are unique up
to a fractional linear transformation, so $z$ must map to
$(au+b)/(cu+d)$ for some constants $a,b,c,d$.  The constants can
be determined from the fact that the condition for a parabolic
bundle $E$ to admit a nilpotent Higgs field (with a possible pole
at $p$ whose residue preserves the parabolic structure) is also
preserved by the map between the two components.  Indeed, that map
preserves both the Hitchin fibration and the projection $\M_H^0\to
\M$, so it maps missing points on Hitchin fibers to missing points
on Hitchin fibers.

So the map between the two components identifies the four points
$u=e_1,e_2,e_3,\infty$ with the four values of $z$ given in eqn.
(\ref{telmo}).  Up to the action of $Q$, this determines the
relation between $z$ and $u$ to be
\begin{equation}\label{zu}u=\frac{e_2\sqrt{e_3-e_1}  + e_3
    \sqrt{e_2-e_1} + (e_1 - \sqrt{e_3- e_1 } \sqrt{e_2-e_1 })\sqrt{
   e_2-e_3} z}{\sqrt{e_3- e_1} + \sqrt{e_2-e_1 } + \sqrt{ e_2-e_3}
   z}.\end{equation}

 $\MH(SL_2)$ is a double cover of the $w-u$ plane given by the
quadratic equation (\ref{jilk}) for $\rho$, and $\MH(SL_2^*)$ is
similarly a double cover of the $w-z$ plane given by the quadratic
equation (\ref{plug}) for $\tilde v$.  An elementary computation
shows that the two quadratic equations are equivalent under the
change of variables
\begin{equation}\tilde v
=\rho\frac{du}{dz}\frac{1}{f(u)}-\frac{1}{2A}\frac{dA}{dz},
\end{equation}
where $A$ was defined in eqn. (\ref{eelg}).

This gives the map between the two spaces.  Finally, it is
straightforward to verify that the complex symplectic form of
$\MH(SL_2^*)$, given in eqn. (\ref{holform}), maps to the complex
symplectic form of $\MH(SL_2)$, which is similarly given by
\begin{equation}\Omega=\sigma_0\frac{du\wedge
d\rho}{f(u)}=\sigma_0\frac{du\wedge d\rho}{\partial P/\partial
w}.\end{equation}

Because of this isomorphism, in comparing $A$-branes to $\cal
D$-modules, as we do in \secref{genus one A to D}, it suffices to
consider the improper component. Whatever we can learn about
$A$-branes and their corresponding $\cal D$-modules from the
geometry of $\MH$ will be the same for the two components.

\subsection{Mirror Symmetry Of Orbifolds}\label{orbifolds}

\def\PPP{{\mathbb P}}
Here we will briefly place some of the issues that we have
considered in a more general context of mirror symmetry.

First of all, let us note the following simple fact about the
classical geometry.  $\MH(SO_3)$ has an $A_1$ singularity
contained in a fiber that consists of a single $\PPP^1$
intersecting at two points. (There are three such singular fibers,
but we consider here just one of them.) If we blow up the $A_1$
singularity, we eliminate the singularity (of the total space) and
create a second $\PPP^1$. The blowup creates a new elliptic
fibration with a singular fiber that now contains two $\PPP^1$'s,
each intersecting at two points. The total space is now smooth,
but has a singular fiber as just described.  This is precisely the
geometry of $\MH(SL_2)$ near its special fibers.

The conclusion therefore is that if we change what physicists
would call the ``Kahler modulus'' by this blowup, we could convert
the local geometry for $SO_3$ to that for $SL_2$.  Let us think
about this local geometry more systematically.  From a
hyper-Kahler point of view, the deformation/resolution of an $A_1$
orbifold singularity is controlled by three real parameters
$\vec\mu$.  If one picks one of the complex structures -- it is
natural to select the complex structure $I$  in which the Hitchin
fibration is holomorphic --  then $\vec \mu$  splits up as a real
parameter $\mu_\RR$ that controls the Kahler class and a complex
parameter $\mu_\C$ that controls the complex structure. Varieties
with the same $\mu_\C$ and different $\mu_\RR$ are birationally
equivalent. A singularity occurs only if $\vec\mu=0$.

In the context of two-dimensional sigma models, there is an
additional parameter, a theta-like angle which in section 6.2 of
\cite{GW}, where this geometry is reviewed, is called $\eta$; it
takes values in $\RR/\Z$.  From the point of view of one of the
complex structures, say $I$, $\eta$ combines with $\mu_\RR$ to a
second complex parameter $\mu_\C'=\mu_\RR+i\eta$.  This parameter
controls the complexified Kahler class associated with the
singularity, while $\mu_\C$ controls the complex structure.

The sigma model is singular if and only if $\vec\mu=\eta=0$ or
equivalently $\mu_\C=\mu_\C'=0$. Having a singularity of the sigma
model is an intrinsic condition and hence is invariant under any
duality. In the context of Type IIA superstring theory, the
singularity of the sigma model at $\vec\mu=\eta=0$ leads to an
$SU(2)$ or $A_1$ gauge symmetry in spacetime; this is possibly the
most satisfying statement of the relation of the group $A_1$ to the
$A_1$ singularity.

In geometric Langlands, we want to treat the $\RR^4/\Z_2$
singularity in the sigma model of $\MH(SO_3)$ as an orbifold
point. The sigma model can be treated as an orbifold precisely
\cite{As} if $\vec\mu=0$, $\eta=1/2$. At this value, we have
$\mu_\C=0$, $\mu_\C'\not=0$. This means that the deformation away
from the point $\mu_\C=\mu_\C'=0$ is best understood as a
resolution, not a deformation, though it is a resolution in a
non-classical sense, involving $\eta={\rm Im}\,\mu_\C'$ rather
than $\mu_\RR={\rm Re}\,\mu_\C'$.  By contrast, to get to the
local geometry of $\MH(SL_2)$, we must make, as we explained at
the outset, a classical resolution, with ${\rm
Re}\,\mu_\C'\not=0$.

Thus, in this situation, the $T$-duality on the fibers of the
Hitchin fibration exchanges a non-classical resolution involving
${\rm Im}\,\mu_\C'$ with a classical resolution involving ${\rm
Re}\,\mu_\C'$.   This is consistent with the following: as the
Hitchin fibration is holomorphic  in complex structure $I$,
$T$-duality on the fibers of this fibration map complex parameters
(such as $\mu_\C$) to complex parameters and Kahler parameters
(such as $\mu_\C'$) to Kahler parameters.

A last comment is that in \secref{higher genus}, we will encounter
an analog of all this involving a singularity $\C^{2n}/\Z_2$,
$n>1$.  In this case, in classical geometry, the singularity can
be neither deformed nor resolved, and in the sigma model, it can
only be treated as an orbifold.  Thus there is no close analog of
the above discussion.

\subsection{Relation To Seiberg-Witten
Theory}\label{SeibergWitten}

The Hitchin fibrations that we have considered here are very similar
to elliptic fibrations that appear in Seiberg-Witten theory.  For
example, modulo some elementary changes of variable, eqn.
(\ref{zot}) for the $SO_3$ moduli space coincides with eqn. (16.24)
of \cite{SW1}.  The latter equation describes the Seiberg-Witten
fibration for the four-dimensional ${\cal N}=2^*$
theory\footnote{This is the ${\cal N}=2$ theory with a massive
hypermultiplet in the adjoint representation.  As the hypermultiplet
mass becomes large, this theory can reduce to the minimal ${\cal
N}=2$ theory, which is the basic case of Seiberg-Witten theory.}
 with gauge group $SU(2)$.  The same singular fibers
of the Hitchin fibration that we are relating here to endoscopy also
play a central role in Seiberg-Witten theory.  Physically, they
describe the appearance of massless charged hypermultiplets;
mathematically, they give the main contribution in the application
of Seiberg-Witten theory to four-manifolds.

The Seiberg-Witten fibration of the ${\cal N}=2^*$ theory was not
originally understood as a Hitchin fibration, but this was later
done \cite{DW} as a step toward understanding the ${\cal N}=2^*$
theory for $SU(N)$.

The fact that the same elliptic fibration appears in two different
problems can be explained by reducing the six-dimensional
world-volume of a certain $M$-theory fivebrane from six to four
dimensions in two different ways.  We consider $M$-theory on the
11-manifold $X=\RR^2\times \mathbb{T}^2\times T^*C \times \RR^3$,
where $C$ is an elliptic curve.  We let $D\subset T^*C$ be a
spectral curve\footnote{This notion is reviewed in
\secref{specreview}.} associated with an $SL_2$ Higgs bundle on $C$
with one point of ramification (for some value of $w$).  And we
consider an $M$-theory fivebrane whose worldvolume is $\RR^2\times
\mathbb{T}^2\times D$. There are two ways to look at this situation,
depending on whether $\mathbb{T}^2$ or $C$ is smaller:

(1) If $\mathbb{T}^2$ is smaller, the first step is to consider
compactification of the fivebrane on $\mathbb{T}^2$.  This  leads to
${\cal N}=4$ super Yang-Mills theory with gauge group $U(2)$
(because the projection $D\to C$ is generically a double cover, the
gauge group is $U(2)$, or effectively $SU(2)$, since the center
plays little role). Geometric Langlands is associated with
subsequent compactification to two dimensions on $C$; the particular
curve $D$ describes the behavior in the presence of a surface
operator associated with ramification, as discussed in section 6.4
of \cite{GW}.

(2) If $C$ is smaller, we consider first the compactification on
$C$, which \cite{Wi3} leads to the ${\cal N}=2^*$ theory on
$\RR^2\times \mathbb{T}^2.$   The elliptic fibration (\ref{zot})
describes this theory on $\mathbb{R}^4$.  The result of further
compactification from $\RR^4$ to $\RR^2\times \mathbb{T}^2$ is that the
elliptic fibration must be endowed with a hyper-Kahler metric
\cite{SW2}.

\section{$A$-Branes And $\cal D$-Modules}\label{genus one A to D}

As promised in \secref{furpro}, we will give more detail on the
relation between an $A$-brane on $\MH$ and the corresponding twisted
$\cal D$-module on $\M$.  The main goal is to explain the
significance of eqn. (\ref{gesto}) that describes the behavior of a
Hitchin fiber restricted to $\tilde T^*\M$ near its ``missing''
divisors (which in that example are points, since $\dim \M = 1$).

For this we need the theory of spectral curves, which will also be a
primary tool in \secref{spectral}.  So we begin with a very brief
sketch of this theory.

Although the arguments in section 11 of \cite{KW} relating
$A$-branes to twisted $\cal D$-modules are fairly clear from a
physical point of view, they are completely conjectural
mathematically.  The following discussion involves an extra layer of
conjecture, since although what we will describe is qualitatively in
accord with what one would expect from \cite{KW}, it is not yet
based on equally clear-cut physics arguments.

For other approaches to these questions, see \cite{NZ,Nadler,Mo}.

\subsection{Spectral Curves}\label{specreview}

The idea  is most simply explained for $GL_n$ or $SL_n$. Let
$(E,\varphi)$ be a Higgs bundle, where $E\to C$ is a vector bundle
of rank $n$ and $\varphi\in H^0(C,K\otimes {\rm ad}(E))$ is the
Higgs field.  Associated to this data, one defines a curve $D$ in
the cotangent bundle $T^*C$ as follows.  Letting $z$ denote a linear
function on the fibers of the cotangent bundle, the curve $D$, which
is known as the spectral curve, is defined by the equation
\begin{equation}\label{gelf}\det(z-\varphi)=0.\end{equation}
The projection $\pi:T^*C\to C$ restricts to an $n$-fold cover
$\pi:D\to C$.

Generically, $D$ is smooth and irreducible and has genus
$g'=n^2g-n^2+1$, where $g$ is the genus of $C$. For ordinary Higgs
bundles, $D\subset T^*C$ is a complete projective curve, but for
ramified Higgs bundles, this is not so as $\varphi$ has poles.  In
that case, one completes $D$ to a projective curve by adding a few
points that are not in $T^*C$.

 One defines a line bundle
$\cal L\to D$ as the cokernel of the map
$(\pi^*(\varphi)-z):\pi^*(E)\to \pi^*(E)\otimes K$. To be more
precise, $\cal L$ is a line bundle when $D$ is smooth but in general
may be a torsion-free sheaf.

So to the Higgs bundle $(E,\varphi)$ on $C$, we associate the
following data: the spectral curve $\pi:D\to C$ and the line bundle
or torsion-free sheaf $\cal L\to D$.  From this data, it is possible
to reconstruct the original  Higgs bundle.  One recovers $E$ by
$E=\pi_*(\cal L)$.  And the
 Higgs field $\varphi:E\to E\otimes K$ can similarly be recovered
from $\pi:D\to C$ as $\pi_*(z)$, where here multiplication by $z$ is
understood as a map from $\cal L$ to $\cal L\otimes \pi^*(T^*C)$. In
this way, the Higgs bundle $(E,\varphi)$ can be recovered from the
pair $(D,\cal L)$.

In this description, the base $\BB$ of the Hitchin fibration
parametrizes the possible spectral curves $D$.  For example, for
$G=SL_2$, the equation defining $D$ reduces to $z^2+\det\,\varphi=0
$, so $\BB$ is parametrized by the quadratic differential
$\det\,\varphi=-\Tr\,\varphi^2/2$.  The Hitchin fibration is the map
that takes a pair $(D,\cal L)$ to $D$. The fiber $\FF$ of the
Hitchin fibration is parametrized by the possible choices of $\cal
L$ for a given spectral curve $D$.

What line bundles $\cal L\to D$ arise in this construction? For
$GL_n$, if we allow all Higgs bundles $(E,\varphi)$, with no
restriction on $\det E$, then all $\cal L\to D$ can arise.  Thus,
$\cal L$ is associated with an arbitrary point in $\Pic\,D$.  If we
wish to associate $(E,\varphi)$ to a local system by solving
Hitchin's equations, then $\det E$ must have degree zero and we must
restrict $\cal L$ to a particular component of $\Pic\,D$.  A small
computation with Riemann-Roch shows that in order for $E=\pi_*(\cal
L)$ to be of degree zero, the degree of $\cal L$ must be
\begin{equation}\label{honcho}{\rm
deg}\,\L=n(n-1)(g-1).\end{equation} (See \secref{multidim} for a
version of this computation.)

For $SL_n$, we must further restrict $\cal L$ so that $\det E$ is
trivial, not just of degree zero.

\subsection{Relation To $A$-Branes}

The condition (\ref{honcho}) on the degree of $\cal L$ can be
usefully stated as follows.  Let $K_C^{1/2}$ and $K_D^{1/2}$ denote
square roots of the canonical bundles of $C$ and $D$, respectively.
Then $\cal L$ has the form $\cal L=\pi^*(K_C^{-1/2})\otimes
K_D^{1/2}\otimes \cal N$, where $\cal N\to D$ has degree zero.  An
evocative way to restate this is to say that we can define $E$ by
\begin{equation}\label{pelikan} K_C^{1/2}\otimes E =
  \pi_*(K_D^{1/2}\otimes \cal N).\end{equation}

The virtue of this last description is that it is closely related to
the theory of $A$-branes.  Let $X$ be a symplectic manifold and
$L\subset X$ a Lagrangian submanifold. We begin with the general
case of a real symplectic manifold with a real Lagrangian
submanifold. To define an $A$-brane $\cal B$ supported on $L$, we
choose roughly speaking a flat vector bundle $V\to L$, and a spin
bundle $\cal W$ on $L$. Then the brane $\cal B$ is associated in
differential geometry with the vector bundle $\cal W\otimes V\to L$.

Now specialize to the case that $X$ is a complex symplectic manifold
and $L$ a complex Lagrangian submanifold of complex dimension $d$.
Then, writing $K_L^{1/2}$ for a square root of the canonical bundle
of $L$ and $\Omega^{0,i}$ for the sheaf of $(0,i)$-forms on $L$, a
spin bundle of $L$ takes the form $\cal W=K_L^{1/2}\otimes
\left(\oplus_{i=0}^{d}\Omega^{0,i}\right)$.  As a result, in
algebraic geometry, one associates $\cal B$ with the sheaf of
sections of $K_L^{1/2}\otimes V\to L$; the tensor product with
$\oplus_{i=0}^{d}\Omega^{0,i}$ simply supplies the $\bar\partial$
resolution of this sheaf.  So in eqn. (\ref{pelikan}), the right
hand side involves precisely the defining data of an $A$-brane
supported on the Lagrangian submanifold $D$, that is, the line
bundle $K_D^{1/2}\otimes \cal N$.  We will see shortly in what sense
the left hand side is also natural.

\remark\label{tormo} A more intrinsic formulation avoids choosing
spin bundles over $D$ and $L$. In the most precise description,
$V\to L$ is not a flat vector bundle, but a flat twisted vector
bundle. The twisting is by a complex gerbe $\eusm G$ of order 2
whose local trivializations correspond to spin structures on the
normal bundle to $L$ in $X$ ($\eusm G$ is obtained from a $\Z_2$
gerbe via the embedding $\{\pm 1\}\subset \C^\times$). The twisting
by $\eusm G$ is related to the fact that branes in $X$ actually have
an interpretation in terms of the $K$-theory of $X$. For example,
see \cite{FrW}.  The key point is that a submanifold $L\subset X$
equipped with a  $\eusm G$-twisted vector bundle defines a class in
$K(X)$, while $L$ equipped with an ordinary vector bundle does not
define such a class.

The gerbe $\eusm G$ will also appear in \secref{gerbes}. Except in
that section, we will not incorporate this extra layer of subtlety.
Our main example of an $A$-brane in most of the paper is a brane
supported on a fiber $\FF$ of the Hitchin fibration. In this case,
the determinant of the normal bundle is trivial and has a global
square root that is also trivial, giving a canonical trivialization
of $\eusm G$.

A rank 1 bundle twisted by $\eusm G$ determines what is  called a
${\rm Spin}_c$ structure.  A {\rm flat} rank 1 $\eusm G$-twisted
bundle corresponds to a flat ${\rm Spin}_c$ structure. In our
present example that $L$ is a curve $D$, a line bundle of the form
$K_D^{1/2}\otimes \cal N\to D$, where $\cal N$ has degree zero,
corresponds naturally to  a flat ${\rm Spin}_c$ structure. That is
an intrinsic way to characterize line bundles of the form
$K_D^{1/2}\otimes \cal N$ without introducing separately either
factor in the tensor product.

\subsection{Map From $A$-Branes To Twisted $\cal
D$-Modules}\label{strategy}

Now we can explain the basic strategy for mapping $A$-branes to
twisted $\cal D$-modules. We do this initially for $A$-branes in a
complex symplectic manifold $X=T^*Y$ for the case that $Y$ is a curve.

We let $D\subset X$ be any smooth curve such that the projection
$\pi:X\to Y$ restricts to a finite cover $\pi:D\to Y$.  In
\secref{specreview}, such a curve $D$ arose as the spectral curve of
a Higgs bundle $(E,\varphi)$ over $C$. Though this is the motivating
example, we have given the base curve a new name $Y$ because in our
application, $Y$ will be in fact not the original curve $C$ but a
moduli space $\M$ of stable parabolic  bundles over $C$.

For dimensional reasons, $D$ is a Lagrangian submanifold with
respect to the holomorphic symplectic structure of $X$. We define a
rank 1 $A$-brane supported on $D$ by picking a flat ${\rm Spin}_c$
structure on $D$ that we write not quite intrinsically as a choice
of line bundle $K_D^{1/2}\otimes \N\to D$, where $\N\to D$ is a line
bundle of degree 0.

Then, roughly speaking, according to eqn. (\ref{pelikan}), we define
a vector bundle $E\to Y$ by $K_Y^{1/2}\otimes
E=\pi_*(K_D^{1/2}\otimes \N)$. And we define a Higgs field $\varphi$
by $\varphi=\pi_*(z)$.  So $(E,\varphi)$ is a Higgs bundle over $Y$.

The key step comes next.  By solving Hitchin's equations, we relate
the Higgs bundle $(E,\varphi)$ to a rank $n$ local system over $Y$
that we also call $E$.  Hence instead of just thinking of
$K_Y^{1/2}\otimes E$ as a vector bundle, we can think of its sheaf
of sections as a twisted $\cal D$-module over $Y$, that is, a sheaf
of modules for the sheaf $\cal D^*$ of differential operators acting
on sections of $K_Y^{1/2}$.

An important detail here is that starting with an $A$-brane defined
by a flat ${\rm Spin}_c$ bundle that we write loosely as
$K_D^{1/2}\otimes \cal N$, we do not quite get intrinsically a
degree zero vector bundle $E\to Y$; rather we get a vector bundle of
non-zero degree that we have written non-intrinsically as
$K_Y^{1/2}\otimes E$. There is no canonical choice of $K_Y^{1/2}$,
so there is also no canonical way to define $E$. Only the tensor
product of the two is canonical, and this is why an $A$-brane in
$T^*Y$ supported on the curve $D$ maps canonically not to an
ordinary $\cal D$-module over $Y$ (such as the sheaf of sections of
a local system $E$) but to a twisted $\cal D$-module (the sheaf of
sections of $K_Y^{1/2}\otimes E$).

Though $E$ is only defined up to a twist by a line bundle of order
2, this twist does not affect Hitchin's equations, and the twisted
$\cal D$-module that emerges from the construction depends only on
the original $A$-brane, and not on any other choice.  What is not
quite canonical is to describe this twisted $\cal D$-module via a
tensor product $K_Y^{1/2}\otimes E$.

\subsection{Poles}\label{poles}

For application to geometric Langlands, it is essential to extend
this to the case that the curve $D\subset T^*Y$ is not compact and
the projection $\pi:D\to Y$ is only generically an $n$-fold cover.
As in eqn. (\ref{normo}), we suppose that locally, near some point
$u_0\in Y$, some branches of $\pi:D\to Y$ go to infinity. We let $u$
denote a local coordinate on $Y$ near $u_0$, and pick a second
coordinate $s$ on $X=T^*Y$ so that the symplectic structure is
$du\wedge ds$.  In the most basic case, we assume that the $n$
branches are described by functions $s_i(u)$ whose singularities at
$u=u_0$ are simple poles. Thus,
\begin{equation}\label{formo}
s_i(u)\sim\frac{c_i}{u-u_0},~~i=1,\dots,n\end{equation} with complex
constants $c_i$, some of which may vanish.

Let $\cal U$ be a small neighborhood of the point $u_0\in Y$ and
trivialize $\cal L$ over $\pi^{-1}(\cal U)$.  Then $E=\pi_*(\L)$ is
naturally trivialized over $\cal U$ and relative to this
trivialization, the Higgs field $\varphi$ is diagonal, $\varphi={\rm
diag}(\varphi_1,\dots,\varphi_n)$ where the $\varphi_i$ are
differentials with poles at $u=u_0$:
\begin{equation}\label{kormo} \varphi_i\sim c_i\frac{du}
{u-u_0}.\end{equation}

Just as in the case that $D$ is compact, an $A$-brane supported on
$D$ still gives rise to a Higgs bundle $(E,\varphi)$ over $Y$
(twisted as above by $K_Y^{1/2}$).  However, $\varphi$ now has a
simple pole at the points in $Y$, such as  $u_0$, at which $\pi:D\to
Y$ fails to be an $n$-fold cover.

We can still  solve Hitchin's equations and associate to this data a
(twisted) local system, since Hitchin's equations work nicely in the
presence of prescribed poles \cite{Sim}.   We have already described
in eqn. (\ref{horseg}) the appropriate polar behavior of the fields
entering in Hitchin's equations.  But now we  write $\alpha^*$,
$\beta^*$, and $\gamma^*$ for the parameters characterizing a
ramified Higgs bundle on the curve $Y$:
\begin{align} \label{horsegg} \notag  A & = \alpha^* \, d\theta+\dots\\
\phi& = \beta^*\frac{dr}{r}-\gamma^* d\theta+\dots.
\end{align} (Here $u-u_0=re^{i\theta}$.) In the rank $n$ case,
the parameters $\alpha^*$, $\beta^*$, and $\gamma^*$ take values in
the Lie algebra of a maximal torus in the compact form of $GL_n$;
thus, they are imaginary\footnote{The elements of the real Lie
algebra of a compact group acts as anti-hermitian matrices in a
unitary representation.} diagonal matrices.

We reserve the names $^L\neg\alpha,{}^L\neg\beta,{}^L\neg\gamma$ for
the ramification parameters on $C$ of the original $^L\neg G$ local
system, and $\alpha,\beta,\gamma$ for their counterparts in the
$A$-model.  Until this point, we have not distinguished in the
notation between the ramification parameters for $^L\neg G$ and $G$.
This is because we have been discussing classical geometry, and the
same facts are relevant to both the $A$-model and the $B$-model.
Henceforth, it will help to be more precise.

As explained in \secref{hitch}, a solution of Hitchin's equations
with a singularity as in eqn. (\ref{horsegg}) determines, depending
on how we look at it, either a Higgs bundle with a singularity at
$u_0$ or a local system with such a singularity. The Higgs bundle
has a pole
\begin{equation}\label{zerlo}\varphi\sim
\sigma^*\frac{du}{u-u_0},~~\sigma^*=\frac{\beta^*+i\gamma^*}{2}.
\end{equation}
The local system has monodromy
\begin{equation}\label{zelmo} M^*=\exp(-2\pi(\alpha^*-i\gamma^*)).
\end{equation}

It is useful to set $\sigma^*={\rm
diag}(\sigma_1^*,\dots,\sigma_n^*)$ and
$\sigma_i^*=\frac{1}{2}(\beta_i^*+i\gamma_i^*)$.  Comparing eqns.
(\ref{kormo}) and (\ref{zerlo}), we see that
\begin{equation}\label{otherwords}\sigma^*_i=c_i\end{equation}
 or in other words
\begin{equation}\beta^*_i+i\gamma^*_i=2c_i.\end{equation}

Thus, given a curve $D\subset T^*Y$ with poles of the type we have
considered, $\beta^*$ and $\gamma^*$ are determined by the
geometry. The eigenvalues of the monodromy $M^*$ have absolute
values $\exp(2\pi i\gamma_i^*)$, which  are determined by $D$. The
arguments of the eigenvalues depend on $\alpha_i^*$. The geometry
of $D$ does not determine $\alpha_i^*$; we can solve Hitchin's
equations and get a twisted $\cal D$-module over $Y$ for any
choice at all of $\alpha_i^*$.

\subsection{Application To Our Example}\label{polynomial}

We can get some insight about $\alpha^*$ by considering the familiar
example -- the geometric Langlands program for $SL_2$ on a curve $C$
of genus 1 with 1 point of ramification. In this problem, we do not
expect to be able to determine $\alpha^*$ uniquely from the data
considered so far; it is perfectly natural that it should also
depend on the underlying parameter $^L\neg\alpha$ of the original
$^L\neg G$ local system. A physics-based method to determine
$\alpha^*$ in this problem is not presently available, but
considerable information comes from very simple arguments. (A
conjectural description of $\alpha^*$ is given at the end of
\secref{zame}.)

The moduli space of parabolic $SL_2$ bundles over $C$ in this
situation is a curve $\M$ of genus zero, as we have described.  We
take the curve $Y$ in the above discussion to be $\M$, and we take
$D$ to be a fiber $\FF$ of the Hitchin fibration, regarded as a
curve in $\tilde T^*Y\subset \MH$.  In considering this example, we
are jumping ahead of our story slightly, because here $\tilde T^*Y$
is an affine deformation of the cotangent bundle of $Y$, while in
\secref{poles}, we considered a curve in the undeformed cotangent
bundle. For the moment we overlook this difference, which will be
incorporated in \secref{twist}.

There are two ramified Higgs bundles in this example of the
geometric Langlands program. The input  is a ramified Higgs bundle
$(E,\varphi)$ over the curve $C$, with parameters
$^L\neg\alpha,{}^L\neg\beta,{}^L\neg\gamma$. The output is a
ramified Higgs bundle over the curve $Y=\M$, with parameters
$\alpha^*,\beta^*,\gamma^*$.  We would like to determine the output
parameters as a function of the input parameters.  (The output
parameters may {\it a priori} be different at different points on
$\M$ where the induced local system is singular.)

In geometric Langlands, one views the input Higgs bundle as
determining a local system over $C$ with monodromy $M=\exp(-2\pi
({}^L\neg\alpha-i{}^L\neg\gamma))$.  The output is a twisted $\cal
D$-module over $\M$, represented by a local system with monodromy
$M^*=\exp(-2\pi(\alpha^*-i\gamma^*))$.  From this point of view, it
is clear that if there is a geometric Langlands duality, then
$\alpha^*-i\gamma^*$ varies holomorphically in
${}^L\neg\alpha-i{}^L\neg\gamma$, and is independent of
${}^L\neg\beta$. The last statement holds because the input local
system does not depend on ${}^L\neg\beta$, though its associated
Higgs bundle does. So in particular, we have
\begin{equation}\label{ufrac}
\frac{\partial\gamma^*}{\partial{}^L\neg\beta}=0.
\end{equation}
{\it A priori}, it is not clear that $\alpha^*-i\gamma^*$ is a
function only of ${}^L\neg\alpha-i{}^L\neg\gamma$ (and not of the
modulus of the input local system over $C$), but in a moment it will
be clear that in our example, this is true for $\gamma^*$ and
therefore by holomorphy also for $\alpha^*$.

Alternatively, Hitchin's equations identify the input local system
with a Higgs bundle $(E,\varphi)$ and  here the natural holomorphic
parameter is $^L\neg\sigma=({}^L\neg\beta+i{}^L\neg\gamma)/2$.
Duality maps this to a multiple of $\sigma=(\beta+i\gamma)/2$, on
which $\MH(G)$ depends holomorphically. So
$\sigma^*=(\beta^*+i\gamma^*)/2$ is holomorphic in ${}^L\neg\sigma$,
and is independent of $^L\neg\alpha$; in particular,
\begin{equation}\label{vfrac}
\frac{\partial\gamma^*}{\partial{}^L\neg\alpha}=0.
\end{equation}

Taken together, the statements in the last two paragraphs mean that
the output triple $(\alpha^*,\beta^*,\gamma^*)$ is a  linear
function of the input triple
$({}^L\neg\alpha,{}^L\neg\beta,{}^L\neg\gamma)$. One way to argue
this begins with the Cauchy-Riemann equations asserting that
$\beta^*+i\gamma^*$ is holomorphic in
${}^L\neg\beta+i{}^L\neg\gamma$:
\begin{align}\notag \frac{\partial\beta^*}{\partial {}^L\neg\beta} &
=\frac{\partial\gamma^*}{\partial {}^L\neg\gamma}\\
\frac{\partial\gamma^*}{\partial {}^L\neg\beta} &
=-\frac{\partial\beta^*}{\partial {}^L\neg\gamma}.\end{align} Using
this and eqn. (\ref{ufrac}), we see that
$\partial^2\gamma^*/\partial{}^L\neg\gamma^2
=\partial^2\beta^*/\partial{}^L\neg\beta\partial{}^L\neg
\gamma = -\partial^2\gamma^*/\partial{}^L\neg\beta^2=0$.  Since
$\gamma^*$ is independent of ${}^L\neg\alpha$ and ${}^L\neg\beta$
according to eqns. (\ref{ufrac}), (\ref{vfrac}), it follows that
$\gamma^*=f({}^L\neg\gamma)$ where $f$ is a possibly inhomogeneous
linear function between the two Lie algebras.  Then holomorphy of
$\alpha^*-i\gamma^*$ in ${}^L\neg\alpha-i{}^L\neg\gamma$ and of
$\beta^*+i\gamma^*$ in ${}^L\neg\beta+i{}^L\neg\gamma$ implies that
\begin{equation}(\alpha^*,\beta^*,\gamma^*)
  =(f({}^L\neg\alpha)+d_1,f({}^L\neg\beta)+d_2, f({}^L\neg\gamma))
\end{equation}
with some constants $d_1$ and $d_2$ and the same linear function
$f$. In particular
\begin{equation} \sigma^*=f({}^L\neg\sigma)+d_2/2\end{equation}
must be a linear function of ${}^L\neg\sigma$.

We can compare this prediction of linearity to the results of
eqns. (\ref{gesto}), where we computed $\sigma^*$ in our example.
(As remarked in \secref{impom}, we consider only the improper
component, as the two are equivalent for these questions.)  From
eqn. (\ref{gesto}) we see that, up to permutation of the two
eigenvalues of $\sigma^*$, we have
\begin{equation}\label{celmo}(\sigma_1^*,\sigma_2^*)=(\sigma_0,0).
\end{equation}
Here $\pm \sigma_0$ are the eigenvalues of ${}^L\neg\sigma$. In
particular, the linearity holds, and moreover the linear function
$f$ is homogeneous and $d_2=0$. We would conjecture (without a solid
physics-based argument) that also $d_1=0$, in which case
$\varrho^*=\alpha^*-i\gamma^*$ is a homogeneous linear function of
$^L\neg\varrho={}^L\neg\alpha-i{}^L\neg\gamma$. Then writing
$^L\neg\varrho={\rm diag}({}^L\neg\varrho_0,-{}^L\neg\varrho_0)$
($\varrho$ is traceless as it takes values in $\mathfrak{sl}_2$) and
$\varrho^*=(\varrho^*_1,\varrho^*_2)$, eqn. (\ref{celmo}) implies
that
\begin{equation}\label{pelmo}(\varrho^*_1,\varrho^*_2)=
({}^L\neg\varrho_0,0).
\end{equation}
The Weyl symmetry between $^L\neg\varrho_0$ and $-{}^L\neg\varrho_0$
has been lost because in \secref{impom}, we have chosen one of two
possible maps $\M_H^0\to \M$ as the starting point in associating a
twisted $\cal D$-module on $\M$ to an $A$-brane on $\M_H^0$.

Now go back to the  ramified $SO_3$ local system $E\to C\backslash
p$ with semi-simple monodromy around $p$. If we pick an extension
over $p$ of the holomorphic (or topological) structure of $E$, then
the monodromy $M$ of $E$ around $p$ makes sense as an element of
$SL_2$ (not just $SO_3$), and  is conjugate to ${\rm
diag}(\lambda,\lambda^{-1})$ with
$\lambda=\exp(-2\pi{}^L\neg\varrho)$. The analogous monodromy $M^*$
of the local system on $\M$ is conjugate to  $${\rm
diag}(\exp(-2\pi\varrho_1^*),\exp(-2\pi\varrho_2^*)).$$ So according
to eqn. (\ref{pelmo}), up to conjugacy the monodromies of the local
system on  $\M$ are
\begin{equation}\label{firstone} M^*={\rm diag}(\lambda,1).
\end{equation}

What has been described so far applies to the generic case of an
$A$-brane on $\MH$ supported on a smooth fiber $\FF$ of the Hitchin
fibration.  Such an $A$-brane is associated with a rank 2 twisted
local system on $\M$ with singularities at a few points, and the
monodromies of the singularities are as just described.  They are
actually the same at each singularity.

For geometric endoscopy, we want the analogous description for the
case that $\FF$ is a union of two components $\FF_1$ and $\FF_2$.
The fractional $A$-branes supported on just one component are
associated with rank 1 twisted local systems over a Zariski open set
in $\M$.  It is explained in Remark \ref{endosc} that each
Lagrangian submanifold $\FF_1$ or $\FF_2$ has a simple pole
analogous to that in eqn. (\ref{gesto}) at precisely 2 of the 4
points $z^*$. Hence the rank 1 local system on $\M$ corresponding to
either of these branes has precisely 2 singularities, at each of
which the monodromy is $\lambda$.

This may seem to entail a contradiction. The structure group of a
rank 1 local system is the abelian group $GL_1$, so the product of
its monodromies must equal 1, rather than $\lambda^2$ as we seem to
get in the last paragraph. The resolution of this question entails a
point that we have been omitting: $\M_H^0$ is not the cotangent
bundle of $\M$ but an affine deformation of this cotangent bundle.
As we discuss next, this leads to a twist of the local system such
that the product of monodromies at its singularities is central but
is not simply equal to 1.

\subsection{The Central Twist}\label{twist}

Our next task will be to understand $A$-branes in $X=\tilde T^*Y$,
where $X$ is an affine deformation of $T^*Y$.  The deformation is
classified by a class $\zeta\in H^1(Y,\Omega^{1,cl}(Y))$.  Viewing
$\zeta$ as a closed $(1,1)$-form, the integral $\int_Y\zeta/2\pi i$
is a complex number that we  call $\zeta_\C$.

$A$-branes in $\tilde T^*Y$ are related to twisted $\cal D$-modules
on $Y$ that are twisted in a more interesting way than we have
discussed so far in this paper.  The relevant $\cal D$-modules are
modules not for the sheaf of differential operators acting on
$K_Y^{1/2}$, but for the sheaf of such operators acting on
$K_Y^{1/2}\otimes \cal T_{\chi}$, where $\cal T_{\chi}$ is a ``line
bundle'' of first Chern class $\chi$. ($\chi$ is a class whose
relation to $\zeta$ will be described.) $\cal T_{\chi}$ does not
really exist as a line bundle unless the class $\chi$ is integral,
but the sheaf $\cal D_{\chi}^*$ of differential operators acting on
$K_Y^{1/2}\otimes \cal T_{\chi}$ does exist in any case.

In the ramified case of geometric Langlands, this can be deduced using
representations of affine Kac--Moody algebras at the critical level
\cite{FG,F:ram} or by analyzing the canonical coisotropic brane
$\A_{cc}$, as in Section 4.4 of \cite{GW}.  Aiming to recover this
result from the viewpoint adopted here, we let $D$ be a curve in $X$
such that the projection $\pi:X\to Y$ restricts to a map $\pi:D\to Y$
that is generically a finite cover. However, we allow $D$ to have
``poles'' over finitely many points in $Y$, as in \secref{poles}. $D$
is also endowed with a suitable line bundle $\cal L$.  It is
instructive to first consider the case that $\pi:D\to Y$ is
generically of degree 1, so that $D$ is a section of $\pi:X\to Y$ over
a Zariski open set in $Y$.  In this case, there must be poles, if
$\zeta\not=0$. Otherwise, $D$ would be a global holomorphic section of
$\pi$, whose existence would trivialize the class $\zeta$.

In the case that $X$ actually is globally the cotangent bundle
$T^*Y$, the curve $D$ corresponds to a meromorphic differential on
$X$.  This differential is the Higgs field $\varphi$ of the rank 1
Higgs bundle that corresponds to $D$.  At each point $u_i\in Y$ at
which $\varphi$ has a pole, one can define its residue, called
$\sigma_i^*$ in eqn. (\ref{zerlo}).  According to the residue
theorem, the sum of residues vanishes:
\begin{equation}\sum_i\sigma_i^*=0.\end{equation}
 When we use Hitchin's equations (which in this rank 1 case
reduce to ordinary Hodge theory) to relate the pair $(D,\cal L)$ to
a rank 1 local system $E\to Y$, the absolute value of the monodromy
at $u_i$ is $\exp(-2\pi {\rm Im}\,\sigma_i^*)$. The vanishing of the
sum of the residues maps to the fact that for a $GL_1$ local system,
the product of these absolute values is 1.  (The phases of the
monodromies are determined by $\alpha_i^*$.)

Now let us consider the case that $X=\tilde T^*Y$ is actually an
affine deformation of $T^*Y$.  In this case, at each singular point
$u_i$, we can still define a residue $\sigma_i^*$; indeed, the
affine deformation is trivial locally, and though there is no
canonical way to trivialize it, a change in the trivialization does
not affect the residue.  What is new in the case of an affine
deformation is that the sum of residues no longer vanishes. Instead,
we have
\begin{equation}\label{turk}\sum_i\sigma_i^*-\zeta_\C=0.\end{equation}
(One way to derive this result is to observe that we could
trivialize $\zeta$ away from an arbitrarily chosen point $r\in Y$.
Relative to this trivialization, $D$ has one more singular point,
with residue $-\zeta_\C$, and eqn. (\ref{turk}) is simply the
vanishing of the sum of all residues.) The fact that the sum of the
residues is a nonzero constant corresponds to the fact that the
product of the absolute values of the monodromies is a constant not
equal to 1, in fact equal to $\exp(2\pi {\rm Im}\,\zeta_\C)$.

This resolves the puzzle mentioned at the end of
\secref{polynomial}.  In that computation, we used coordinates $z$
on $Y=\M$ and $v$ on the fiber of $\tilde T^*\M$.  The choice of
these coordinates trivialized the affine deformation except at
$z=\infty$.  In particular, the deformation was trivialized near the
poles of $D$ and hence did not affect the residues of those poles.
With this trivialization, $D$ has an additional pole at $z=\infty$,
and of course, the sum of all residues, including that last one,
vanishes.

In differential geometry, one can represent $\zeta$ as a
closed\footnote{Of course, on a curve, every $(1,1)$-form is closed,
but the construction is actually meaningful in higher dimensions, as
we discuss in \secref{multidim}.} $(1,1)$-form and describe this
situation as follows. A meromorphic section $D$ of $\pi:\tilde
T^*Y\to Y$ does not correspond to a meromorphic differential
$\omega$. Rather it corresponds to a $(1,0)$-form on $Y$, with poles
as above, and obeying
\begin{equation}\label{exce}\bar\partial\varphi = \zeta.\end{equation}
Integration of this formula over $C$ leads back to eqn.
(\ref{turk}); one integrates the left hand side using the poles of
$\varphi$ and expresses the integral of the right hand side in terms
of $\zeta_\C$.

Now let us consider the rank $n$ case of this.  In other words, we
suppose that $D\subset \tilde T^*Y$ is a curve such that $\pi:D\to
Y$ is generically an $n$-fold cover.  $D$ moreover is endowed with a
line bundle $\cal L$.  Can we associate a $GL_n$ Higgs bundle to the
pair $(D,\cal L)$?  We can define the bundle $E$ as usual by
$E=\pi_*(\cal L)$, and we can also imitate the usual definition of
the Higgs field, $\varphi=\pi^*(z)$.  Now, however, $\varphi$ is not
a holomorphic map $E\to E\otimes K_Y$.  Rather, each branch of
$\varphi$, or more exactly the contribution to $\varphi$ from each
branch of $\pi:D\to Y$, is twisted exactly as in eqn. (\ref{exce}).
Consequently, the equation obeyed by $\varphi$ is
\begin{equation}\bar\partial_E\varphi=1\otimes \zeta,\end{equation}
where $\bar\partial_E$ is the $\bar\partial$ operator of the bundle
$E$, and $1:E\to E$ is the identity.

This type of twisting makes sense for Hitchin's equations
(\ref{otto}), which we can slightly generalize to
\begin{align}\label{toto}\notag F-\phi\wedge \phi&=1\otimes i\,\zeta'
\\\notag d_A\star \phi&=1\otimes i\,{\rm Im}\,\zeta\\
d_A \phi& = 1\otimes i\, {\rm Re}\,\zeta.\end{align}  Along with the
deformation by $\zeta$ seen in the last paragraph, which has been
rewritten in terms of real differential geometry, we have included
an additional real\footnote{The reason for the factors of $i$ on the
right hand sides of these equations is that $\phi$ and $F$ are forms
valued  in the Lie algebra of $U_n$, the compact form of $GL_n$; we
represent this Lie algebra by anti-hermitian $n\times n$ matrices.}
closed $(1,1)$-form $\zeta'$ in the first equation. Up to trivial
equivalence, the equations depend only on the cohomology classes of
$\zeta'$, ${\rm Im}\,\zeta$, ${\rm Re}\,\zeta$. Those classes are a
triple of real parameters, like the usual parameters
$(\alpha,\beta,\gamma)$ of ramified Higgs bundles.

We expect that $A$-branes in $\tilde T^*Y$ should be mapped to
twisted $\cal D$-modules on $Y$ by solving this twisted form of
Hitchin's equations.  The right choice of $\zeta'$ is not determined
by the geometry of the $A$-brane, somewhat like the value  of
$\alpha^*$ in \secref{poles}.  This issue will be discussed in
\secref{zame}.

After solving the twisted equations (\ref{toto}), we form the
complex-valued connection $\CA=A+i\phi$.  Usually, Hitchin's
equations imply that $\CA$ is flat.  In the present case, however,
$\CA$ is not flat, but rather has central curvature ${\eusm
F}_0=i(\zeta'+i\,{\rm Re}\,\zeta)$.  A connection with central
curvature $\eusm F_0$ describes not a $\cal D$-module but a twisted
$\cal D$-module, twisted by $\T_\chi$, a ``line bundle'' of first
Chern class $\chi=\eusm F_0/2\pi i$.  In a more intrinsic
description, the twisting is really by $K_Y^{1/2}\otimes \T_\chi$,
where the role of $K_Y^{1/2}$ has been described earlier.

\subsection{The $B$-Field}\label{zame}
At two points we have run into important parameters that are not
determined by the geometry of an $A$-brane, namely the ramification
parameter $\alpha^*$ at a pole of $D\subset \tilde T^*Y$, and the
twisting class $\zeta'$ of eqn. (\ref{toto}).  In the abstract, any
values of these parameters make sense and can be used to construct
twisted $\cal D$-modules by solving the twisted version of Hitchin's
equations with singularities.  We would like to know what values are
relevant in the context of geometric Langlands.

For the twisting parameter, we can make an argument based on
holomorphy that is very similar to what we have already said for the
ramification parameter.  The cohomology class $\chi$, like the
twisting parameter $\varrho^*$, must be holomorphic in
${}^L\neg\alpha-i{}{}^L\neg\gamma$.  As we have just seen, ${\rm
Im}\,\chi$ is determined by the affine deformation by which a
Zariski open set in $\MH$ differs from the cotangent bundle $T^*\M$.
In \cite{GW}, by analyzing the geometry of $\MH$, it was shown that
this is linear in ${}^L\neg\gamma$, and hence ${\rm Re}\,\chi$ is
similarly linear in ${}^L\neg\alpha$.

However, we would like to explain more directly how $\zeta'={\rm
Re}\,\chi$ originates in the quantum field theory. Up until this
point, in our study of $A$-branes on a symplectic manifold such as
$X=\tilde T^*Y$, we have omitted an important aspect of the
$A$-model, namely the $B$-field.  Thus, we have defined the
$A$-model purely in terms of the symplectic form $\omega$ of $X$. In
general, the $A$-model depends holomorphically on the complexified
symplectic form $\hat \omega=\omega+iB$.

In the absence of the $B$-field, $A$-branes on a Lagrangian
submanifold $X$ are characterized by a rank $n$ vector bundle $V\to
L$ with a flat unitary connection $A$.  (For brevity, we omit here
the twisting by $K_L^{1/2}$.)  Thus, writing $F$ for the curvature
of $A$, the condition is $F=0$.  For $B\not=0$, however, the
condition is modified to include a central
curvature:\footnote{Physicists usually omit the factor of $i$ on the
right hand side because of defining $F$ to take values in hermitian
rather than anti-hermitian $n\times n$ matrices.}
\begin{equation}\label{twisted}F=1\otimes i B.\end{equation}
Thus, the twisting field $\zeta'$ of eqn. (\ref{toto}) is precisely
$B$.  $B$ is really only defined modulo the addition of an exact
two-form.

 As analyzed in \cite{GW}, under
duality from the $B$-model of $^L\neg G$ to the $A$-model of $G$,
$^L\neg\alpha$ maps to a quantum parameter $\eta$ of the $A$-model,
and in turn the $B$-field on $\MH(G)$ is linear in $\eta$, in the
following sense. $\eta$ takes values in $\mathfrak t$, the Lie
algebra of the compact form of $G$. The cohomology class $[B]$ of
$B$ is an element of $H^2(\MH,\mathbb{R})$, which has a natural
identification as $\mathbb{R}\oplus\mathfrak t$.  (For simplicity, we
consider the case of only one ramification point; otherwise, one has
a similar analysis with more parameters.) Relative to this,
$[B]=c\oplus \eta$, where $c$ is a constant that vanishes in the
usual case of geometric Langlands (and is nonzero in a
generalization known as quantum geometric Langlands).

Consider first the case that $\M_H^0$ reduces to $T^*\M$. This
actually occurs if $^L\neg \beta={}^L\neg\gamma=0$,  as shown in
\cite{GW}. Then $Y=\M$ can be embedded as a Lagrangian submanifold.
Consider the special case of eqn. (\ref{toto}) with $\varphi=0$. If
we set $\zeta=0$ and $\zeta'=B$, this reduces to eqn.
(\ref{twisted}) and therefore is the right equation to describe a
rank $n$ $A$-brane supported on $\M\subset T^*\M$. Therefore the
twisting parameter $\zeta'$ of the category of $A$-branes is
precisely $B$, or more precisely, its cohomology class; hence,
$\zeta'$ coincides with the ramification parameter $^L\neg\alpha$ of
the underlying $^L\neg G$ ramified local system. Solutions with
$\varphi\not=0$ describe the possible deformation of the $A$-brane
from a rank $n$ brane supported on $\M$ to a rank 1 brane supported
on a more general Lagrangian submanifold $D\subset T^*\M$ ($D$ is an
$n$-fold cover of $\M$).  $A$-branes supported on $D$ are twisted in
the same way, since the twisting applies to the whole $A$-category.

In geometric Langlands, we are actually interested in a case in
which $D\subset T^*\M$ has poles.  However, $B$ when restricted to
$T^*\M$ is a pullback from $\M$ (modulo addition of an irrelevant
exact two-form) and the twisting is just as in the last paragraph.
Turning on ${}^L\neg\beta$ and ${}^L\neg\gamma$, which are multiples
of $\beta$ and $\gamma$, deforms the classical geometry from $T^*\M$
to $\tilde T^*\M$, but (as analyzed in \cite{GW}) does not change
the cohomology class of $B$ and so does not affect this contribution
to the twisting.

In summary, in the standard case of geometric Langlands, with one
ramification point, $\zeta'$ is equal to $0\oplus{} ^L\neg\alpha$.
With $k$ ramification points,
$H^2(\MH,\RR)=\RR\oplus\left(\oplus_{i=1}^k\mathfrak
t^{(i)}\right)$, where the $\mathfrak t^{(i)}$ are copies of
$\mathfrak t$.  Each ramification point has a corresponding
parameter $^L\neg \alpha^{(i)}$ and a similar argument shows that
$\zeta'=0\oplus\left(\oplus_{i=1}^k{}^L\neg\alpha^{(i)}\right)$.

\subsubsection{Conjectural Description of
$\al_i^*$}\label{conjdesc}

Now that we have determined $\zeta'$, we discuss the parameters
$\al_i^*$, where $i$, as before, labels the divisors in $\M$ where
the Hitchin fiber ``goes to infinity.''

Here some general remarks on $B$-fields might be helpful. A general
$B$-field of the type that we consider describes mathematically a
$U(1)$-gerbe with a connection (similar objects will also be
discussed in \secref{gerbes}). A gauge transformation of such an
object is made by taking the tensor product with a line bundle ${\mc
L}$ with a unitary connection $\nabla$.   In our situation, the
gerbe is actually trivial, though not canonically so, which is why
its connection may be represented by the globally defined two-form
$B$ that was discussed above. Under a gauge transformation, this
two-form gets shifted by the curvature of $\nabla$.

Now let $\FF$ be a compact fiber of the Hitchin fibration, and
$\FF_0$ its restriction to $\tilde T^*\M$.  The cohomology class of
$B$ restricted to $\FF$ actually vanishes.  This can be seen from
the analysis in \cite{GW}, and ensures that a rank 1 $A$-brane
supported on $\FF$ can exist.  Such an $A$-brane is defined by
giving a flat trivialization $\cal T$ of the gerbe, that is, a
trivialization relative to which the connection form vanishes.
(Often, it is assumed that the gerbe has an {\it a priori}
trivialization, and then the $A$-brane is defined by a flat line
bundle over $\FF$.  However, the intrinsic formulation is that the
$A$-brane is defined by a flat trivialization of the gerbe.)
 On the other hand, upon restricting $B$ to $\FF_0$, we
may use a gauge in which $B$ is a pullback from $\M$. As such, $B$
may well be non-zero, but we can take it to be smooth and hence
bounded in norm. The difference between the two gauges is given by a
line bundle ${\mc L}\to \FF_0$ with a connection $\nabla$. The
connection is not flat, but its curvature is uniformly bounded,
since $B$ is bounded on $\M$. Therefore the monodromies of ${\mc L}$
around the divisors in $\FF \backslash \FF_0$ are well-defined. We
conjecture that these are actually responsible for the ramification
parameters $\alpha_i^*$.

More precisely, in the neighborhood of each divisor $D_i$ in $\M$
over which the Hitchin fiber $\FF$ goes to infinity, $\FF$ has $n$
branches. Therefore the connection $\nabla$ gives rise to a
collection of monodromies $\la_i^j \in U(1), j=1,\ldots,n$, one for
each branch. We conjecture that $\exp(-2\pi \al_i^*)$ is in the
conjugacy class of $\on{diag}(\la_i^1,\ldots,\la_i^n)$. This
completes our (conjectural) description of the $\al_i^*$.

Finally, we obtain the following description of the ${\mc D}$-module
${\cal F}$ associated to the $A$-brane $(\FF,{\cal T})$.  The
curvature of ${\mc L}\to \FF_0$ is of type $(1,1)$, since $B$ has
this property, so the connection $\nabla$ endows ${\mc L}$ with a
holomorphic structure.  We extend ${\mc L}$ as a holomorphic line
bundle over $\FF$ by saying that a holomorphic section $\psi$ of
${\mc L}\to \FF_0$ is holomorphic on $\FF$ if its magnitude
$|\psi|$, computed using the unitary structure of ${\mc L}$, is
bounded. Moreover, we say that  $\psi$  vanishes to order $\alpha$
near an irreducible component $D'$  of $\FF\backslash\FF_0$ if near
$D'$ one has $|\psi|\sim |z|^\alpha$, with $z$ a normal coordinate
near $D'$. This characterization of the order of vanishing of a
section is well-defined because the curvature of ${\mc L}$ is
bounded.  It endows  ${\mc L}\to \FF$ with parabolic structure along
$\FF\backslash\FF_0$.

Then  $E=\pi_*(\cal L)$, where $\pi: \FF_0 \to \M$ is the
projection, has a natural parabolic structure at the divisors $D_i$.
It is also equipped in the usual way with a natural Higgs field
$\varphi$.

Our proposal is that the restriction of ${\cal F}$ to the complement
of the divisors $D_i$ is a local system obtained by the non-abelian
Hodge transformation of the parabolic Higgs bundle $(E,\varphi)$.

\subsection{Tame And Irregular Singularities}\label{tame}

So far we have considered only the case that the exponents $d_i$
describing the singular behavior of the Lagrangian curve $D\subset
T^*Y$
\begin{equation}\label{normox}s_i(u)\sim  c_i
(u-u_0)^{-d_i},~i=1,\dots,n\end{equation} are equal to 1.  Here we
will briefly relax this assumption.

One case is that the exponents are positive (so $D$ does go to
infinity) but less than 1.  This is only possible if $D$, upon
being completed by adding a point or points at infinity, is
ramified over $u_0$.   For example, taking $u_0=0$, let us
consider a spectral curve $\det(s-\varphi)=0$ with
$\det\,\varphi=1/u$.  The equation for the spectral curve is then
$s^2+u^{-1}=0$, with branches $s=\pm (-u)^{1/2}$, so the exponents
$d_i$ are equal to $1/2$.  A suitable choice of $\varphi$ is
\begin{equation}\label{polfo}\varphi=du\begin{pmatrix}0 & u^{-1}\\
-1 & 0 \end{pmatrix}.\end{equation} $\varphi$ has a pole, but the
polar residue is nilpotent, so $\det\,\varphi$ has only a single
pole, not a double one.   The analysis of \cite{Sim} applies to this
situation; it corresponds to $\beta^*=\gamma^*=0$, as is reviewed in
\cite{GW}, Section 3.3.  Since $\gamma^*=0$, monodromies of the
local system on $Y$ arising from an $A$-brane supported on $D$ will
have modulus 1.

An opposite case is that the exponents $d_i$ are greater than 1. In
this case, $\varphi$ has a pole at $u_0$ of greater than first
order. Hitchin's equations for Higgs bundles in which $\varphi$ has
such an irregular pole have been analyzed in \cite{BB}. The
corresponding local system has an irregular singularity at $u_0$.  The
irregular part of the connection can be read off directly from the
singularity of the Higgs field, in contrast with the subtle interplay
of parameters $\alpha,\beta,\gamma$ that arises in the regular
case. For an informal explanation, see \cite{Wi2}.

\subsection{The Multi-Dimensional Case}\label{multidim}

For geometric Langlands, we require a generalization of all this to
$A$-branes in $\tilde T^*Y$, where $Y$ has dimension greater than 1.
The correspondence between Higgs bundles and local systems has an
extension to higher dimensions \cite{Sim2,Sim3}, and this also
generalizes to Higgs bundles with poles \cite{Bi,Mo}. The
generalization to allow poles is crucial for geometric Langlands,
because a fiber $\FF$ of the Hitchin fibration, when restricted to
$T^*\M$, essentially always has poles. However, here we will
consider only the case without poles, which will suffice for
explaining the basic idea.

A Higgs bundle on a complex manifold $Y$ of dimension $n$ is defined
to be a pair $(E,\varphi)$, where $E\to Y$ is a holomorphic vector
bundle, and $\varphi:E\to E\otimes \Omega^1(Y)$ is a holomorphic map
such that $\varphi^2=0$  ($\varphi^2$ is defined by composing
endomorphisms of $E$ and using the cup product $\wedge^2\Omega^1\to
\Omega^2$).  The last condition is trivial in the case of a curve.
The basic result is that if $Y$ is Kahler, then there is a natural
correspondence between Higgs bundles on $Y$ such that the rational
Chern classes of $E$ vanish and local systems on $Y$.  The
correspondence is made by solving a multi-dimensional generalization
of Hitchin's equation.  (If $Y$ is projective, this operation is
equivalent to solving Hitchin's equations simultaneously on every
curve in $Y$.)

There is also a natural correspondence between a Higgs bundle
$(E,\varphi)$ and a spectral variety $D\subset T^*Y$ endowed with a
line bundle $\cal L$.  Letting $\pi:T^*Y\to Y$ denote the
projection, the correspondence is made in one direction by setting
$E=\pi_*(\cal L)$ and $\varphi=\pi_*(z)$, where $z$ is linear  on
the fiber of the cotangent bundle.

We want to state this correspondence for a rank 1 $A$-brane
supported on a smoothly embedded compact Lagrangian submanifold
$D\subset T^*Y$.  (Examples of this situation for ${\rm dim}\,Y>1$
are actually extremely scarce, and this is another reason that the
extension to allow poles of $D$ is  essential.) Such a rank 1
$A$-brane  exists if and only if the canonical bundle of $D$ admits
a square root (this is related to Remark \ref{tormo}). To specify
such a brane, we pick a flat ${\rm Spin}_c$ structure over $D$ that
we describe somewhat non-intrinsically by a choice of line bundle
$K_D^{1/2}\otimes {\cal N}$, where $\cal N$ has zero first Chern
class.   Then we define\footnote{If $K_Y$ does not have a global
square root, then $E$ must be understood as a twisted vector bundle,
twisted by a certain gerbe.  In any event, that is the most
intrinsic formulation.} a vector bundle $E\to Y$ by
$K_Y^{1/2}\otimes E=\pi_*(K_D^{1/2}\otimes \cal N)$, or in other
words $E=\pi_*(\cal L)$ with $\cal L=K_D^{1/2}\otimes
\pi^*(K_Y^{-1/2})\otimes \cal N$. This extends to a Higgs bundle
$(E,\varphi)$ with $\varphi=\pi_*(z)$.  The rank of $E$ coincides
with the degree $n$ of the cover $\pi:D\to Y$.

The Higgs bundle $(E,\varphi)$ will correspond to a local system
over $Y$ if the Chern character of $E$, which we denote as ${\rm
ch}(E)$, is equal to $n$. An argument via Riemann-Roch shows that
this is always the case.  In stating this argument, we use the $\hat
A$ class, which is related to the Todd class of a complex manifold
$X$ by  $\hat A(X)={\rm Td}(X)\ch(K_X^{1/2})$. The $\hat A$ class is
defined for any real vector bundle $V$, and has the property that
$\hat A(V)=\hat A(V^*)$, where $V^*$ is the dual to $V$.

The Riemann-Roch formula says in this situation that
\begin{equation}\pi_*({\rm Td}(D)\ch(K_D^{1/2}))={\rm
Td}(Y)\ch(K_Y^{1/2})\ch(E).\end{equation} We have used the fact that
as $c_1(\cal N)=0$, $\ch(\cal N)=1$. In terms of $\hat A$ classes,
this says that
\begin{equation} \pi_*(\hat A(D))=\hat A(Y){\rm
ch}(E).\end{equation} Hence ${\rm ch}(E)$ equals $n$ if and only if
$\pi_*(\hat A(D))=n\hat A(Y)$.  Since this is a statement in
rational cohomology, it is equivalent to show that
\begin{equation} 2\pi_*(\hat A(D))=2n\hat A(Y).\end{equation}
Now by definition $\hat A(D)$ is the $\hat A$ class of the tangent
bundle $TD$ of $D$.  So $2\hat A(D)=\hat A(TD\oplus TD)$.  As
$D\subset T^*Y$ is Lagrangian, the normal bundle and tangent bundle
to $D$ are isomorphic (as real vector bundles). Their direct sum is
$T(T^*Y)|_D$, that is, the restriction to $D$ of the tangent bundle
$T(T^*Y)$ to $T^*Y$.   So we can replace $TD\oplus TD$ by
$T(T^*Y)|_D$, and interpret $2\pi_*(\hat A(D))$ as $\pi_*(\hat
A(T(T^*Y)|_D))$. As $Y\subset T^*Y$ is also Lagrangian, we can
similarly interpret $2\hat A(Y)$ as $\hat A(T(T^*Y)|_Y)$. So the
formula we want is that
\begin{equation}\label{hopeful}\pi_*(\hat A(T(T^*Y)|_D))=n\hat
A(T(T^*Y)|_Y).\end{equation} Since $T^*Y$ is contractible to $Y$,
any vector bundle $V\to T^*Y$ is isomorphic to the pullback from $Y$
of its restriction to $Y$, that is $V\cong \pi^*(V|_Y)$. Setting
$V=T(T^*Y)$, we have  $T(T^*Y)=\pi^*(T(T^*Y)|_Y)$. Hence in
particular $T(T^*Y)|_D$ is the restriction to $D$ of
$\pi^*(T(T^*Y)|_Y)$. The left hand side of eqn. (\ref{hopeful}) is
therefore  $\pi_*(\hat A(\pi^*(T(T^*Y)|_Y)))$.   In general, for any
map $\pi:X\to Y$ and real vector bundle $V\to Y$, one has $\hat
A(\pi^*(V))=\pi^*(\hat A(V))$.  So the left hand side of
(\ref{hopeful}) is $\pi_*(\pi^*(\hat A(T(T^*Y)|_Y)))$. For a map
$\pi:D\to Y$ of degree $n$, the composition $\pi_*\pi^*$ acting on
cohomology is multiplication by $n$, and now the validity of eqn.
(\ref{hopeful}) is clear.

The local system on $Y$ corresponding to the Higgs bundle
$(E,\varphi)$ is then the desired ${\mc D}$-module corresponding to
our starting point, the  rank 1 $A$-brane supported on the
Lagrangian submanifold $D \subset T^* Y$. This is a very special
case since we have assumed $D$ to be compact. In most of the
examples relevant to the geometric Langlands Program (such as the
ones considered earlier in this section) this is not so, and the
corresponding ${\mc D}$-modules are not represented by local systems
on the entire $Y$, only on an open subset of $Y$. We have explained
in Sections \ref{polynomial}--\ref{zame} how to construct the
corresponding local systems on an open subset of $Y$ in the case
when $\dim Y = 1$. A similar picture hopefully holds in the
multi-dimensional case.  For some of the necessary analysis, see
\cite{Bi}, \cite{Mo}.

\section{Spectral Covers, Hecke Operators, and Higher Genus}
\label{spectral}

The analysis of \secref{genus one} was primarily based on direct
computations.  One can learn much more using the technique of
spectral curves, which was briefly introduced in
\secref{specreview}.  In \secref{oneagain}, we reconsider the genus
1 example from this point of view, in \secref{higher genus} we apply
similar ideas to the case of genus greater than 1, and in
\secref{Hecke} we use these methods to analyze the action of Hecke
operators on $A$-branes arising in the geometric Langlands program.

\subsection{Genus One Revisited}\label{oneagain}

We return first to the example of \secref{genus one}, involving a
curve $C$ of genus 1, described by a cubic equation
\begin{equation}\label{cubic} y^2=f(x),\end{equation}
where $f(x)=(x-e_1)(x-e_2)(x-e_3)$.  We consider $SL_2$ Higgs bundles
$(E,\varphi)$ ramified at the point $p$ defined by
$x=y=\infty$. Ramification means that near $p$, the eigenvalues of
$\varphi$ behave as $\pm \sigma_0\, dx/2x$, and hence
$\det\,\varphi\sim -\sigma_0^2 dx^2/4x^2\sim
-(\sigma_0^2/4)x(dx/y)^2$. Suppose that $\sigma_0 \neq 0$. Then the
general form of a quadratic differential on $C$ that is holomorphic
except for this behavior near $p$ is
$\det\,\varphi=-(\sigma_0^2/4)(x-a)(dx/y)^2$, for some complex
constant $a$.  So after absorbing in $z$ a factor of
$(dx/y)\sigma_0/2$, the equation (\ref{gelf}) of the spectral curve
$D$ becomes \begin{equation}\label{cover} z^2=x-a.\end{equation}

 $D$ is, therefore, a  cover of the $x$-plane described by the pair
of equations (\ref{cubic}) and (\ref{cover}). Eliminating $x$ by
means of the second equation, the spectral curve can be described by
the equation
\begin{equation}\label{helor}y^2=\prod_{i=1}^3(z^2+a-e_i).\end{equation}
$D$ is a smooth curve of genus 2, unless  $a$ is equal to $e_1,$
$e_2, $ or $e_3$, in which case $D$ reduces to a curve of genus 1
with two points identified.  For example, if $a=e_1$, we let
$w=y/z$, obeying
\begin{equation}\label{pelor}w^2=(z^2+e_1-e_2)(z^2+e_1-e_3).\end{equation}
This equation describes a smooth curve $D'$ of genus 1. $D$ is
obtained from $D'$ by identifying the two points with $z=0$,
$w=\pm\sqrt{(e_1-e_2)(e_1-e_3)}$. These two points, which we will
call $q'$ and $q''$, correspond to just one point $q$ in $D$, as
they are both characterized by $y=z=0$.

$D'$ is an unramified double cover of $C$.  Indeed, $D'$ has the
freely acting symmetry $\tau:w\to -w,\,z\to -z$.  The invariants
are $x=z^2+e_1$, as well as  $w^2$, which can be expressed in
terms of $x$ via eqn. (\ref{pelor}), and $y=zw$, which obeys
$y^2=\prod_{i=1}^3(x-e_i)$, the defining equation of $C$.  So $C$
is the quotient $D'/\{1,\tau\}$.

We write $p'$ and $p''$ for the two points in $D'$ with
$z=\infty$; they lie above the point $p$ at infinity in $C$.   The
symmetry $\tau$ of $D'$ acts freely and is of order 2, so if we
regard $D'$ as an elliptic curve with, say, $p'$ as the origin,
then $\tau$ is the shift by an element of order 2.  This element
is simply $p''$, since $\tau$ exchanges $p'$ and $p''$. If $r$ is
any point in $C$ and $r',r''$ are the points in $D'$ lying above
$r$, then $\tau$ exchanges $r',r''$ so
\begin{equation}\label{useful}r''-r'=p''-p'.\end{equation}
The divisor $p'+p''-q'-q''$ is the divisor of the function $z$, so
its divisor class is trivial. Together with eqn. (\ref{useful}),
this implies that the points $q',q''$ are of order 2.

\subsubsection{Fiber Of The Hitchin Fibration}\label{fibergof}

If $D$ is smooth, the fiber of the Hitchin fibration is a smooth
curve of genus 1, the Prym variety of the double cover $\pi:D\to
C$. Let us discuss what happens when $D$ is singular, for example
at $a=e_1$.

We want to find the line bundles, or more generally torsion-free
sheaves, on $D$ that push down to vector bundles $E\to C$ of trivial
determinant.  We will describe our line bundles and sheaves in terms
of data on the smooth curve $D'$.  So let us begin by ``pushing down''
the trivial line bundle ${\cal O}_{D'}$, using the projection
$\rho:D'\to C$.  We compute the pushdown using the fact that the
quotient $D'/\{1,\tau\}$ is $C$. We can decompose $\rho_*({\cal
O}_{D'})$ in subsheaves that are even or odd under $\tau$. The even
part is ${\cal O}_C$, and the odd part is a locally free sheaf ${\cal
S}\to C$. ${\cal S}$ is uniquely determined by the fact that it is
nontrivial and $\rho^*({\cal S}^{-1})$ has a global section over $D'$.
We can describe ${\cal S}$ via the divisor $-p+q$, since this pulls
back on $D'$ to the trivial divisor $-p'-p''+q'+q''$.

Thus, ${\cal O}_{D'}$ does {\it not} have the property that
$\rho_*({\cal O}_{D'})$ has trivial determinant.  Its determinant is
${\cal O}(p)^{-1}\otimes {\cal O}(q)$.  However, let $p^*$ be either
of the points $p',p''$ and let $q^*$ be either of the points
$q',q''$. Let ${\cal T}_0={\cal O}(p^*)\otimes {\cal O}(q^*)^{-1}$.
Then $\det(\rho_*({\cal T}_0))={\cal O}_C$.  All we need here is
that $p^*$ pushes down to $p$ and $q^*$ to $q$; this implies that
$\det(\rho_*({\cal T}_0))=\CO(p)\otimes \CO(q)^{-1}\otimes
\det(\rho_*({\cal O}_{D'}))=\CO_{C}$. It looks like we have 4
choices of ${\cal T}$, but up to isomorphism there are only 2; in
view of eqn. (\ref{useful}), if we reverse the choice of $q^*$ while
also reversing the choice of $p^*$, ${\cal T}$ is unchanged.

So we have found two  line bundles ${\cal T}_0\to D'$ that push
down to  $SL_2$-bundles $E\to C$. They are the only ones.  But if
we work on the singular curve $D$ rather than its normalization
$D'$,  there are more choices. We first replace the degree zero
line bundle ${\cal T}_0$ with the degree 1 line bundle ${\cal
T}_1(p^*)={\cal O}(p^*)$. We think of ${\cal T}_1(p^*)$ as a line
bundle on $D'$ that is trivialized away from $p^*$. We cannot
interpret ${\cal T}_1(p^*)$ as a line bundle on $D$ (as opposed to
$D'$) since in its definition, we have not taken account of the
identification of the two points $q'$ and $q''$ on $D'$ to a
single point $q \in D$. To do this, we pick $\lambda\in \C^\times$
and define a line bundle ${\cal T}(\lambda;p^*)$ over $D$ by
saying that a section of ${\cal T}(\lambda;p^*)$ is a section $f$
of ${\cal T}_1$ such that $f(q')=\lambda f(q'')$. This gives
(since there are two choices of $p^*$) two families of complex
line bundles over $D$, each parametrized by $\C^\times$.

To compactify these families, we set $\lambda=-v/u$, where $u$ and
$v$ will be understood as homogeneous coordinates for
$\mathbb{CP}^1$, and replace the condition $f(q')=\lambda f(q'')$
by
\begin{equation} \label{hh}u f(q')+v f(q'')=0.\end{equation}
For any $u,v$, the sheaf of sections of ${\cal T}_1(p^*)$ that
obey this condition is a torsion-free sheaf ${\cal R}(u,v;p^*)$ on
$D$ whose pushdown to $C$ is a rank two vector bundle of trivial
determinant. If $u,v\not=0$, this sheaf is locally free and is the
sheaf of sections of the line bundle ${\cal T}(-v/u;p^*)$.  If $u$
or $v$ vanishes, we get a torsion-free but not locally free sheaf
on $D$. This torsion-free sheaf is the pullback from $D'$ of the
line bundle ${\cal O}(p^*)\otimes {\cal O}(q^*)^{-1}$, where $q^*$
is $q'$ if $v=0$ (and eqn. (\ref{hh}) reduces to $f(q')=0$), or
$q''$ if $u=0$.

The four line bundles ${\cal O}(p^*)\otimes {\cal O}(q^*)^{-1}\to
D'$, with the different choices of $p^*$ and $q^*$, are isomorphic
in pairs, because of the relation  (\ref{useful}), with $r=q$.
Consequently, ${\cal R}(0,1;p')$ is isomorphic to ${\cal
R}(1,0;p'')$, and similarly with $p',p''$ exchanged.

This construction gives two families ${\cal R}(u,v;p^*)$ of
torsion-free sheaves on $D$.  Each family is parametrized by
$\mathbb{CP}^1$, with homogeneous coordinates $u,v$.  The two
$\mathbb{CP}^1$'s meet at two points, because of the remark in the
last paragraph.   The fiber of the Hitchin fibration for $SL_2$ is
the union of these $\mathbb{CP}^1$'s.  This agrees with the picture
that we developed in \secref{hitchfib} by direct computation.

\subsubsection{The Improper Component}\label{improper}

We can similarly use spectral curves to describe the Hitchin
fibration for the improper component of the $SL_2$ moduli space,
introduced in \secref{secomp}.  This component parametrizes Higgs
bundles $(E,\varphi)$ with $\det\,E={\cal O}(r)$, $r$ being a
point in $C$.  Here we will treat an issue that was omitted in
\secref{secomp}: the dependence on $r$.  So we refer to the
improper component as $\MH(SL_2^*;r)$.

This space is independent of $r$ up to a not quite canonical
isomorphism; given another point  $\tilde r$, we pick a line
bundle ${\cal N}$ whose square is isomorphic to $\CO(\tilde
r)\otimes \CO(r)^{-1}$, and then tensoring with ${\cal N}$ gives a
map from $\MH(SL_2^*;r)$ to $\MH(SL_2^*;\tilde r)$.   The choice
of ${\cal N}$ is unique modulo tensoring with a line bundle of
order 2, so the identification of $\MH(SL_2^*;r)$ with
$\MH(SL_2^*;\tilde r)$ is unique modulo the action of the group
$Q$ of line bundles of order 2. (Hence $\MH(SL_2^*;r)$ becomes
naturally independent of $r$ if one divides by $Q$; this gives the
moduli space $\MH(SO_3;\theta)$ of $SO_3$ Higgs bundles with
non-zero second Stieffel-Whitney class, whose definition requires
no choice of  $r$.) If $\tilde r$ is close to $r$, we can resolve
the ambiguity by asking that ${\cal N}$ should be near the
identity (in the Picard group of $C$), so locally there is a
natural identification of $\MH(SL_2^*;r)$ with $\MH(SL_2^*;\tilde
r)$. Hence there is a natural monodromy action, the group of
monodromies being simply $Q$.

Regardless of $\det E$, the spectral curve for a Higgs bundle
$(E,\varphi)$ is defined by the equation $\det(z-\varphi)=0$.
Hence, the relevant spectral curves $D$ for the improper component
of the Hitchin fibration are the same as for the proper component.
The difference is only that now the fiber of the Hitchin fibration
parametrizes line bundles ${\cal R}\to D$, or more generally
torsion-free sheaves, such that $\det(\pi_*({\cal R}))=\CO(r)$
(rather than $\det(\pi_*({\cal R}))=\CO$).   Just as before, the
fiber of the Hitchin fibration is smooth if $D$ is smooth; we want
to consider the special fibers for which $D$ is not smooth, but
has for normalization a smooth genus 1 curve $D'$.

  To construct those
special fibers, we repeat the previous construction, now beginning
not with the line bundle $\CO(p^*)\to D'$, but with ${\cal
T}_1=\CO(p^*)\otimes \CO(r^*)$, where $r^* $ is either of
$r',r''$. There are seemingly four choices of ${\cal T}_1$,
involving the choices of $p^*$ and $r^*$, but since
$r''-r'=p''-p'$ on the elliptic curve $D'$, there are only two
choices up to isomorphism. Just as for the proper component of the
moduli space, we associate to either of these line bundles over
$D'$ a family ${\cal T}(u,v;p^*,r^*)$ of torsion-free sheaves on
$D$, by imposing eqn. (\ref{hh}).  Each family is parametrized by
$\mathbb{CP}^1$, and, as before, ${\cal T}(0,1;p^*,r')$ is
isomorphic to ${\cal T}(1,0;p^*,r'')$, and vice-versa.  So the two
$\mathbb{CP}^1$'s meet at two points. Their union is the fiber of
the Hitchin fibration.

Now we can consider monodromies when $r$ varies in $C$.  These
will exchange the two choices of $r^*$, and so will exchange the
two components of the Hitchin fiber. This agrees with the fact
that the monodromy group is the group $Q$ of line bundles of order
2, and that the action of $Q$ exchanges the two components of the
Hitchin fiber, as we saw in \secref{symgroup}.

\subsection{Extension To Higher Genus}\label{higher genus}

Our next goal is to apply the same methods to the case that $C$ is
a smooth curve of genus $g>1$. We will see that the results are
similar. For simplicity, we omit ramification. For $g=1$, the
unramified case is rather degenerate, but that is not so for
$g>1$.

\subsubsection{The Spectral Curves}\label{curves}

What sort of Higgs bundles over $C$ will be related to endoscopy?
We consider endoscopic $SO_3$ local systems whose structure group
reduces to $O_2$, the subgroup consisting of elements of the form
\begin{equation}\begin{pmatrix}* & * & 0 \\ * & * & 0 \\
                               0&0&\pm 1\end{pmatrix},\end{equation}
but not to a proper subgroup. For such local systems, the
automorphism group is equal to $\Z_2$.  The correspondence between
local systems and Higgs bundles given by Hitchin's equations is
compatible with any reduction of the structure group.  So $SO_3$
local systems with structure group reducing to $O_2$ correspond to
Higgs bundles with the same structure group. Those which do not
further reduce to a proper subgroup have the group of automorphisms
equal to $\Z_2$. From now on we will restrict ourselves to these
Higgs bundles.

If we lift such a Higgs bundle to $SL_2$, the structure group lifts
to what we will call $O_2^*$, a double cover of $O_2$ generated by
the diagonal elements of $SL_2$
\begin{equation}\begin{pmatrix}* & 0 \\ 0 & *
\end{pmatrix}\end{equation}
together with the element
\begin{equation}\begin{pmatrix} 0 & 1 \\ -1 & 0
\end{pmatrix}.\end{equation}

An $SL_2$ spectral curve $D$ is defined by an equation
\begin{equation}z^2+\det\,\varphi=0,\end{equation}
where $\det\,\varphi$ is a quadratic differential on $C$.  For $C$
of genus $g$, a quadratic differential has  $4g-4$ zeroes. $D$ is
smooth if and only if the zeroes of $\det\,\varphi$ are distinct, in
which case the genus of $D$ is $4g-3$.

What happens when the structure group reduces to $O_2^*$?
 The Lie
algebra of $O_2^*$ is simply the abelian algebra of traceless
diagonal matrices.  So if $(E,\varphi)$ is an $SL_2$ Higgs bundle
whose structure group reduces to $O_2^*$, then $\varphi$ locally
takes the form
\begin{equation}\varphi=\begin{pmatrix} \omega & 0 \\ 0 &-\omega
\end{pmatrix},\end{equation}
 where $\omega$ is a holomorphic differential on $C$. This leads to
$\det\varphi = -\omega^2$, as a result of which any zeroes of
$\det\varphi$ are double zeroes. Near a double zero at, say,
$x=0$, with $x$ a local parameter on $C$, the equation for $D$
looks something like $z^2-x^2=0$; the point $z=x=0$ is a double
point. Each double point reduces by 1 the genus of the
normalization of $D$. For an $O_2^*$ Higgs bundle, the zeroes are
all double zeroes, so there are  $2g-2$ double points and the
normalization of $D$ is a curve $D'$ of genus $2g-1$.

The fact that the zeroes of $\det\varphi$ are double zeroes does
not imply that globally $\det\varphi=-\omega^2$ for a holomorphic
section  $\omega$ of the canonical bundle $K$.  Rather, it implies
that $\det\varphi=-\omega^2$ where $\omega$ is a holomorphic
section of $K\otimes \cal V$, for some line bundle $\cal V\to C$
of order 2. The case of interest to us is that $\cal V$ is
non-trivial. (The case that $\cal V$ is trivial is related to
Higgs bundles whose structure group reduces to $GL_1$ rather than
$O_2^*$.)
 Associated with the choice of $\cal V$ is an unramified
double cover $\pi':D'\to C$.  This is a curve of genus $2g-1$, and
is the normalization of $D$. Let $\tau:D'\to D'$ be the covering
map that commutes with $\pi'$. Then $C=D'/\{1,\tau\}$. The element
$\omega\in H^0(C,K\otimes \cal V)$ vanishes at $2g-2$ points
$p_1,\dots,p_{2g-2}$.  Above each such point $p_i$ there are two
points $p_i'$, $p_i''$ in $D'$,  exchanged by $\tau$, but only a
single point $\hat p_i\in D$.  $D$ is obtained from $D'$ by gluing
together the pairs of points $p_i'$ and $p_i''$.

What we have just described is a natural analog of the result of
\secref{oneagain} for a genus 1 curve $C$ with one point of
ramification. A curve of genus 1 has precisely three non-trivial
unramified double covers; these are the normalizations $D'$ of the
three singular spectral curves $D$.

\subsubsection{The Prym}    \label{prym}

An open dense subset of the Hitchin fiber $\FF$ for $SL_2$ consists
of line bundles ${\cal L}\to D$ such that $\det\pi_*({\cal
L})={\cal O}$. The full Hitchin fiber parametrizes certain
torsion-free sheaves as well as these line bundles.

Given such a line bundle ${\cal L}$, let ${\cal L}'$ be its pullback
to $D'$.  Then $\det\pi'_*({\cal L}')={\cal
O}(p_1+\dots+p_{2g-2})$. (When $\cal L$ is lifted to $D'$, we drop the
requirement that a section must have equal values at the points $p_i'$
and $p_i''$ lying above $p_i$, and this leads to the claimed result.)
The space of such line bundles is\footnote{This statement means that
if ${\cal L}'$ and ${\cal L}''$ are two line bundles with
$\det\pi'_*(\cal L')=\det\pi'_*(\cal L'')=\O(p_1+\dots+p_{2g-2})$,
then ${\cal L}''={\cal L}'\otimes \cal N$ for a unique $\cal N$ with
$\det\pi'_*(\cal N)=\cal O$.} a ``torsor'' for the group of line
bundles ${\cal N}\to D'$ with $\det\pi'_*({\cal N})={\cal O}$. Since
$\pi':D'\to C$ is unramified, $\det\pi'_*({\cal N})$ is the same as
${\rm Nm}(\cal N)$, the\footnote{For any map of curves $\pi:D\to C$,
the norm is a map from line bundles over $D$ to line bundles over $C$
defined as follows. The norm of $\cal N=\O(\sum_i n_i p_i)\to D$, for
integers $n_i$ and points $p_i\in D$, is defined as ${\rm Nm}(\cal N)=
\O(\sum_i n_i\pi(p_i))$.} ``norm'' of $\cal N$. The group of line
bundles of trivial norm is called the Prym variety of $\pi':D'\to
C$. (Sometimes the term ``Prym variety'' is taken to refer to the
connected component of this group.)  We write $\mbox{\bf \em P}$ for
the Prym and $\mbox{\bf \em T}$ for its torsor of line bundles $\cal
L'\to D'$ obeying $\det\pi'_*({\cal L}')={\cal
O}(p_1+\dots+p_{2g-2})$.

We can easily construct a large family of line bundles over $D'$
of trivial norm.  We take any points $s_i\in C$, $i=1,\dots,k$
(allowing some of the points to coincide), denote as $s_i'$ and
$s_i''$ the points in $D'$ lying above $s_i$ (with any choice of
which one is which), and define a line bundle ${\cal N}$ by
\begin{equation}    \label{tsen}
{\cal N}=\bigotimes_{i=1}^k\O(s_i'-s_i'').\end{equation} Conversely,
one can show that every point of $\mbox{\bf \em P}$ can be represented
by a line bundle of this form. To see that (we thank T. Pantev for
showing us this argument), let $E$ be a divisor on $D'$ such that
${\mc N} = {\mc O}(E)$. Then $\on{Nm}({\mc N}) = {\mc
O}_C(\pi'(E))$. If ${\mc N} \in \mbox{\bf \em P}$, then $\pi'(E) =
(f)$, the divisor of a rational function $f$ on $C$. But it follows
from Tsen's theorem that the norm map from the field of rational
functions on $D'$ to the field of rational functions on $C$ is
surjective.\footnote{The norm of a rational function $g$ on $D'$ is by
definition the product of $g$ and $\tau(g)$, where $\tau$ is the
involution on $D'$ corresponding to the cover $D'\to C$.} Hence there
exists a rational function $g$ on $D'$ whose norm is $f$. Then
$\pi'((g)) = (f)$. Let $E' = D - (g)$. Then $\pi'(E') = 0$ and therefore
$E'$ has the form $\sum_{i=1}^k \left( s_i'- s_i'' \right)$. Thus, we
obtain eqn. \eqref{tsen}.

Now let us find the connected components of the Prym.  Obviously,
the part of $\mbox{\bf \em P}$ that we can construct for fixed $k$ is
connected, since $C$ itself is connected and the points $s_i$ may
move freely.  Line bundles that differ by an exchange
$s_i'\leftrightarrow s_i''$ lie in the same connected component of
$\mbox{\bf \em P}$, since these two points are exchanged under monodromy
of $s_i$. Line bundles that differ by changing $k$ by a multiple
of 2 are also in the same connected component; $k$ is reduced by 2
if we take 2 of the $s_i$ to be equal, with the points labeled
$s_i'$ and $s_i''$ chosen properly, and use the identity
$(s_i'-s_i'')+(s_i''-s_i')=0$.

On the other hand, $\mbox{\bf \em P}$ actually has two connected
components. This is shown in \cite{Mum}, and also follows by a purely
topological argument (see the discussion of eqn.
(\ref{condi})).\footnote{Another proof is presented in \cite{Ngo1},
Section 11.} The argument in the last paragraph shows that the only
possible invariant is the value of $k$ modulo 2.  So it must be that
the component of $\mbox{\bf \em P}$ connected to the identity is
characterized by even $k$, while the disconnected component is
characterized by odd $k$.

The torsor $\TT$ is non-canonically isomorphic to
$\mbox{\bf \em P}$,
so it likewise has two components. We proceed as in
\secref{fibergof} to construct the fiber $\FF$ of the Hitchin
fibration from $\TT$. Let ${\cal L}'$ be any line bundle
representing a point in $\TT$.  For $i=1,\dots,2g-2$,  pick a
pair of homogeneous coordinates $(u_i,v_i)$, and define a line
bundle ${\cal L}'(u_1,v_1;\dots;u_{2g-2},v_{2g-2})\to D$ by saying
that a section of this line bundle is a section $f$ of $\cal L'\to
D'$ that obeys
\begin{equation}\label{gludata}
u_if(p_i')+v_if(p_i'')=0,~~i=1,\dots,2g-2.\end{equation} If, for all
$i$, $u_i$ and $v_i$ are both non-zero, this construction gives a
family of line bundles over $D$, representing (if we also let $\cal
L'$ vary in $\TT$) a Zariski open set of the Hitchin fiber $\FF$. This
open set has two connected components $\wt\FF_1$ and $\wt\FF_2$,
because $\cal L'$ may lie in either component of $\TT$. They are
isomorphic to $(\C^\times)^{2g-2}$-bundles over the two components of
$\TT$.

As in \secref{fibergof}, to get the full Hitchin fiber, we must
compactify by including torsion-free sheaves that are obtained by
allowing $u_i$ or $v_i$ to vanish, for each $i$.  The compactified
fiber has two irreducible components, which we call $\FF_1$ and
$\FF_2$. They are isomorphic to $(\CP^1)^{2g-2}$-bundles over the two
components of $\TT$. However, just as in the genus 1 example, $\FF_1$
and $\FF_2$ intersect over the divisors on which $u_i$ or $v_i$
vanish.  The reason for this is that starting with a line bundle
${\cal L}'\to D'$ and taking $u_i= 0$ in the construction of the last
paragraph is equivalent to starting with a different line bundle
${\cal L}''={\cal L}'\otimes \O(p_i'-p_i'')$ and taking $v_i= 0$.  But
the operation ${\cal L}'\to {\cal L}''$ exchanges the two connected
components of $\TT$.

\subsubsection{The Improper Component}\label{improperly}

We also want to understand the Higgs bundles $(E,\varphi)$ where $\det
E=\O(r)$, $r$ being a specified point in $C$.  The Hitchin fiber in
this case can be analyzed via the same methods. We simply have to
start with a different torsor $\TT(r)$ for the same Prym variety
$\mbox{\bf \em P}$. $\TT(r)$ parametrizes line bundles $ \cal L'\to
D'$ with $\det\pi'_*(\cal L')=\cal O(p_1+\dots+p_{2g-2}+r)$. $\TT(r)$
again has two connected components, exchanged by tensor product with
$\O(s'-s'')$ for any $s\in C$. Correspondingly, the fiber of the
Hitchin fibration has two components, meeting along the divisors that
parametrize torsion-free sheaves that are not locally free.

Let $\FF$ and $\FF^*(r)$ be the Hitchin fibers for Higgs bundles
$(E,\varphi)$ with, respectively, $\det\,E=\O$ and
$\det\,E=\O(r)$. Then $\FF$ and $\FF^*(r)$ are non-canonically
isomorphic.  To make an isomorphism, we simply pick a point $r'\in
D'$ lying above $r$. Then, for $\L'\in \TT$, we define a line
bundle ${\cal L}'_{r'}\in \TT(r)$ by
$\L'_{r'}=\L'\otimes\O(r')$. We map $\FF$ to $\FF^*(r)$ by
$\L'(u_i,v_i)\to\L'_{r'}(u_i,v_i)$. This is an isomorphism between
$\FF$ and $\FF^*(r)$, but it is not quite canonical since it depends
on the choice of $r'$.

Now consider the effect of a monodromy in $r$ that exchanges the
two points $r'$ and $r''$.  The effect of this is to map
$\L'_{r'}$ to $\L'_{r''}=\L'_{r'}\otimes \O(r''-r')$. This
operation exchanges the two components of $\FF^*(r)$.  Thus,
monodromy in $r$ exchanges the two components of $\FF^*(r)$, just as
we saw in the genus 1 example at the end of \secref{improper}.

\subsubsection{Topological Point Of View}\label{topolview}

Here we will explain from a topological point of view the fact
that the Prym $\mbox{\bf \em P}$ for an unramified (but connected)
double cover $D'\to C$ has two components.

This Prym is the fiber of the Hitchin fibration for certain
$O_2^*$ local systems.  We recall that $O_2^*$ is the subgroup of
$SL_2$ generated by the diagonal matrices together with the
element
\begin{equation}\label{belk}\begin{pmatrix} 0 & 1\\ -1 & 0
\end{pmatrix}.\end{equation}
Under the double cover $SL_2\to SO_3$, $O_2^*$ projects to
$O_2\subset SO_3$.  Topologically, $O_2^*$ has two components; the
component containing the identity consists of diagonal elements,
and the other component consists of elements of $O_2^*$ that are
not diagonal.

The statement that the Prym has two components is equivalent to
the statement that even after we pick an unramified double cover
$D'\to C$, the corresponding  component of $\MH(O_2^*)$  actually
has two components. (The base of the Hitchin fibration for a given
$D'$ is connected, so the components of $\MH(O_2^*)$ are simply
the components of the Prym.)  Using the relation of $O_2^*$ Higgs
bundles to $O_2^*$ local systems, the question of determining the
components of $\MH(O_2^*)$ is equivalent to the analogous question
about $O_2^*$ local systems and can be answered topologically.

 Picking suitable generators of the fundamental
group of $C$, and writing $A_i,B_j$, $i,j=1,\dots,g$ for the
monodromies of an $O_2^*$ local system, we have
\begin{equation}\label{condi}[A_1,B_1][A_2,B_2]\cdots
[A_g,B_g]=1,\end{equation} where $[A,B] = ABA^{-1}B^{-1}$. We specify
$\cal V$ by saying, for example, that $B_g$ lies in the disconnected
component of $O_2^*$ and all other $A_i$ and $B_j$ in the connected
component.  Since the connected component is abelian,
eqn. (\ref{condi}) reduces to $A_gB_gA_g^{-1}B_g^{-1}=1$, which
implies (for $B_g$ in the disconnected component) that $A_g$ is one of
the two central elements of $O_2^*$.  The two components of $\mbox{\bf
\em P}$ are distinguished by the choice of $A_g$.

\subsubsection{Analog For $SO_3$}\label{analog}

To understand endoscopy, we must consider the fiber of the Hitchin
fibration for $SO_3$ rather than $SL_2$.  Just as in eqn.
(\ref{turgid}), the moduli space of $SO_3$ Higgs bundles is
obtained from the moduli space of $SL_2$ Higgs bundles by dividing
by the group $Q=H^1(C,\Z_2)$ of line bundles of order 2.  $Q$ acts
on a Higgs bundle $(E,\varphi)$ by $E\to E\otimes {\cal R}$, where
$\cal R$ is a line bundle of order 2.

This operation does not affect $\det \varphi$, so the action of
$Q$ commutes with the Hitchin fibration.  The action on a fiber of
the Hitchin fibration is easily described; if $E=\pi_*({\cal L})$
for a line bundle ${\cal L}$ over the spectral cover, then
$E\otimes \cal R=\pi_*({\cal L}\otimes \pi^*({\cal R}))$.  So the
action of $Q$ on the fiber of the Hitchin fibration is by ${\cal
L}\to {\cal L}\otimes \pi^*(\cal R)$.

In the  case of a smooth spectral curve $D$, the fiber $\FF$ of
the Hitchin fibration is a complex torus, and the operation ${\cal
L}\to {\cal L}\otimes \pi^*(\cal R)$ is a translation on this
torus. It acts without fixed points. Now let us see what happens
in the case related to endoscopy, when the normalization of $D$ is
an unramified double cover $D'\to C$.   $\FF$ is usefully
described,\footnote{The action of $Q$ on the improper fiber
$\FF^*(r)$ can be considered similarly.} as we have seen, in terms
of the torsor $\TT$ that parametrizes line bundles $\L'\to D'$
with $\det\pi_*(\L')=\O(p_1+\dots+p_{2g-2})$. $Q$ acts on this
torsor by $\L'\to \L'\otimes \cal R$.  $Q$ also acts on the gluing
data of eqn. (\ref{gludata}), as  we will discuss momentarily.

A first basic fact about this endoscopic case is that   \cite{Mum}
the action of $Q$ exchanges the two components of the Prym
$\mbox{\bf \em P}$, and hence of the torsor $\TT$. (Of course, a
subgroup of $Q$ of index 2 maps a given component to itself.) This
fact has a simple topological explanation using eqn.
(\ref{condi}). $Q$ acts by independent sign changes on all $A_i$
and $B_j$; the two components of $\mbox{\bf \em P}$ are exchanged
by any element of $Q$ that changes the sign of $A_g$.

 So after dividing by $Q$, the Hitchin fiber
for $SO_3$ has only one component, in contrast to the situation
for $SL_2$. This is familiar from \secref{struc}.

Second, and again familiar from \secref{struc}, the  moduli space
of $SO_3$ Higgs bundles has singularities that arise because the
action of $Q$ is not quite free.  How can this occur?  If $\cal
R\to C$ is a line bundle of order 2 whose pullback to $D'$ is
non-trivial, then the operation $\cal L\to \cal L\otimes \cal R$
acts freely on $\TT$ and hence on $\FF$.

However, $\pi':D'\to C$ is an unramified double cover associated
with a line bundle $\cal V\to C$ of order 2, and tautologically
the pullback  $(\pi')^*(\cal V)$ of $\cal V$ to $D'$ is trivial.
Hence the element of $Q$ corresponding to $\cal V$ acts trivially
on $\TT$.  But it acts non-trivially on the gluing condition
of eqn. (\ref{gludata}).  Triviality of $(\pi')^*(\cal V)$ means
that this line bundle has an everywhere non-zero global section
$w$. This section is odd, rather than even, under the covering map
$\tau:D'\to D'$. (Otherwise $w$ would descend to a section of
$\cal V\to C$, contradicting the fact that  $\cal V$ is
non-trivial.) The action of $\cal V$ modifies the gluing condition
of eqn. (\ref{gludata}) to $u_i\tilde f(p_i')+v_i\tilde
f(p_i'')=0$, where $\tilde f = f w$. Since $w$ is odd under the
covering map, which exchanges $p_i'$ and $p_i''$ for all $i$, the
effect of this is to transform the gluing data by
\begin{equation}\label{zono}(u_i,v_i)\to
(u_i,-v_i),~~i=1,\dots,2g-2.\end{equation} Bearing in mind that
the pair $u_i,v_i$ are homogeneous coordinates for a copy of
$\CP^1$, the condition for a fixed point of $\cal V$ is that $u_i$
or $v_i$ must vanish for all $i$. This is $2g-2$ conditions.

We also adjusted $2g-2$ parameters so that the spectral curve $D$
has for its normalization a double cover $D'\to C$.  As a check, and
also a confirmation that fixed points of this kind only occur for
the sort of spectral curves that we have assumed, we observe that
the fact that the $Q$ action preserves the complex symplectic
structure of the moduli space of Higgs bundles, together with the
Lagrangian nature of the Hitchin fibers, implies that the number of
parameters on the fiber of the Hitchin fibration that must be
adjusted to get a fixed point equals the number of parameters that
must be adjusted on the base.

So altogether, the $\Z_2$ fixed points that we have found occur in
codimension\ $4g-4$. The local structure is $\C^{4g-4}/\Z_2$. These
singularities are not $A_1$ singularities, since they arise in
complex codimension greater than 2.  But the generalities of
\secref{braneduals} still apply. There are two inequivalent
$B$-branes $\cal B_+$ and $\cal B_-$ supported at a $\Z_2$ orbifold
singularity of any codimension; our basic proposal is that {\em
mirror symmetry maps this fact to the fact that the corresponding
Hitchin fiber for $SL_2$ has two components, $\FF_1$ and $\FF_2$.}
Thus, two $B$-branes $\cal B_+$ and $\cal B_-$ give rise to two
$A$-branes $\cal A_1$ and $\cal A_2$ supported on these two
components. Which one of them corresponds to $\cal B_+$ and which to
$\cal B_-$ is a subtle issue, which will be discussed in
\secref{gerbes} (this is the reason for the choice of notation,
$\cal A_1, \cal A_2$, and not $\cal A_+, \cal A_-$).

The locus of $\Z_2$ singularities has dimension $2g-2$.  There are
$g-1$ parameters in picking a spectral curve $z^2+\det\,\varphi=0$,
where $\det\,\varphi$ has only double zeroes, and $g-1$ more
parameters in picking an appropriate line bundle $\cal L'\to D'$.

\subsubsection{The Transfer}    \label{transfer}

Let us now consider this result from the point of view of $SO_3$
local systems.  An $SO_3$ local system should represent a $\Z_2$
orbifold singularity of $\MH$ if its structure group reduces to
$O_2$, but not to a proper subgroup thereof.  The moduli space of
$O_2$ Higgs bundles has dimension $2g-2$, which agrees with the
dimension of the above-described family of orbifold singularities.

More specifically, the Higgs bundles representing the
singularities are precisely $O_2$ Higgs bundles. $\MH(O_2)$ is
simply the quotient by $Q=H^1(C,\Z_2)$ of $\MH(O_2^*)$.  So its
Hitchin fibration is easily understood: the base is the same as it
is for $O_2^*$, and the fiber is the quotient of the Prym by $Q$.
This is the same as the singular locus of $\MH(SO_3)$ that we have
just described.

An argument just like that surrounding eqn. (\ref{condi}) shows
that $\MH(O_2)$ has two components, even after an unramified
double cover $D'\to C$ is specified.  Indeed, from the point of
view of an $O_2$ local system $U$, the choice of unramified double
cover specifies the Stieffel-Whitney class $w_1(U)$, and the two
components for a given choice of $w_1(U)$ differ by the value of
$w_2(U)$.  The map $U\to U\oplus \det\,U$ from an $O_2$ local
system to an $SO_3$ local system kills $w_1$ and leaves $w_2$
unchanged.  So all components of $\MH(O_2)$ with a given value of
$w_2$ appear as $\Z_2$ orbifold singularities in the component of
$\MH(SO_3)$ labeled by the same value of $w_2$.

The fact that $\MH(O_2)$ appears as a locus of singularities in
$\MH(SO_3)$ is not special to the pair $O_2$ and $SO_3$.  For any
reductive Lie group $\LG$ and reductive subgroup $\LH$, one has a
natural embedding $\MH(\LH)\subset \MH(\LG)$.  If the centralizer of
$\LH$ in $\LG$ is non-trivial, then $\MH(\LH)$ will be a locus of
singularities.  This embedding leads to a natural functor (direct
image) from the category of $B$-branes on $\MH(\LH)$ to the category
of $B$-branes on $\MH(\LG)$, and this will have to give rise to a
functor from the category of $A$-branes on $\MH(H)$ to the category of
$A$-branes on $\MH(G)$, as shown on the following diagram:
$$
\begin{CD}
B\text{-branes on } \MH(\LG) @>{\sim}>> A\text{-branes on } \MH(G) \\
@AAA @AA{\text{transfer}}A \\
B\text{-branes on } \MH(\LH) @>{\sim}>> A\text{-branes on } \MH(H)
\end{CD}
$$

This is closely related to what in the Langlands Program is called
the {\em transfer} or the {\em functoriality principle} (see
\cite{L,Arthur:funct}). In the classical setting (discussed in more
detail in \secref{classical}), this means that for any homomorphism
of the dual groups $\LH \to \LG$ one expects to have a map (the
``transfer'') from the set of equivalences classes of irreducible
automorphic representations of $H({\mathbb A})$ (more precisely,
their $L$-packets) to those of $G({\mathbb A})$. In the geometric
setting, automorphic representations are replaced by ${\mc
D}$-modules on $\Bun_G$, or $A$-branes on $\MH(G)$, and the transfer
becomes a functor between appropriate categories associated to $H$
and $G$, as in the above diagram. This functor should be compatible
with the action of the Hecke/'t Hooft operators (discussed in the
next section) on the two categories.

Mirror symmetry of the Hitchin fibrations provides a natural setup for
constructing such a functor.  In physical terms, one might hope to
study this situation by introducing a {\em supersymmetric domain
wall}, with ${\cal N}=4$ supersymmetric Yang-Mills theory of gauge
group $\LH_c$ (the compact form of $\LH$) on one side of the domain
wall, and the same theory with gauge group $\LG_c$ on the other
side. For some examples of string-theoretic constructions of such
domain walls, see \cite{Wi}.

In our example $\LH = O_2$ and $\LG = SO_3$. On the dual side we have
$G = SL_2$ and $H(F)$ is a twisted torus in $SL_2(F)$, where $F$ is
the field of rational functions on $C$, which is defined as
follows. Consider the moduli space $\MH(O_2,0)$ of $O_2$-Higgs bundles
with $w_2=0$, embedded into $\MH(SO_3,0)$. It has components
parametrized by the set $J_2$ of unramified double covers $D' \to
C$. For each $\psi \in J_2^\times$ corresponding to an unramified
cover we define $H_\psi(F)$ as the group of non-zero rational
functions $f$ on $D'$ such that $\tau(f) = f^{-1}$, where $\tau$ is
the involution of the cover. It is naturally realized as a subgroup of
$SL_2(F)$. The transfer functor of the diagram above is then
constructed as follows.

Each component $\MH(O_2,0)_\psi$ of $\MH(O_2,0)$ corresponding to
$\psi \in J_2^\times$ is a toric fibration over the corresponding
locus $\BB_\psi$ in the Hitchin base $\BB$. We have $\BB_\psi =
H^0(C,K \otimes {\mc L}_\psi)$, where ${\mc L}_\psi$ is the line
bundle on $C$ corresponding to $\psi$, and the map $\BB_\psi \to \BB =
H^0(C,K^2)$ is given by $\omega \mapsto \omega^2$. Since we wish to
avoid the local systems that reduce to proper subgroups of $O_2$, we
consider the complement $\MH(O_2,0)_\psi^\circ$ of the zero fiber in
$\MH(O_2,0)_\psi$. This is a subvariety in $\MH(SO_3,0)$ projecting
onto $\BB_\psi^\circ = (\BB_\psi \bs 0) \subset \BB$. The union of
these varieties is precisely the ``elliptic endoscopic locus'' of
$\MH(SO_3,0)$ (see \cite{Ngo1} and \secref{other}). For each point $b
\in \BB_\psi^\circ$ the Hitchin fiber $\FF_b(O_2)_\psi \subset
\MH(O_2,0)_\psi$ is identified with the moduli space of rank one
unitary local systems on each of the two components, $\FF_{b,1}$ and
$\FF_{b,2}$, of the corresponding singular Hitchin fiber $\FF_b$ in
the dual moduli space $\MH(SL_2)$. Indeed, as explained in
\secref{prym}, each $\FF_{b,i}$ is isomorphic to a
$(\CP^1)^{2g-2}$-bundle over an abelian variety
$\FF_b(O_2)_\psi^\vee$, which is dual to $\FF_b(O_2)_\psi$. The
transfer (outside of the zero Hitchin fiber) is then implemented via
the fiberwise $T$--duality of the toric fibration
$\MH(O_2,0)_\psi^\circ$. In particular, the skyscraper $B$-brane
supported at a point $\E \in \FF_b(O_2)_\psi \subset
\MH(O_2,0)^\circ_\psi$ gives rise to a magnetic eigenbrane on $\FF_b
\subset \MH(SL_2)$, which is the sum of the pull-backs of the
corresponding rank one unitary local system on $\FF_b(O_2)_\psi^\vee$
to $\FF_{b,1}$ and $\FF_{b,2}$.

In addition, $\MH(O_2)^\circ_\psi$ is embedded into $\MH(SO_3)$ as the
locus of $\Z_2$-orbifold singularities. This leads to a ``doubling''
of the category of $B$-branes supported on this component. On the dual
side this is reflected in the fact that the dual Hitchin fibers have
two components, $\FF_{b,1}$ and $\FF_{b,2}$, also leading to a
``doubling'' of the corresponding category of $A$-branes.

More general examples will be considered in \secref{other}.

\subsection{'t Hooft/Hecke Operators}\label{Hecke}

Much of the richness of the geometric Langlands program comes from
the fact that the ${\cal D}$-modules dual to local systems are
eigensheaves for the geometric Hecke operators.  In the quantum
field theory approach, this arises from a duality between line
operators that are known as Wilson and 't Hooft operators.

We want to describe the refinement of  this picture associated
with endoscopy.  For simplicity, we focus on our usual example
with $^L\neg G=SO_3$ and a local system with automorphism group
$\Z_2$.

\subsubsection{Review Of Wilson Operators} First we describe the
action of Wilson operators in general (see Section 8 of
\cite{KW}).  We let $(\hat A,\hat\phi)$ be the universal solution
of the $SO_3$ Hitchin equations over $\MH(SO_3)\times C$.  Thus,
$\hat A$ is a unitary connection on an $SO_3$ bundle $\eurm W\to
\MH(SO_3)\times C$, and $\hat\phi$ a section of ${\rm ad}(\eurm
W)\otimes \Omega^1_C\to \MH(SO_3)\times C$, such that if we
restrict to $m\times C$ for a point $m\in \MH(SO_3)$, then $(\hat
A,\hat \phi)$ is gauge-equivalent to the solution of Hitchin's
equations determined by $m$.  As usual, the restriction of $(\hat
A,\hat\phi)$ to $m\times C$ determines either an $SO_3$ local
system over $C$ or an $SO_3 $ Higgs bundle.

A Wilson operator in $^L\neg G$ gauge theory is associated to the
choice of a point $p\in C$ and a representation $^L\neg R$ of $^L\neg
G$.  For simplicity, we will take $^L\neg R$ to be the
three-dimensional representation of $^L\neg G=SO_3$, and we write
$W_p$ for the corresponding Wilson operator.  The action of $W_p$ on
$B$-branes can be described as follows.  Let $\cal B$ be a $B$-brane
associated with a coherent sheaf (or a complex of coherent sheaves)
${\cal K}\to \MH$.  Then $W_p\cdot {\cal B}$ is the $B$-brane
associated with the sheaf ${\cal K}\otimes {\eurm W}|_p$, where
${\eurm W}|_p$ is the restriction of ${\eurm W}$ to $\MH\times p$.
(We understand $\eurm W$ as a rank 3 vector bundle with structure
group $SO_3$.)  Thus, the action of $W_p$ on coherent sheaves is
\begin{equation}\label{donkey} {\cal K}\to {\cal K}\otimes {\eurm
    W}|_p.\end{equation}

This formula makes sense for any of the complex structures that
make up the hyper-Kahler structure of $\MH$, since $\eurm W$, when
restricted to $\MH\times p$, is holomorphic in any complex
structure.  (It carries a natural connection whose curvature is of
type $(1,1)$ for each complex structure.)  In geometric Langlands,
one is interested primarily in complex structure $J$, in which
$\MH(SO_3)$ parametrizes $SO_3$ local systems, but the existence
of the other complex structures simplifies some computations, as
we will see.

A very important special case is that the brane ${\cal B}$ is
simultaneously a $B$-brane for each of the complex structures of
$\MH$.  We then call $\cal B$ a brane of type $(B,B,B)$ -- a $\cal
B$-brane in complex structures $I$, $J$, or $K$ (or any linear
combination).    Examples are a brane supported at a point in
$\MH$ -- the case that we consider momentarily -- and a brane
defined by an inclusion $\MH({}^L\neg G')\subset \MH({}^L\neg G)$,
for some subgroup ${}^L\neg G'\subset {}^L\neg G$. In this case,
$W_r\cdot {\cal B}$ is again a brane of type $(B,B,B)$. Thus, the
action of $W_r$ preserves the full topological symmetry of type
$(B,B,B)$ (that is, of the $B$-model in any complex structure).

Of particular interest are the eigenbranes of the Wilson
operators, also called electric eigenbranes. One defines the
tensor product of a brane with a vector space $V$ as follows:  if
$\cal B$ is defined by a sheaf $\cal K$, then $\cal B\otimes V$ is
defined by the sheaf $\cal K\otimes V$. A brane $\cal B$ is called
an eigenbrane of $W_p$ if
\begin{equation}\label{onerel}W_p\cdot {\cal B}={\cal B}\otimes
  V_p,\end{equation}
for some vector space $V_p$. We will call $V_p$ the multiplier.  In
complex structure $J$, the $B$-model of $\MH$ is related to a
four-dimensional topological field theory and general arguments can be
used (see Section 6.4 of \cite{KW}) to show that if
eqn. (\ref{onerel}) holds for one value of $p$, then it holds for all
$p$ and $V_p$ varies as the fiber of a local system.

Comparing to (\ref{donkey}), we see that for a brane $\cal B$ to be
an electric eigenbrane, the bundle ${\eurm W}|_p$ must be trivial --
equivalent to a constant vector space -- when restricted to the
support of $\cal B$.  This is so if the support of $\cal B$ is a
smooth point $m\in \MH(SO_3)$.  More specifically, we take $\cal B$
to be the brane (known in the physics literature as a zero-brane)
associated with a skyscraper sheaf supported at $m$. Such a brane is
an electric eigenbrane with multiplier the vector space $\eurm
W|_{m\times p}$, that is, the restriction of $\eurm W\to \MH\times
C$ to the point $m\times p\in \MH\times C$:
\begin{equation}\label{tworel}W_p\cdot {\cal B}={\cal B}\otimes
\eurm W|_{m\times p}.\end{equation}

This statement holds in any complex structure, so we can think of
$\cal B$ as an eigenbrane of type $(B,B,B)$.  In other words, it
is a $B$-brane in complex structure $I, J$, or $K$ on $\MH$ (or
any combination thereof), and furthermore is an eigenbrane in any
complex structure.  In the geometric Langlands Program, one cares
primarily about complex structure $J$, but the fact that the
zero-brane is simultaneously an eigenbrane in all three complex
structures facilitates computations, as will become clear.

\subsubsection{$B$-Branes At An Orbifold
Singularity}\label{borbifold}

We want to repeat this analysis for the case of a brane supported
at a $\Z_2$ orbifold singularity $r\in \MH$. Such a singularity is
associated with an $SO_3$ local system whose structure group
reduces to $O_2$.  We recall that $O_2$ is embedded in $SO_3$ as
the subgroup
\begin{equation}\begin{pmatrix} * & * & 0 \\ * & * & 0 \\ 0 & 0 &
\pm 1\end{pmatrix}\end{equation}  and that any $SO_3$ local system
whose structure group reduces to $O_2$ has symmetry group $\Z_2$,
generated by the central element of $O_2$:
\begin{equation}\begin{pmatrix} -1 & 0 & 0 \\ 0 & -1 & 0 \\ 0 & 0
& 1\end{pmatrix}.\end{equation} As usual, we will consider a generic
local systems of this type whose group of automorphisms is precisely
this $\Z_2$. In the present context, $\eurm W|_r$, the restriction
of $\eurm W$ to $r\times C$, is a local system whose structure group
reduces to $O_2$, so it has a decomposition
\begin{equation}\label{inview}\eurm W|_r=U \oplus \det\,
U,\end{equation} where $U$ is a rank 2 local system, with
structure group $O_2$, and $\det\,U$ is its determinant.  The
central generator of $\Z_2$ acts as $-1$ on $U$ and as $+1$ on
$\det\,U$.

As explained in \secref{braneduals}, the category of branes
supported at the orbifold singularity $r$ is generated by two
irreducible objects $\cal B_+$ and $\cal B_-$. Each is associated
with a skyscraper sheaf supported at $r$.  They differ by whether
the non-trivial element of $\Z_2$ acts on this sheaf as
multiplication by $+1$ or by $-1$.

What happens when we act on $\cal B_+$ or $\cal B_-$ by the Wilson
operator $W_p$?  Since $\cal B_+$ and $\cal B_-$ both have skyscraper
support at $r$, $W_p$ acts on either of them by tensor product with
the three-dimensional vector space $\eurm W|_{r\times p}$, the fiber
of $\eurm W$ at $r\times p$.  However, we should be more precise to
keep track of the $\Z_2$ action. In view of eqn.  (\ref{inview}),
there is a decomposition $\eurm W|_{r\times p}=U|_p\oplus \det\,U|_p$,
where the non-trivial element of $\Z_2$ acts as $-1$ on the first
summand and as $+1$ on the second summand.  So we have
\begin{align}\label{elko}\notag W_p\cdot \cal B_+ & = \left( \cal
  B_-\otimes U|_p \right) \oplus \left( \cal B_+\otimes \det U|_p
  \right) \\ W_p\cdot \cal B_- & = \left( \cal B_+\otimes U|_p \right)
  \oplus \left( \cal B_-\otimes \det U|_p \right).\end{align}

A less precise but illuminating way to rewrite this is as follows.
$\det U|_p$ is a one-dimensional vector space on which $\Z_2$ acts
trivially, so $\cal B_\pm \otimes \det U|_p$ is isomorphic,
non-canonically, to $\cal B_\pm$.  And $U|_p$ is a two-dimensional
vector space on which the non-trivial element of $\Z_2$ acts as
multiplication by $-1$.  So $\cal B_\pm \otimes U|_p$ is isomorphic,
non-canonically, to the sum of two copies of $\cal B_\mp$.  Thus up to
isomorphism we have
\begin{align}\label{helko}\notag W_p\cdot \cal B_+ & = \cal B_+ +
  2\cal B_-\\ W_p\cdot \cal B_- & = \cal B_- + 2\cal
  B_+.\end{align}

We want to understand the magnetic dual of these statements. In
Sections \ref{modif} and \ref{interpretation}, we review geometric
Langlands duality for generic Hitchin fibers, and then in
\secref{reducible}, we consider the behavior for special Hitchin
fibers related to endoscopy.

\subsubsection{$\varphi$-Invariant Hecke Modifications}\label{modif}

The magnetic dual of a Wilson operator $W_p$ is an 't Hooft operator
$T_p$. For the definition of 't Hooft operators and their relation
to the usual Hecke operators of the geometric Langlands program, see
Sections 9 and 10 of \cite{KW}. An $A$-brane $\A$ that is an
eigenbrane for the 't Hooft operators, in the sense that, for every
't Hooft operator $T_p$,
\begin{equation}T_p\cdot \A=\A \otimes V_p\end{equation}
for some vector space $V_p$, is known as a magnetic eigenbrane.
Wilson operators of $^L\neg G$ gauge theory are classified by a
choice of representation of $^L\neg G$, and   't Hooft operators of
$G$ gauge theory  are likewise classified by representations of
$^L\neg G$. Electric-magnetic duality is expected to map Wilson
operators to 't Hooft operators and electric eigenbranes to magnetic
eigenbranes.

Let us review the action of an 't Hooft operator $T_p$ on a Higgs
bundle $(E,\varphi)$. In case $\varphi=0$, the possible Hecke
modifications are the usual ones considered in the geometric
Langlands program; they are parametrized by a subvariety of the
affine Grassmannian known as a Schubert variety $\eusm V$, which
depends on a choice of representation $^L\neg R$ of the dual group
$^L\neg G$. For instance, continuing with our example, if $G=SL_2$
and $^L\neg R$ is the three-dimensional representation of $^L\neg
G=SO_3$, then a generic point in $\eusm V$ corresponds to a Hecke
modification of an $SL_2$ bundle $E$ of the following sort: for some
local decomposition of $E$ as a sum of line bundles $\cal N_1\oplus
\N_2$, $E$ is mapped to $\N_1(p)\oplus \N_2(-p)$. Letting $\N_1$ and
$\N_2$ vary, this gives a two-parameter family of Hecke
modifications of $E$. A  family of modifications of $E$ of this type
can degenerate to a trivial modification, and $\eusm V$ contains a
point corresponding to the trivial Hecke transformation.

What we have just described, for this example, is the possible
action of the 't Hooft operator in the most degenerate case that
$\varphi=0$.  If instead $\varphi\not=0$, one must restrict to
Hecke modifications that are in a certain sense
$\varphi$-invariant.  For $G=SL_2$, and assuming $\varphi$ to be
regular semi-simple at the point $p$, $\varphi$-invariance means
that the decomposition $E=\N_1\oplus \N_2$ must be compatible with
the action of $\varphi$, in the sense that $\varphi:E\to E\otimes
K$ maps $\N_1$ to $\N_1\otimes K$ and $\N_2$ to $\N_2\otimes K$.
There are precisely two possible choices of $\N_1$ and $\N_2$:
locally, as $\varphi(p)$ is regular semi-simple, we can
diagonalize $\varphi$
\begin{equation}\varphi=\begin{pmatrix}a & 0 \\ 0 & -a
\end{pmatrix}, \end{equation}
and $\N_1$ and $\N_2$ must equal, up to permutation, the two
``eigenspaces.''

In addition to these two non-trivial $\varphi$-invariant Hecke
modifications, we must remember to include the trivial Hecke
modification (since it corresponds to a point in $\eusm V$), which
is also $\varphi$-invariant. Altogether then there are in this
example three $\varphi$-invariant Hecke modifications, namely a
trivial one and two non-trivial ones, a statement that, as we will
see, is dual to the fact that the representation of $^L\neg G=SO_3$
that we started with is three-dimensional.

Now let us see what the $\varphi$-invariant Hecke modifications
look like from the point of view of the spectral curve $\pi:D\to
C$.  We consider first the case of a generic spectral curve, given
by an equation $\det(z-\varphi)=0$. A $\varphi$-invariant Hecke
modification leaves fixed the characteristic polynomial of
$\varphi$ and hence maps each fiber $\FF$ of the Hitchin fibration
to itself.  How does it act on $\FF$?

A point $p\in C$ at which $\varphi$ is regular semi-simple lies
under two distinct points $p',p''\in D$.  The bundle $E$ is
$\pi_*(\cal L)$ for some line bundle $\cal L\to D$, and
$\varphi=\pi_*(z)$. The latter condition means that the
eigenspaces of $\varphi(p)$ correspond to the two distinct values
of $z$ lying above $p$, or in other words to the two points $p'$
and $p''$.  This being so, a non-trivial $\varphi$-invariant Hecke
modification of $(E,\varphi)$ at the point $p$ comes from a
transformation of $\cal L$ of the specific form
\begin{equation}\label{yelgo}\L\to \L\otimes \O(p'-p'')\end{equation}
for one or another of the two possible labellings of the two points
$p',p''$ lying above $p$. (This notion of a Hecke modification of a
Higgs bundle $(E,\varphi)$ is mathematically natural and was taken
as the starting point in \cite{DP}.)  When this is pushed down to
$C$, it modifies $E$ in the desired fashion.

Now we can see why an $A$-brane $\A_{\FF}$ supported on a fiber $\FF$ of
the Hitchin fibration and endowed with a flat line bundle $\cal R$ is
a magnetic eigenbrane, that is an eigenbrane for the 't Hooft operator
$T_p$.  First of all, $T_p$ maps $\FF$ to itself, since it preserves the
characteristic polynomial of $\varphi$.  Since $T_p$ preserves the
support of $\A_{\FF}$, it is conceivable for $\A_{\FF}$ to be an
eigenbrane for $T_p$.

Now, assuming that we choose $p$ so that $\varphi(p)$ is regular
semi-simple (and we will only treat this case), the evaluation of
$T_p\cdot \A_{\FF}$ comes from a sum of contributions from the three
$\varphi$-invariant Hecke modifications that were just described.  One
of them is the trivial Hecke modification, and this leaves $\A_{\FF}$
invariant. The other two come from transformations $\L\to \L\otimes
\O(p'-p'')$ (for some labeling of the two points). Such a
transformation can be interpreted as an isomorphism $\Phi:\FF\to \FF$
of the Hitchin fiber. If the labeling of the two points $p'$ and $p''$
is reversed, then $\Phi$ is replaced by $\Phi^{-1}$.

$\FF$ is a complex torus, and $\Phi$ is a ``translation'' of $\FF$ by
a constant vector.  In general, if $\cal R\to \FF$ is a flat line
bundle over a complex torus and $\Phi:\FF\to \FF$ is a translation,
then $\Phi^*(\cal R)=\cal R\otimes \cal V$ for some
one-dimensional vector space $\cal V$.  From this it follows that
$\A$ is an eigenbrane for $T_p$.  In fact, we have
\begin{equation}\label{hopscotch} T_p\cdot \A_{\FF}=\A_{\FF}\otimes
(\C\oplus \cal V\oplus \cal V^{-1}),\end{equation} where the three
contributions on the right come respectively from the trivial Hecke
modification and the non-trivial modifications that involve $\Phi$ and
$\Phi^{-1}$.

Let us compare this to what we had on the electric side.  There,
we considered a brane $\cal B$ whose support in $\MH(SO_3)$
corresponds to a Higgs bundle $(E,\varphi)$.   It obeyed $W_p\cdot
\cal B =\cal B\otimes E_p,$ where $E_p$ is the fiber at $p$ of
$E$. Since we assume that $\varphi(p)$ is regular semi-simple (a
property that is invariant under duality), we can decompose $E_p$
according to the eigenspaces of $\varphi(p)$.  $\varphi(p)$ has
two eigenspaces  with non-zero eigenvalues; they are dual to each
other because of the quadratic form on $E_p$, so we call  them
${\cal X}$ and ${\cal X}^{-1}$.  They have a natural isomorphism
to $\cal V$ and $\cal V^{-1}$; this can be established by
arguments similar to those used in demonstrating the geometric
Langlands duality for $GL_1$. See \cite{DP} for a version of this
calculation.  The kernel of $\varphi(p)$ corresponds to the
summand $\C$.

\subsubsection{Interpretation}\label{interpretation}
What perhaps most needs clarification is the interpretation of the
result just described.

That computation was made using the Hitchin fibration and other
tools appropriate for the $B$-model in complex structure $I$.
However, for the geometric Langlands program, we are really
interested in $A$-branes for the $A$-model in symplectic structure
$\omega_K$.

We observed earlier that the electric eigenbrane supported at a
point is a brane of type $(B,B,B)$ -- that is, it is a $B$-brane in
each of complex structures $I,J$, and $K$.  The dual statement is
that the dual magnetic eigenbrane is a brane of type $(B,A,A)$.
 A brane of type $(B,A,A)$ is a brane that
is simultaneously a $B$-brane in complex structure $I$, and an
$A$-brane for the $A$-models of symplectic structure $\omega_J$ and
$\omega_K$.

To explicitly see that a brane $\A_{\FF}$ supported on a fiber $\FF$  of
the Hitchin fibration and endowed with a unitary flat line bundle
$\L$ is of type $(B,A,A)$, we reason as follows. $\FF$ is a complex
Lagrangian submanifold in complex structure $I$. $\A_{\FF}$ is a
$B$-brane in complex structure $I$ because $\FF$ is holomorphic in
that complex structure, and a flat line bundle such as $\cal L$ is
also holomorphic. $\FF$ is Lagrangian for
$\Omega_I=\omega_J+i\omega_K$, and hence is Lagrangian for
$\omega_J$ and $\omega_K$.  The most standard type of $A$-brane is
a Lagrangian submanifold endowed with a unitary flat line bundle;
$\A_{\FF}$ qualifies whether we take the symplectic structure to be
$\omega_J$ or $\omega_K$.

The dual of the fact that the Wilson operator $W_p$ preserves
supersymmetry of type $(B,B,B)$ is that the 't Hooft operator
$T_p$ preserves supersymmetry of type $(B,A,A)$.  This statement
means that if $\tilde \A$ is a brane of type $(B,A,A)$ then so is
$T_p\cdot \tilde \A$. (More generally, if $\tilde \A$ preserves
any linear combination of these supersymmetries, then so does
$T_p\cdot \tilde \A$.)  To actually identify which brane of type
$(B,A,A)$ is $T_p\cdot \tilde \A$, we can compute using whatever
one of the supersymmetries is most convenient. Algebraic geometry
is powerful, so it is likely to be convenient to compute $T_p\cdot
\tilde\A$ viewed as a $B$-brane in complex structure $I$.  This is
enough to determine $T_p\cdot \tilde\A$ as a brane of type
$(B,A,A)$, and in particular, as an $A$-brane of type $K$, which
is what one wants for geometric Langlands.

The reason for the last statement is that given a $B$-brane of
type $I$, there is at most one way to endow it with the structure
of a brane of type $(B,A,A)$.  Let us spell this out concretely in
the present context. In \secref{modif}, we used  spectral curves
and algebraic geometry to construct the brane $T_p\cdot \A_{\FF}$ in
terms of a holomorphic subvariety $\FF\subset \MH(SL_2)$ and a
holomorphic vector bundle ${\cal K}\to \FF$.  This data determines a
$B$-brane in complex structure $I$.  Endowing this brane with a
structure of type $(B,A,A)$ means endowing ${\cal K}$ with a
hermitian metric such that the induced unitary connection is flat;
$\FF$ with such a flat bundle is a brane of type $(B,A,A)$.  The
flat unitary connection with which a holomorphic bundle ${\cal K}$
can be so endowed is unique up to isomorphism if it exists.  In
the present context, because $T_p$ preserves supersymmetry of type
$(B,A,A)$, we expect such a flat connection to exist, and the
computation in \secref{modif} shows that it does.

A final comment concerns the ``multiplier'' $V_p$ in the formula
expressing the fact that $\A_{\FF}$ is a magnetic eigenbrane:
\begin{equation}T_p\cdot \A_{\FF}=\A_{\FF}\otimes V_p.\end{equation}
We want to discuss what happens when $p$ varies. In geometric
Langlands, we expect that $V_p$ will be the fiber of a local
system, that is a flat vector bundle.  That is not what we get in
the most obvious way from the description based on the spectral
curve and complex geometry in complex structure $I$.  From that
point of view, we obtain $V_p$ as the fiber of a rank three
holomorphic vector bundle $V\to C$ (defined to begin with where
$\varphi$ is regular semi-simple) that does not have any obvious
flat structure. However, multiplication by the coordinate $z$ of
the spectral curve gives a $K$-valued endomorphism that we will
call $\theta$, so the multiplier is actually a Higgs bundle
$(V,\theta)$. By solving Hitchin's equations, we associate with
this Higgs bundle a rank three local system, and this we expect
will be isomorphic to the $SO_3$ local system with which we began.

\subsubsection{Reducible Fibers}\label{reducible}

Now we are ready to consider the situation related to endoscopy. We
consider a special fiber $\FF$ of the Hitchin fibration that is a
union of two irreducible components $\FF_1$ and $\FF_2$ that
intersect each other on a divisor.   This being so, we can construct
rank 1 $A$-branes $\A_{1}$ and $\A_{2}$ supported on $\FF_1$ or
$\FF_2$.  These branes are unique if $\FF_1$ and $\FF_2$ are
simply-connected, as in the case of a curve of genus 1 with 1 point
of ramification; otherwise, they depend on parameters that we are
not indicating explicitly.

In the derivation of eqn. (\ref{hopscotch}) describing the action of
$T_p$, a key ingredient was the map $\Phi:\FF\to \FF$ by $\L\to
\L\otimes \O(p'-p'')$.  The essential new fact in the case that
$\FF$ is reducible is simply that $\Phi$ exchanges the two component
of $\FF$.  This was how we characterized the two components in
\secref{curves}.  Likewise $\Phi^{-1}$ exchanges the two components.
Hence $\Phi$ or $\Phi^{-1}$ exchange $\A_{1}$ and $\A_{2}$. Since
$T_p$ acts by $1+\Phi+\Phi^{-1}$, it follows that we have up to
isomorphism
\begin{align}\label{helpmet} \notag T_p\cdot \A_{1}& =
\A_{1}+2\A_{2}\\ T_p\cdot \A_{2}& = \A_{2}+2\A_{1}.\end{align} This
is in perfect parallel with the formula (\ref{helko}) for the
electric case.

If $\A_{1}$ and $\A_{2}$ have moduli, this should be described a
little more precisely.  $\A_{1}$ depends on the choice of a suitable
line bundle $\L\to \FF_1$, and we should take $\A_{2}$ to be the
brane associated with the line bundle $\Phi^*(\L)\to \FF_2$. Note
that $\Phi^*(\L)$ and $(\Phi^{-1})^*(\L)$ are isomorphic, though not
canonically so.

One expects to get the more precise result analogous to (\ref{elko})
via the procedure of \secref{interpretation}.  One uses standard
methods of algebraic geometry to construct $T_p\cdot \A_{1}$ and
$T_p\cdot \A_{2}$ as $B$-branes in complex structure $I$.  This will
give a result more precise than (\ref{helpmet}):
\begin{align}\label{helpmeto} \notag T_p\cdot \A_{1}& =
\left( \A_{1}\otimes {\cal J}_1 \right) \oplus \left( \A_{2}\otimes
{\cal J}_2 \right) \\ T_p\cdot \A_{2}& = \left( \A_{2}\otimes {\cal
J}'_1 \right) \oplus \left( \A_{1}\otimes {\cal J}'_2
\right),\end{align} with vector spaces $\cal J_1$, $\cal J_2$, etc.,
of dimensions indicated by the subscripts.  All these admit natural
$K$-valued endomorphisms $\theta_1$, $\theta_2$, etc., coming from the
Higgs field (that is, multiplication by the coordinate $z$ of the
spectral curve), and $(\cal J_1,\theta_1)$, etc., are Higgs bundles
over $C$. Relating these Higgs bundles to local systems via Hitchin's
equations, one expects to arrive at the analog of (\ref{elko}),
\begin{align}\label{elko1}\notag T_p\cdot \cal A_1 & = \left( \cal A_2\otimes
  U|_p \right) \oplus \left( \cal A_1\otimes \det U|_p \right) \\
T_p\cdot \cal A_2 & = \left( \cal A_1\otimes U|_p \right) \oplus
\left( \cal A_2\otimes \det U|_p \right).\end{align}

\subsubsection{The Reciprocal Case}\label{imthooft}

We can apply similar techniques to the reciprocal case $^L\neg
G=SL_2$, $G=SO_3$.

For gauge group $SL_2$, the basic Wilson operator to consider is the
operator $\tilde W_p$  associated with the two-dimensional
representation. Roughly speaking, it acts by the obvious analog of
eqn. (\ref{donkey}). Letting $(\eurm E,\hat\varphi)$ denote the
universal Higgs bundle over $\MH(SL_2)\times C$, $\tilde W_p$ acts
on the sheaf $\cal K$ defining a $B$-brane $\cal B$ by
\begin{equation}\label{roughly} \cal K\to \cal K\otimes \eurm E|_p
\end{equation}
where $\eurm E|_p$ is the restriction to $\MH\times p$ of the
universal rank two bundle $\eurm E\to \MH\times C$. $\tilde W_p$
obeys
\begin{equation}\label{hoppy}\tilde W_p^2=1+W_p,\end{equation}
expressing the fact that the tensor product of the two-dimensional
representation with itself is a direct sum of the trivial
representation and the three-dimensional representation; they
correspond to the terms 1 and $W_p$ on the right hand side of eqn.
(\ref{hoppy}).

An important subtlety reflects the fact that the center of $SL_2$ acts
non-trivially in the two-dimensional representation. The universal
bundle $\eurm E$ does not exist as a vector bundle in the usual
sense. Rather, it must be understood as a twisted vector bundle,
twisted by a certain $\C^\times$ gerbe over $\MH$.  The gerbe in
question is induced from a $\Z_2$ gerbe ($\Z_2$ being here the center
of $SL_2$) by the inclusion $\Z_2\subset \C^\times$.  As a result of
the fact that $\eurm E|_p$ is a twisted vector bundle, the tensor
product with it maps ordinary sheaves over $\MH$ to twisted ones, and
vice-versa. This means, in the language of \cite{KW}, that the action
of $\tilde W_p$ on a brane shifts the discrete electric flux
$\mbox{\bf \em e}_0$, which is a character of the center of $SL_2$.
The dual statement is that the dual 't Hooft operator $\tilde T_p$
shifts the discrete magnetic flux $\mbox{\bf \em m}_0$, which is the
second Stieffel-Whitney class $w_2$ of an $SO_3$ bundle.

Roughly speaking, a skyscraper sheaf supported at a smooth point
$r\in \MH(SL_2)$ gives an electric eigenbrane $\B$, just as in our
earlier discussion for $^L\neg G=SO_3$.  But this is slightly
oversimplified. The skyscraper sheaf supported at the point $r$
makes sense as either an ordinary sheaf or a twisted one, since
the twisting involves a gerbe that is trivial when restricted to a
smooth point.  The tensor product with $\eurm E$ exchanges the two
cases. So if we write $\B$ and $\B'$ for the ordinary and twisted
versions of the brane related to the skyscraper sheaf, then the
action of the Wilson operator is
\begin{align}\notag\label{hopo}
\tilde W_p\cdot \B &= \B' \otimes {\eurm E}|_{r\times p}\\
               \tilde W_p\cdot \B' &= \B \otimes {\eurm E}|_{r\times
               p}.\end{align}
The sum $\hat\B=\B\oplus \B'$ is therefore an electric eigenbrane
in the usual sense:
\begin{equation}\label{opo}
\tilde W_p\cdot \hat\B=\hat B\otimes {\eurm E}|_{r\times p}.\end{equation}

The action of the dual 't Hooft operator $\tilde T_p$ on an $SO_3$
bundle $W$ or $SO_3$ Higgs bundle $(W,\varphi)$ is quite similar to
what has been described in \secref{modif}.  It is convenient to
describe the action in terms of an $SL_2$ Higgs bundle $(E,\varphi)$
such that $W={\rm ad}(E)$, where the degree of $\det \,E$ is
congruent mod 2 to $w_2(W)$.  $E$ is not quite uniquely determined,
but the following statements, when expressed in terms of $W={\rm
ad}(E)$, do not depend on the choice of $E$. Relative to some local
decomposition of $E$ as a sum of line bundles $E=\N_1\oplus \N_2$,
$\tilde T_p$ acts by $E\to \N_1(p)\oplus \N_2$. This operation
reverses the reduction mod 2 of the degree of $\det \, E$, so it
reverses $\mbox{\bf \em m}_0$, as expected.

If $\varphi=0$, the freedom to choose $\N_1$ leads to a family of
possible Hecke modifications parametrized by $\CP^1$.  This
parameter space is compact (reflecting the fact that the
two-dimensional representation of $SL_2$ is minuscule) so we do not
need to add anything to compactify it.  If $\varphi\not=0$, $\tilde
T_p$ acts by a $\varphi$-invariant Hecke modification.  If
$\varphi(p)$ is regular semi-simple, this means that, in the last
paragraph, we must take $\N_1$ to be one of the two eigenspaces of
$\varphi$.  The existence of two choices is dual to the fact that
the representation of $SL_2$ that we started with is
two-dimensional.

In terms of spectral curves, the action of a $\varphi$-invariant
Hecke modification can be described very similarly to eqn.
(\ref{yelgo}). The line bundle $\L\to D$ such that $E=\pi_*(\cal L)$
transforms by \begin{equation}\label{uret}\L\to \L(p^*)
\end{equation}
where $p^*$ may be either of the two points $p'$ and $p''$ lying
above $p$.

Now $\MH(SO_3)$ has two components, classified by $w_2(W)$.  We
write $\FF$ and $\FF_\theta$ for the fibers of the Hitchin fibration
for these two components. For $p^*$ equal to $p'$ or $p''$,
(\ref{uret}) corresponds to two maps $\Phi':\FF\leftrightarrow
\FF_\theta$ and $\Phi'':\FF\leftrightarrow \FF_\theta$, each of
which exchanges these two components.   Acting on a brane supported
on $\FF$ with a flat bundle $\cal L$, $\tilde T_p$ therefore gives a
brane supported on $\FF_\theta$ with flat bundle
$(\Phi')^*(\L)\oplus (\Phi'')^*(\L)$, and similarly with $\FF$ and
$\FF_\theta$ exchanged. These are all branes of type $(B,A,A)$ and
can be analyzed by arguments similar to those that we have already
described.  This leads to formulas dual to (\ref{hopo}) if we
consider branes supported on only one component, or to (\ref{opo}),
if we form ordinary magnetic eigenbranes with support on both
components.

\subsubsection{The Improper 't Hooft Operator}\label{impropop}
 None of this gives a good example of geometric endoscopy, because an
$SL_2$ local system cannot have a finite automorphism group that does
not merely reduce to the center of $SL_2$. However, we can see
endoscopy at work if we take $\tilde T_p$ to act on the $A$-model not
of $\MH(SO_3)$ but of its cover. We recall that $\MH(SO_3,0)$ and
$\MH(SO_3,\theta)$ have covers (with covering group the finite abelian
group $Q=H^1(C,\Z_2)$) that are the proper and improper components of
the $SL_2$ moduli space, $\MH(SL_2)$ and $\MH(SL_2^*;p)$. The improper
component depends on the choice of a basepoint $p\in C$, as explained
in \secref{improperly}; it is convenient to take this to be the point
$p$ at which we will apply the 't Hooft operator.

After lifting to the these covering spaces, the action of the 't
Hooft operator $\tilde T_p$ is much as we have already described.
Mapping from $\MH(SL_2)$ to  $\MH(SL_2^*;p)$, $\tilde T_p$ tensors
the line bundle on the spectral curve by $\CO(p')\oplus \CO(p'')$;
mapping from $\MH(SL_2^*;p)$ to $\MH(SL_2)$, it tensors that line
bundle by $\pi^*(\CO(-p))\otimes (\CO(p')\oplus
\CO(p''))=\CO(-p')\oplus \CO(-p'')$.  These formulas are compatible
with the fact that the two components parametrize, respectively,
Higgs bundles $(E,\varphi)$ with $\det\,E=\O$ and $\det\,E=\O(p)$.
 Ensuring this has necessitated the prefactor $\pi^*(\CO(-p))$ in one
of the formulas.

As long as the spectral curve is smooth, the action of $\tilde
T_p$ on $A$-branes of $SL_2$, exchanging the two components, is
similar to what we described earlier for $SO_3$.  Now, however, we
can consider the case that the Hitchin fiber has two irreducible
components.  We write $\FF_1$ and $\FF_2$ for the two components
of the special Hitchin fiber of $\MH(SL_2)$, and $\FF^*_1$ and
$\FF^*_2$ for the two components of the special Hitchin fiber of
$\MH(SL_2^*;p)$.  We likewise write $\A_1,\,\A_2$ and
$\A^*_1,\,\A^*_2$ for $A$-branes of the usual type supported on
these fibers.

In the action of the 't Hooft operator, the two operations of
tensoring the line bundle on the spectral cover by $\CO(p')$ and
by $\CO(p'')$ differ by the tensor product with $\CO(p'-p'')$.
This is the basic operation that exchanges the two components. We
can label the points to that $\CO(p')$ maps $\FF_1$ to $\FF_1^*$
and $\FF_2$ to $\FF_2^*$, while $\CO(p'')$ maps $\FF_1$ to
$\FF_2^*$ and $\FF_2$ to $\FF_1^*$.  (The tensor product with
$\CO(p')^{-1}$ or $\CO(p'')^{-1}$ is of course the inverse
operation.)  There is no natural way to say which is which, since
$p'$ and $p''$ are exchanged by monodromy in $p$, and this
monodromy similarly exchanges $\FF^*_1$ and $\FF^*_2$, as we noted
at the end of \secref{improperly}.

The action of the 't Hooft operator $\tilde T_p$ on branes
$\A_{{1,2}}$ and $\A^*_{{1,2}}$ is schematically
\begin{align}\notag\label{schema} \tilde T_p \cdot
  \A_{1}&=\A^*_{{1}}+\A^*_{2}\\
                \tilde T_p \cdot
                \A_{2}&=\A^*_{{1}}+\A^*_{2},\end{align}
and similar formulas with $\A_i$ and $\A_i^*$ exchanged.  These
formulas and the analogous ones for the action of the 't Hooft
operator $T_p$ dual to the three-dimensional representation are
compatible with the relation
\begin{equation}\label{silgo} \tilde T_p^2=1+T_p.\end{equation}
This relation is dual to eqn. (\ref{hoppy}).

Finally let us discuss how natural is the operator $\tilde T_p$.
As an operator acting on branes on $\MH(SO_3)$, it is completely
natural, being dual to the two-dimensional representation of the
dual group $^L\neg G=SL_2$.  However, as an operation acting on
branes on $\MH(SL_2)$, we cannot expect $\tilde T_p$ to be
entirely natural, since it is supposed to be dual to the
two-dimensional representation, which does not exist as a
representation of $^L\neg G=SO_3$.  The unnaturalness shows up in
the fact that if we want $\tilde T_p$ to act on branes on
$\MH(SL_2)$, it maps them to branes on $\MH(SL_2^*;p)$, a space
whose definition depends on $p$, albeit relatively weakly. By
contrast, for any reductive group $G$, 't Hooft operators
associated with representations of $^L\neg G$ always have a
completely natural action on branes on $\MH(G)$.

As an operator on branes on the $SL_2$ moduli spaces, we call
$\tilde T_p$ the improper 't Hooft operator.

\section{Categories Of Eigensheaves}    \label{categories}

In the previous sections we have constructed $A$-branes on the moduli
space of Higgs bundles which satisfy a property very similar to, but
not quite the same, as the usual Hecke property. As explained in
\secref{genus one A to D}, to each of these $A$-branes we should be able
to attach a ${\mc D}$-module on $\Bun_G$. These ${\mc D}$-modules
should then satisfy the same Hecke property. In this section we
explain the general framework in which we can interpret this property
as a natural generalization of the standard notion of Hecke
eigensheaf.

\subsection{Generalities On Categories}    \label{gen cat}

Let $H$ be a reductive algebraic group over $\C$ and $\Rep(H)$ the
tensor category of finite-dimensional representations of $H$. We
consider below two types of abelian categories associated to
$H$:


\smallskip

\noindent (1) Categories equipped with an action of $H$. This means
that any $h \in H$ gives rise to a functor $F_h$ on the category ${\mc
C}$ sending each object $M$ of ${\mc C}$ to an object $F_h(M)$, and
each morphism $f: M \to M'$ to a morphism $F_h(f): F_h(M) \to
F_h(M')$.

These functors have to satisfy natural compatibilities; in particular,
for any $h, h' \in H$ there is an isomorphism $i_{h,h'}: F_{hh'}
\simeq F_h \circ F_{h'}$, and we have $i_{hh',h''} i_{h,h'} =
i_{h,h'h''} i_{h',h''}$.

\smallskip

Example: the category $\on{Coh}(X)$ of coherent sheaves (or
$B$-branes) on an algebraic variety $X$ equip\-ped with an action of
$H$.

\smallskip

\noindent (2) Categories with a monoidal action of the tensor category
$\Rep(H)$. This means that any object $V \in \Rep(H)$ defines a
functor on ${\mc C}$, $M \mapsto V \star M$, and these functors are
compatible with the tensor structure on $\Rep(H)$.

\smallskip

Example: the category $\on{Coh}(X/H)$ of coherent sheaves on the
quotient $X/H$, where $X$ is as in point (1) above.\footnote{Here we
may suppose for simplicity that this action is free, but this is not
necessary if we are willing to consider $X/H$ as an algebraic stack --
or an orbifold, in the case of a finite group $H$.} Then for each $V
\in \Rep(H)$ we have a vector bundle $\mc V$ on $X/H$ associated with
the principal $H$-bundle $X \to X/H$,
$$
{\mc V} = X \underset{H}\times V.
$$
Its sheaf of sections (also denoted by ${\mc V}$) is an object of
$\on{Coh}(X/H)$. For any other object $M$ of $\on{Coh}(X/H)$ we set
$$
V \star M := {\mc V} \underset{\OO_{X/H}}\otimes M.
$$


One can pass between the categories of these two types:

\smallskip

The passage (1) $\longrightarrow$ (2) is called {\em
equivariantization};

\smallskip

The passage (2) $\longrightarrow$ (1) is called {\em
de-equivariantization}.

\smallskip

These two procedures are inverse to each other. Here we only give a
brief sketch as this material is fairly well-known (see, e.g.,
\cite{AG}).

\smallskip

Equivariantization is done as follows: given a category ${\mc C}$ of
type (1), we construct a new category ${\mc C}'$ of type (2). Its
objects are the data $(M,(\mu_h)_{h \in H})$, where $M \in {\mc C}$
and
$$
\mu_h: F_h(M) \overset{\sim}\longrightarrow M, \qquad h \in H,
$$
is a collection of isomorphisms such that $\mu_{hh'} = \mu_h \circ
\mu_{h'}$. Given $V \in \Rep(H)$, we define a new object of ${\mc
  C}'$,
$$
V \star (M,(\mu_h)_{h \in H}) = (\underline{V} \otimes_{\C} M,(\rho(h)
\otimes \mu_h)_{h \in H}),
$$
where $\underline{V}$ is the vector space underlying $V$, and $\rho: H
\to \on{End}(\underline{V})$ is the representation of $H$ on
$V$. Thus, we see that $\Rep(H)$ acts on ${\mc C}'$.

\smallskip

Example: if ${\mc C} = \on{Coh}(X)$, then ${\mc C}'$ is the category
of $H$-equivariant coherent sheaves on $X$ which is the same as the
category $\on{Coh}(X/H)$ of coherent sheaves on $X/H$.

\medskip

De-equivariantization is similar, and this is where the notion of
``Hecke eigenobject'' naturally appears.

Given a category ${\mc C}'$ of type (2), we define a new category ${\mc
C}$ of type (1). Its objects are the data $(M,(\al_V)_{V \in
\on{Rep}(H)})$, where $M \in {\mc C}'$ and
\begin{equation}    \label{al general}
\al_V: V \star M \overset{\sim}\longrightarrow \underline{V}
\otimes_{\C} M, \qquad V \in \Rep(H),
\end{equation}
is a collection of isomorphisms compatible with the tensor product
structure on $\Rep(H)$. We will call $(M,(\al_V)_{V
\in \on{Rep}(H)})$ a {\em Hecke eigenobject} of the category ${\mc
  C}'$.

The group $H$ naturally acts on ${\mc C}$: for $h \in H$ we define
\begin{equation}    \label{action of H}
F_h((M,(\al_V)_{V \in \on{Rep}(H)})) = (M,((\rho(h) \otimes
\on{id}) \circ \al_V)_{V \in \on{Rep}(H)}).
\end{equation}
In other words, $M$ stays the same, but we twist the isomorphism
$\al_V$ by $h$. Thus, ${\mc C}$ is indeed a category of type (1).

\subsection{Examples}    \label{exa}

The simplest example of a category of type (2) is the category
$\Rep(H)$ equipped with the natural monoidal action on
itself. Consider the corresponding de-equivariantized category ${\mc
C}'$. Let ${\mc O}_H$ be the algebra of functions on $H$, that is, the
regular representation of $H$. We have an isomorphism
$$
{\mc O}_H = \bigoplus_{V \in \Irrep(H)} V \otimes V^*
$$
respecting both left and right actions of $H$ on the two sides (here
$\Irrep(H)$ is the set of equivalence classes of irreducible
representations of $H$, and for each representation $V$ we denote by
$V^*$ the dual representation).

The data of the isomorphisms $\al_V$ in eqn. \eqref{al general} and
the compatibilities between them may be neatly summarized as the
structure of an ${\mc O}_H$-module on $M$, compatible with the left
action of $H$. Thus, ${\mc C}'$ is the category of $H$-equivariant
coherent sheaves on $H$. This category is equivalent to the category
$\on{Vect}$ of vector spaces. Indeed, we have a functor $G: \on{Vect}
\to {\mc C}'$ sending $U \in \on{Vect}$ to $U \otimes {\mc O}_H$. Its
quasi-inverse functor ${\mc C}' \to \on{Vect}$ is defined by sending
${\mc M} \in {\mc C}'$ to its fiber at $1 \in H$. The natural action
of $H$ on ${\mc C}'$ described above becomes the trivial action on
$\on{Vect}$. The corresponding $H$-equivariant category of $\on{Vect}$
has as its objects vector spaces equipped with an action of $H$. Thus,
by equivariantizing ${\mc C}'$, we recover the category $\Rep(H)$ we
started with, as promised.

\smallskip

It is useful to consider the case when $H = \Z_2 = \{ 1,-1 \}$ in more
concrete terms. In this case $\Rep(H)$ has two irreducible
one-dimensional representations: the trivial one, $I$, and the sign
representation, $S$.

Let ${\mc C}$ be the following category of type (2): it has two
irreducible objects $A_+$ and $A_-$, and the $\Rep(\Z_2)$ acts as
follows:
\begin{equation}    \label{action of Z2}
I \star A_\pm = A_\pm, \qquad S \star A_\pm = A_\mp.
\end{equation}
All other objects of ${\mc C}$ are direct sums of copies of $A_+$ and
$A_-$.

Let us assume first that there is a preferred object among $A_+$ and
$A_-$ corresponding to the trivial representation, say, $A_+$.

What does the corresponding category ${\mc C}'$ of type (1) look like?
By definition, its objects have the form $(M,(\al_I,\al_S))$, where $M
\in {\mc C}$ and
$$
\al_I: I \star M \simeq M, \qquad \al_S: S \star M \simeq M
$$
are isomorphisms satisfying the following conditions. First of all,
$I \star M$ is $M$ and $\al_I$ is the identity map $M \to M$. Second,
$$
(\al_S)^2: M = I \star M = (S \otimes S) \star M = S \star (S \star M)
\to S \star M \to M
$$
is the identity. Thus, our data may be recorded as pairs $(M,\al)$,
where $M \in {\mc C}$ and $\al=\al_S$ is an isomorphism $S \star M \to
M$ such that $\al^2 = \on{id}$.

It is easy to see that if $M$ is an object of ${\mc C}$ such that
there exists $\al: S \star M \simeq M$ satisfying the above property,
then $M$ is a direct sum of copies of $$A := A_+ \oplus A_-.$$ Let us
look at objects of ${\mc C}'$ of the form $(A,\al)$. We have $S \star
A = S(A_+) \oplus S(A_-) = A$. Therefore $\al: A \to A$ is determined
by two non-zero scalars $\ep_\pm$ corresponding to the action of $\al$
on $A_\pm$ (we have assumed that both $A_+$ and $A_-$ are irreducible
and non-isomorphic). The condition that $\al^2 = \on{id}$ implies that
$\ep_+ \ep_- = 1$. Given two such objects, $(A,(\ep_+,\ep_-))$ and
$(A,(\ep'_+,\ep'_-))$, an isomorphism between them is a pair of
non-zero scalars $(\la_+,\la_-)$ acting on $A_+$ and $A_-$ such that
$\ep_+/\ep'_+ = \la_+/\la_-$. Thus, $\la_+$ is determined by $\la_-$
and vice versa.

Thus, the category ${\mc C}'$ is very simple: up to an isomorphism
there is a unique irreducible object, $A$, and all other objects are
isomorphic to a direct sum of copies of this object. In fact, ${\mc
C}'$ is equivalent to the category $\on{Vect}$ of vector spaces. The
functor ${\mc C}' \to \on{Vect}$ sends $M \in {\mc C}'$ to
$\on{Hom}(A_+,M)$.

In the study of geometric endoscopy for $SL_2$ we encounter a category
of $A$-branes similar to the category ${\mc C}$. It also has two
irreducible objects, ${\mc A}_+$ and ${\mc A}_-$, and $\Rep(\Z_2)$
acts on it as in formula \eqref{action of Z2}. However, there is no
preferred object among $\A_+$ and $\A_-$; in other words, there is no
canonical equivalence between ${\mc C}$ and $\Rep(\Z_2)$. It is
natural to ask how this ambiguity is reflected in the category ${\mc
C}'$.

The answer is clear: the category ${\mc C}'$ is still equivalent to
$\on{Vect}$, but there are two such equivalences and there is no way
to choose one of them over the other. The corresponding functors ${\mc
C}' \to \on{Vect}$ send $M \in {\mc C}'$ to $\on{Hom}({\mc
A}_\pm,M)$. Under both of these equivalences all objects
$(A,(\ep_+,\ep_-))$, with $\ep_+ \ep_- = 1$, of ${\mc C}'$ go to the
one-dimensional vector space $\C$ (viewed as an object of
$\on{Vect}$), but the isomorphism
$$
(A,(\ep_+,\ep_-)) \to (A,(\ep'_+,\ep'_-))
$$
given by $(\la_+,\la_-)$, as above, goes either to the isomorphism $\C
\to \C$ given multiplication by $\la_+$ or by $\la_-$.

Thus, to each of the two objects, $\A_+$ and $\A_-$, corresponds a
particular equivalence ${\mc C}' \simeq \on{Vect}$, but inasmuch as we
cannot choose between $\A_+$ and $\A_-$, we cannot choose one of these
equivalences over the other (it is easy to see that these
equivalences are not isomorphic to each other as functors ${\mc C}'
\to \on{Vect}$). This is what replaces the ambiguity between $\A_+$
and $\A_-$ in the category ${\mc C}$ of type (2) when we pass to the
category ${\mc C}'$ of type (1).

The existence of two different equivalences of categories ${\mc C}'
\simeq \on{Vect}$ certainly looks like a more subtle and esoteric
notion than the more concrete notion that the category ${\mc C}$ has
two indistinguishable objects. This is a good illustration of why,
from the practical point of view, it is often better to work with a
category of type (2) than with a category of type (1).

The point is however that the two descriptions are equivalent to each
other. To convince ourselves of that, it is instructive to see how we
can recover ${\mc C}$ from ${\mc C}'$. Let us apply the
equivariantization procedure to ${\mc C}'$. Hence we define a new
category ${\mc C}''$, whose objects are pairs $(M,\mu)$, where $M$ is
an object of ${\mc C}'$ and $\mu = \mu_{-1}$ is an isomorphism between
$M$ and the new object $M'$ obtained by applying the functor
corresponding to $-1 \in \Z_2$ to $M$. Note that $\mu_1 = \on{id}: M
\to M$ and so we must have
$$
\mu^2 = \on{id}
$$
to satisfy the relations of $\Z_2$. If $M = (\A,(\ep_+,\ep_-))$, then
$M' = (\A,(-\ep_+,-\ep_-))$, and we have exactly two isomorphisms
$\mu_\pm$ between $M$ and $M'$, satisfying $\mu^2=\on{id}$: one acts
by $\pm 1$ on $\A_\pm \subset \A$, and the other acts as $\mp 1$. Thus,
we see that there are two non-isomorphic objects in ${\mc C}''$, which
correspond under a canonical equivalence ${\mc C}'' \simeq {\mc C}$ to
$\A_+$ and $\A_-$, respectively. In fact, it is instructive to think of
$\frac{1}{2}(1+\mu_\pm)$ as a projector onto $\A_\pm$ in $A$.

\subsection{Hecke Eigensheaves}    \label{hecke sheaves}

The reason why we have discussed all these subtleties in such great
detail is that the category that most closely matches the usual notion
of Hecke eigenfunctions of the classical theory of automorphic forms
is a category of type (1), so to understand fully the connection to
the classical theory of automorphic forms we have to go through a
category of type (1). However, the category of $A$-branes considered
in \secref{spectral} is naturally a category of type (2). Therefore we
have to make a link between the two types of categories.

Let us recall the traditional definition of Hecke eigensheaves used in
the geometric Langlands Program (see, e.g., \cite{BD} or
\cite{F:houches}, Section 6.1). These are ${\mc D}$-modules on
$\Bun_G$, the moduli stack of $G$-bundles on a curve $\CC$, satisfying
the Hecke eigenobject property. To explain this more precisely, recall
that for each finite-dimensional representation $V$ of the dual group
$\LG$ we have a Hecke functor $H_V$ acting from the category of ${\mc
D}$-modules on $\Bun_G$ to the category of ${\mc D}$-modules on $\CC
\times \Bun_G$. We will not recall the definition of these functors
here, referring the reader to \cite{BD} and \cite{F:houches}. We note
that these functors are closely related to the 't Hooft operators
discussed in \secref{Hecke}, as explained in \cite{KW}.

Let $\E$ be a flat $\LG$-bundle on $\CC$.\footnote{Here we consider
for simplicity the unramified case, but the definition is easily
generalized to the ramified case; we simply omit the points of
ramification of $\E$.} A {\em Hecke eigensheaf} with ``eigenvalue'' $\E$
is by definition a collection of data
\begin{equation}    \label{hecke eigensheaf}
({\mc F},(\alpha_V)_{V \in \Rep(\LG)}),
\end{equation}
where ${\mc F}$ is a ${\mc D}$-module on $\Bun_G$ and $(\alpha_V)$ is
a collection of isomorphisms
\begin{equation}    \label{al V}
\alpha_V: H_V({\mc F}) \overset{\sim}\longrightarrow  V_\E \boxtimes
{\mc F},
\end{equation}
where
$$
V_\E = \E \underset{\LG}\times V
$$
is the flat vector bundle on $\CC$ associated to $V$.

These isomorphisms must satisfy natural compatibility conditions with
respect to the composition of the Hecke functors $H_V$ on the LHS and
the tensor product of representation on the RHS of \eqref{al V}, as
well as the natural associativity condition.\footnote{We will see in
\secref{comm cond} that an additional commutativity condition needs to
be imposed on the isomorphisms $\al_V$.}

Let ${\mc Aut}_\E$ be the category of all Hecke eigensheaves with
eigenvalue $\E$. To make contact with the categories of type (1)
studied in the previous section, let us fix a point $x \in C$. Then
the restriction of the Hecke operator $H_V$ to $x$ is a functor
$H_{V,x}$ from the category of ${\mc D}$-modules on $\Bun_G$ to
itself. For a Hecke eigensheaf \eqref{hecke eigensheaf}, by
restricting the isomorphisms $\al_V$ to $x$, we obtain a compatible
collection of isomorphisms
\begin{equation}    \label{al V x}
\alpha_{V,x}: H_{V,x}({\mc F}) \overset{\sim}\longrightarrow  V_{\E,x}
\otimes {\mc F}.
\end{equation}
Here
$$
V_{\E,x} = \E_x \underset{\LG}\times V,
$$
where $\E_x$ is the fiber of $\E$ at $x$, is a vector space
isomorphic to $V$. The data of ${\mc F}$ and $(\al_{V,x})_{V \in
\Rep(\LG)})$, is precisely the kind of data that we used above in
the definition of the de-equivariantized category. A Hecke eigensheaf
\eqref{hecke eigensheaf} is therefore an object of this type, except
that instead of just one collection of Hecke isomorphisms \eqref{al
general} we have an entire family of such collections parametrized
by points of the curve $\CC$. The condition \eqref{al V} actually
contains much more information than the data of the isomorphisms
\eqref{al V x} for all $x \in \CC$, because formula \eqref{al V}
also describes the dependence of the ``eigenvalues'' $V_{\E,x}$ of
the Hecke operators on $x$: they ``vary'' according to the local
system $\E$.

If we only impose the condition \eqref{al V x} (for a particular $x
\in C$), then the corresponding category carries an action of the
group $\LG$, or, more precisely, its twist by the $\LG$-torsor $\E_x$,
that is, $\LG_{\E_x} = \E_x \underset{\LG}\times \LG$. This action is
defined as in formula \eqref{action of H}. If, on the other hand, we
impose the full Hecke condition \eqref{al general}, then the group
$\LG_{\E_x}$ no longer acts. Rather, we have an action of the (global)
group of automorphisms of $\E$, that is, $\Gamma(C,\LG_\E)$, where
$\LG_{\E} = \E \underset{\LG}\times \LG$.

Suppose that the group of automorphisms of our local system $\E$ is
trivial. This means, in particular, that the center of $\LG$ is
trivial, so that $\LG$ is a semi-simple group of adjoint type. In this
case it is expected that there is a unique irreducible ${\mc
D}$-module ${\mc F}$ satisfying the Hecke property \eqref{al general},
and all other ${\mc D}$-modules satisfying this property are direct
sums of its copies. If so, then there is essentially a unique way to
construct the isomorphisms $\alpha_V$ in \eqref{al V} for the
irreducible ${\mc D}$-modules ${\mc F}$, and this is why in this case
one usually suppresses the data of $(\alpha_V)$.


However, these data become important in the case when $\E$ has a
non-trivial group of automorphisms
$$
\Gamma := \Gamma(C,\LG_\E).
$$
Then we define an action of the group $\Gamma$ on the category ${\mc
Aut}_\E$ as in formula \eqref{action of H}. Namely, given an object
$({\mc F},(\alpha_V))$ of the category ${\mc Aut}_\E$ as in formula
\eqref{hecke eigensheaf}, we construct a new object $({\mc
F},(\alpha^s_V))$, where
$$
\alpha^s_V = (s \otimes 1) \circ \alpha_V.
$$
This new object may not be isomorphic to the old one (and even if it
is, it may be isomorphic to it in different ways).

\subsection{Category Of Hecke Eigensheaves In The Endoscopic Example}

Let us discuss the category of Hecke eigensheaves in our endoscopic
example, when $G = SL_2, \LG = SO_3$, and $\Gamma = \Z_2$. We expect
that in this case any ${\mc D}$-module satisfying the Hecke eigensheaf
property is a direct sum of copies of a ${\mc D}$-module, which we
will denote by ${\mc F}$. As explained in Sections \ref{dmodules} and
\ref{genus one A to D}, the ${\mc D}$-module ${\mc F}$ corresponds to
an $A$-brane ${\mc A}$ on a singular fiber of $\MH(G)$, which is a
magnetic eigenbrane with respect to the 't Hooft operators. As we
explained in \secref{spectral}, this $A$-branes is reducible:
$$
\A = \A_+ \oplus \A_-,
$$
where the $A$-branes $\A_\pm$ are irreducible. Furthermore, there is
not a preferred one among them (see \secref{gerbes} for a more
detailed discussion of this point). Actually, because of this
ambiguity, we previously used the notation $\A_1$ and $\A_2$ for these
$A$-branes, in order to emphasize that they do not correspond
canonically to the $B$-branes $\B_+$ and $\B_-$. But from now on we
will use the notation $\A_\pm$.

Therefore we expect that the ${\mc D}$-module ${\mc F}$ is also
reducible:
$$
{\mc F} = {\mc F}_+ \oplus {\mc F}_-,
$$
and each ${\mc F}_\pm$ is irreducible. We also expect that neither
of them is preferred over the other one.

Recall that the notion of an eigensheaf, as defined above, includes
the isomorphisms $\alpha_V$ for all representations $V$ of $SO_3$. By
using the compatibility with the tensor product structure, we find
that everything is determined by the adjoint representation of $SO_3$,
which we denote by $W$, as before. A Hecke eigensheaf may therefore be
viewed as a pair $({\mc F},\alpha)$, where
\begin{equation}    \label{h eig}
\alpha: H_W({\mc F}) \overset{\sim}\longrightarrow  W_\E \boxtimes
{\mc F}.
\end{equation}

In the endoscopic case the structure group of our $SO_3$-local system
$\E$ is reduced to the subgroup
$$
O_2 = \Z_2 \ltimes \C^\times
$$
(but not to a smaller subgroup).  Denote by $U$ the defining
two-dimensional representation of $O_2$. Then $\det U$ is the
one-dimensional sign representation induced by the homomorphism $O_2
\to \Z_2$. We have
$$
W = (\det U \otimes I) \oplus (U \otimes S),
$$
as a representation of $O_2 \times \Z_2$, where $\Z_2$ is the
centralizer of $O_2$ in $SO_3$ (which coincides with the
center of $O_2$), and, as before, $S$ is the sign representation
of $\Z_2$, and $I$ is the trivial representation of $\Z_2$. Therefore
we have the following decomposition of the corresponding local system:
\begin{equation}    \label{WE}
W_\E = (\det U_\E \otimes I) \oplus (U_\E \otimes S).
\end{equation}

We may twist the isomorphism $\al$ by the action of the non-trivial
element $\tau \in \Z_2 = \Gamma$, which is in the group of
automorphisms of our local system. This way we obtain a new Hecke
eigensheaf, that is, a pair $({\mc F},\alpha')$, where $\al' = (\tau
\otimes 1) \circ \al$. The objects $({\mc F},\alpha)$ and $({\mc
F},\alpha')$ of ${\mc Aut}_\E$ are isomorphic, but non-canonically.
There are in fact two natural isomorphisms, equal to $\pm 1$ on ${\mc
F}_\pm$ or $\mp 1$ on ${\mc F}_\pm$, and there is no natural way to
choose between them.

We expect that any object of the category ${\mc Aut}_\E$ of Hecke
eigensheaves is isomorphic to a direct sum of copies of $({\mc
F},\alpha)$. As in the toy model discussed in \secref{exa}, this means
that the category ${\mc Aut}_\E$ is equivalent to the category of
vector spaces, but in two different ways, corresponding to the choice
of ${\mc F}_+$ and ${\mc F}_-$.


\subsection{Fractional Hecke Eigensheaves}    \label{type 2}

Next, we introduce a category of Hecke eigensheaves of type (2).

Suppose again that we are given an $\LG$-local system $\E$ on a curve
$\CC$, and let $\Gamma$ be the group of its automorphisms. To simplify
our discussion below, we will identify $\Gamma$ with a subgroup of
$\LG$ by picking a point $x \in C$ and choosing a trivialization of
the fiber $\E_x$ of $\E$ at $x$. (This allows us to assign to $\E$ a
homomorphism $\pi_1(C,x) \to \LG$. The group $\Gamma$ may then be
defined as the centralizer of its image.)

Suppose that we are given an abelian subcategory ${\mc C}$ of the
category of ${\mc D}$-modules on $\Bun_G$ equipped with an action of
the tensor category $\Rep(\Gamma)$. In other words, for each $R \in
\Rep(\Gamma)$ we have a functor
$$
M \mapsto R \star M,
$$
and these functors compose in a way compatible with the tensor product
structure on $\Rep(\Gamma)$. The category of Hecke eigensheaves of
type (2) with ``eigenvalue'' $\E$ will have as objects the following
data:
\begin{equation}    \label{obj type 2}
({\mc F},(\alpha_V)_{V \in \Rep(\LG)}),
\end{equation}
where ${\mc F}$ is an object of ${\mc C}$, and the $\alpha_V$ are
isomorphisms defined below.

Denote by $\on{Res}_{\Gamma}(V)$ the restriction of a representation $V$ of
$\LG$ to $\Gamma$. If $\Rep(\Gamma)$ is a semi-simple category (which is the
case, for example, if $\Gamma$ is a finite group), then we obtain a
decomposition
$$
\on{Res}_{\Gamma}(V) = \bigoplus_i F_i \otimes R_i,
$$
where the $R_i$ are irreducible representations of $\Gamma$ and $F_i$
is the corresponding representation of the centralizer of $\Gamma$ in
$\LG$.  Twisting by $\E$, we obtain a local system
$(\on{Res}_{\Gamma}(V))_\E$ on $\CC$ with a commuting action of
$\Gamma$, which decomposes as follows:
$$
(\on{Res}_{\Gamma}(V))_\E = \bigoplus_i (F_i)_\E \otimes R_i.
$$
Note that since $\Gamma$ is the group of automorphisms of $\E$, the
structure group of $\E$ is reduced to the centralizer of $\Gamma$ in
$G$, and $F_i$ is a representation of this centralizer. Therefore
$F_i$ may be twisted by $\E$, and the resulting local system (or flat
vector bundle) on $C$ is denoted by $(F_i)_\E$.

The isomorphisms $\al_V$ have the form
\begin{equation}    \label{Res isom}
\al_V: H_V(M) \overset{\sim}\longrightarrow (\on{Res}_{\Gamma}(V))_\E \;
\star M = \bigoplus_i (F_i)_{\E} \boxtimes (R_i \star M), \qquad M \in
{\mc C},
\end{equation}
and they have to be compatible in the obvious sense. We will denote
the category with objects \eqref{obj type 2} satisfying the above
conditions by ${\mc Aut}'_\E$.\footnote{As we will see in \secref{comm
cond}, we also need to impose an additional commutativity condition on
the isomorphisms $\al_V$.}

The category of Hecke eigensheaves of type (2) has the advantage of
being more concrete than the category of type (1), and it matches more
closely the structure of the categories of $A$- and $B$-branes that we
have found in \secref{spectral}.

However, the category ${\mc Aut}_\E$ may be reconstructed from ${\mc
Aut}'_\E$ by the procedure of de-equiva\-riantization along the lines
of \secref{gen cat}. Conversely, applying the procedure of
equivariantization to ${\mc Aut}_\E$, we obtain a category that is
equivalent to ${\mc Aut}'_\E$.

What does the category ${\mc Aut}'_\E$ look like in our main example
of geometric endoscopy? In this case the category ${\mc C}$ should
have two irreducible objects, ${\mc F}_+$ and ${\mc F}_-$, which are
the ${\mc D}$-modules on $\Bun_G$ corresponding to the fractional
$A$-branes $\A_+$ and $\A_-$. The category $\Rep(\Z_2)$ acts on them
as follows: the sign representation $S$ of $\Gamma=\Z_2$ permutes
them,
$$
S \star {\mc F}_\pm = {\mc F}_\mp,
$$
while the trivial representation $I$ acts identically.

Since the category of representations of $SO_3$ is generated by the
adjoint representation $W$, it is sufficient to formulate the Hecke
property \eqref{Res isom} only for the adjoint representation $W$ of
$SO_3$. It reads
\begin{align}    \notag
H_W({\mc F}_+) &\simeq (\det U_\E \boxtimes {\mc F}_+) \oplus (U_\E
\boxtimes {\mc F}_-), \\ \label{pm} H_W({\mc F}_-) &\simeq (\det U_\E
\boxtimes {\mc F}_-) \oplus (U_\E \boxtimes {\mc F}_+),
\end{align}
where $\det U_\E$ and $U_\E$ are the summands of $W_\E$ defined in
formula \eqref{WE}. This matches the action of the 't Hooft operators
on the $A$-branes given by formula \eqref{elko1}. Since that formula
describes the behavior of the fractional branes $\A_\pm$, we will call
the property expressed by formulas \eqref{Res isom} and \eqref{pm} the
{\em fractional Hecke property}, and the corresponding ${\mc
D}$-modules {\em fractional Hecke eigensheaves}. On the other hand, we
will call the ordinary Hecke property \eqref{al V} the {\em regular
Hecke property} and the ${\mc D}$-modules satisfying it {\em regular
Hecke eigensheaves} (the reason for this terminology is that such an
eigensheaf corresponds to the regular representation of the group
$\Gamma$).

\subsection{Other Examples}    \label{other examples}

In this section we look at other examples of categories of fractional
Hecke eigensheaves. First, we consider the abelian case, when $G =
\C^\times$ or an arbitrary torus. Then we consider the case of
$\LG$-local systems (where $G$ is an arbitrary reductive group), whose
group of automorphisms is the center of $\LG$. Finally, we discuss the
geometric Eisenstein series.

\subsubsection{The Case Of $\C^\times$}

Let $G=\C^\times$, so that $\LG=\C^\times$ as well. Let $\E$ be a
$\C^\times$-local system on $C$, which we will view as a rank one
local system. The objects of the category ${\mc Aut}_\E$ are the data
$({\mc F},\al)$, where ${\mc F}$ is a ${\mc D}$-module on $\Pic$, the
Picard scheme of $C$, and
\begin{equation}    \label{ab type 1}
\al: H_1({\mc F}) \overset{\sim}\longrightarrow \E \boxtimes {\mc F}.
\end{equation}
Here $H_1$ is the Hecke functor corresponding to the identity
character $\C^\times \to \C^\times$,
$$
H_1({\mc F}) = p^*({\mc F}),
$$
where
$$
p: C \times \Pic \to \Pic, \qquad (x,{\mc L}) \mapsto {\mc L}(-x).
$$
This Hecke property automatically implies the Hecke property for the
Hecke functors $H_n, n \in \Z$, corresponding to other characters of
$\C^\times$. The ${\mc D}$-module Fourier--Mukai equivalence, due to
\cite{Laumon:fm,Roth}, implies that ${\mc F}$ is a direct sum of
copies of a particular flat line bundle on $\Pic$. From now on we will
denote by ${\mc F}$ this flat line bundle.

Since $\Pic$ is the disjoint union of its components $\Pic_n, n \in
\Z$, corresponding to line bundles of degree $n$, each ${\mc
D}$-module ${\mc F}$ on $\Pic$ decomposes into a direct sum
\begin{equation}    \label{dir sum Z}
{\mc F} = \bigoplus_{n \in \Z} {\mc F}_n.
\end{equation}

The automorphism group of any rank one local system $\E$ is
$\C^\times$. Hence we have an action of $\C^\times$ on the category
${\mc Aut}_\E$. It is given by the formula
$$
\lambda \cdot ({\mc F},\al) \mapsto ({\mc F},\lambda \alpha), \qquad
\lambda \in \C^\times.
$$

Now let us describe the corresponding category of fractional Hecke
eigensheaves. Let ${\mc C}$ be the category whose objects are the
direct sums of the ${\mc D}$-modules ${\mc F}_n, n \in \Z$, appearing
in the decomposition \eqref{dir sum Z} of the Hecke eigensheaf ${\mc
F}$. These are flat vector bundles on the components $\Pic_n \subset
\Pic$. The action of the category $\Rep(\C^\times)$ on ${\mc C}$ is
given by the formula
$$
[m] \star {\mc F}_n = {\mc F}_{n+m},
$$
where $[m]$ denotes the one-dimensional representation corresponding
to the character $\C^\times \to \C^\times, a \mapsto a^m$.

An object of the category ${\mc Aut}'_\E$ of fractional Hecke
eigensheaves is then given by the data $({\mc K},\al)$, where ${\mc K}
\in {\mc C}$ and
\begin{equation}    \label{ab type 2}
\al: H_1({\mc K}) \overset{\sim}\longrightarrow \E \boxtimes ([1] \star
   {\mc K}).
\end{equation}
The isomorphism $\al$ gives rise to a unique isomorphism
$$
\al_n: H_n({\mc K}) \overset{\sim}\longrightarrow \E^{\otimes n}
\boxtimes ([n] \star {\mc K}), \qquad n \in \Z,
$$
satisfying the required properties.

Note that the regular Hecke property of ${\mc F}$ given by formula
\eqref{ab type 1} is equivalent to the fractional Hecke property for
${\mc F}_n, n \in \Z$. Therefore we obtain an isomorphism \eqref{ab
type 2} for ${\mc K} = {\mc F}_m$ and hence for all objects of ${\mc
C}$.

This example of a category of fractional Hecke eigensheaves should be
contrasted with our main example for $G=SL_2$ coming from the
endoscopy. In the latter case the group of automorphisms is a finite
group $\Z_2$. Hence we have a finite decomposition of a regular Hecke
eigensheaf, ${\mc F} = {\mc F}_+ \oplus {\mc F}_-$, with the summands
that are labeled, albeit non-canonically, by irreducible
representations of $\Z_2$. In the other example the group of
automorphisms is the Lie group $\C^\times$. Hence a regular Hecke
eigensheaf decomposes into a direct sum \eqref{dir sum Z} of
infinitely many summands that are labeled, this time canonically, by
irreducible representations of $\C^\times$.

An important difference is that in the endoscopic example there is no
canonical equivalence between the category ${\mc C}$ of fractional
Hecke eigensheaves and the category $\Rep(\Z_2)$ (in other words, we
cannot distinguish between ${\mc F}_+$ and ${\mc F}_-$), whereas in
the other example we identify it canonically with $\Rep(\C^\times)$
(indeed, there is a canonical object ${\mc F}_0$ supported on the
component $\Pic_0$ of line bundles of degree $0$). The reason for this
will be discussed in \secref{gerbes}.

\subsubsection{Arbitrary Torus}

This example of category of fractional Hecke eigensheaves generalizes
in a straightforward way to the case when $G$ is an arbitrary
torus. In this case the corresponding moduli space $\Bun_T$ decomposes
into a disjoint union of components $\Bun_{T,\chi}$, where $\chi$ runs
over the lattice $\check{P}$ of characters of $^L\neg T$. The regular
Hecke eigensheaf ${\mc F}$ therefore decomposes into a direct sum
\begin{equation}    \label{dec T}
{\mc F} = \bigoplus_{\chi \in \check{P}} {\mc F}_{\chi},
\end{equation}
where ${\mc F}_{\chi}$ is supported on $\Bun_{T,\chi}$. The ${\mc
D}$-modules ${\mc F}_{\chi}$ are then the building blocks of the
category of fractional Hecke eigensheaves. They satisfy the fractional
Hecke property
\begin{equation}    \label{al la}
H_\la({\mc F}_{\chi}) \simeq {\mc F}_{\la+\chi} = [\la] \star {\mc
F}_{\chi},
\end{equation}
where $H_\la$ is the Hecke functor corresponding to a character $\la
\in \check{P}$. The objects of the category ${\mc Aut}'_\E$ are
collections $({\mc K},(\al_\la)_{\la \in \check{P}})$, where ${\mc K}$
is a direct sum of copies of ${\mc F}_\chi$ with some multiplicities
and $\al_\la$ form a system of compatible isomorphisms \eqref{al la}.

\subsubsection{The Center As The Automorphism Group}    \label{center}

This has an analogue for an arbitrary reductive group $G$. Namely, let
$\LZ$ be the center of the dual group $\LG$. Suppose that $\E$ is an
$\LG$-local system on $C$ whose group of automorphisms is precisely
$\LZ$. A generic local system has this property. In this case we
expect that there is a unique regular Hecke eigensheaf ${\mc F}$
with the eigenvalue $\E$ which is irreducible on each connected
component of $\Bun_G$ (and any other Hecke eigensheaf is a direct sum
of copies of ${\mc F}$). The connected components of $\Bun_G$ are
labeled precisely by the lattice $\check{P}$ of characters of $\LZ$,
so we are in the same situation as above: the sheaf ${\mc F}$
decomposes
$$
{\mc F} = \bigoplus_{\chi \in \check{P}} {\mc F}_{\chi},
$$
where ${\mc F}_{\chi}$ is supported on the component corresponding to
$\chi$. The ${\mc D}$-modules ${\mc F}_{\chi}$ are then the building
blocks of the category of fractional Hecke eigensheaves, satisfying
the property \eqref{al la}. The objects of the category
${\mc Aut}'_\E$ are then defined as in the case when $G$ is a torus.
We expect that this category is equivalent to the category
$\Rep(\LZ)$. This matches the structure of the category of
$B$-branes at the point of the moduli stack of $\LG$-local systems
corresponding to $\E$, which is also equivalent to $\Rep(\LZ)$, because
$\LZ$ is assumed to be the group of automorphisms of $\E$ (see
\secref{braneduals}).

\subsubsection{Geometric Eisenstein Series}

As our last example, we describe the category of fractional Hecke
eigensheaves for non-abelian groups using the construction of
geometric Eisenstein series from \cite{Laumon:eis,BG}. More precisely,
let $G$ be a reductive algebraic group and $T$ its maximal torus. Then
$^L\neg T$ is a maximal torus in $\LG$. Let $\E$ be an $^L\neg
T$-local system, which we view as an $\LG$-local system. Suppose that
$\E$ is generic, in the sense that the rank one local systems
corresponding to the roots of $\LG$ are all non-trivial. Then the
group of automorphisms of $\E$ is equal to $^L\neg T$. In this case
there is a geometric construction of a Hecke eigensheaf ${\mc F}_G$ on
$\Bun_G$ with the eigenvalue $\E$ \cite{Laumon:eis,BG}, starting from
a Hecke eigensheaf ${\mc F}_T$ with respect to $\E$, considered as an
$^L\neg T$-local system.

By construction, the decomposition \eqref{dec T} gives rise to a
decomposition of ${\mc F}_G$,
\begin{equation}    \label{dec G}
{\mc F}_G = \bigoplus_{\chi \in \check{P}} {\mc F}_{G,\chi}.
\end{equation}
The ${\mc D}$-modules ${\mc F}_{G,\chi}$ are then the building blocks
of the category of fractional Hecke eigensheaves corresponding to
$\E$. We define the category ${\mc C}$ as the category whose objects
are direct sums of the ${\mc D}$-modules ${\mc F}_{G,\chi}, \chi \in
\check{P}$. The category $\Rep(^L\neg T)$ acts on it by the formula
$$
[\la] \star {\mc F}_{G,\chi} = {\mc F}_{G,\chi+\la}.
$$
The objects of the corresponding category ${\mc Aut}'_\E$ are
collections $({\mc K},(\al_\la)_{\la \in \check{P}})$, where ${\mc K}
\in {\mc C}$ and the $\al_\la$'s form a system of compatible
isomorphisms (see formula \eqref{Res isom})
$$
\al_\la: H_V({\mc F}_{G,\chi}) \overset{\sim}\longrightarrow
\bigoplus_\mu V(\mu) \otimes ([\la] \star {\mc F}_{G,\chi}) =
\bigoplus_\mu V(\mu) \otimes {\mc F}_{G,\chi+\mu}.
$$
Here $V$ is a representation of $\LG$, and $V(\mu)$ is the subspace of
$V$ of weight $\mu$, so that we have
$$
\on{Res}_{^L\neg T}(V) = \bigoplus_\mu V(\mu).
$$

Note that in the abelian case when $G=T$ the decomposition \eqref{dec
T} has geometric origin: it corresponds to the splitting of $\Bun_T$
into a union of connected components. However, for non-abelian $G$ the
decomposition \eqref{dec G} of the corresponding Eisenstein sheaves is
not directly linked in any obvious way to the geometry of the
underlying moduli stack of $G$-bundles. This is similar to what
happens in the endoscopic case.

It would be interesting to construct explicitly the $A$-branes
corresponding to the Eisenstein sheaves ${\mc F}_{G}$ and ${\mc
F}_{G,\chi}$ using the mirror symmetry of the Hitchin fibrations, by
analogy with the endoscopic example. On the $B$-model side, this
corresponds to a more complicated singularity, with a continuous group
of automorphisms (the maximal torus of $\LG$) rather than a finite
group, as in the endoscopic examples (this is reflected in particular
in the fact that the Hitchin fibers now have infinitely many
irreducible components). Because of that, the analysis of the mirror
symmetry becomes more subtle in this case. We plan to discuss this in
more detail elsewhere.

\section{The Classical Story}    \label{classical}

In this section we recall the set-up of endoscopy and $L$-packets in
the classical theory of automorphic forms, discovered by Labesse and
Langlands \cite{LL}. We will discuss potential implications of the
geometric picture outlined above for the classical theory in the next
section.

\subsection{Local And Global Langlands Correspondence}
\label{recoll}

Let us first recall the general setup of the classical and the
geometric Langlands correspondence. For simplicity, we will restrict
ourselves here to the unramified situation.

The geometric Langlands correspondence predicts, in the first
approximation, that for each irreducible $\LG$-local system $\E$ on
$\CC$ there exists a unique (up to an isomorphism) Hecke eigensheaf
$({\mc F}_\E,(\al_V))$ on $\Bun_G$ with eigenvalue $\E$, a notion
discussed in detail in \secref{hecke sheaves}.

We wish to recall the relation between the geometric Langlands
correspondence and the classical Langlands correspondence, in the case
when the curve $\CC$ is defined over a finite field $k = {\mathbb
F}_q$. So let $\CC$ be such a curve and $F$ the field of rational
functions of $\CC$. For example, if $\CC = \pone$, then $F$ consists
of fractions $f(z)/g(z)$, where $f(z)$ and $g(z)$ are polynomials (in
variable $z$) over $k$ which do not have common factors (and $g(z) =
z^m + \ldots$ is monic). For each closed point $x$ of $\CC$ we have
the local field $F_x \simeq k_x\ppart$, where $t$ is a local
coordinate at $x$, and its ring of integers $\OO_x \simeq
k_x[[t]]$. Here $k_x$ is the residue field of $x$, which is in general
a finite extension of $k$, and hence is isomorphic to ${\mathbb
F}_{q^m}$ for some $m \geq 1$. This number is called the {\em degree}
of $x$ and is denoted by $\deg(x)$. For example, in the case when
$\CC=\pone$ closed points correspond to irreducible monic polynomials
in $k[z]$, and in addition there is the point $\infty$. The points
with residue field $k$ correspond to polynomials of degree one, $z-a,
a \in k$. A general irreducible monic polynomial $P(z)$ of degree $n$
corresponds to a closed point of $\CC$ of degree $n$. The field $k_x$
is just the quotient of $k[z]$ by the principal ideal generated by
$P(z)$.

The ring ${\mathbb A}_F$ of ad\'eles of $F$ is by definition the
restricted product
$$
{\mathbb A}_F = \prod_{x \in \CC}{}' \; F_x,
$$
where the word ``restricted'' (and the prime in the notation) refers
to the fact that elements of ${\mathbb A}_F$ are collections
$(f_x)_{x \in \CC}$, where $f_x \in \OO_x$ for all but finitely many
$x \in \CC$.

Let $\on{Gal}(\ol{F}/F)$ be the Galois group of $F$, the group of
automorphisms of the separable closure $\ol{F}$ of $F$ (obtained by
adjoining to $F$ the roots of all separable polynomials with
coefficients in $F$), which preserve $F$ pointwise. We have a natural
homomorphism $\on{Gal}(\ol{F}/F) \to \on{Gal}(\ol{k}/k)$. The group
$\on{Gal}(\ol{k}/k)$ is topologically generated by the Frobenius
automorphism $\on{Fr}: y \mapsto y^q$, and is isomorphic to the
pro-finite completion $\wh\Z$ of the group of integers $\Z$. The
preimage of $\Z \subset \wh\Z$ in $\on{Gal}(\ol{F}/F)$ is the {\em
Weil group} $W_F$ of $F$.

The Weil group (or, more precisely, its unramified quotient) is the
arithmetic analogue of the fundamental group of $\CC$. Therefore the
arithmetic analogue of a $\LG$-local system on $\CC$ is a
homomorphism\footnote{Here we need to consider $\LG$ over the field
$\ol{\mathbb Q}_\ell$, where $\ell$ does not divide $q$, and the
so-called $\ell$-adic homomorphisms (see, e.g., \cite{F:houches},
Section 2.2).}
$$
\sigma: W_F \to \LG.
$$
The global Langlands conjecture predicts, roughly speaking, that to
each $\sigma$ corresponds an automorphic representation
$\pi(\sigma)$ of the group $G({\mathbb A}_F)$. This means that it may
be realized in a certain space of functions on the quotient $G(F) \bs
G({\mathbb A}_F)$.\footnote{This requires some explanation if
$\pi(\sigma)$ does not appear in the ``discrete spectrum'', but we
will ignore this technical issue here.}

It is useful to relate this global Langlands conjecture to the local
ones. Recall that for each closed point $x$ of $C$ (we will write $x
\in \CC$) we have the local field $F_x \simeq k_x\ppart$, which is a
completion of $F$. We define its Weil group $W_{F_x}$ in the same way
as above, as the preimage of $\Z \subset \wh\Z =
\on{Gal}(\ol{k}_x/k_x)$ in $\on{Gal}(\ol{F}_x/F_x)$ under the
homomorphism $\on{Gal}(\ol{F}_x/F_x) \to \on{Gal}(\ol{k}_x/k_x)$. The
group $W_{F_x}$ may be realized as a subgroup of the global Weil group
$W_F$, but non-canonically. However, its conjugacy class in $W_F$ is
canonical. Hence the equivalence class of $\sigma: W_F \to \LG$ as
above gives rise to an equivalence class of homomorphisms
$$
\sigma_x: W_{F_x} \to \LG.
$$
The local Langlands conjecture predicts, roughly speaking, that to
each $\sigma_x$ we can associate a (smooth) irreducible representation
$\pi_x$ of $G(F_x)$. Taking their restricted tensor product (see
below), we obtain an irreducible representation
\begin{equation}    \label{product}
\pi = \pi(\sigma) = \bigotimes_{x \in \CC}{}' \; \pi_x
\end{equation}
of the ad\`elic group $G({\mathbb A}_F)$. The compatibility between
the global and local Langlands conjectures is the statement that this
$\pi$ is automorphic. $\pi$ is then the automorphic representation
corresponding to the global homomorphism $\sigma$. Schematically,
\begin{align*}
\sigma_x &\overset{\on{local}}\longrightarrow \pi_x \\ \sigma
&\overset{\on{global}}\longrightarrow \pi = \bigotimes_{x \in \CC}{}'
\pi_x.
\end{align*}

Here the homomorphism $\sigma$ is assumed to be unramified at all but
finitely many points of $\CC$. This means for all but finitely many $x
\in C$ the homomorphism $\sigma_x: W_{F_x} \to \LG$ factors through
the quotient $W_{F_x} \to \Z$. The generator of this quotient $\Z
\subset \on{Gal}(\ol{k}_x/k_x)$ is called the Frobenius element
associated to $x$ and is denoted by $\on{Fr}_x$. Since $\sigma_x$ is
only well-defined up to conjugation, we obtain that
$\sigma_x(\on{Fr}_x)$ gives rise to a well-defined conjugacy class in
$\LG$. It is believed that the conjugacy classes obtained this way are
always semi-simple. Thus, we obtain a semi-simple conjugacy class
$\sigma_x(\on{Fr}_x)$ in $\LG$ for all but finitely many $x \in \CC$.

The irreducible representation $\pi_x$ corresponding to an unramified
$\sigma_x$ is also {\em unramified}, which means that its subspace of
invariant vectors $(\pi_x)^{G(\OO_x)} \subset \pi_x$ under the maximal
compact subgroup $G(\OO_x) \subset G(F_x)$ is one-dimensional. We will
fix, once and for all, a non-zero $G(\OO_x)$-invariant vector $v_x \in
\pi_x$ for all such $x \in \CC$. Then the prime in formula
\eqref{product} (indicating the ``restricted tensor product'') means
that this space is spanned by vectors of the form $\bigotimes_{x \in
\CC} w_x$, where for all but finitely many points $x$ we have $w_x =
v_x$.

Since the vector $v_x$ is $G(\OO_x)$-invariant, it is an eigenvector
of the spherical Hecke algebra, defined as the algebra of compactly
supported $G(\OO_x)$ bi-invariant functions on $G(F_x)$. By the {\em
Satake correspondence}, this algebra is isomorphic to the
representation ring $\Rep(\LG)$ (see, e.g., \cite{F:houches}, Section
5.2). Therefore to each $V \in \Rep(\LG)$ corresponds an element of
the spherical Hecke algebra, which we denote by $T_{V,x}$. These
operators, which are function theoretic analogues of the Hecke
functors $H_{V,x}$ discussed in \secref{hecke sheaves}, act on
$\pi_x$. The vector $v_x \in \pi_x$ is a joint eigenvector with
respect to this action, and the eigenvalues are recorded by the
conjugacy class of $\sigma_x(\on{Fr}_x)$. Namely, we
have\footnote{Here and below we skip factors of the form $q_x^{m}$
on the right hand sides of this and other similar formulas.}
\begin{equation}    \label{hecke vector}
T_{V,x} \cdot v_x = \on{Tr}(\sigma(\on{Fr}_x),V) v_x.
\end{equation}
Actually, this property determines $\pi_x$ uniquely.


\subsection{$L$-packets}    \label{L-packets}

The picture of the local and global Langlands correspondences outlined
above is correct (and has been proved) for $G=GL_n$. But for other
groups one needs to make some adjustments. The most important
adjustment is that in general the local Langlands correspondence
assigns to each $\sigma_x$ not a single equivalence class of
irreducible representations $\pi_x$, but a collection $\{ \pi_{x,\al}
\}_{\al \in A_x}$, called a (local) $L$-{\em packet}. Usually, this
happens for infinitely many points $x \in \CC$, and so picking a
particular representation in each of these $L$-packets, we obtain
infinitely many representations of $G({\mathbb A}_F)$. It turns out
that in general not all of them are automorphic, and those which are
automorphic may occur in the space of functions on $G(F) \bs
G({\mathbb A}_F)$ with different multiplicities. There is a beautiful
combinatorial formula for this multiplicity for those homomorphisms
$\sigma: W_F \to \LG$ which factor through the groups $\LH$ dual to
the endoscopic group of $G$ (see \cite{LL,K:cusp}).

This phenomenon was first discovered in the case of $G=SL_2$ by
Labesse and Langlands in \cite{LL}. Let us briefly summarize their
results in the unramified case.

In this case $\LG = SO_3 = PGL_2$. Let $E$ be an unramified quadratic
extension of $F$. This is the field of functions $k(\CC')$ of a degree
two unramified covering $\CC'$ of our curve $\CC$. Let $\mu$ be a
character, that is, a one-dimensional ($\ell$-adic) representation of
the Weil group $W_E$. The group $W_F$ contains $W_E$ as a normal
subgroup, and the quotient is isomorphic to $\on{Gal}(E/F) = \{ 1,\tau
\}$. Let
$$
\sigma: W_F \to PGL_2
$$
be the projectivization of the two-dimensional representation
$\wt\sigma$ of $W_F$ induced from $\mu$,
$$
\wt\sigma = \on{Ind}_{W_E}^{W_F} \mu.
$$
It is clear that the image of $\sigma$ belongs to the subgroup $O_2
\subset SO_3 = PGL_2$. We will assume in what follows that the image
does not belong to the subgroup $\Z_2 \times \Z_2 \subset O_2$ or to
the connected component $SO_2$ of $O_2$. In this case the centralizer
of the image of $\sigma$ is equal to $\Z_2$. The corresponding
homomorphisms $\sigma: W_F \to PGL_2$ are the arithmetic analogues of
the $SO_3=PGL_2$-local systems that we have considered above in the
geometric setting.

Let us now look at the local homomorphisms
$$
\sigma_x: W_{F_x} \to PGL_2
$$
corresponding to $\sigma$. There are two possibilities:

\begin{itemize}

\item The point $x$ is split in $E/F$. This means that the preimage of
$x$ in $\CC'$ consists of two closed points, which we will denote by
$y_1(x)$ and $y_2(x)$, and their residue fields are both equal to the
residue field of $x$. In this case $W_{F_x}$ is isomorphic to both
$W_{E_{y_1(x)}}$ and $W_{E_{y_2(x)}}$, which, when realized as
subgroups of $W_F$, are conjugate to each other:
$$
\wt\tau W_{E_{y_1(x)}} \wt\tau^{-1} = W_{E_{y_2(x)}},
$$
where $\wt\tau$ projects onto the non-trivial element $\tau \in
\on{Gal}(E/F)$ under the homomorphism $W_F \to \on{Gal}(E/F)$.

\item The point $x$ is not split. Then the preimage of $x$ in $\CC'$
consists of one closed point, which we will denote by $y(x)$, and its
residue field $k_{y(x)}$ is a quadratic extension of the residue
field $k_x$ of $x$. Thus, if $k_x \simeq {\mathbb F}_{q^n}$, then
$k_{y(x)} \simeq {\mathbb F}_{q^{2n}}$. In this case $W_{F_x}$
contains  $W_{E_{y(x)}}$ as a normal subgroup, and the quotient is
isomorphic to $\on{Gal}(E_{y(x)}/F_x) \simeq \Z_2$.

\end{itemize}

In the first case $\sigma_x$ is equivalent to the projectivization of
the two-dimensional representation $\wt\sigma_x$ of $W_{F_x}$ defined
by the formula (here we identify $W_{F_x}$ with $W_{E_{y_1(x)}}$):
\begin{equation}
g \mapsto \begin{pmatrix} \mu_{y_1(x)}(g) & 0 \\ 0 &
\mu_{y_2(x)}(\wt\tau g \wt\tau^{-1}) \end{pmatrix}, \qquad g \in
W_{E_{y_1(x)}},
\end{equation}
where $\mu_{y_i(x)}$ denotes the restriction of $\mu$ to
$W_{E_{y_i(x)}}$.

In the second case $\sigma_x$ is isomorphic to the projectivization of
the two-dimensional induced representation
\begin{equation}    \label{ind local}
\wt\sigma_x = \on{Ind}_{W_{E_{y(x)}}}^{W_{F_x}} \mu_{y(x)}.
\end{equation}

As shown by Labesse--Langlands, the local $L$-packet contains either
one or two irreducible representations of $SL_2(F_x)$. In both cases
they appear as irreducible summands of a single irreducible
representation of $GL_2(F_x)$.

{}From now on we will focus on the unramified homomorphisms from the
Weil group $W_F$ to $PGL_2$. We have already assumed that the covering
$\CC' \to \CC$ is unramified, and from now on we will assume that the
character $\mu$ of $W_E$ is unramified as well. Then our $\sigma$ is
unramified, and hence so are the local homomorphisms $\sigma_x:
W_{F_x} \to PGL_2$.\footnote{Note however that there are ramified
characters $\mu$ such that the corresponding homomorphism $\sigma$ is
unramified. Imposing the condition that $\mu$ be unramified is similar
to considering $PGL_2$-local systems with $w_2 = 0$ in the geometric
theory.} Hence $\sigma_x$ is determined by the image
$\sigma_x(\on{Fr}_x)$ in $PGL_2$ of the Frobenius conjugacy class
$\on{Fr}_x$. There are two possibilities: if $\sigma_x(\on{Fr}_x)$ is
the conjugacy class of the element
\begin{equation}    \label{trans}
\begin{pmatrix} 0 & 1 \\ 1 & 0 \end{pmatrix} \in PGL_2,
\end{equation}
whose centralizer has two connected components, then the $L$-packet
consists of two irreducible representations of $SL_2(F_x)$; otherwise
the centralizer is connected, and the $L$-packet consists of a single
irreducible representation.

The former scenario may occur in two ways. First, suppose that the
point $x$ is non-split. Choose the basis $\{ 1,\on{Fr}_x \}$ in the
induced representation \eqref{ind local}. Then we have
$$
\wt\sigma_x(\on{Fr}_x) = \begin{pmatrix} 0 & \mu(\on{Fr}_{y(x)})
\\ 1 & 0 \end{pmatrix}.
$$
Therefore we obtain that $\sigma_x(\on{Fr}_x)$ is the conjugacy class
of \eqref{trans}. The second possibility is that $x$ is split, and the
values of $\mu$ on $\on{Fr}_{y_1(x)}$ and $\on{Fr}_{y_2(x)}$ differ by
the minus sign. In this case $\wt\sigma_x(\on{Fr}_x)$ is conjugate to
\begin{equation}    \label{at split}
\mu(\on{Fr}_{y_i(x)}) \begin{pmatrix} 1 & 0
\\ 0 & -1 \end{pmatrix} \sim \begin{pmatrix} 1 & 0
\\ 0 & -1 \end{pmatrix} \in PGL_2, \qquad i=1,2,
\end{equation}
Therefore $\sigma_x(\on{Fr}_x)$ is again the conjugacy class of
\eqref{trans} in $PGL_2$. However, we will see below that the second
case does not play an important role in the endoscopy.

In both of these cases, the corresponding $L$-packet consists of two
irreducible representations, $\pi'_x$ and $\pi''_x$, of
$SL_2(F_x)$. We will discuss them in more detail in the next section.

Thus, we see that at infinitely many points of $\CC$ we have a binary
choice between the two possible irreducible representations $\pi_x$ of
$SL_2(F_x)$ in the tensor product \eqref{product}. However, it turns
out that only about one half of those choices -- at the non-split
points of $\CC$ (with respect to the covering $\CC' \to \CC$) -- gives
rise to automorphic representations of $SL_2({\mathbb A}_F)$. More
precisely, if we choose particular representations $\pi_x$ in the
local $L$-packets at all but one non-split point of $\CC$, say $y$,
then only one of the two representations $\{ \pi'_y, \pi''_y \}$ of
the local $L$-packet at the remaining point $y$ will complete the
tensor product $\bigotimes'_{x \neq y} \pi_x$ to an automorphic
representation $\pi$ of $SL_2({\mathbb A}_F)$.

The precise description of the collections of irreducible
representations whose tensor products are automorphic will be
presented in the next section. This description was first given by
Labesse and Langlands \cite{LL} using the trace formula. In the
Appendix we will give an alternative derivation of this description
using an explicit formula for the Hecke eigenfunctions due to Weil
\cite{Weil} and Jacquet--Langlands \cite{JL} using the Fourier
transform of the Whittaker functions.

\subsection{Spaces Of Invariant Vectors}    \label{inv vec}

Our goal in what follows is to pass from the classical to the
geometric setting. The first step in this direction is to cut down the
infinite-dimensional representation $\pi = \bigotimes'_{x \in \CC}
\pi_x$ of $G({\mathbb A}_F)$ to a finite-dimensional vector space
$\pi^K$ of $K$-fixed vectors, where $K$ is a compact subgroup of
$G({\mathbb A}_F)$. We take $K$ to be the product
$$
K = \prod_{x \in \CC} K_x
$$
of compact subgroups $K_x \subset G(F_x) \simeq G\ppart$. A typical
example is the subgroup $G(\OO_x) = G[[t]]$. As we discussed in
\secref{recoll}, any vector in $\pi$ is invariant under the subgroup
that is the product of $G(\OO_x)$ for all but finitely many $x$. If
$\pi$ is automorphic, then $\pi^K$ is realized in the space of
functions on the double quotient $G(F) \bs G({\mathbb A}_F)/K$, which
are Hecke eigenfunctions for all $x \in \CC$ for which $K_x =
G(\OO_x)$. This double quotient has a geometric interpretation as the
set of ${\mathbb F}_q$-points of a moduli stack of $G$-bundles on
$\CC$ with parabolic (or level) structures at those points of $\CC$
where $K_x$ is not maximal compact (e.g., choosing $K_x$ to be the
Iwahori subgroup corresponds to fixing a Borel reduction in the fibers
of the $G$-bundles at $x$, etc.), and this fact will be used below in
order to relate the classical and the geometric Langlands conjectures.

In the unramified case the existence of non-trivial $L$-packets is
related to the fact that there are inequivalent choices for a maximal
compact subgroup of $G(F_x)$. For example, for $G=SL_2$ there are two
inequivalent choices: one of them is $SL_2[[t]]$, and the other one is
\begin{equation}    \label{another compact}
\begin{pmatrix}
t & 0 \\ 0 & 1
\end{pmatrix} SL_2[[t]] \begin{pmatrix}
t^{-1} & 0 \\ 0 & 1
\end{pmatrix}
\end{equation}
(they are conjugate in $GL_2\ppart$, but not in $SL_2\ppart$).  We
will denote these two subgroups by $K'$ and $K''$,
respectively. Naively, it looks like $K'$ is a preferred choice, but
that's because we have tacitly chosen as our initial datum the
``constant'' group scheme $SL_2$ over the curve $\CC$. However, our
initial datum should really be the group $SL_2$ over the field $F$ of
rational functions on $\CC$, or, in other words, a group scheme over
the {\em generic point} of $\CC$. There are many ways to extend it to
a group scheme over the entire $\CC$, and we have the freedom of
extending it in such a way that at $x \in \CC$ we get as the group of
sections over the disc $D_x$, a subgroup of $SL_2(F_x) = SL_2\ppart$
that is conjugate to either $K'$ or $K''$.\footnote{As an analogy,
consider the datum of a line bundle over $\CC \bs \{ x_1,\ldots,x_m
\}$ -- it can be extended to a line bundle on the entire curve $\CC$
in many different ways.} This makes it clear that there is in {\em a
priori} no canonical choice between the subgroups $K'$ and $K''$.

However, in what follows we will fix a particular extension of the
group scheme $SL_2$ from the generic point of $\CC$ to the entire
$\CC$; namely, the ``constant'' group scheme $\CC \times SL_2$. Thus,
we will have a preferred compact subgroup $K'_x = SL_2[[t]]$ at each
point $x \in \CC$.

This ambiguity in the choice of a maximal compact subgroup of
$SL_2(F_x)$ is closely related to the structure of the unramified local
$L$-packets $\{ \pi'_x,\pi''_x \}$. Namely, choosing appropriate
notation, the spaces of invariant vectors $(\pi'_x)^{K'_x}$ and
$(\pi''_x)^{K''_x}$ are one-dimensional, whereas $(\pi'_x)^{K''_x} =
(\pi''_x)^{K'_x} = 0$.

Let us return to the setting of the previous section. Thus, we are
given an unramified degree two covering $\CC'$ of $\CC$ and an
unramified character $\mu$ of $W_E$, where $E$ is the field of
functions on $\CC'$. We let $\sigma$ be a homomorphism $W_F \to PGL_2$
defined as above. Now let us fix one of the two maximal compact
subgroups $K_x$ at each point of $x$ at which $\sigma_x(\on{Fr}_x)$ is
conjugate to \eqref{trans}. We will assume that
$K_x = K'_x = SL_2[[t]]$ at all but finitely many points.

Thus, we have an irreducible representation $\pi_x$ of $SL_2(F_x)$
such that $(\pi_x)^{K_x}$ is one-dimensional at all of those points.
Denote by $S$ the finite set of the {\em non-split} points $x \in \CC$
(with respect to the covering $\CC' \to \CC$) such that $K_x =
K''_x$. Let $\# S$ be its cardinality.
The following statement is due to \cite{LL} (for an alternative proof,
see the Appendix).

\begin{thm}    \label{descr aut}
The representation
\begin{equation}    \label{tensor}
\bigotimes_{x \in \CC}{}' \; \pi_x
\end{equation}
of $SL_2({\mathbb A}_F)$ is an automorphic representation if and only
if $\# S$ is even.
\end{thm}

Labesse and Langlands formulate this condition in the following neat
form, which allows a generalization to other groups (see
\cite{K:cusp}). Denote by $S_\sigma$ the group of automorphisms of our
homomorphism $\sigma: W_F \to PGL_2$, that is, the centralizer of the
image of $\sigma$ in $PGL_2$.\footnote{This is the arithmetic analogue
of the group $\Gamma$ discussed in the previous section. We denote it
by $S_\sigma$ in order to follow the standard notation.} Let
$S^0_\sigma$ be its connected component. Likewise, for each $x \in
\CC$, let $S_{\sigma_x}$ be the group of automorphisms of $\sigma_x:
W_F \to PGL_2$ and $S^0_{\sigma_x}$ its connected component. We have
natural homomorphisms $S_\sigma \to S_{\sigma_x}$ and $S^0_\sigma \to
S^0_{\sigma_x}$, and hence a homomorphism
\begin{equation}    \label{S and S0}
S_\sigma/S^0_\sigma \to S_{\sigma_x}/S^0_{\sigma_x}.
\end{equation}
In our case, for generic $\sigma$ in the class that we are considering
here we have $S_\sigma = S_\sigma/S^0_\sigma = \Z_2$, generated (with
respect to the natural basis in the induced representation) by the
element
\begin{equation}    \label{glob aut}
\begin{pmatrix} 1 & 0 \\ 0 & -1 \end{pmatrix}
\end{equation}
of $PGL_2$ (this is the centralizer of $O_2 \subset PGL_2$).

Now consider the local groups. If $x$ is a split point of $\CC$ and
the ratio of the eigenvalues of $\sigma_x(\on{Fr}_x)$ is not equal to
$-1$, then $S_{\sigma_x}$ is a connected torus, and so
$S_{\sigma_x}/S^0_{\sigma_x}$ is trivial. Next, consider the case of
split points for which the ratio of the eigenvalues of
$\sigma_x(\on{Fr}_x)$ is equal to $-1$. Then the group $S_{\sigma_x}$
is the subgroup $O_2$ of $PGL_2$, which is the centralizer of
\eqref{glob aut} (see formula \eqref{at split}), and so
$S_\sigma/S^0_\sigma = \Z_2$. But the global automorphism \eqref{glob
aut} lands in the connected component $S^0_{\sigma_x}$ of
$S_{\sigma_x}$, so the homomorphism \eqref{S and S0} is trivial in
this case. Finally, if $x$ is non-split, we find that $S_{\sigma_x}$
is the centralizer of \eqref{trans}, which is also isomorphic to
$O_2$. However, now the element \eqref{glob aut} lands in the other
connected component, and so the homomorphism \eqref{S and S0} is
non-trivial in this case.

The idea of \cite{LL} is that irreducible representations of
$SL_2(F_x)$ from the local $L$-packet should be labeled by irreducible
representations of $S_{\sigma_x}/S^0_{\sigma_x}$. Thus, if this group
is trivial, there is only one irreducible representation in the
$L$-packet. If this group is isomorphic to $\Z_2$, then there are
two. According to our conventions, the representation with non-trivial
space of invariants of $K'_x$ will correspond to the trivial
representation of $\Z_2$, and the one with non-trivial invariants of
$K''_x$ will correspond to the sign representation of $\Z_2$. Now the
tensor product \eqref{tensor} gives rise to an irreducible
representation of the group $\prod_{x \in \CC}
S_{\sigma_x}/S^0_{\sigma_x}$, on which all but finitely many factors
act trivially.

\thmref{descr aut} may then be reformulated as saying that
\eqref{tensor} is automorphic if and only if $S_\sigma/S^0_\sigma$
acts trivially on the corresponding representation of $\prod_{x \in
\CC} S_{\sigma_x}/S^0_{\sigma_x}$, via the diagonal homomorphism
$$
S_\sigma/S^0_\sigma \to \prod_{x \in \CC} S_{\sigma_x}/S^0_{\sigma_x}.
$$

Note that, according to the above discussion, if $x$ is split, then
the homomorphism \eqref{S and S0} has trivial image, even if the group
$S_{\sigma_x}/S^0_{\sigma_x}$ is non-trivial.  Therefore we may choose
either of the two irreducible representations of $SL_2(F_x)$ from the
local $L$-packet associated to such a point as $\pi_x$, and in both
cases the corresponding representations \eqref{tensor} will
simultaneously be automorphic or not. In this sense, the split points
do not affect the automorphy of the representation \eqref{tensor},
unlike the non-split points, for which it is crucial which one of the
two members of the $L$-packet we choose as the local factor of
\eqref{tensor}.

Suppose now that $\# S$ is even and so the representation
\eqref{tensor} is automorphic. Then the one-dimensional vector space
\begin{equation}    \label{one dim}
\bigotimes_{x \in \CC} (\pi_x)^{K_x}
\end{equation}
may be realized in the space of functions on
\begin{equation}    \label{first component}
SL_2(F) \bs SL_2({\mathbb A}_F)/\prod_{x \in \CC} K_x.
\end{equation}
Moreover, any non-zero vector in \eqref{one dim} gives rise to a Hecke
eigenfunction $f$ on \eqref{first component} with the eigenvalues
prescribed by the conjugacy class $\sigma_x(\on{Fr}_x)$. This means
that it is an eigenfunction of the Hecke operator $T_{W,x}$
corresponding to the adjoint representation $W$ of $PGL_2$ and a point
$x \in \CC$, that is,
\begin{equation}    \label{hecke W}
T_{W,x} \cdot f = \on{Tr}(\sigma_x(\on{Fr}_x),W) f,
\end{equation}
where $\on{Fr}_x$ is the Frobenius conjugacy class corresponding to
$x$ in $W_F$. Here $T_{W,x}$ is a generator of the spherical Hecke algebra
of $K_x$ bi-invariant compactly supported functions on
$SL_2(F_x)$. For either choice of $K_x$ this algebra is canonically
isomorphic to $\Rep(PGL_2)$, and under this isomorphism $T_{W,x}$
corresponds to the class of the adjoint representation of $PGL_2$ (see
\secref{pos char} below for more details on the Hecke property).

Finally, suppose that $\# S$ is odd. Then the representation
\eqref{tensor} is not automorphic. Hence the one-dimensional vector
space \eqref{one dim} cannot possibly be realized in the space of
functions on \eqref{first component}. In other words, any function on
\eqref{first component} satisfying \eqref{hecke W} is necessarily
identically equal to zero.

\medskip

This should be contrasted with the generic situation, when the image
of $\sigma: W_F \to PGL_2$ has trivial centralizer (recall that we are
focusing here exclusively on the unramified homomorphisms
$\sigma$). In this case the group $S_\sigma/S^0_\sigma$ is trivial, so
all representations \eqref{tensor}, where the local factors $\pi_x$
are arbitrary representations from the local $L$-packets corresponding
to $\sigma_x$, will be automorphic. If $\sigma_x(\on{Fr}_x)$ is
generic, then $S_{\sigma_x}/S^0_{\sigma_x}$ is trivial, and the
corresponding $L$-packet contains one irreducible representation
$\pi_x$ of $SL_2(F_x)$. This $\pi_x$ has one-dimensional spaces of
invariants under {\em both} $K'_x$ and $K''_x$. It is also possible
that for some $x \in \CC$, $\sigma_x(\on{Fr}_x)$ is in the conjugacy
class of \eqref{tensor}. Then $S_{\sigma_x}/S^0_{\sigma_x}$ is equal
to $\Z_2$, and the $L$-packet contains two irreducible
representations, one of which has non-zero invariants with respect to
$K'_x$, and the other -- with respect to $K''_x$. However, inserting
either of them as the local factor of $\pi$ at $x$, we will obtain an
automorphic representation of $SL_2({\mathbb A}_F)$.

Thus, we find that for a generic $\sigma$, given any choices of $K_x$
(that is, $K_x=K'_x$ or $K_x=K''_x)$, the corresponding space of Hecke
eigenfunctions satisfying \eqref{first component} on \eqref{first
component} is one-dimensional.

\subsection{The Improper Hecke Operators}    \label{comparison}

In the geometric theory it was useful to consider, in addition to the
``proper'' Hecke operators corresponding to the three-dimensional
adjoint representation of $SO_3=PGL_2$, the ``improper'' ones
corresponding to the two-dimensional projective representation of this
group. These operators also have counterparts in the classical
theory. They are defined as follows. Given a point $x \in \CC$ and a
coordinate $t$ at $x$, the operator $\wt{T}_x$ is the integral
operator acting from functions on the double quotient \eqref{first
component} with $K_x = K'_x = SL_2[[t]]$ to functions on the same
double quotient, but with $K'_x$ replaced by the subgroup $K''_x$
from eqn. \eqref{another compact}. It is given by the formula
\begin{equation}    \label{improper hecke classical}
(\wt{T}_x \cdot f)(g) = \int_{M_x} f(gh) dh,
\end{equation}
where
$$
M_x = \begin{pmatrix}
t & 0 \\ 0 & 1
\end{pmatrix} SL_2[[t]] \begin{pmatrix}
t & 0 \\ 0 & 1
\end{pmatrix}.
$$

Now suppose that we have a Hecke eigenfunction $f'$ (resp., $f''$) on
\eqref{first component} with $K_x=K'_x$ (resp., $K_x=K''_x$), and with
$K_y, y \neq x$, being the same, and satisfying the (proper) Hecke
eigenfunction property \eqref{hecke W}. Each of these two functions is
unique up to a scalar. Suppose that we can lift $\sigma$ to an
unramified homomorphism $\wt\sigma: W_F \to GL_2$. Then we can
normalize the functions $f'$ and $f''$ in such a way that both are
equal to the restrictions to the appropriate double quotient
\eqref{first component} of a Hecke eigenfunction $\wt{f}$ for
$GL_2({\mathbb A}_F)$ corresponding to $\wt\sigma$ (see the Appendix
for more details). Since all Hecke operators for $GL_2$ commute with
each other, we find that the function $\wt{T}_x \cdot f'$ is a
function on \eqref{first component} with $K_x=K''_x$, which also
satisfies \eqref{hecke W}. Moreover, we have
\begin{equation}    \label{Tx prime}
\wt{T}_x \cdot f' = \on{Tr}(\wt\sigma_x(\on{Fr}_x),V) f'',
\end{equation}
where $V$ is the two-dimensional representation of $GL_2$. Thus,
$\wt{T}_x$ is an intertwining operator between the two spaces of Hecke
eigenfunctions for $SL_2({\mathbb A}_F)$ if and only if
the trace $\on{Tr}(\wt\sigma_x(\on{Fr}_x),V)$ is non-zero.

But this trace is equal to zero precisely when $\sigma_x(\on{Fr}_x)$
is the conjugacy class of \eqref{trans}, and this is the special case
when we have a non-trivial $L$-packet at $x$!  In this case $f'$ and
$f''$ correspond to two {\em non-isomorphic} representations of
$SL_2(F_x)$ (one of which could be automorphic and the other one
not). Formula \eqref{Tx prime} shows that in this case $\wt{T}_x \cdot
f' = 0$. Thus, the improper Hecke operators give us another way to
observe the non-triviality of the $L$-packets. They underscore the
discrepancy between two spaces of Hecke eigenfunctions on the double
quotients \eqref{first component} corresponding to $K_x=K'_x$ and
$K_x=K''_x$ in the case of the endoscopic $\sigma$: one of the two
spaces could be one-dimensional and the other equal to zero. An
analogue of this phenomenon may be observed geometrically, as we will
see in \secref{improp} below.

\section{From Hecke Eigensheaves To Hecke Eigenfunctions}
\label{from}

We now wish to replace Hecke eigenfunctions by Hecke eigensheaves,
geometric objects that allow us to link the classical Langlands
correspondence to the geometric Langlands correspondence and
ultimately to the mirror symmetry of the Hitchin fibrations for the
dual groups discussed in the previous sections.

\subsection{Hecke Eigensheaves In Positive Characteristic}
\label{pos char}

In our previous discussion of Hecke eigensheaves in \secref{hecke
sheaves}, we had assumed that our curve was defined over $\C$. Then a
Hecke eigensheaf corresponding to an $\LG$-local system $\E$ on $C$ is
a ${\mc D}$-module ${\mc F}$ on $\Bun_G$ together with the additional
data of isomorphisms $\al_V$ (see formula \eqref{al V}).\footnote{It
is expected that this ${\mc D}$-module is holonomic and has regular
singularities.} By using the Riemann--Hilbert correspondence, we may
then switch from ${\mc D}$-modules to perverse sheaves (this is
explained, e.g., in \cite{F:houches}, Section 3.4). Thus, Hecke
eigensheaves may be viewed as objects of the category of perverse
sheaves on $\Bun_G$, equipped with the isomorphisms \eqref{al V}.

Now we replace a complex curve by a curve $C$ defined over a finite
field $k = {\mathbb F}_q$. The notion of perverse sheaves in
characteristic 0 has an analogue for algebraic varieties (or algebraic
stacks) over a finite field (these are objects of the derived category
of $\ell$-adic sheaves \cite{BBD}). We have the moduli stack $\Bun_G$
of $G$-bundles on our curve $C$ defined over $k$. This is an
algebraic stack over $k$. Therefore we have the notion of a Hecke
eigensheaf on $\Bun_G$ corresponding to an unramified homomorphism
$\sigma: W_F \to \LG$.  Namely, we view $\sigma$ as an $\ell$-adic
$\LG$-local system $\E$ on $C$. In other words, for each
representation $V$ of $\LG$ the corresponding twist
$$
V_\E = \E \underset{\LG}\times V
$$
is a locally constant $\ell$-adic sheaf on $C$, and these sheaves are
compatible with respect to the tensor product structure on
representations of $\LG$. We also have Hecke functors $H_V, V \in
\Rep(\LG)$, defined in the same way as over $\C$.

A Hecke eigensheaf with ``eigenvalue'' $\E$ (or $\sigma$) is, by
definition, a perverse ($\ell$-adic) sheaf ${\mc F}$ on $\Bun_G$
together with the additional data of isomorphisms (compare with
\eqref{al V})
\begin{equation}    \label{al V1}
\alpha_V: H_V({\mc F}) \overset{\sim}\longrightarrow V_\E \boxtimes
{\mc F}.
\end{equation}
These isomorphisms should be compatible with the tensor product
structures and associativity on both sides. We will see below that to
ensure the passage from Hecke eigensheaves to Hecke eigenfunctions we
need to impose an additional {\em equivariance condition}.

To explain this passage in more detail, we recall that for any
algebraic variety (or algebraic stack) $Y$ over ${\mathbb F}_q$, we
may assign a function on the set of ${\mathbb F}_q$-points of $Y$ to
any $\ell$-adic sheaf (or a complex) ${\mc F}$ on $Y$ (see
\cite{De,Laumon:const}). Indeed, let $y$ be an $\Fq$-point of $Y$ and
$\ol{y}$ the $\ol{{\mathbb F}}_q$-point corresponding to an inclusion
$\Fq \hookrightarrow \ol{{\mathbb F}}_q$. Then the pull-back of $\F$
with respect to the composition $\ol{y} \to y \to Y$ is a
$(\ell$-adic) sheaf on a point $\on{Spec} \ol{{\mathbb F}}_q$. The
data of such a sheaf is the same as the data of a $\ol{\mathbb
Q}_\ell$-vector space, which we may think of as the stalk
$\F_{\ol{y}}$ of $\F$ at $\ol{y}$. There is an additional piece of
data on this vector space. Indeed, the Galois group
$\on{Gal}(\ol{{\mathbb F}}_q/\Fq)$ is the symmetry group of the
morphism $\ol{y} \to y$, and therefore it acts on $\F_{\ol{y}}$. In
particular, we have an action of the (geometric) Frobenius element
$\on{Fr}_{{y}}$, corresponding (the inverse of) the generator of the
Galois group of ${\mathbb F}_q$, acting as $x \mapsto x^q$. This
automorphism depends on the choice of the morphism $\ol{y} \to y$, but
its conjugacy class is independent of any choices. Thus, we obtain a
conjugacy class of automorphisms of the stalk $\F_{\ol{y}}$. Therefore
the trace of the geometric Frobenius automorphism is canonically
assigned to ${\mc F}$ and $y$. We will denote it by
$\on{Tr}(\on{Fr}_y,{\mc F})$.

More generally, if ${\mc F}$ is a complex of $\ell$-adic sheaves, we
take the alternating sum of the traces of $\on{Fr}_{{y}}$ on the
stalk cohomologies of ${\mc F}$ at $\ol{y}$. Hence we obtain a
function $\text{\tt f}_{\mc F,\Fq}$ on the set of ${\mathbb
F}_q$-points of $Y$, whose value at $y \in Y({\mathbb F}_q)$ is
$$
\text{\tt f}_{\mc F,\Fq}(y) = \sum_i (-1)^i
\on{Tr}(\on{Fr}_{y},H^i_{\ol{y}}({\mc F})).
$$
Similarly, for each $n>1$ we define a function $\text{\tt f}_{\mc
F,{\mathbb F}_{q^n}}$ on the set of ${\mathbb F}_{q^n}$-points of $Y$
by the formula
$$
\text{\tt f}_{\mc F,{\mathbb F}_{q^n}}(y) = \sum_i (-1)^i
\on{Tr}(\on{Fr}_{y},H^i_{\ol{y}}({\mc F})), \qquad y \in
Y({\mathbb F}_{q^n})
$$
(now $\on{Fr}_{{y}}$ corresponds to the automorphism $y \mapsto
y^{q^n}$).

The maps ${\mc F} \to \text{\tt f}_{\mc F,{\mathbb F}_{q^n}}$
intertwine the natural operations on complexes of sheaves with natural
operations on functions (see \cite{Laumon:const}, Sect. 1.2). For
example, pull-back of a sheaf corresponds to the pull-back of a
function, and push-forward of a sheaf with compact support corresponds
to the fiberwise integration of a function. This passage from sheaves
to functions is referred to as Grothendieck's {\em
faisceaux--fonctions dictionnaire}.

If $Y = \Bun_G$, then the set of ${\mathbb F}_q$-points of $Y$ is
naturally isomorphic to the double quotient
\begin{equation}    \label{double quot}
G(F) \bs G({\mathbb A}_F)/G({\mc O}_F),
\end{equation}
where
$$
{\mc O}_F = \prod_{x \in C} \OO_x
$$
(see, e.g., \cite{F:houches}, Section 3.2). Therefore any perverse
sheaf ${\mc F}$ on $\Bun_G$ gives rise to a function $\text{\tt
f}_{\mc F,\Fq}$ on the double quotient \eqref{double quot}. Suppose
now that $({\mc F},(\al_V))$ is a Hecke eigensheaf on
$\Bun_G$. Consider the corresponding function $\text{\tt f}_{\mc
F,\Fq}$ on the set $\Bun_G(\Fq)$, isomorphic to the double quotient
\eqref{double quot}, and its transform under the Hecke functor $H_V$,
restricted to
$$
(\CC \times \Bun_G)(\Fq) = \CC(\Fq) \times \Bun_G(\Fq).
$$
The action of the Hecke functor $H_V$ on sheaves becomes the action of
the corresponding Hecke operators $T_{V,x}$ on functions. Hence for
each $x \in \CC(\Fq)$ the left hand side of \eqref{al V1} gives rise
to the function $T_{V,x} \cdot \text{\tt f}_{\mc F,\Fq}$, whereas the
right hand side becomes $\on{Tr}(\on{Fr}_{x},V_{\E}) \text{\tt f}_{\mc
F,\Fq}$. Hence the isomorphism \eqref{al V1} implies that
\begin{equation}    \label{geom hecke}
T_{V,x} \cdot \text{\tt f}_{\mc F,\Fq} = \on{Tr}(\on{Fr}_{x},V_{\E})
\text{\tt f}_{\mc F,\Fq} = \on{Tr}(\sigma_x(\on{Fr}_x),V) \text{\tt
f}_{\mc F,\Fq}, \qquad \forall x \in C(\Fq)
\end{equation}
(see \cite{F:houches}, Section 3.8, for more details).

This is the sought-after Hecke eigenfunction property, but there is a
caveat: {\em a priori} this condition is satisfied only for the
$\Fq$-points of $\CC$. In contrast, an unramified Hecke eigenfunction
with respect to $\sigma$ is supposed to be an eigenfunction of the
Hecke operators for {\em all} closed points of $\CC$, with arbitrary
residue fields. To ensure that this property holds for the function
$\text{\tt f}_{\mc F,\Fq}$ at all points $x \in \CC$, we have to
impose an additional condition on the perverse sheaf ${\mc F}$;
namely, the $S_2$-equivariance of the iterated Hecke functor from
\cite{FGV}, Sect. 1.1. This will be discussed in the next section.

\subsection{Equivariance And Commutativity Conditions For Hecke
  Eigensheaves}    \label{comm cond}

Recall that the Hecke functor $H_V$ acts from the derived category
of sheaves on $\Bun_G$ to the derived category of sheaves on $\CC
\times \Bun_G$. Applying this functor again, we obtain the {\em
iterated Hecke functor} $H^{\boxtimes 2}_V$ from the derived
category of sheaves on $\Bun_G$ to the derived category of sheaves
on $\CC \times \CC \times \Bun_G$. A Hecke eigensheaf ${\mc F}$ with
``eigenvalue'' $\E$ is equipped with an isomorphism
$$
\al_V: H_V({\mc F}) \simeq V_\E \boxtimes {\mc F},
$$
which gives rise to an isomorphism
$$
\al_V^{\boxtimes 2}: H_V^{\boxtimes 2}({\mc F}) \simeq V_\E \boxtimes
V_\E \boxtimes {\mc F}.
$$
Away from the diagonal $\Delta \subset \CC \times \CC$ we have a natural
action of the symmetric group $S_2$ on both sides of this
isomorphism. The extra condition that we need to impose is that
$\al_V^{\boxtimes 2}$ {\em is an $S_2$-equivariant isomorphism}.

This condition implies that for the $m$th iterated Hecke functor
$H_V^{\boxtimes m}$ acting from from the derived category of sheaves
on $\Bun_G$ to the derived category of sheaves on $\CC^m \times \Bun_G$,
the isomorphism
$$
\al_V^{\boxtimes m}: H_V^{\boxtimes m}({\mc F}) \simeq
(V_\E)^{\boxtimes m} \boxtimes {\mc F}
$$
is $S_m$-equivariant outside the union $\Delta$ of pairwise diagonals
in $\CC^m$.

Suppose that the $S_2$-equivariance condition holds. Then ${\mc F}$ is
an eigensheaf with respect to the symmetrized Hecke functor
$H_V^{(m)}$ acting from the derived category of sheaves on $\Bun_G$ to
the derived category of sheaves on $(\on{Sym}^m \CC \bs \Delta) \times
\Bun_G$, that is, we have an isomorphism
$$
\al_V^{(m)}: H_V^{(m)}({\mc F}) \simeq V_\E^{(m)} \boxtimes {\mc F}
$$
on $(\on{Sym}^m \CC \bs \Delta) \times \Bun_G$, where
\begin{equation}    \label{symm}
V_\E^{(m)} = (p_*((V_\E)^{\boxtimes m}))^{S_m},
\end{equation}
and $p: \CC^m \to \on{Sym}^m \CC$ is the symmetrization map.

Now observe that any closed point $x$ of $\CC$ of degree $m$ gives rise
to an $\Fq$-point $D(x)$ in $\on{Sym}^m \CC$ (an effective divisor of
degree $m$). Moreover, it is easy to see that
$$
\on{Tr}(\on{Fr}_{x},V_\E) = \on{Tr}(\on{Fr}_{D(x)},V_\E^{(m)}).
$$
Therefore, restricting $\al_V^{(m)}({\mc F})$ to $D(x) \times \Bun_G$
and evaluating the traces of the Frobenius on $\Fq$-points there, we
find that formula \eqref{geom hecke} holds for all closed points $x$
of degree $m$. Thus, the $S_2$-equivariance condition guarantees that
the function $\text{\tt f}_{\mc F,\Fq}$ is truly a Hecke eigenfunction
on \eqref{double quot} with respect to the local system $\E$ (or
homomorphism $\sigma: W_F \to \LG$).

\medskip

The fact that the ``naive'' Hecke eigensheaf property \eqref{al V1}
does not by itself imply the Hecke eigenfunction property for those
closed points whose residue field is a non-trivial extension of $\Fq$,
the field of definition of our curve $\CC$, comes as a bit of a
surprise. However, the $S_2$-equivariance that is needed to ensure
that the Hecke eigenfunction property does hold everywhere is a very
natural condition. In fact, it is a special case of the following
general {\em commutativity condition} for the Hecke functors,
introduced in \cite{FGV}, Sect. 1.4.

For $V, W \in \Rep(\LG)$, let $H_V$ and $H_W$ be the corresponding
Hecke functors from the derived category of sheaves on $\Bun_G$ to the
derived category of sheaves on $\CC \times \Bun_G$. We then have the
iterated functors $H_V \circ H_W$ from the derived category of sheaves
on $\Bun_G$ to the derived category of sheaves on $\CC \times \CC
\times \Bun_G$. Given a Hecke eigensheaf $({\mc F},(\al_V))$, we have
isomorphisms
$$
\al_V \circ \al_W: (H_V \circ H_W)({\mc F}) \simeq \E_V \boxtimes \E_W
\boxtimes {\mc F}.
$$
On the other hand, over $\CC \times \CC \bs \Delta$ we have a natural
identification
$$
(H_V \circ H_W)({\mc F})|_{\CC \times \CC \bs \Delta} \simeq
\sigma^* \circ (H_W \circ H_V)({\mc F})|_{\CC \times \CC \bs \Delta},
$$
where $\sigma$ is the transposition on $\CC \times \CC \bs \Delta$.
The commutativity condition is that the diagram
$$
\begin{CD}
(H_V \circ H_W)({\mc F})|_{\CC \times \CC \bs \Delta} @>{\al_V \circ
    \al_W}>> \E_V \boxtimes \E_W \boxtimes {\mc F}|_{\CC \times \CC \bs
    \Delta} \\ @VVV @VVV \\ \sigma^* \circ (H_W \circ H_V)({\mc
    F})|_{\CC \times \CC \bs \Delta} @>{\sigma^*(\al_W \circ \al_V)}>>
    \sigma^*(\E_W \boxtimes \E_V) \boxtimes {\mc F}|_{\CC \times \CC \bs
    \Delta}
\end{CD}
$$
is commutative. If $V=W$, we obtain the above $S_2$-equivariance
condition.

To explain the meaning of this commutativity condition, let us recall
from \cite{MV} the geometric Satake equivalence between the category
of Hecke functors supported at a fixed point $x \in \CC$, which is the
category of equivariant perverse sheaves on the affine Grassmannian,
and the category $\Rep(\LG)$. This is an equivalence of tensor
categories, which means that in addition to being compatible with the
tensor products in both categories, it is also compatible with the
commutativity and associativity constraints. On the former category the
commutativity constraint is defined (following V. Drinfeld) as a
certain limit of the transposition of the Hecke functors defined at
distinct points of $\CC$, when the points coalesce.

The notion of (regular) Hecke eigensheaf may be viewed as a natural
generalization of the notion of a fiber functor from the category of
the Hecke functors supported at one point $x \in C$ to $\Rep(\LG)$,
when we allow the point $x$ to move along the curve. From this point
of view, asking that the isomorphisms $\al_V$ be compatible with the
tensor product structures and associativity is akin to asking for the
fiber functor to be a monoidal functor (that is, one compatible with
the tensor products and associativity constraint). But we know from
the Tannakian theory that this is not sufficient for establishing an
equivalence of a tensor category and the category of representations
of an algebraic group. For that we also need the fiber functor to be
compatible with the commutativity constraint. Since the commutativity
constraint on the category of Hecke functors supported at one point
appears as the limit the transposition of the two Hecke functors
supported at two different points of $\CC$, the commutativity
constraint itself appears as the limit of the above commutativity
condition when the two points coalesce.

Therefore we see that it is quite natural to require that a Hecke
eigensheaf satisfy the commutativity condition. An interesting fact
that we have observed above is that part of this condition (for $V=W$)
is also necessary for ensuring that, when working in positive
characteristic, the function associated to a Hecke eigensheaf is a
Hecke eigenfunction for all closed points of the curve.

\subsection{Back To $SL_2$}

As we have seen in the previous section, the geometric counterpart of
the double quotient \eqref{double quot} is the moduli stack $\Bun_G$
of $G$-bundles on $\CC$. In fact, \eqref{double quot} is the set of
$\Fq$-points of $\Bun_G$. Therefore Hecke eigensheaves on $\Bun_G$
give rise to Hecke eigenfunctions on \eqref{double quot}, as explained
above.

On the other hand, we have seen in \secref{inv vec} that in order to
understand properly the $L$-packets of (unramified) automorphic
representations for $G=SL_2$ we need to consider more general double
quotients \eqref{first component}, where $K_x = K'_x$ or $K''_x$, and
$K_x=K'_x=SL_2[[t]]$ for all but finite many closed points $x \in
\CC$. If all $K_x=SL_2[[t]]$, then \eqref{first component} is the set
of $\Fq$-points of $\Bun_{SL_2}$. What about the more general
quotients \eqref{first component}? The answer is clear: these are the
sets of $\Fq$-points of the ``improper'' versions of $\Bun_{SL_2}$;
namely, the moduli stacks $\Bun^{{\mc L}}_{SL_2}$ of rank two vector
bundles with the determinant being the line bundle ${\mc L} = {\mc
O}(D)$. Here $D$ is the set of points where $K_x=K''_x$, which we view
as an effective divisor on $\CC$.

We have already encountered these moduli stacks in the case of curves
over $\C$ in Sections \ref{secomp}, \ref{improper} and
\ref{improperly}. At that time we remarked that if ${\mc L} = {\mc L}'
\otimes {\mc N}^2$, where ${\mc N}$ is a line bundle on $\CC$, then we
may identify $\Bun^{{\mc L}}_{SL_2}$ with $\Bun^{{\mc L}'}_{SL_2}$ by
tensoring a rank two vector bundle with ${\mc N}$. Therefore
$\Bun^{{\mc L}}_{SL_2}$ really depends not on ${\mc L}$ but on its
image in the quotient of the Picard group $\on{Pic}(\C)$ (which is the
set of $\C$-points of the Picard scheme $\on{Pic}$ of $\CC$) by the
subgroup of squares. This quotient is isomorphic to $\Z_2$, and so
there is a unique improper component $\Bun^{\mc L}_{SL_2}$ in this
case (for which we may choose ${\mc L} = \OO(p)$, where $p$ is a point
of $\CC$), up to a non-canonical isomorphism.

Now consider a curve $\CC$ defined over $\Fq$. Here again the improper
stacks $\Bun^{\mc L}_{SL_2}$ are classified, up to an isomorphism, by
the quotient $\on{Pic}(\Fq)/\on{Pic}(\Fq)^2$. But now this quotient is
much bigger. To describe it more precisely, let us recall \cite{Serre}
that the abelian class field theory identifies $\on{Pic}(\Fq)$ with
the maximal unramified abelian quotient of the Weil group $W_F$ of the
function field $F$ of $\CC$. In other words, $\on{Pic}(\Fq)$ is
isomorphic to a dense subgroup of the Galois group of the maximal
unramified abelian extension $F^{\on{ab,un}}$ of $F$, defined in the
same way as the Weil group of $F$. Namely, it is the preimage of $\Z
\subset \wh\Z$ under the homomorphism $\on{Gal}(F^{\on{ab,un}}/F) \to
\on{Gal}(\ol{\mathbb F}_q/\Fq) = \wh\Z$. Therefore we obtain that
$\on{Pic}(\Fq)/\on{Pic}(\Fq)^2$ is the maximal quotient of
$\on{Gal}(F^{\on{ab,un}}/F)$ such that all of its elements have order
$2$. It is also the dual group of the group of unramified quadratic
extensions of $F$. Indeed, each such extension $E/F$ gives rise to a
quadratic character of $\on{Pic}(\Fq)$, which factors through
$\on{Pic}(\Fq)/\on{Pic}(\Fq)^2$.

Let us choose a representative ${\mc L} = {\mc O}(D)$ of this group,
where $D$ is a subset of the set of closed points of $\CC$. Then we
have the algebraic moduli stack $\Bun^{\mc L}_{SL_2}$ of rank two
vector bundles on $\CC$ with the determinant ${\mc L}$, whose set of
$\Fq$-points is the double quotient \eqref{first component} with the
above choice of subgroups $K_x$. In the geometric theory the notion of
Hecke eigenfunction on this set becomes that of (regular) Hecke
eigensheaf, defined in the same way as for the proper moduli stack
$\Bun_{SL_2}$ (corresponding to ${\mc L} = \OO$).

\subsection{From Curves Over $\C$ To Curves Over $\Fq$}

Let us go back to a curve $\CC$ over $\C$ and choose a $PGL_2$-local
system $\E$ whose structure group is reduced to $O_2 \subset PGL_2$,
but not to its proper subgroup. Since $\E$ comes from an irreducible
rank two local system, we expect that the category of regular Hecke
eigensheaves (in the sense of \secref{type 2}) with eigenvalue $\E$ on
$\Bun_{SL_2}$ has one irreducible object (up to an isomorphism). Let
${\mc F}$ be the underlying ${\mc D}$-module on $\Bun_{SL_2}$. In
\secref{spectral} we have discussed the $A$-brane ${\mc A}$
corresponding to ${\mc F}$, which is represented by a rank one unitary
local system on the singular Hitchin fiber, which has two irreducible
components. We have observed that ${\mc A}$ splits into two
$A$-branes, ${\mc A}_+$ and ${\mc A}_-$ supported on the two
irreducible components of the Hitchin fiber. Therefore we expect that
the ${\mc D}$-module ${\mc F}$ also splits into a direct sum,
\begin{equation}    \label{dir sum}
{\mc F} = {\mc F}_+ \oplus {\mc F}_-,
\end{equation}
of two irreducible ${\mc D}$-modules on $\Bun_{SL_2}$ corresponding to
the two ${\mc A}$-branes on the singular Hitchin fiber. Moreover,
since the $A$-branes $\A_\pm$ are fractional eigenbranes with respect
to the 't Hooft operators, we expect that the sheaves ${\mc F}_\pm$
satisfy the fractional Hecke property introduced in \secref{type 2}.

This leads us to postulate that {\em the same phenomenon should also
occur for curves over a finite field $\Fq$}. Namely, the regular Hecke
eigensheaf ${\mc F}$ corresponding to an $\ell$-adic local system $\E$
on a curve $\CC$ defined over $\Fq$, of the kind discussed above,
should also split as a direct sum \eqref{dir sum}. Moreover, these
sheaves should satisfy the fractional Hecke property introduced in
\secref{type 2} and hence give rise to a category of fractional Hecke
eigensheaves.  Next, in the setting of curves over finite fields we
can pass from $\ell$-adic perverse sheaves on $\Bun_{SL_2}$, to
functions. Thus, each of the sheaves ${\mc F}_\pm$ should give rise to
a function $f_\pm$ on the double quotient \eqref{double quot}, which
is the set of $\Fq$-points of $\Bun_{SL_2}$. The fractional Hecke
property of the sheaves ${\mc F}_\pm$ translates into a certain
property of the corresponding functions $f_\pm$.

Thus, we started with $A$-branes and ended up with automorphic
functions satisfying the fractional Hecke property. Schematically,
this passage looks as follows:

$$
\boxed{\text{$A$-branes}} \; \overset{\on{over} \C}\Longrightarrow \;
\boxed{\D\text{-modules}}  \; \overset{\on{over} \C}\Longrightarrow \;
\boxed{\text{perverse sheaves}} \; \overset{\on{over}
\Fq}\Longrightarrow \; \boxed{\text{functions}}
$$

\bigskip

We will see below that the fractional Hecke property means in
particular that not only $f_+ + f_-$ is a Hecke eigenfunction, in the
ordinary sense, but $f_+ - f_-$ is a Hecke eigenfunction as well, but
with respect to a {\em different} homomorphism $\sigma': W_F \to
PGL_2$. We will show that $\sigma'$ really exists, and is in fact
canonically attached to the original homomorphism $\sigma$. This
will provide the first consistency check for our predictions.

\subsection{Fractional Hecke Property}

Let $\CC$ be a curve over $\Fq$ and $\E$ an endoscopic $\ell$-adic
$PGL_2$-local system on $\CC$ (corresponding to an unramified
homomorphism $\sigma: W_F \to PGL_2$). This means that its structure
group is reduced to $O_2$, but not to a proper subgroup. Then the
group of automorphisms of $\E$ (equivalently, the centralizer of the
image of $\sigma$) is $\Z_2$. Let $D$ be a finite set of closed points
of $\CC$. Denote by ${\mc F}^D$ a regular Hecke eigensheaf on
$\Bun^{\OO(D)}_{SL_2}$ with the ``eigenvalue'' $\E$ (in the sense of
\secref{type 2}). Motivated by our results on $A$-branes in the
analogous situation for curves over $\C$, we conjecture that ${\mc
F}^D$ splits as a direct sum
\begin{equation}    \label{dir sum1}
{\mc F}^D = {\mc F}^D_+ \oplus {\mc F}^D_-
\end{equation}
of perverse sheaves ${\mc F}^D_\pm$ which satisfy the following
fractional Hecke property with respect to $\E$, introduced in
\secref{type 2} (and so we also call them the {\em fractional Hecke
eigensheaves}):
\begin{align}    \label{FS1}
\al_{+}: H_W({\mc F}^D_+) &\overset{\sim}\longrightarrow (\det U_\E
\boxtimes {\mc F}^D_+) \oplus (U_\E \boxtimes {\mc F}^D_-), \\
\label{FS2} \al_-: H_W({\mc F}^D_-) &\overset{\sim}\longrightarrow
(U_\E \boxtimes {\mc F}^D_+) \oplus (\det U_\E \boxtimes {\mc F}^D_-).
\end{align}

Here $W$ is the adjoint representation of $PGL_2$ and we use the
decomposition of the rank three local system $W_\E$ on $\CC$ with
respect to the action of its group $\Z_2$ of automorphisms as in
formula \eqref{WE},
\begin{equation}    \label{decomp W}
W_\E = \left( \det U_\E \otimes I \right) \oplus \left( U_\E \otimes S
\right),
\end{equation}
where $I$ and $S$ are the trivial and sign representations of $\Z_2$,
respectively, and $\det U_\E$ and $U_\E$ are the rank one and two
local systems on $\CC$ defined as follows. Recall that by our
assumption the $PGL_2$-local system $\E$ is reduced to $O_2$, so we
view it as an $O_2$-local system. We then set
$$
U_\E = \E \underset{O_2}\times U,
$$
where, as before, $U$ is the defining two-dimensional representation
of $O_2$.

\subsection{Fractional Hecke Eigenfunctions}    \label{splitting}

We now analyze the implications of formulas \eqref{FS1} and
\eqref{FS2} for the functions associated to ${\mc F}^D_\pm$,
$$
f^D_\pm = \text{\tt f}_{{\mc F}^D_\pm,\Fq},
$$
on the set $\Bun^{\OO(D)}_{SL_2}(\Fq)$, which is isomorphic to the
double quotient \eqref{first component}. Formula \eqref{dir sum1}
implies that
\begin{equation}    \label{sum}
f^D = f^D_+ + f^D_-,
\end{equation}
where $f^D = \text{\tt f}_{{\mc F}^D,\Fq}$ is the function on
$\Bun^{\OO(D)}_{SL_2}(\Fq)$ associated to the regular Hecke eigensheaf
${\mc F}^D$.

To simplify our notation, in what follows, when no ambiguity arises,
we will suppress the upper index $D$.

By restricting the Hecke correspondence to $x \times
\Bun^{\OO(D)}_{SL_2}$, where $x \in \CC(\Fq)$ and evaluating the trace
of the Frobenius at the $\Fq$-points there, we obtain from formulas
\eqref{FS1} and \eqref{FS2} that the functions $f_\pm$ satisfy the
following property:
\begin{equation}    \label{mod hecke fun}
T_{W,x} \cdot \begin{pmatrix} f_+ \\ f_- \end{pmatrix} =
\begin{pmatrix} a_x & b_x \\ b_x & a_x \end{pmatrix} \begin{pmatrix}
  f_+ \\ f_- \end{pmatrix},  \qquad  x \in \CC(\Fq),
\end{equation}
where
\begin{align*}
a_x &= \on{Tr}(\on{Fr}_x,\det U_\E) = \det(\sigma_x(\on{Fr}_x),U), \\
b_x &= \on{Tr}(\on{Fr}_x,U_\E) = \on{Tr}(\sigma_x(\on{Fr}_x),U).
\end{align*}
Here we view $\sigma_x$ as a homomorphism $W_{F_x} \to O_2$. To
compute these numbers, we recall the description of the Frobenius
conjugacy classes from \secref{L-packets}. We find that the conjugacy
class of $\sigma_x(\on{Fr}_x)$ in $O_2$ contains the matrix
$$
\begin{pmatrix}
  \frac{\mu(\on{Fr}_{y_1(x)})}{\mu(\on{Fr}_{y_2(x)})} & 0 \\
  0 & \frac{\mu(\on{Fr}_{y_2(x)})}{\mu(\on{Fr}_{y_1(x)})}
\end{pmatrix},
$$
if $x$ is split, and the matrix
$$
\begin{pmatrix} 1 & 0 \\ 0 & -1 \end{pmatrix},
$$
if $x$ is non-split. Therefore we find that
\begin{align}    \label{ax}
a_x &= \begin{cases} 1, & \on{ if \;} x
  \on{\; is \; split}, \\ -1, & \on{ if \;} x \on{\; is \; non-split},
  \end{cases} \\ \label{bx} b_x &=
  \begin{cases}   \dfrac{\mu(\on{Fr}_{y_1(x)})}{\mu(\on{Fr}_{y_2(x)})} +
  \dfrac{\mu(\on{Fr}_{y_2(x)})}{\mu(\on{Fr}_{y_1(x)})}, & \on{ if \;}
  x \on{\; is \; split}, \\ 0, & \on{ if \;} x \on{\; is \;
  non-split},
\end{cases}
\end{align}

Formula \eqref{mod hecke fun} implies that the sum $f_+ + f_-$
is an eigenfunction of the Hecke operators
\begin{equation}    \label{hecke f+}
T_{W,x} \cdot (f_+ + f_-) = (a_x+b_x) (f_+ + f_-), \qquad x \in
\CC(\Fq),
\end{equation}
where
\begin{align} \notag
a_x+b_x = \on{Tr}(\on{Fr}_x,W_\E) &= \on{Tr}(\sigma_x(\on{Fr}_x),W)
\\ \label{al+} &= \begin{cases} 1 +
  \dfrac{\mu(\on{Fr}_{y_1(x)})}{\mu(\on{Fr}_{y_2(x)})} +
  \dfrac{\mu(\on{Fr}_{y_2(x)})}{\mu(\on{Fr}_{y_1(x)})}, & \on{ if \;}
  x \on{\; is \; split}, \\ -1, & \on{ if \;} x \on{\; is \;
  non-split}. \end{cases}
\end{align}
Thus, formula \eqref{hecke f+} expresses the usual Hecke property of
the function $f = f_+ + f_-$ associated to the sheaf ${\mc F} = {\mc
F}_+ \oplus {\mc F}_-$ with respect to $\E$ (or $\sigma$).

But we also find that the difference $f_+ - f_-$ is a Hecke
eigenfunction with a different set of eigenvalues; namely,
\begin{equation}    \label{hecke f-}
T_{W,x} \cdot (f_+ - f_-) = (a_x-b_x) (f_+ - f_-), \qquad x \in
\CC(\Fq),
\end{equation}
where
\begin{equation}    \label{al-}
a_x-b_x = \begin{cases} 1 -
  \dfrac{\mu(\on{Fr}_{y_1(x)})}{\mu(\on{Fr}_{y_2(x)})} -
  \dfrac{\mu(\on{Fr}_{y_2(x)})}{\mu(\on{Fr}_{y_1(x)})}, & \on{ if \;}
  x \on{\; is \; split}, \\  -1, &
  \on{ if \;} x \on{\; is \; non-split}, \end{cases}.
\end{equation}
However, we have to remember that the Hecke property \eqref{hecke f-}
holds only for $\Fq$-points of $\CC$. Indeed, we have started with the
geometric Hecke property \eqref{FS1}--\eqref{FS2}. The Hecke
correspondence relates $\Bun^{\OO(D)}_{SL_2}$ and $\CC \times
\Bun^{\OO(D)}_{SL_2}$. The functions $f_\pm$ are obtained from ${\mc
F}_\pm$ by taking the trace of the Frobenius at the $\Fq$-points of
$\Bun^{\OO(D)}_{SL_2}$. To obtain a Hecke property for them, we need
to consider the Hecke correspondence on the sets of $\Fq$-points of
these two stacks. The only Hecke operators we can reach this way are
those corresponding to the $\Fq$-points of $\CC$. The resulting action
of the Hecke operators is expressed by equation \eqref{mod hecke fun}.

This formula does not uniquely determine the function $f_+ - f_-$. It
would be uniquely determined (at least for generic $\sigma$'s of the
type we are considering, which correspond to irreducible two-dimensional
representations of $W_F$) only if it were a Hecke eigenfunction for
{\em all} closed points of $\CC$, not just its $\Fq$-points.

In order to incorporate closed points of $\CC$ with the residue field
${\mathbb F}_{q^m}, m>1$, we need to consider more general Hecke
correspondence $H_W^{(m)}$ over the $m$th symmetric power of $\CC$
(with the union $\Delta$ of pairwise diagonals removed). This requires
an additional $S_2$-equivariance condition on the isomorphisms
$\al_{\pm}$, similar to the one in the case of regular Hecke
eigensheaves. This condition is defined in exactly the same way as in
\secref{comm cond}.

Assuming that this $S_2$-equivariance condition holds, we obtain
isomorphisms
\begin{multline*}
\al_+^{(2)}: H_W^{(2)}({\mc F}_+) \overset{\sim}\longrightarrow \\
((\det U_\E)^{(2)} \oplus U_\E^{(2)})|_{\on{Sym}^2 \CC \bs \Delta}
\boxtimes {\mc F}_+ \oplus \on{Sym}(\det U_\E \boxtimes U_\E \oplus
U_\E \boxtimes \det U_\E)|_{\on{Sym}^2 \CC \bs \Delta} \boxtimes {\mc
F}_-,
\end{multline*}
\begin{multline*}
\al_-^{(2)}: H_W^{(2)}({\mc F}_-) \overset{\sim}\longrightarrow \\
\on{Sym}(\det U_\E \boxtimes U_\E \oplus U_\E \boxtimes \det
U_\E)|_{\on{Sym}^2 \CC \bs \Delta} \boxtimes {\mc F}_+ \oplus ((\det
U_\E)^{(2)} \oplus U_\E^{(2)})|_{\on{Sym}^2 \CC \bs \Delta} \boxtimes
{\mc F}_-.
\end{multline*}
Here we use notation \eqref{symm}.

Suppose that this condition is satisfied. Recall from \secref{comm
cond} that any closed point $x$ of $\CC$ such that $\deg(x)=m$ (for
the definition of $\deg(x)$, see \secref{recoll}) gives rise to an
$\Fq$-point $D(x)$ in $\on{Sym}^m \CC$ (an effective divisor of degree
$m$), and we have
$$
\on{Tr}(\on{Fr}_{D(x)},\E^{(m)}) = \on{Tr}(\on{Fr}_{x},\E)
$$
for any local system $\E$ on $\CC$. In particular, suppose that $x$ is
a closed point of $\CC$ of degree $2$, that is, with the residue field
isomorphic to ${\mathbb F}_{q^2}$, and let $D(x)$ be the corresponding
$\Fq$-point of $\on{Sym}^2 \CC$. Then we have
$$
\on{Tr}(\on{Fr}_{D(x)},\det U_\E^{(2)} \oplus U_\E^{(2)}) =
\on{Tr}(\on{Fr}_{D(x)},W_\E^{(2)}) = \on{Tr}(\on{Fr}_{x},W_\E) =
a_x+b_x,
$$
but
$$
\on{Tr}(\on{Fr}_{D(x)},\on{Sym}(\det U_\E \boxtimes U_\E \oplus U_\E
\boxtimes \det U_\E)) = 0.
$$
Therefore the isomorphisms $\al_\pm^{(2)}$ imply that {\em both} $f_+$
and $f_-$ are Hecke eigenfunctions (in the ordinary sense) with the
eigenvalue $(a_x+b_x)$, and so we have
\begin{equation}    \label{even}
T_{W,x} \cdot (f_+ - f_-) = (a_x+b_x)(f_+ - f_-), \qquad x \in
\CC, \quad \deg(x) = 2.
\end{equation}

Next, we analyze in the same way what happens at the closed points of
$\CC$ of arbitrary degree $m$. The $S_2$-invariance condition implies
the $S_m$-invariance condition, as in \secref{comm cond}. We then find
that for odd $m$ the function $f_+ - f_-$ satisfies formula
\eqref{hecke f-}, and for even $m$ it satisfies formula
\eqref{even}. Thus, we have
\begin{equation}    \label{f- final}
T_{W,x} \cdot (f_+ - f_-) = (a_x + (-1)^m b_x)(f_+ - f_-), \qquad x \in
\CC, \quad \deg(x) = m.
\end{equation}

According to the Langlands correspondence for $GL_2$
\cite{Dr,Drinfeld}, a formula like this may only hold if the
eigenvalues of $T_{W,x}$, the numbers $a_x + (-1)^m b_x$, are equal to
$\on{Tr}(\sigma'_x(\on{Fr}_x,W)$ for some homomorphism $\sigma': W_F
\to PGL_2$.  This gives us an opportunity to test our prediction that
there exists a decomposition \eqref{dir sum1}.

In fact, it is easy to construct a homomorphism $\sigma'$ with this
property. Observe that any homomorphism $\sigma: W_F \to O_2$ may be
twisted by a quadratic character $\rho: W_F \to \Z_2 = \{ \pm 1 \}$,
where the group $\Z_2$ is identified with the center of $O_2$. We
denote this operation by $\sigma \mapsto \sigma \otimes \rho$.

In particular, the quadratic extension of the scalars ${\mathbb
F}_{q^2}/\Fq$ defines a quadratic extension ${\mathbb
F}_{q^2}(\CC)/\Fq(\CC)$ (recall that $F = \Fq(\CC)$) and hence a
quadratic character of $W_F$, which we will denote by $\nu$. This
character is determined by the following property:
\begin{equation}    \label{nu}
\nu(\on{Fr}_x) = (-1)^m, \qquad x \in \CC, \quad \deg(x) = m.
\end{equation}
Let $\sigma' = \sigma \otimes \nu$. Then we find that
$$
\on{Tr}(\sigma'_x(\on{Fr}_x),W) = \begin{cases} 1 + (-1)^m
  \dfrac{\mu(\on{Fr}_{y_1(x)})}{\mu(\on{Fr}_{y_2(x)})} + (-1)^m
  \dfrac{\mu(\on{Fr}_{y_2(x)})}{\mu(\on{Fr}_{y_1(x)})}, & \on{ if \;}
  x \on{\; is \; split}, \\ -1, & \on{ if \;} x \on{\; is \;
  non-split}, \end{cases}
$$
for all $x \in \CC(\Fq)$. Hence
$$
\on{Tr}(\sigma'_x(\on{Fr}_x),W) = a_x + (-1)^m b_x.
$$
Therefore we obtain that a Hecke eigenfunction with the eigenvalues
$a_x + (-1)^m b_x$, as in formula \eqref{f- final}, does exist! Let us
denote this function by $f'$.

Using this function, we can now solve for $f_+$ and $f_-$:
\begin{equation}    \label{solve}
f_\pm = \frac{1}{2} (f \pm f').
\end{equation}
These are the functions corresponding to the fractional Hecke
eigensheaves ${\mc F}_\pm$ whose existence we have
conjectured.\footnote{After this paper appeared on the arXiv, we
learned from S. Lysenko that the sheaves ${\mc F}_\pm$ may be
constructed using his results on theta-lifting \cite{Ly1,Ly2}.}

It is natural to ask: what is the representation theoretical meaning
of the functions $f_\pm$ and the equations that they satisfy? These
equations are given by formula \eqref{mod hecke fun} for $x \in
\CC$ of odd degree $m$, and
\begin{equation}    \label{mod hecke fun1}
T_{W,x} \cdot f_\pm = (a_x + b_x)f_\pm,
\end{equation}
for $x \in \CC$ of even degree $m$.

Recall that the Hecke eienfunctions, such as $f = f_+ + f_-$ and $f' =
f_+ - f_-$, may be interpreted as matrix coefficients of automorphic
representations of $SL_2({\mathbb A}_F)$, and their (regular) Hecke
property is the result of the Satake isomorphism identifying the
spherical Hecke algebra with the representation ring of the Langlands
dual group $PGL_2$. It would be interesting to find a similar
interpretation of the {\em fractional Hecke eigenfunctions} $f_\pm$
and equations \eqref{mod hecke fun}, \eqref{mod hecke fun1}.

\subsection{The Improper Hecke Functors}    \label{improp}

In addition to the ``proper'' Hecke functors $H_W$ acting on the
categories of ${\mc D}$-modules on $\Bun_{SL_2}^{{\mc O}(D)}$, there
are also ``improper'' Hecke functors $\wt{H}_x$ acting from the
category of ${\mc D}$-modules on $\Bun_{SL_2}^{{\mc O}(D)}$ to the
category of ${\mc D}$-modules on $\Bun_{SL_2}^{{\mc O}(D+x)}$. They
are defined via the Hecke correspondence between the two moduli stacks
consisting of pairs of rank two bundles ${\mc M} \in \Bun_{SL_2}^{{\mc
O}(D)}$ and ${\mc M}' \in \Bun_{SL_2}^{{\mc O}(D+x)}$ such that ${\mc
M} \subset {\mc M}'$ as a coherent sheaf. These functors are
categorical analogues of the improper Hecke operators $\wt{T}_x$
introduced in \secref{comparison}. The corresponding operators on
$A$-branes are the improper 't Hooft operators discussed in
\secref{Hecke}.

In formula \eqref{schema} we have computed the action of the improper
't Hooft operators on the branes $\A_\pm$. Based in this formula, we
conjecture that the improper Hecke operators should act on the
fractional Hecke eigensheaves ${\mc F}_\pm^D$ as follows:
\begin{align*}
\wt{H}_x({\mc F}^D_+) &\simeq {\mc F}^{D+x}_+ \oplus {\mc F}^{D+x}_-,
\\ \wt{H}_x({\mc F}^D_-) &\simeq {\mc F}^{D+x}_+ \oplus {\mc
F}^{D+x}_-.
\end{align*}
This should hold for all points $x \in \CC$ if $\CC$ is defined over
$\C$, and all split points, if $\CC$ is defined over $\Fq$.

This formula indicates that the improper Hecke functor fails to
establish an equivalence between the categories of fractional Hecke
eigensheaves on $\Bun_{SL_2}^{{\mc O}(D)}$ and $\Bun^{{\mc
O}(D+x)}_{SL_2}$ for the endoscopic local systems. This may be viewed
as a geometric counterpart of the vanishing of the improper Hecke
operator acting on functions discussed in \secref{comparison}, which,
as we have seen, is closely related to the structure of the global
$L$-packets of automorphic representations associated to endoscopic
$\sigma: W_F \to PGL_2$.

It would be more difficult to see an analogue of the phenomenon of
$L$-packets in the framework of the categories of regular Hecke
eigensheaves. Indeed, for a regular Hecke eigensheaf ${\mc F}^D = {\mc
F}_+^D \oplus {\mc F}_-^D$ we have
\begin{equation}    \label{multipl}
\wt{H}_x({\mc F}^D) \simeq V \otimes {\mc F}^{D+x},
\end{equation}
where $V$ is a two-dimensional vector space. In the classical setting
explained in \secref{comparison} $V$ is replaced by the trace
$\on{Tr}(\wt\sigma_x(\on{Fr}_x),V)$. It vanishes precisely when
$\sigma_x(\on{Fr}_x)$ is the conjugacy class of \eqref{trans}, and
this vanishing manifests the non-trivial structure of the $L$-packets
in this case. In contrast, in the geometric setting, $V$ itself
appears as the multiplier in formula \eqref{multipl}. It is not clear
what should replace the vanishing of the trace in this context.

However, in the framework of the categories of fractional Hecke
eigensheaves the picture is more transparent. Now we have a category
with two irreducible objects, ${\mc F}_+^D$ and ${\mc F}_-^D$. The
functor $\wt{H}_x$ sends both of them to ${\mc F}^{D+x} = {\mc
F}^{D+x}_+ \oplus {\mc F}^{D+x}_-$, and hence does not set up an
equivalence between the categories corresponding to $D$ and $D+x$.

\subsection{$L$-packets Associated To $\sigma$ And $\sigma'$}
\label{rel pac}

The above discussion shows that the representations $\pi =
\bigotimes'_x \pi_x$ of $SL_2({\mathbb A})$ corresponding to $\sigma:
W_F \to O_2 \subset PGL_2$ and $\sigma' = \sigma \otimes \nu$, where
$\nu$ is given by formula \eqref{nu}, are linked together. Let us
compare the $L$-packets (or, equivalently, the multiplicities of the
representations $\pi$) corresponding to $\sigma$ and $\sigma'$.

Recall that for each finite subset $D \subset \CC$ we have the space
of $\prod_x K_x$-invariant vectors in $\pi$, where $K_x = K''_x$, if
$x \in D$, and $K_x = K'_x$, otherwise. If $\pi$ is automorphic, then
this space of invariants is realized in the space of Hecke
eigenfunctions on the double quotient \eqref{first component}, which
is $\Bun^{\OO(D)}_{SL_2}(\Fq)$. As explained in \secref{inv vec}, this
space is either one- or zero-dimensional. In the former case, there is
a unique Hecke eigenfunction (up to a scalar), which we denote by
$f^D_\sigma$. In the latter case, any Hecke eigenfunction that we
construct has to vanish. A criterion as to whether it is one- or
zero-dimensional is given in \thmref{descr aut} (following \cite{LL}),
and it amounts to a description of the global $L$-packets.

This criterion is as follows. Consider the case when we can lift
$\sigma$ to $\wt\sigma: W_F \to GL_2$ and represent $\wt\sigma$ as
$\on{Ind}_{W_E}^{W_F} \mu$, where $E$ is the field of functions on a
double covering $\CC' \to \CC$, and $\mu$ is a character of
$W_E$. Then let $S \subset D$ be the set of points in $D$ which are
non-split in $E$. The dimension of the space of Hecke eigenfunctions
on $\Bun^{\OO(D)}_{SL_2}(\Fq)$ with respect to $\sigma$ is then equal
to $1$, if $\# S$ is even, and to $0$, if $\# S$ is odd.

Note that the extension $E/F$ corresponds to the quadratic character
of $W_F$ obtained via the composition $\kappa: W_F
\overset{\sigma}\longrightarrow O_2 \to \Z_2$. We will call $E$ the
field {\em affiliated with} $\sigma$.

\begin{lem}
There exists a quadratic character $\phi: W_E \to \Z_2$, where $E$ is
affiliated with $\sigma$, such that $\sigma' = \sigma \otimes \nu$ is
equivalent to the projectivization of the two-dimensional
representation $\on{Ind}_{W_E}^{W_F} (\mu \otimes \phi)$ of $W_F$.
\end{lem}

\begin{proof}
Let us choose $\wt\tau \in W_F$ which projects onto the non-trivial
element $\tau$ of $W_F/W_E = \Z_2$. By Cebotarev's theorem, the
condition stated in the lemma is equivalent to saying that $\phi$
satisfies
$$
\frac{\phi(\wt{\tau} h \wt\tau^{-1})}{\phi(h)} = \nu(h), \qquad h \in
W_E.
$$
The existence of such $\phi$ may be derived from the Hasse--Minkowski
theorem, as was explained to us by B. Poonen. We will not present his
argument here, as this would take us too far afield.
\end{proof}

Since the same quadratic extension is affiliated with both $\sigma$
and $\sigma'$, we find that the criterion for the dimensionality of
the space of Hecke eigenfunctions on $\Bun^{\OO(D)}_{SL_2}(\Fq)$ with
respect to $\sigma$ and $\sigma'$ is {\em the same}. Therefore the
Hecke eigenfunctions $f^D_{\sigma}$ and $f^D_{\sigma'}$ have to vanish
simultaneously when $\# S$ is odd, where $S \subset D$ is the subset
of points of $D$ which are non-split in $E$.

Now recall that the functions $f^D_{\sigma,\pm}$, corresponding to the
sheaves ${\mc F}^D_{\sigma,\pm}$ on $\Bun^{\OO(D)}_{SL_2}$, are given
by formula \eqref{solve},
\begin{equation}    \label{formula pm}
f^D_{\sigma,\pm} = \frac{1}{2}(f^D_{\sigma} \pm f^D_{\sigma'}).
\end{equation}
Hence we find that both of these functions have to vanish if $\#
S$ is odd.

Therefore we obtain that for odd $\# S$ the fractional Hecke
eigensheaves ${\mc F}^D_{\sigma,\pm}$ on $\Bun^{\OO(D)}_{SL_2}$ are
such that the corresponding functions $f^D_{\sigma,\pm}$ on the set of
$\Fq$-points of $\Bun^{\OO(D)}_{SL_2}$ are identically equal to
$0$. However, this does not mean that the sheaves themselves are equal
to $0$. This only means that the traces of the Frobenius on the stalks
of ${\mc F}^D_{\sigma,\pm}$ at the $\Fq$-points of
$\Bun^{\OO(D)}_{SL_2}$ are equal to $0$. But this does not mean that
the traces of the Frobenius on the stalks at ${\mathbb
F}_{q^n}$-points are equal to $0$ for $n>1$ (which would have implied
that the sheaves are identically zero, see \cite{Laumon:const}). In
fact, it is easy to see that the latter are non-zero for general $n$.

Before explaining this, we note a general fact about compatibility
of Hecke eigensheaves with base change. For each $n>1$ we have
the curve
$$
\CC_n = \CC \underset{\on{Spec} \Fq}\times \on{Spec} {\mathbb F}_{q^n}
$$
over ${\mathbb F}_{q^n}$. The moduli stack $\Bun_{G,\CC_n}$ of
$G$-bundles on $\CC_n$ is equivalent to the base change of
the moduli stack $\Bun_G = \Bun_{G,\CC}$ of $G$-bundles on $\CC$,
$$
\Bun_{G,\CC_n} = \Bun_G \underset{\on{Spec} \Fq}\times \on{Spec}
{\mathbb F}_{q^n}.
$$
Let $\E_n$ be the pull-back of the $^L\neg G$-local system $\E$ on $\CC$
to $\CC_n$. The geometric Langlands correspondence is compatible with
base change, in the sense that the pull-back ${\mc F}_n$ of a Hecke
eigensheaf ${\mc F}$ with eigenvalue $\E$ from $\Bun_{G,\CC}$ to
$\Bun_{G,\CC_n}$ is a Hecke eigensheaf with the eigenvalue $\E_n$ (see
\cite{Laumon:duke}).

Let us consider now the traces of the Frobenius on the stalks of our
sheaves ${\mc F}^D_{\sigma,\pm}$ at ${\mathbb F}_{q^m}$-points of
$\Bun^{\OO(D)}_{SL_2,\CC}$ with $m>1$. As an example, let us look at
the set of ${\mathbb F}_{q^2}$-points of $\Bun^{\OO(D)}_{SL_2,\CC}$,
which is is the same as the set of ${\mathbb F}_{q^2}$-points of the
moduli stack $\Bun^{\OO(D)}_{SL_2,\CC_2}$ of $SL_2$-bundles on
$\CC_2$. The pull-back of a Hecke eigensheaf ${\mc F}^D_{\sigma}$ to
$\Bun^{\OO(D)}_{SL_2,\CC_2}$ is a Hecke eigensheaf ${\mc
F}^D_{\sigma_2}$, where $\sigma_2$ is the restriction of $\sigma$ to
$W_{F_2}$, with $F_2 = {\mathbb F}_{q^2}(\CC)$. Therefore the
pull-back of ${\mc F}^D_{\sigma,\pm}$ is ${\mc
F}^D_{\sigma_2,\pm}$. Suppose, for example, that $D = [y]$, where $y$
is an $\Fq$-point of $\C$ which is non-split in the quadratic
extension $E$ of $F$ used in defining $\sigma$. Then the set $S$
appearing in the statement of \thmref{descr aut} consists of a single
point $y$, and according to this theorem, the functions
$f^{[y]}_{\sigma,\pm}$ on the set $\Bun^{\OO(y)}_{SL_2,\CC}(\Fq)$
associated to ${\mc F}^D_{\sigma,\pm}$ have to vanish.

However, the ${\mathbb F}_{q^2}$-point of $\CC_2$ corresponding to $y$
(which we will also denote by $y$) is split in the corresponding
quadratic extension $E_2$ of $F_2$. Therefore, by \thmref{descr aut},
the functions $f^D_{\sigma_2,\pm}$ on
$\Bun^{\OO(y)}_{SL_2,\CC_2}({\mathbb F}_{q^2}) =
\Bun^{\OO(y)}_{SL_2,\CC}({\mathbb F}_{q^2})$ are non-zero. According
to the above compatibility property with base change, these functions
coincide with the functions on $\Bun^{\OO(D)}_{SL_2,\CC}({\mathbb
F}_{q^2})$ corresponding to the sheaves ${\mc
F}^{[y]}_{\sigma,\pm}$. Hence the sheaves themselves are non-zero!

This elementary example shows that even if the functions on the set of
$\Fq$-points of $\Bun^{\OO(D)}_{SL_2,\CC}$ associated to the sheaves
${\mc F}^D_{\sigma,\pm}$ are equal to zero, the corresponding
functions on the sets of ${\mathbb F}_{q^m}$-points with $m>1$, are
not all equal to zero simultaneously, and hence the sheaves ${\mc
F}^D_{\sigma,\pm}$ are non-zero.

\subsection{Abelian Case}

In \secref{other examples} we have discussed other examples of
fractional Hecke eigensheaves. We now revisit them in the case when
the underlying curve $\CC$ is defined over $\Fq$. It is instructive to
look at the corresponding functions on the sets of $\Fq$-points of
$\Bun_G$ and to express them in terms of the ordinary Hecke
eigenfunctions, the way we did in the endoscopic example for $G=SL_2$
above (see formula \eqref{formula pm}).

Consider first the case when $G$ is a one-dimensional torus. The
corresponding moduli space, the Picard variety $\Pic$, breaks into
connected components $\Pic_n, n \in \Z$, and the Hecke eigensheaf
${\mc F}_\sigma$ corresponding to a one-dimensional ($\ell$-adic)
representation $\sigma$ of the Weil group $W_F$, breaks into a direct
sum
\begin{equation}    \label{dir sum Z1}
{\mc F}_\sigma = \bigoplus_{n \in \Z} {\mc F}_{\sigma,n},
\end{equation}
where ${\mc F}_{\sigma,n}$ is supported on $\Pic_n$. This is an
analogue of the decomposition \eqref{dir sum1}. Let $f_\sigma$
(resp., $f_{\sigma,n}$) be the function on $\Pic(\Fq)$ (resp.,
$\Pic_n(\Fq)$) corresponding to ${\mc F}_\sigma$ (resp., ${\mc
F}_{\sigma,n}$). Then we have
$$
f_\sigma = \sum_{n \in \Z} f_{\sigma,n}.
$$
This is analogous to formula \eqref{sum}. We now wish to express the
functions $f_{\sigma,n}$ in terms of (ordinary) Hecke eigenfunctions
$f_{\sigma'}$, similarly to formula \eqref{formula pm}.

This is achieved by a simple Fourier transform. Namely, for each
non-zero number $\ga \in \C^\times$ (in what follows we identify
$\overline{\mathbb Q}_\ell$ with $\C$) we define a one-dimensional
representation $\al_\ga$ of $W_F$ as the composition of the
homomorphism
\begin{equation}    \label{res}
\on{res}: W_F \to W_{\Fq} = \Z,
\end{equation}
obtained by restricting to the scalars $\Fq \subset F$, and the
homomorphism
$$
\Z \to \C^\times, \qquad 1 \mapsto \ga.
$$
Now let $\sigma_\ga = \sigma \otimes \al_\ga$ be the twist of $\sigma$
by $\al_\ga$. Then we have the following obvious formula
\begin{equation}    \label{f n}
f_{\sigma,n} = \frac{1}{2\pi i} \int_{|\ga| = 1} f_{\sigma_\ga}
\; \ga^{-n-1} d\ga,
\end{equation}
expressing the functions $f_{\sigma,n}$ as integrals of the ordinary
Hecke eigenfunctions corresponding to the twists $\sigma_\ga$ of
$\sigma$ by $\al_\ga, |\ga|=1$.

Formula \eqref{f n} is an analogue of formula \eqref{formula pm} which
we had in the endoscopic case, in the sense that in both cases the
functions satisfying the fractional Hecke property (that is,
$f_{\sigma,\pm}$ in the endoscopic case, and $f_{\sigma,n}, n \in \Z$,
in the abelian case) are expressed via Fourier transform of ordinary
Hecke eigenfunctions. The difference is that in the first case the
Fourier transform is performed with respect to the finite group
$\Z_2$, which is the group of automorphisms of an endoscopic
homomorphism $\sigma: W_F \to PGL_2$, whereas in the second case it is
performed with respect to a continuous group of automorphisms (or
rather, its compact form $U_1$). This is the reason why a finite sum
in \eqref{formula pm} is replaced by an integral in \eqref{f n}. We
will see other examples of this kind of Fourier transform with respect
to more general finite groups of automorphisms in \secref{more
general}.

In a similar way we can obtain functions satisfying the fractional
Hecke property associated to other types of local systems discussed in
\secref{other examples}: for more general tori, for local systems
whose group of automorphisms is the center of $\LG$, and for the
Eisenstein series. It would be interesting to find a direct
representation-theoretic interpretation of these functions and the
fractional Hecke property that they satisfy.

\subsection{The Iwahori Case}

We have discussed above the Hecke eigensheaves on the moduli stacks
$\Bun^{\OO(D)}_{SL_2}$ and the corresponding Hecke eigenfunctions.
However, in our most detailed example of $A$-branes corresponding to
the elliptic curves in \secref{genus one} we have considered a
slightly different moduli space corresponding to ramified Higgs
bundles. In this case the relevant moduli stack is
$\Bun^{\OO(D)}_{SL_2,I_p}$ which parametrizes rank two vector
bundles on $\CC$ with determinant $\OO(D)$ and a parabolic structure
at a fixed point $p$ of $\CC$ (that is, a choice of a line in the
fiber of the bundle at $p$). It is instructive to look at how the
story with $L$-packets discussed in \secref{inv vec} plays out in
this case.

Let $\CC$ be again defined over $\Fq$. Then the set of $\Fq$-points of
$\Bun^{\OO(D)}_{SL_2,I_p}$ is isomorphic to the double quotient
\begin{equation}    \label{Iw double}
SL_2(F) \bs SL_2({\mathbb A}_F)/\left( \prod_{x \neq p} K_x \times I_p
\right).
\end{equation}
Here $I_p = K'_p \cap K''_p$ is the Iwahori subgroup of $SL_2(F_p)$,
and $K_x = K''_x$ for $x \in D$, $K_x = K'_x$, otherwise. Let us
suppose that $p$ is a non-split point of $\CC$, with respect to the
unramified covering $\CC' \to \CC$ affiliated with an unramified
homomorphism $\sigma: W_F \to O_2$. Then the local $L$-packet
corresponding to $p$ and a homomorphism $\sigma: W_F \to PGL_2$
constructed as above consists of two irreducible representations,
$\pi'_p$ and $\pi''_p$, but now both $(\pi'_p)^{I_p}$ and
$(\pi''_p)^{I_p}$ are one-dimensional.

Let us fix the local factors $\pi_x, x \neq p$. Then we have two
non-isomorphic irreducible representations of $SL_2({\mathbb A}_F)$,
$$
\bigotimes_{x \neq p} \pi_x \otimes \pi'_p \qquad
\on{and} \qquad
\bigotimes_{x \neq p} \pi_x \otimes \pi''_p.
$$
According to \thmref{descr aut}, only one of them is automorphic;
that is, may be realized as a constituent of an appropriate space of
functions on $SL_2(F) \bs SL_2({\mathbb A}_F)$. However, their spaces of
invariants with respect to the subgroup
$$
\prod_{x \neq p} K_x \times I_p
$$
are both one-dimensional. Therefore no matter which one of them is
automorphic, we will have a one-dimensional space of Hecke
eigenfunctions on the double quotient \eqref{Iw double}. Thus, the
function on the set of $\Fq$-points of $\Bun^{\OO(D)}_{SL_2,I_p}$
associated to a regular Hecke eigensheaf will be non-zero. In the
same way as above, we then obtain that the functions $f^D_\pm$
associated to fractional Hecke eigensheaves are also non-zero in this
case. This constitutes an important difference between the double
quotients \eqref{first component} and \eqref{Iw double}.

\section{Other groups}    \label{other}

In this section we sketch a generalization of our results and
conjectures to the case of an arbitrary semi-simple simply-connected
Lie group $G$ (the latter assumption is not essential and is made to
simplify the exposition). Then $\LG$ is a semi-simple Lie group of
adjoint type.

\subsection{Overview}

Recall that we have two dual moduli spaces of Higgs bundles, ${\mc
M}_H(G)$ and ${\mc M}_H(\LG)$, and the corresponding dual Hitchin
fibrations \eqref{celp}. The geometric Langlands correspondence is
interpreted in \cite{KW} as the homological mirror symmetry between
these two moduli spaces that reduces to the fiberwise $T$--duality on
generic fibers which are smooth dual tori. Under this mirror symmetry,
the categories of $B$-branes on ${\mc M}_H(G)$ and $A$-branes on ${\mc
M}_H(\LG)$ are supposed to be equivalent. We are interested in the
$A$-branes on ${\mc M}_H(\LG)$ corresponding to the $B$-branes
supported at the orbifold singular points of ${\mc M}_H(\LG)$. Such a
singular point may be viewed as an $\LG$-local system $\E$ with a
non-trivial, but finite, group of automorphisms $\Gamma$. We call such
a local system ``elliptic endoscopic'', or simply ``endoscopic'', for
brevity. Then $\E$ is reduced to one or more of the dual (elliptic)
endoscopic subgroups $\LH \subset \LG$, which are defined as the
centralizers of non-trivial elements of $\Gamma$.

The category of $B$-branes (or, equivalently, coherent sheaves)
supported at such a point $\E$ is equivalent to the category
$\Rep(\Gamma)$ of representations of $\Gamma$. The objects of the
corresponding category of $A$-branes are supported on the Hitchin
fiber $\FF_b$ in ${\mc M}_H(G)$ dual to the Hitchin fiber $^L\neg
\FF_b$ in ${\mc M}_H(\LG)$ containing the point $\E$. Thus, $b$ is the
image of $\E \in {\mc M}_H(\LG)$ in the Hitchin base $\BB$. The
question that we take up in this section is to describe the categories
of $A$-branes corresponding to the elliptic endoscopic $\LG$-local
systems and their properties under the action of the 't Hooft/Hecke
operators.

In the previous sections we analyzed in detail the case of
$G=SL_2$. In this case, the only dual elliptic endoscopic subgroup is
$O_2 \subset SO_3 = \LG$, and generic local systems which are reduced
to $O_2$ have the automorphism group $\Gamma = \Z_2$. The
corresponding category of $B$-branes is equivalent to
$\Rep(\Z_2)$.\footnote{This equivalence is non-canonical due to the
twist by a gerbe described below in \secref{gerbes}, but we will
ignore this issue for now.} On the $A$-model side this corresponds to
the fact that the Hitchin fiber $\FF_b$ has two irreducible
components. Therefore the dual category of $A$-branes has two
irreducible objects supported on those components. These fractional
$A$-branes have additional parameters; namely, rank one unitary local
systems, which correspond to $O_2$-local systems. Thus, we obtain a
concrete realization of the transfer (also known as the functoriality
principle, or, in the physics interpretation, the domain wall
phenomenon) corresponding to the homomorphism $O_2 \to SO_3$ in the
geometric setting, as explained in \secref{transfer}. Finally, the two
fractional $A$-branes satisfy the fractional Hecke property, as
explained in \secref{Hecke}. In \secref{from} we have interpreted
these results for $A$-branes in terms of the corresponding ${\mc
D}$-modules on $\Bun_{SL_2}$ and automorphic functions when the curve
$\CC$ is defined over a finite field (see \eqref{passage}). In this
section we propose a generalization of this picture.

\subsection{Categories Of Branes Corresponding To The Endoscopic Local
  Systems}

Let $G$ be a semi-simple simply-connected complex Lie group, and $\LG$
its Langlands dual group (of adjoint type, with the trivial center).

An $\LG$-local system $\E$ will be called {\em elliptic endoscopic}, or
simply {\em endoscopic}, if its group of automorphisms is a
non-trivial finite group, which we will denote by $\Gamma$. The {\em
dual endoscopic groups} $\LH_s$ associated with such a local
system\footnote{Here we follow the tradition of calling $H_s$ rather
than $\LH_s$ the endoscopic groups.} are by definition the
centralizers of non-trivial elements $s \in \Gamma, s \neq 1$. The
structure group of $\E$ may be reduced to any of the $\LH_s$. As we
have already pointed out in \secref{transfer}, for any subgroup $\LH
\subset \LG$ we have a natural inclusion ${\mc M}_H(\LH) \subset {\mc
M}_H(\LG)$. Therefore we see that an endoscopic local system $\E$ lies
in the intersection of the images of ${\mc M}_H(\LH_s), s \in \Gamma,
s \neq 1$, in ${\mc M}_H(\LG)$.

Note that $\E \in {\mc M}_H(\LG)$ may also be viewed as a Higgs
bundle (in the complex structure $I$). Its group of automorphisms
as a Higgs bundle will also be $\Gamma$, and this Higgs bundle
will be reduced to the subgroups $\LH_s, s \in \Gamma$.

An endoscopic local system $\E$, viewed as a point of the moduli space
${\mc M}_H(\LG)$, is an orbifold point. Denote by $B\on{-branes}_\E$
the category of $B$-branes (coherent sheaves) supported at $\E$. This
category is equivalent to $\Rep(\Gamma)$, although there may not be a
canonical equivalence, as explained in \secref{gerbes} below (in
that case, let us choose such an equivalence). Then for each
representation $R$ of $\Gamma$ we have a $B$-brane ${\mc B}_R$ in the
category $B\on{-branes}_\E$. In the same way as in \secref{borbifold},
we find that the action of the Wilson operators $W_{V,p}$, where $V
\in \Rep(\LG)$ and $p \in \CC$, on these branes is given by the
formula
\begin{equation}    \label{Wp}
W_{V,p} \cdot {\mc B}_R = \sum_{R' \in \Irrep(\Gamma)} V(R')_{\E_p}
\otimes {\mc B}_{R' \otimes R}, \qquad R \in \Rep(\Gamma).
\end{equation}
Here we use the decomposition of $V$ with respect to the action of
$\LG \times \Gamma$,
\begin{equation}    \label{decomp0}
V = \bigoplus_{R' \in \Irrep(\Gamma)} V(R') \otimes R',
\end{equation}
where $\Irrep(\Gamma)$ is the set of equivalence classes of
irreducible representations of $\Gamma$. Also, for any representation
$U$ of $\LG$ we use the notation $$U_{\E_p} = \E_p
\underset{\LG}\times U,$$ where $\E_p$ is the $\LG$-torsor which is
the fiber of $\E$ at $p \in \CC$.

Formula \eqref{Wp} implies that the eigenbranes of the Wilson
operators are direct sums of copies of the $B$-brane corresponding to
the {\em regular representation}
$$
\on{Reg}(\Gamma) = \bigoplus_{R \in \Irrep(\LG)} R^* \otimes R
$$
of $\Gamma$, where $R^*$ is the dual of $R$. This $B$-brane is
\begin{equation}    \label{reg B}
{\mc B}_{\on{Reg}(\Gamma)} = \bigoplus_{R \in \Irrep(\LG)} R^* \otimes
{\mc B}_R.
\end{equation}
We have
\begin{equation}    \label{eig wilson}
W_{V,p} \cdot {\mc B}_{\on{Reg}(\Gamma)} = V_{\E_p} \otimes {\mc
  B}_{\on{Reg}(\Gamma)}.
\end{equation}

Now we consider the mirror dual category $A\on{-branes}_\E$ of
$A$-branes on $\MH(G)$. By analogy with the case of $G=SL_2$ that was
explained in detail in the previous sections, we expect that these
$A$-branes are supported on the Hitchin fiber $\FF_b$ in ${\mc
M}_H(G)$, where $b$ is the image of $\E \in {\mc M}_H(\LG)$ in the
Hitchin base $\BB$. In the case of $G=SL_2$ we saw that for an
endoscopic local system $\E$ different $A$-branes correspond to
different components of $\FF_b$. We would like to understand what
happens in general.

One complication is that it is quite possible that there are points
on the Hitchin fiber $^L\neg \FF_b$ which correspond to $\LG$-local
systems with infinite groups of automorphisms. An example is the
zero fiber $^L\neg \FF_0$ at $0 \in \BB$, which is the nilpotent
cone. Points of this fiber may have either finite or infinite groups
of automorphisms. For instance, for $\LG = SO_3$, it includes the
trivial local system, for which $\Gamma = SO_3$, as well as local
systems that reduce to the subgroup $\Z_2 \times \Z_2$, for which
$\Gamma = \Z_2 \times \Z_2$, and also irreducible local systems.

If the Hitchin fiber $^L\neg \FF_b$ contains local systems with
infinite automorphism groups, then we cannot expect that the structure
of the dual Hitchin fiber $\FF_b$ is controlled by endoscopic points
of $^L\neg \FF_b$. Indeed, the same $\FF_b$ would carry objects of the
categories of $A$-branes mirror dual to the categories of $B$-branes
supported at those local systems. Therefore the structure of $\FF_b$
should be more complicated in this case. We hope to discuss this
more general case elsewhere, but for now we will restrict ourselves to
the situation when infinite groups of automorphisms do not occur.

For $SL_n$, a useful condition that ensures that automorphism groups
are finite is that the spectral curve is reduced and irreducible.
This is equivalent to requiring that the characteristic polynomial
of the Higgs field $\varphi$ is irreducible (which in particular
implies that $\varphi(x)$ is regular semi-simple for generic $x\in
C$). This criterion has a simple analog\footnote{Regard $\varphi$ as
a matrix in the adjoint representation and set
$P(y)=\det(y-\varphi)$. For $G$ simple of rank $r$, we have
generically $P(y)=y^rQ(r)$ if $Q$ is simply-laced; the condition we
want is then that $Q$ is irreducible. For $G$ not simply-laced,
generically $P(y)=y^rQ(r)R(r)$, where $Q(r)$ and $R(r)$ are
contributions from long and short roots, respectively. In this case,
the condition is that $Q$ or equivalently $R$ should be irreducible.
Note that if $(E,\varphi)$ is a Higgs bundle such that $\varphi$
obeys this criterion of irreducibility, then $(E,\varphi)$ is
automatically stable; $E$ has no non-trivial $\varphi$-invariant
subsheaves and hence no destabilizing ones. So over this locus, one
can apply Hecke operators while working with stable Higgs bundles
only. } for any $G$ which, however, is stronger than needed to
ensure that automorphism groups are finite. A weaker criterion is
given by Ng\^{o} in \cite{Ngo1}, Definition 7.5. By \cite{Ngo1},
Corollaire 7.6, it is equivalent to the following. Let $\PP_b$ be
the generalized Prym variety associated to $b$, defined in
\cite{Ngo1}, Section 4 (in the case when $G=SL_2$ this is the Prym
variety of the spectral curve associated to $b$ discussed in
\secref{prym}). Then the condition is that the group $\pi_0(\PP_b)$ of
components of $\PP_b$ is finite. Following \cite{Ngo2}, we write
$\BB^{\on{ani}}$ for the corresponding locus in $\BB$.

Let us suppose then that $b \in \BB^{\on{ani}}$. In this case the
Hitchin fiber is reduced and contains an open dense subset which is a
torsor over the abelian group $\PP_b$ \cite{Ngo1}. Thus, the set of
irreducible components of $\FF_b$ is a torsor over the group
$\pi_0(\PP_b)$ of components of $\PP_b$, and all components of $\FF_b$
have multiplicity one. The group $\pi_0(\PP_b)$ is finite (by our
assumption that $b \in \BB^{\on{ani}}$) and {\em abelian}.

Now recall from \cite{KW} that for generic $b \in \BB$ the $A$-branes
corresponding to any rank one unitary local system on $\FF_b$ (which
is a smooth torus) are eigenbranes of the 't Hooft operators that are
dual to the Wilson operators acting on the $B$-branes (see
\secref{Hecke}). We conjecture that the same is true for any $b \in
\BB^{\on{ani}}$. In addition, we conjecture that each irreducible
object ${\mc A}_R$ of the category $A\on{-branes}_\E$ corresponding to
the irreducible object ${\mc B}_R$ of $B\on{-branes}_\E$ under the
equivalence $A\on{-branes}_\E \simeq B\on{-branes}_\E$ is supported on
a union of irreducible components of $\FF_b$.

In particular, suppose that $\Gamma = \Gamma_b$ is the largest
possible group of automorphisms among the local systems in the dual
Hitchin fiber $^L\neg \FF_b$. Then it is natural to expect that each
${\mc A}_R$ is supported on a particular irreducible component of
$\FF_b$ and that there is a bijection (perhaps, non-canonical, as for
$G=SL_2$) between $\Irrep(\Gamma_b)$ and the set of irreducible
components of $\FF_b$, and hence the set $\pi_0(\PP_b)$. But
$\pi_0(\PP_b)$ is an abelian group. This suggests that
$\Irrep(\Gamma_b)$ also has a natural abelian group structure and that
$\Gamma_b$ is in fact an abelian group that is dual to
$\pi_0(\PP_b)$. Thus, we arrive at the following conjecture.

\begin{conj}    \label{conj 1}
Let $\E$ be an elliptic endoscopic $\LG$-local system with the group
of automorphisms $\Gamma$ such that the image $b$ of $\E$ in $\BB$
lies in $\BB^{\on{ani}}$. Then the group $\Gamma$ is abelian and its
dual group $\wh\Gamma$ may be identified with a quotient of the
group $\pi_0(\PP_b)$ of components of the generalized Prym variety
$\PP_b$ corresponding to $b$.

Furthermore, if $\Gamma_b$ is the largest group of automorphisms of
the local systems in the dual Hitchin fiber $^L\neg \FF_b$, then
$\Gamma_b$ is isomorphic to the dual group of $\pi_0(\PP_b)$.
\end{conj}

The results presented in \secref{spectral} confirm this conjecture in
the case when $G=SL_2$ (see also the footnote on page 10).

\subsection{Fractional Eigenbranes And Eigensheaves}    \label{conjec}

Let $\E$ be an endoscopic $\LG$-local system. We have the mirror dual
categories $B\on{-branes}_\E$ and $A\on{-branes}_\E$ discussed in the
previous subsection. The former is equivalent to $\Rep(\Gamma)$ and
contains irreducible objects $\B_R$ attached to irreducible
representations $R$ of $\Gamma$. The corresponding $A$-branes are
denoted by $\A_R$. Therefore the $B$-brane \eqref{reg B} corresponds
to the $A$-brane
\begin{equation}    \label{reg A}
{\mc A}_{\on{Reg}(\Gamma)} = \bigoplus_{R \in \Irrep(\LG)} R^* \otimes
{\mc A}_R.
\end{equation}
In light of \conjref{conj 1}, in the case when the projection of
$\E$ onto $\BB$ is in $\BB^{\on{ani}}$ this decomposition should
reflect the decomposition of the Hitchin fiber $\FF_b$ into (unions
of) irreducible components.

Since the $B$-brane $\B_{\on{Reg}(\Gamma)}$ is an eigenbrane of the
Wilson operators (see formula \eqref{eig wilson}),
$\A_{\on{Reg}(\Gamma)}$ should be an eigenbrane of the 't Hooft
operators:
\begin{equation}    \label{reg eig}
T_{V,p} \cdot {\mc A}_{\on{Reg}(\Gamma)} = V_{\E_p} \otimes {\mc
  A}_{\on{Reg}(\Gamma)}.
\end{equation}
We call ${\mc A}_{\on{Reg}(\Gamma)}$ the {\em regular
eigenbrane}.

Furthermore, formula \eqref{Wp} for the action of the Wilson operators
on the $B$-branes $B_R, R \in \Irrep(\Gamma)$, implies the following
formula for the action of the 't Hooft operators $T_{V,p}, V \in
\Rep(\LG)$, on the corresponding $A$-branes $\A_R$:
\begin{equation}    \label{Tp}
T_{V,p} \cdot {\mc A}_R = \sum_{R' \in \Irrep(\Gamma)} V(R')_{\E_p}
\otimes {\mc A}_{R' \otimes R}, \qquad R \in \Rep(\Gamma).
\end{equation}
We call the $A$-branes $\A_R$ the {\em fractional eigenbranes}.

Note that we expect formulas \eqref{reg eig} and \eqref{Tp} to hold
regardless of whether the projection of $\E$ onto $\BB$ lies in
$\BB^{\on{ani}}$. If it is, then we expect $\Gamma$ to be abelian (see
\conjref{conj 1}); otherwise, it may well be non-abelian, as can be
seen from explicit examples.

As explained in \secref{dmodules}, we expect that to each $A$-brane on
$\MH(G)$ one may associate a ${\mc D}$-module on
$\Bun_G$. Furthermore, the properties of the $A$-branes under the
action of the 't Hooft operators should translate to similar
properties of the corresponding ${\mc D}$-modules under the action of
the Hecke operators. Therefore we predict that any Hecke eigensheaf on
$\Bun_G$ with the eigenvalue $\E$ (which is an endoscopic local
system) is a direct sum of copies of the following ${\mc D}$-module:
\begin{equation}    \label{reg F}
{\mc F}_{\on{Reg}(\Gamma)} = \bigoplus_{R \in \Irrep(\LG)} R^* \otimes
{\mc F}_R,
\end{equation}
satisfying the regular Hecke property
$$
H_V({\mc F}_{\on{Reg}(\Gamma)}) \simeq V_\E \boxtimes {\mc
F}_{\on{Reg}(\Gamma)}
$$
(see formula \eqref{al V}). Furthermore, we predict that its
constituents ${\mc F}_R$ are irreducible ${\mc D}$-modules on $\Bun_G$
which satisfy an analogue of formula \eqref{Tp},
\begin{equation}    \label{Res Isom1}
H_V({\mc F}_R) \simeq \sum_{R' \in \Irrep(\Gamma)} V(R')_{\E}
\boxtimes {\mc F}_{R' \otimes R}, \qquad R \in \Rep(\Gamma).
\end{equation}
This is a variant of formula \eqref{Res isom}, which means that the
${\mc D}$-modules ${\mc F}_R, R \in \Irrep(\Gamma)$, satisfy the {\em
fractional Hecke property} and hence are {\em fractional Hecke
eigensheaves}. In the case when $G=SL_2$ these are the ${\mc
D}$-modules ${\mc F}_\pm$ discussed in \secref{type 2}, and formula
\eqref{Res Isom1} coincides with formula \eqref{pm}.

Thus, we obtain a concrete conjecture about the structure of (regular)
Hecke eigensheaves corresponding to endoscopic local systems: they
split into direct sums of irreducible ${\mc D}$-modules satisfying the
fractional Hecke property \eqref{Res Isom1}. We have derived this
conjecture from the homological mirror symmetry of the dual Hitchin
fibrations, using the passage from $A$-branes to ${\mc
D}$-modules. Alternatively, one may look at it from the point of view
of a non-abelian version of the Fourier--Mukai transform
\cite{Laumon:fm,Roth}, suggested by A. Beilinson and V. Drinfeld,
which is supposed to be an equivalence of certain categories (whose
precise definition is presently unknown) of ${\mc O}$-modules on the
moduli stack of $\LG$-local systems on $\CC$ and ${\mc D}$-modules on
$\Bun_G$ (see, e.g., \cite{F:houches}, Section 6.2).

\subsection{Computations With Hecke Eigenfunctions}    \label{more
  general}

In the previous section we have made conjectures about the
structure of Hecke eigensheaves corresponding to endoscopic local
systems. So far, we have discussed local systems on a complex
curve $\CC$. However, we conjecture that the same pattern will
also hold if we consider instead $\ell$-adic local systems defined
on a curve over a finite field, or equivalently, $\ell$-adic
homomorphisms $\sigma: W_F \to \LG$, where $W_F$ is the Weil group
of the function field $F$ of this curve. Then the analogue of the
group $\Gamma$ is the centralizer of the image of $\sigma$, which
is traditionally denoted by $S_\sigma$. But here we will stick to
the same notation $\Gamma$. In this context there is a new
feature; namely, the Grothendieck {\em faisceaux--fonctions
dictionnaire}. This enables us to pass from Hecke eigensheaves
(which are now viewed as perverse sheaves on $\Bun_G$) to the
corresponding automorphic functions on a double quotient of the
ad\`elic group $G({\mathbb A}_F)$ and gives us an opportunity
to test our conjectures. We have already done this in the case
when $G=SL_2$ in \secref{splitting} and shown that such functions
indeed exist. Here we extend our analysis to the general situation
considered above.

\subsubsection{Abelian Case}

Suppose that we have an endoscopic homomorphism $\sigma: W_F \to
\LG$. This means that the centralizer $\Gamma$ of its image is a
non-trivial finite group (we are still under the assumption that $\LG$
is a semi-simple group of adjoint type). Let us start with the case
when $\Gamma$ is abelian. Recall that in a similar situation over $\C$
we expect to have the irreducible ${\mc D}$-modules labeled (perhaps
slightly non-canonically) by one-dimensional representations
(characters) of $\Gamma$. Thus, we have a ${\mc D}$-module ${\mc
F}_\chi$ for each $\chi \in \wh{\Gamma} = \Irrep(\Gamma)$. They have
to satisfy the fractional Hecke property \eqref{Res Isom1}. After
passing to curves over a finite field, we should have the
corresponding perverse sheaves on $\Bun_G$ satisfying the same
property, to which we associate automorphic functions $f_{\chi}$. The
fractional Hecke property for the sheaves translates to the following
equations on these functions:
\begin{equation}    \label{T}
T_{V,x} \cdot f_\chi = \sum_{\mu \in \wh\Gamma} a_{V,\mu,x} f_{\chi
  \cdot \mu}, \qquad \chi \in \wh\Gamma.
\end{equation}
Here $V$ is a representation of $\LG$, which decomposes as follows
\begin{equation}    \label{decomp}
V = \bigoplus_{\mu \in \wh\Gamma} V(\mu)
\end{equation}
under the action of $\Gamma$, and
\begin{equation}    \label{aV}
a_{V,\mu,x} = \on{Tr}(\sigma(\on{Fr}_x),V(\mu)),
\end{equation}
where $\sigma: W_F \to \LG$ is the object replacing the local system
$\E$ (since $\sigma$ lands in the centralizer of $\Gamma$, by our
assumption, the right hand side is well-defined). $T_{V,x}$ is a
classical Hecke operator corresponding to a closed point $x \in \CC$
(see \secref{recoll}).

We will now show that functions $f_\chi$ satisfying the fractional Hecke
property \eqref{T} do exist and may be obtained by a kind of Fourier
transform over $\Gamma$ from the ordinary Hecke eigenfunctions. This
will generalize the formulas obtained in \secref{splitting} in the
case of $G=SL_2$.

For simplicity we will assume here that $x$ is a closed point of $C$
with the residue field $\Fq$ equal to the ground field. For other
closed points the computation needs to be modified along the lines of
\secref{splitting}.

{}From \eqref{T}, by doing Fourier transform on $\Gamma$, we find the
eigenfunctions of $T_{V,x}$:
\begin{equation}    \label{eigenfunction}
\wh{f}_\ga = \sum_{\chi \in \wh\Gamma} \chi(\ga)
f_{\chi}, \qquad \ga \in \Gamma,
\end{equation}
with the eigenvalues
\begin{equation}    \label{A}
A_{x,\ga} = \sum_{\mu \in \wh\Gamma} \mu(\ga) a_{V,\mu,x}.
\end{equation}
In particular,
$$
A_{x,1} = \sum_{\mu \in \wh\Gamma} a_{V,\mu,x} =
\on{Tr}(\sigma(\on{Fr}_x),V),
$$
so
$$
\wh{f}_1 = \sum_{\chi \in \wh\Gamma} f_{\chi}
$$
is a Hecke eigenfunction corresponding to $\sigma$, as expected.

But what about the other functions $\wh{f}_\ga$ with $\ga \neq 1$?
We claim that they are also Hecke eigenfunctions, but corresponding to
other homomorphisms
$$
\sigma_\ga: W_F \to \LG.
$$
Namely, recall that we have a homomorphism
$$
\on{res}: W_F \to W_{\Fq} = \Z,
$$
by restricting to the scalars $\Fq \subset F$. Let
$$
\al_\ga: W_F \to \Gamma
$$
be the homomorphism given by the composition of $\on{res}$ and
the homomorphism $\Z \to \Gamma$ sending $1 \mapsto \gamma$.

Since $\Gamma$ centralizes the image of $\sigma$, the formula
$$
\sigma_\ga(g) = \sigma(g) \al_\ga(g)
$$
defines a homomorphism $W_F \to \LG$, for each $\ga \in \Gamma$.
We claim that
$$
A_{x,\ga} = \on{Tr}(\sigma_\ga(\on{Fr}_x),V),
$$
and so the function $\wh{f}_\ga$ is in fact a Hecke eigenfunction
corresponding to $\sigma_\ga: W_F \to \LG$!  (We recall that in the
above computation we have assumed that $x$ is an $\Fq$-point of $C$.
For other closed points we obtain the same result by applying the
analysis of \secref{splitting}.)

Now, making the inverse Fourier transform, we express the functions
$f_\chi$ corresponding to the sheaves ${\mc F}_\chi$ in terms of the
ordinary Hecke eigenfunctions:
\begin{equation}    \label{inverse}
f_\chi = \frac{1}{|\Gamma|} \sum_{\ga \in \Gamma} \chi(\ga)
\wh{f}_\ga.
\end{equation}
This generalizes our formula
$$
f_\pm = \frac{1}{2}(f_\sigma \pm f_{\sigma'})
$$
in the case of $SL_2$ (see formula \eqref{formula pm}).

The existence of the functions $f_\chi$ satisfying the function
theoretic analogue of the fractional Hecke property \eqref{T} (which we
had learned from the $A$-branes) provides a consistency check for
our predictions.

To summarize: we have found that the geometrically ``correct'' objects
(corresponding to irreducible perverse sheaves ${\mc F}_\chi$, the
``fractional'' Hecke eigensheaves) are not the ordinary Hecke
eigenfunctions, but their linear combinations (obtained by a finite
Fourier transform) corresponding to a collection of Galois
representations $\{ \sigma_\ga \}$ labeled by $\ga \in \Gamma$. These
are constructed as simple twists of $\sigma$. We note that the Fourier
transform in formula \eqref{inverse} is somewhat reminiscent of the
Fourier transform observed by Lusztig in the theory of character
sheaves \cite{Lu}.

\subsubsection{Non-abelian Case}

Let us consider now the case when $\Gamma$ is non-abelian. Over $\C$
this means, assuming \conjref{conj 1}, that the corresponding point of
the Hitchin moduli space $\MH(\LG)$ is not generically regular
semi-simple. In this case, we expect that some of the components of
the dual Hitchin fiber $\FF_b$ have multiplicities greater than $1$,
which should be equal to the dimensions of the corresponding
irreducible representations $R$ of $\Gamma$.

According to the conjectures of \secref{conjec}, transported to the
realm of curves over finite fields, we have irreducible perverse
sheaves ${\mc F}_R, R \in \Irrep(\Gamma)$, on $\Bun_G$ satisfying the
fractional Hecke property \eqref{Res Isom1}. Let $f_R, R \in
\Irrep(\Gamma)$, be the corresponding automorphic functions. We then
have an analogue of formula \eqref{T},
\begin{equation}    \label{T1}
T_{V,x} \cdot f_R = \sum_{R' \in \Irrep(\Gamma)} a_{V,R',x} f_{R
  \otimes R'}, \qquad R \in \Irrep(\Gamma),
\end{equation}
where we use the decomposition \eqref{decomp0}, and set
\begin{equation}    \label{aV1}
a_{V,R',x} = \on{Tr}(\sigma(\on{Fr}_x),V(R')).
\end{equation}
As before,
$$
R \otimes R' = \bigoplus_{R'' \in \Irrep(\Gamma)} (R'')^{\oplus
  m^{R,R'}_{R''}},
$$
and, by definition,
$$
f_{R \otimes R'} = \sum_{R'' \in \Irrep(\Gamma)} m^{R,R'}_{R''}
f_{R''}.
$$

Let us find the eigenfunctions of $T_{V,x}$. They are labeled by
the conjugacy classes $[\ga]$ in $\Gamma$ and are given by the formula
\begin{equation}    \label{eigenfunction1}
\wh{f}_{[\ga]} = \sum_{R \in \Irrep(\Gamma)} \on{Tr}([\ga],R)
f_R, \qquad [\ga] \in \Gamma.
\end{equation}
The corresponding eigenvalue is
\begin{equation}    \label{A1}
A_{x,[\ga]} = \sum_{R' \in \Irrep(\Gamma)} \on{Tr}([\ga],R')
a_{V,R',x}.
\end{equation}
In particular,
$$
A_{x,[1]} = \sum_{R' \in \Irrep(\Gamma)} \dim(R') a_{V,R',x} =
\on{Tr}(\sigma(\on{Fr}_x),V),
$$
and so
$$
\wh{f}_{[1]} = \sum_{R \in \Irrep(\Gamma)} \dim(R) f_{R}
$$
is a Hecke eigenfunction corresponding to $\sigma$, as
expected.

The other functions $\wh{f}_{[\ga]}$ with $[\ga] \neq [1]$ are also
Hecke eigenfunctions, but corresponding to other homomorphisms
$$
\sigma_{[\ga]}: W_F \to \LG,
$$
defined by the formula
$$
\sigma_{[\ga]}(g) = \sigma(g) \al_\ga(g),
$$
where
$$
\al_\ga: W_F \to \Gamma
$$
is the homomorphism given by the composition of $\on{res}$ (see
formula \eqref{res}) and the homomorphism $\Z \to \Gamma$ sending $1
\mapsto \gamma$, and $\gamma$ is an arbitrary element of
$[\ga]$. Clearly, the equivalence class of $\sigma_{[\ga]}$ depends
only on $[\ga]$ and not on the choice of $\gamma$.

The corresponding Hecke eigenvalue \eqref{A1} is
$$
A_{x,[\ga]} = \on{Tr}(\sigma_\ga(\on{Fr}_x),V),
$$
and so the function $\wh{f}_\ga$ is in fact a Hecke eigenfunction
corresponding to $\sigma_\ga: W_F \to \LG$, as desired.

Now we can express the functions $f_R$ corresponding to the sheaves
${\mc F}_R$ in terms of the ordinary Hecke eigenfunctions:
$$
f_\chi = \frac{1}{|\Gamma|} \sum_{[\ga] \in \Gamma} \on{Tr}([\ga],R)
\wh{f}_{[\ga]},
$$
as in the abelian case. Thus, functions satisfying the fractional
Hecke property \eqref{T1} do exist in the non-abelian case as well.
This provides a consistency check for our conjectures from
\secref{conjec}.

\section{Gerbes}    \label{gerbes}

The goal of this section is to elucidate a tricky point that arose in
\secref{braneduals}.  If $r$ is a point in $\MH(SO_3)$ corresponding
to an $SO_3$ local system with automorphism group $\Gamma=\Z_2$, then
there are two branes supported at $r$, namely $\B_+$ and $\B_-$.  The
corresponding fiber of the Hitchin fibration for $SL_2$ is the union
of two components $\FF_1$ and $\FF_2$, and accordingly in the dual
$A$-model there are two $A$-branes.  The central claim of this paper
is that the $A$-branes $\FF_1$ and $\FF_2$ are dual to the $B$-branes
$\B_+$ and $\B_-$.  But which of $\FF_1$ and $\FF_2$ is dual to $\B_+$
and which is dual to $\B_-$?  There is no natural way to decide, and
indeed, $\FF_1$ and $\FF_2$ are exchanged by the symmetry group
$Q=\Z_2\times \Z_2$.  By contrast, the branes $\B_+$ and $\B_-$ are
not equivalent; $\B_+$ corresponds to the trivial representation of
$\Z_2$, and $\B_-$ to a non-trivial representation.

\subsection{A Subtlety}\label{subtlety}

We claim that this question reflects a subtlety in the mirror symmetry
of $\MH(G)$ and $\MH(\LG)$ that has nothing to do with endoscopy.  Let
us start by asking whether the Hitchin fibration has a
section.\footnote{For a general study of this type of question in a
much more general context, see \cite{DG}.} As explained in
\cite{Hi,Hi4}, a section can always be constructed if one picks a spin
structure on the Riemann surface $C$, that is, a square root $K^{1/2}$
of the canonical bundle of $C$.  Let $E=K^{-1/2}\oplus K^{1/2}$.
Although $E$ is unstable, it is possible for a Higgs bundle
$(E,\varphi)$ to be stable.  This is so precisely if, modulo a
(unique) automorphism of $E$,
\begin{equation}\label{polho} \varphi=\begin{pmatrix}0 & 1 \\ w & 0
\end{pmatrix},\end{equation}
where $w$ is a quadratic differential.  Then $\det\,\varphi = -w$,
so the pair $(E,\varphi)$ maps, under the Hitchin fibration, to
the point in the base $\BB$ determined by the quadratic
differential $-w$.  Since every point in $\BB$ arises for a
unique $(E,\varphi)$ of this form, this gives a section of the
Hitchin fibration.  (Higgs bundles of this form are sometimes
called classical opers, reflecting their analogy to the opers of
\cite{BD}.)

The section obtained this way is not completely canonical, since it
depends on the choice of $K^{1/2}$.  However, the same construction
(replacing $E$ by $H={\rm ad}(E)$) makes sense for $SO_3$, and here
the choice of $K^{1/2}$ does not matter.  So for $SO_3$, the Hitchin
fibration has a natural section, but for $SL_2$, a choice of section
depends on a choice of $K^{1/2}$.

 This distinction is actually visible in the formulas of
\secref{dualbranes}.  The $SO_3$ moduli space is described in eqn.
(\ref{zot}); the Hitchin fibration has a natural section given by
$z=t=\infty$.  The $SL_2$ moduli space is described in
(\ref{elmy}), and there is no natural section of the Hitchin
fibration.

For any $G$, one repeats this construction, starting with a
$G$-bundle that is associated to $E$ via the choice of a principal
$\mathfrak{sl}_2$ subalgebra of $\mathfrak g$.  Higgs bundles
$(E,\varphi)$ with such an $E$ always give a section of the Hitchin
fibration. This section is independent of the choice of $K^{1/2}$ if
and only if the subgroup of $G$ that corresponds to the principal
$\mathfrak{sl}_2$ subalgebra is $SO_3$ rather than $SL_2$.  (For
example, if $G$ is of adjoint type, the Hitchin fibration has a
natural section, and similarly if $G$ is $SL_{2n+1}$.)

\subsection{A Conundrum For Mirror Symmetry}\label{conundrum}

These facts lead to a puzzle for the proper statement of the
mirror symmetry between $\MH(SO_3)$ and $\MH(SL_2)$.

Let $^L\neg \FF$ and $\FF$ be corresponding fibers of the Hitchin
fibrations of $SO_3$ and $SL_2$.  Naively speaking, they are dual
tori, meaning that $^L\neg \FF$ parametrizes flat unitary line bundles
over $\FF$, and vice-versa.  However, there is an immediate problem:
the space of flat unitary line bundles over a complex torus always
has a distinguished point, corresponding to the trivial line bundle.
So if $^L\neg \FF$ and $\FF$ are dual tori in this sense, then each must
have a distinguished point, associated with the trivial line bundle
on the other.  This contradicts the fact that, although $^L\neg \FF$
does have a distinguished point (its intersection with the section
of the Hitchin fibration described above), $\FF$ does not.

In fact, what is proved in \cite{HT} is not that $^L\neg \FF$ and $\FF$
are dual in this naive sense, but that the abelian variety $^L\neg
\FF$ is dual to the abelian variety for which  $\FF$ is a torsor.

For mirror symmetry between $SO_3$ and $SL_2$ theories, what this
means is that the $A$-model of $\MH(SL_2)$ is dual, not quite to the
$B$-model of $\MH(SO_3)$, but to a slightly twisted version of this
$B$-model.

A $B$-brane on a complex manifold $X$ is a coherent sheaf on $X$ (or
an object of the corresponding derived category, that is, a complex of
such sheaves, modulo a certain equivalence).  Now let $\cal G$ be a
$\C^\times$ gerbe on $X$.  For every $\cal G$, there is a $\cal
G$-twisted version of the category of $B$-branes; for the notion of a
$\cal G$-twisted sheaf, see for example Section 2.3 of \cite{DJ} or
Definition 2.1.2.2 of \cite{Li}.  A $\cal G$-twisted coherent sheaf of
rank 1 is a trivialization of $\cal G$.  A direct sum of $n$
trivializations is an example of a $\cal G$-twisted sheaf of rank $n$.

For our present problem, we need a $\C^\times$ gerbe over $\MH(SO_3)$
that is trivial but not canonically trivial.  In general, let $X$ be
any space and ${\cal L}\to X$ a complex line bundle.  Then there is a
canonically defined gerbe $\cal G$ whose (local) trivializations are
square roots of ${\cal L}$. More precisely, the objects of the
category associated to an open subset $U \subset X$ are pairs $({\mc
M},\alpha)$, where ${\mc M}$ is a line bundle on $U$ and $\alpha$ is
an isomorphism between ${\mc M}^2$ and ${\mc L}$. This gerbe $\cal G$
is trivial globally if and only if ${\cal L}$ has a global square
root; and it can be trivialized in a unique way (up to sign) if and
only if ${\cal L}$ has a unique global square root. ${\cal G}$ is a
$\C^\times$ gerbe, but actually it is associated with a $\Z_2$ gerbe
via the embedding $\{\pm 1\}\subset\C^\times$, so it has a natural
flat gerbe connection (in physical language, it is associated with a
flat $B$-field over $X$).

We apply this construction to the case that $X$ is $\MH(G)$ for
some reductive Lie group $G$ and $\cal L$ is $K_X$, the canonical
bundle of $X$.  Considering the square roots of $K_X$ gives a flat
$\C^\times$ gerbe $\cal G$ over $X$.  This gerbe is actually
trivial, because $K_X$ does have global square roots.  This point
is explained in great detail in Section 4 of \cite{BD}, where it
enters for reasons somewhat analogous to our present
considerations.  The construction is as follows.  Given any spin
structure ${\cal S}$ on $C$ (that is, a square root of the
canonical bundle of $C$) and a $G$-bundle $E\to C$, one considers
the Pfaffian line ${\cal L}_{\cal S,E}$ of the Dirac operator for
spin structure ${\cal S}$ twisted by ${\rm ad}(E)$. As $E$ varies,
${\cal L}_{\cal S,E}$ varies as the fiber of a line bundle  ${\cal
L}_{\cal S}\to\MH(G)$ that is a square root of the canonical
bundle of $\MH(G)$.  So the gerbe $\cal G$ has a natural
trivialization for each choice of spin structure ${\cal S}\to C$.

For $G$ simply-connected, $\MH(G)$ is also simply-connected, so the
square root of the canonical bundle obtained this way is independent
of the choice of $\cal S$, up to isomorphism.  This is so, for
example, for $G=SL_2$.  In this situation, ${\cal G}$ is canonically
trivial.

In general, if ${\cal S}$ and ${\cal S}'$ are two spin structures
on $C$,  the Pfaffian construction gives two square roots ${\cal
L}_{{\cal S}'}$ and ${\cal L}_{{\cal S}}$ of the canonical bundle
of $\MH(G)$.  They must differ by the tensor product with a line
bundle ${\cal U}({\cal S}',{\cal S})$ of order 2:
\begin{equation}{\cal L}_{{\cal S}'}= {\cal L}_{\cal S}\otimes
{\cal U}({\cal S}',{\cal S}).\end{equation} Obviously, for three
spin structures ${\cal S},{\cal S}',{\cal S}''\to C$, we have
\begin{equation}\label{ojo} {\cal U}({\cal S}'',{\cal S})= {\cal
    U}({\cal S}'',{\cal S}') \otimes{\cal U}({\cal S}',{\cal S}).
\end{equation}

A particularly simple example of this is for $G=SO_3$.  There is a
natural isomorphism between $H^1(C,\Z_2)$ and the orbifold
fundamental group $\pi_1(\MH(SO_3))$.  So there is a natural map
from a line bundle ${\cal V}\to C$ of order 2 to an orbifold line
bundle $\cal T(\cal V)\to \MH(SO_3)$ of order 2.  If ${\cal
S}',{\cal S}\to C$ are two spin structures, then ${\cal S}'={\cal
S}\otimes {\cal V}$ for some line bundle ${\cal V}$ of order 2,
and
\begin{equation}{\cal U}(\cal S',\cal S)={\cal T}({\cal V}),
\end{equation}
a statement that is clearly compatible with eqn.
(\ref{ojo}).

The precise statement of mirror symmetry between $\MH(SO_3)$ and
$\MH(SL_2)$ is as follows.  The $\cal G$-twisted $B$-model of $SO_3$
is dual to the $A$-model of $SL_2$.  And conversely, the $\cal
G$-twisted $A$-model of $\MH(SO_3)$ is dual to the $B$-model of
$\MH(SL_2)$. In general, the $A$-model can be twisted by a {\it flat}
complex gerbe, such as ${\cal G}$.

A similar twisting by a gerbe should also be implemented in the
non-abelian Fourier--Mukai transform formulation of the geometric
Langlands correspondence suggested by Beilinson and Drinfeld.

\subsection{Application}\label{application}

Now let us reconsider the question with which we began: the duality
between the $B$-branes $\cal B_+$ and $\cal B_-$ and the $A$-branes
$\FF_1$ and $\FF_2$.  Which $A$-brane corresponds to $\cal B_+$?

First of all, the problem only exists because the Hitchin fibration
for $SL_2$ has no natural section.  Given such a section $s$, we
would be able to pick out a distinguished component $\FF_1$ or $\FF_2$
of the special Hitchin fiber, namely the one that intersects $s$.

The resolution of the problem is that the duality involves, not the
ordinary $B$-model of $\MH(SO_3)$, but the ${\cal G}$-twisted
$B$-model.  In the ordinary $B$-model, as between $\cal B_+$ and
$\cal B_-$, there is a distinguished one, namely the one on which
the automorphism group $\Z_2$ of the $SO_3$ local system acts
trivially.  But in the $\cal G$-twisted $B$-model, things are
different.  Although the two $B$-branes transform oppositely under
$\Gamma=\Z_2$, to make sense of which transforms trivially and which
transforms non-trivially, we would have to first trivialize $\cal
G$.

$\cal G$, however, has no natural trivialization; rather, it has a
family of trivializations depending on the choice of a spin
structure on $C$.  Different trivializations would give different
interpretations of which of the two $B$-branes is invariant under
$\Gamma$ and which is not.

Indeed, two of these trivializations differ by tensoring by one of
the line bundles $\T({\cal V})$, for some ${\cal V}\in
H^1(C,\Z_2)$. If $r\in \MH(SO_3)$ is one of the orbifold singularities,
with symmetry group $\Gamma=\Z_2$, then for suitable ${\cal V}$,
the non-trivial element of $\Gamma$ acts on the fiber of $\T({\cal
V})$ as multiplication by $-1$.  When this is the case, the choice
of which $B$-brane is $\Gamma$-invariant and which is not is
reversed by tensoring by $\T({\cal V})$.

\subsection{Dual Symmetry Groups}\label{dualg}

For more understanding, we should describe another interpretation of
some of the facts that we have exploited.

The group $Q=H^1(C,\Z_2)$ acts on $\MH(SL_2)$ in a manner familiar
from \secref{symgroup}: an element of $Q$ corresponds to a line
bundle ${\cal V}\to C$ of order 2, which acts on a Higgs bundle
$(E,\varphi)$ by $E\to E\otimes {\cal V}$.  This gives a
geometrical action of $Q$ on $\MH(SL_2)$, preserving its
hyper-Kahler structure, so it gives an action of $Q$ on the
$A$-model and the $B$-model of $\MH(SL_2)$.

Dually, an isomorphic group must act\footnote{See Section 7.2 of
\cite{KW} for another explanation.}  on the $B$-model and the
$A$-model of the dual moduli space $\MH(SO_3)$. The key to this is
something we already exploited above: the natural correspondence
${\cal V}\to \T({\cal V})$ from a line bundle ${\cal V}\to C$ of order
2 to a line bundle $\T({\cal V})\to \MH(SO_3)$ of order 2. Since
$\T({\cal V})$ is a coherent sheaf, the tensor product with $\T({\cal
V})$ makes sense as a symmetry of the $B$-model; since it is a flat
line bundle, it makes sense as a symmetry of the $A$-model.

The duality between $\MH(SL_2)$ and $\MH(SO_3)$ exchanges the action
of $Q$ on the $A$- and $B$-models of $\MH(SL_2)$ coming from its
geometrical action on this space with the action of $Q$ on the
${\cal G}$-twisted $B$- and $A$- models of $\MH(SO_3)$ by tensor
product.

  $Q$
exchanges the two $A$-branes supported on the special Hitchin
fiber of $\MH(SL_2)$ for geometrical reasons.  It exchanges the
two $B$-branes supported at the orbifold singularity because the
non-trivial element of $\Gamma=\Z_2$ acts as $-1$ on the relevant
fiber of some of the line bundles $\T({\cal V})$.

\subsection{Relation To The Usual Statement Of Geometric
Langlands}\label{usual}

In this paper, in order to explore endoscopy, we have primarily
compared the $\cal G$-twisted $B$-model of $SO_3$ to the $A$-model
of $SL_2$.  However, it is also of interest to compare the
$B$-model of $SL_2$ to the $\cal G$-twisted $A$-model of $SO_3$.
What does the $\cal G$-twisting do in that context?

It is shown in Section 11 of \cite{KW} that, for any $G$, an
$A$-brane on $\MH(G)$ is equivalent to a twisted ${\cal D}$-module
on $\M(G)$, the moduli space of $G$-bundles.  Here\footnote{For
simplicity, we consider only the unramified case. Ramification
leads to a further twisting that does not affect our main claim.}
a twisted ${\cal D}$-module is a sheaf of modules for a sheaf of
algebras that we call ${\cal D}^*$, the differential operators
acting on a square root of the canonical bundle $K_\M$ of $\M$.
The sheaf of algebras ${\cal D}^*$ does not depend on a global
choice of square root of $K_\M$ (or even on the global existence
of such a square root, though it does in fact exist).

There is a slight tension between this and the usual statement of
geometric Langlands duality: the right hand side of the duality is
supposed to involve an ordinary ${\cal D}$-module, rather than a
twisted ${\cal D}$-module.

Now if there is a canonical global square root of $K_\M$, then
this distinction is inessential.  Given such a line bundle, we can
consider the differential operators that map $\CO_\M$ to
$K_\M^{1/2}$ and the sheaf of such operators is a ``bi-module''
for the pair of (sheaves of) algebras $\D$ and $\D^*$, i.e. the
ordinary and twisted differential operators.  This bi-module is a
``Morita equivalence bi-module'' that establishes an equivalence
between the categories of ordinary and twisted $\D$-modules.

For $SL_2$, there is such a canonical choice of $K_\M^{1/2}$, but
for $SO_3$ there is not.  So a $\D^*$-module on $\M(SO_3)$, such
as we would get from an $A$-brane of $\MH(SO_3)$, is not
canonically the same thing as an ordinary $\D$-module on
$\M(SO_3)$, such as we expect in the geometric Langlands program.

What reconciles the two viewpoints is that the $A$-model on
$\MH(SO_3)$ that arises in $S$-duality is $\cal G$-twisted. While
an ordinary $A$-brane maps to a  $\D^*$-module on $\M(SO_3)$, a
$\cal G$-twisted $A$-brane maps to a $\cal G$-twisted
$\D^*$-module on the same space. But a $\cal G$-twisted
$\D^*$-module on $\M(SO_3)$ maps canonically to an ordinary
$\D$-module on the same space.  The reason for this is that
although a square root $K_\M^{1/2}$ does not exist canonically as
a line bundle, it does exist canonically -- and tautologically --
as a trivialization of the gerbe $\cal G$.

We have described this for the dual pair of groups $SO_3$ and
$SL_2$ that has been our main example.  More generally, for any
reductive group $G$, one defines the gerbe $\cal G$ of square
roots of the canonical bundle.  Given any dual pair $^L\neg G$ and
$G$, the underlying gauge theory duality is an isomorphism between
the $\cal G$-twisted $B$-model of $^L\neg G$ and the $\cal
G$-twisted $A$-model of $G$ (and vice-versa).  In our example, we
have seen the twisting on only one side, simply because the gerbe
$\cal G$ is canonically trivial for $SL_2$.

\section{Appendix. $L$-packets for $SL_2$.}

In this Appendix we sketch a proof of \thmref{descr aut} using the
Whittaker functions. The construction of automorphic functions for
$GL_2$ via a Fourier transform of Whittaker functions was introduced
by H. Jacquet and R. Langlands \cite{JL} using a result of A. Weil
\cite{Weil}. In what follows we use the presentation and notation of
\cite{Drinfeld}.

We will fix a non-zero rational differential $\omega$ on $\CC$ and
denote by $\delta = \sum_x \delta_x [x]$ its divisor of zeros and
poles. Let $\psi: \Fq \to \ol{\mathbb Q}^\times_\ell$ be a non-trivial
additive character. It gives rise to a character $\Psi: {\mathbb
A}_F \to \ol{\mathbb Q}^\times_\ell$ defined by formula
$$
\Psi((f_x)) = \prod_{x \in \CC} \psi\left( \on{Tr}_{k_x/\Fq}(\on{Res}_x
(f_x \omega)) \right).
$$
By residue formula, its value on $F \subset {\mathbb A}_F$ is equal to
$1$. Hence $\Psi$ gives rise to a character of ${\mathbb A}_F/F$.

Let $B \subset GL_2$ be the Borel subgroup of upper triangular
matrices. Denote by $V$ the space of locally constant functions $f:
GL_2({\mathbb A}_F) \to \ol{\mathbb Q}_\ell$ such that

\begin{itemize}

\item[(1)] $f(\al x) = f(x)$ for all $x \in GL_2({\mathbb A}_F), \al
\in B(F)$;

\item[(2)] $\ds \int_{{\mathbb A}_F/F} f\left( \begin{pmatrix} 1 & z
\\ 0 & 1 \end{pmatrix} x \right) dz = 0$ for all $x \in GL_2({\mathbb
A}_F)$;

\item[(3)] $f(xu) = f(x)$ for all $x \in GL_2({\mathbb A}_F), u \in
GL_2(\OO_F)$.

\end{itemize}

\medskip

Denote by $W$ the space of locally constant functions $\phi:
GL_2({\mathbb A}_F) \to \ol{\mathbb Q}_\ell$ such that

\begin{itemize}

\item[(1')] $\ds \phi\left(\begin{pmatrix} a & z \\ 0 & a
\end{pmatrix} x \right) = \Psi(z) \phi(x)$ for all $x \in
GL_2({\mathbb A}_F), z \in {\mathbb A}_F, a \in F^\times$;

\item[(2')] $f(xu) = f(x)$ for all $x \in GL_2({\mathbb A}_F), u \in
GL_2(\OO_F)$.

\end{itemize}

\medskip

Define maps between these two spaces by the formulas
\begin{align}    \label{V to W}
f \in V &\mapsto \phi \in W; \qquad \phi(x) = \int_{{\mathbb A}_F/F}
f\left( \begin{pmatrix} 1 & z \\ 0 & 1 \end{pmatrix} x \right)
\Psi(-z) dz, \\ \label{W to V}
\phi \in W &\mapsto f \in V; \qquad f(x) = \sum_{a \in F^\times}
\phi\left( \begin{pmatrix} a & 0 \\ 0 & 1 \end{pmatrix} x \right).
\end{align}

According to \cite{JL}, these maps are mutually inverse
isomorphisms. In addition, they intertwine the (right) action of the
spherical Hecke algebra of $GL_2$ on both spaces. Therefore the spaces
of Hecke eigenfunctions in the two spaces are isomorphic. The
corresponding eigenvalues are determined by a collection of $GL_2$
conjugacy classes $\ga_x, x \in \CC$.

It is known that, for any such collection $(\ga_x)_{x \in \CC}$, the
space of Hecke eigenfunctions in $W$ is spanned by the so-called {\em
Whittaker function}. To write down an explicit formula for this
function \cite{Weil}, let us denote by $V_{m,k}, m \geq k$, the
irreducible representation of $GL_2$ with the highest weight $(m,k)$,
that is $\on{Sym}^{m-k} V \otimes (\det V)^{\otimes k}$, where $V$ is
the defining two-dimensional representation. Note that using the above
conditions (1') and (2'), a function $\phi \in W$ is uniquely
determined by its values on elements of the form
\begin{equation}    \label{special element}
\ds \left( \begin{pmatrix} t_x^{m_x} & 0 \\ 0 & t_x^{k_x} \end{pmatrix}
\right)_{x \in \CC} \; \in \; GL_2({\mathbb A}_F)
\end{equation}
(here, as before, $t_x$ denotes a uniformizer, that is, a formal
coordinate, at $x$). Given a collection $(\ga_x)_{x \in X}$ of
conjugacy classes, the corresponding Hecke eigenfunction in $W$ is
then determined (up to some inessential non-zero factors) by the
formula
\begin{equation}    \label{whit}
\phi \left( \begin{pmatrix} t_x^{m_x} & 0 \\ 0 & t_x^{k_x} \end{pmatrix}
\right) = \prod_{x \in \CC} \on{Tr}(\ga_x,V_{m_x+\delta_x,k_x})
\end{equation}
(if $m_x+\delta_x<k_x$ for some $x \in \CC$, then the right hand side
is equal to $0$, by definition).

A cuspidal automorphic Hecke eigenfunction for $GL_2$ with the
eigenvalues corresponding to a collection $(\ga_x)_{x \in X}$ as above
is, by definition, a Hecke eigenfunction that satisfies the above
conditions (2), (3), and a stronger condition than (1); namely, that
$f(\al x) = f(x)$ for all $x \in GL_2({\mathbb A}_F), \al \in
GL_2(F)$. In other words, it is a function on the double quotient
\begin{equation}    \label{gl2}
\Bun_{GL_2}(\Fq) = GL_2(F) \bs GL_2({\mathbb A}_F)/GL_2(\OO_F).
\end{equation}
The above results that imply that, if it exists, this function must be
obtained by applying the transform \eqref{W to V} to the Whittaker
function with the eigenvalues $(\ga_x)_{x \in \CC}$. In particular, if
exists, it is unique up to a scalar multiple.

According to the results of Drinfeld \cite{Dr,Drinfeld}, if
$\wt\sigma: W_F \to GL_2$ is irreducible and unramified, then the
vector space of Hecke eigenfunction on \eqref{gl2} with respect to
$\wt\sigma$ is one-dimensional and consists of cuspidal
functions. Therefore a generator of this space is given by the
operator \eqref{W to V} applied to the Whittaker function with the
eigenvalues $(\wt\sigma(\on{Fr}_x))$. Let us denote this automorphic
function by $f_{\wt\sigma}$ and the corresponding Whittaker function
by $W_{\wt\sigma}$.

Now we switch from $GL_2$ to $SL_2$. Recall that we would like to find
the dimension of the space of Hecke eigenfunctions on
$\Bun^{\OO(D)}_{SL_2}(\Fq)$, which is the double quotient \eqref{first
component} (here, as before, $D$ is a finite subset of $\CC$, which we
view as an effective divisor). We embed this component into the double
quotient \eqref{gl2} by sending $(g_x) \mapsto (\ol{g}_x)$, where
$$
\ol{g}_x = g_x \begin{pmatrix} t_x & 0 \\ 0 & 1 \end{pmatrix},
$$
if $x \in D$ and $\ol{g}_x = g_x$, otherwise. We will assume that
$\sigma: W_F \to O_2 \subset PGL_2$ may be lifted to a homomorphism
$\wt\sigma: W_F \to GL_2$. It is easy to see that the space of Hecke
eigenfunctions on $\Bun^{\OO(D)}_{SL_2}(\Fq)$ with respect to $\sigma$
is equal to the restriction of the space of Hecke eigenfunctions on
\eqref{gl2} to the image of this embedding with respect to
$\wt\sigma$.

Observe that for generic $\sigma$ of the above form the representation
$\wt\sigma$ will be irreducible. Thus, to determine whether the space
of Hecke eigenfunctions on $\Bun^{\OO(D)}_{SL_2}(\Fq)$ with respect to
$\sigma$ is zero- or one-dimensional, we need to determine whether the
restriction of the function $f_{\wt\sigma}$ constructed above to the
image of $\Bun^{\OO(D)}_{SL_2}(\Fq)$ is zero or not. This, in turn, is
determined by whether the Whittaker function $W_{\wt\sigma}$ is equal
to zero on all elements of the form \eqref{special element}, where the
divisor
$$
\sum_{x \in \CC} (m_x+k_x) [x]
$$
is linearly equivalent to $D$.

Let $\al: W_F \to \Z_2$ be the quadratic character obtained as the
composition of $\sigma$ and the homomorphism $O_2 \to \Z_2$. It
corresponds to a quadratic extension $E/F$, which we have called in
\secref{rel pac} affiliated with $\sigma$. We have
$$
\alpha(\on{Fr}_x) = \begin{cases} 1, & x \on{ \; is \; split \; in \;
  } E \\ -1, & x \on{ \; is \; non-split \; in \; } E. \end{cases}
$$
For any divisor $M = \sum_{x \in \CC} M_x [x]$ on $\CC$ let us set
$$
\langle M \rangle \; = \sum_{x \; \on{non-split \; in} E} M_x,
$$

Recall that we have the divisor $\delta$ of the differential $\omega$
used in the definition of the character $\Psi$, and we have the
divisor $D = \sum_x D_x [x]$, with $D_x=0$ or $1$.

We claim that the values of the Whittaker function $W_{\wt\sigma}$ are
zero on all elements of the form \eqref{special element}, with
$$
\sum_{x \in \CC} (m_x+k_x) [x]
$$
linearly equivalent to $D$, if and only if $\langle \delta+D \rangle$
is odd. Indeed, this set contains the element $(g_x)$, where
$$
g_x = \begin{pmatrix} t_x & 0 \\ 0 & 1 \end{pmatrix}, \qquad x \in D,
$$
and $g_x = 1$, otherwise. According to formula \eqref{whit}, the value
of $W_{\wt\sigma}$ on this element is equal to
$$
\prod_{x \in \CC} \on{Tr}(\wt\sigma(\on{Fr}_x),V_{\delta_x+D_x,0})
$$
If $\langle \delta+D \rangle$ is odd, then there is at least one
non-split point $y$ such that $\delta_y+D_y$ is odd. But since
$\wt\sigma(\on{Fr}_y)$ is conjugate to a scalar multiple of the matrix
\eqref{trans} in this case, we find that
$$
\on{Tr}(\wt\sigma(\on{Fr}_y),V_{\delta_y+D_y,0}) = 0.
$$
Therefore $W_{\wt\sigma}$ is equal to $0$ at this point. All other
points that we need to check have the form \eqref{special element},
with
$$
\sum_{x \in \CC} (m_x+k_x) [x] = D + (F),
$$
where
$$
(F) = \sum_{x \in \CC} n_x [x]
$$
is the divisor of zeros and poles of a rational function $F$ on
$\CC$. But it follows from the abelian class field theory (see, e.g.,
\cite{Serre}) that for any rational function $F$ on $\CC$ we have
\begin{equation}    \label{reciprocity}
\prod_{x \in \CC} \al(\on{Fr}_x)^{n_x} = \prod_{x \on{non-split \; in}
E} (-1)^{n_x} = 1.
\end{equation}

Therefore for any element \eqref{special element} satisfying the above
conditions there again exists at least one non-split point $z \in
\CC$ such that $m_z + \delta_z - k_z$ is odd. Formula \eqref{whit}
then shows that the value of $W_{\wt\sigma}$ on all such elements is
$0$. Therefore we find that in this case the restriction of
$f_{\wt\sigma}$ to the image of $\Bun^{\OO(D)}_{SL_2}(\Fq)$ in
\eqref{gl2} is equal to $0$. Hence there are no non-zero Hecke
eigenfunctions on $\Bun^{\OO(D)}_{SL_2}(\Fq)$ with respect to
$\sigma$, which is what we wanted to show.

On the other hand, if $\langle \delta+D \rangle$ is even, then it is
easy to see that the restriction of $f_{\wt\sigma}$ to the image of
$\Bun^{\OO(D)}_{SL_2}(\Fq)$ is non-zero. Hence the corresponding space
of Hecke eigenfunctions on $\Bun^{\OO(D)}_{SL_2}(\Fq)$ is
one-dimensional.

To complete the proof of \thmref{descr aut}, it remains to observe
that
$$
\langle \delta \rangle = \sum_{x \; \on{non-split \; in} E} \delta_x
$$
is always even (see Theorem 13 in \cite{Weil:arith},
Ch. XIII, Section 12).\footnote{We thank B. Poonen for pointing out
this reference.} Hence the parity of $\langle \delta+D \rangle$ is the
same as that of $\langle D \rangle$, which is the same as the number
of elements in the subset $S \subset D$ consisting of all non-split
points of $D$.

\end{document}